\documentclass[UTF8,10pt]{article}

\usepackage[left=1.91cm,right=1.91cm,top=2.54cm,bottom=2.54cm]{geometry}
\usepackage{amsfonts}
	\usepackage{amsmath}
	\usepackage{amssymb}
	\usepackage{pgfplots}
	\usepackage{tikz}
	\usetikzlibrary{arrows}
	\usepackage{titlesec}
	\usepackage{bm}
	\usepackage{appendix}
	\usepackage{tikz-cd}
	\usepackage{mathtools}
	\usepackage{amsthm}
	\usepackage{setspace}
	\usepackage{enumitem}
	\setenumerate[1]{itemsep=0pt,partopsep=0pt,parsep=\parskip,topsep=5pt}
	\setitemize[1]{itemsep=0pt,partopsep=0pt,parsep=\parskip,topsep=5pt}
	\setdescription{itemsep=0pt,partopsep=0pt,parsep=\parskip,topsep=5pt}
	\usepackage[colorlinks,linkcolor=black,anchorcolor=black,citecolor=black]{hyperref}
	\usepackage[numbers]{natbib}
	\usepackage{cases}
	\usepackage{fancyhdr}

	\usepackage{txfonts}	
	\usepackage{verbatim}

	\theoremstyle{definition}
	\newtheorem{thm}{Theorem}[section]

	\newtheorem{defn}{Definition}[section]
	\newtheorem{lem}[thm]{Lemma}
	\newtheorem{prop}[thm]{Proposition}
	
	\newtheorem{cor}[thm]{Corollary}
	
	\newtheorem{rmk}{Remark}[section]

	\newcommand{\R}{\mathbb{R}}

	\newcommand{\N}{\mathbb{N}}
	
	\newcommand{\T}{\mathbb{T}}
	\newcommand{\TT}{\mathcal{T}}
	\newcommand{\TP}{\overline{\p}{}}
	
	\newcommand{\TL}{\overline{\Delta}{}}
	
	\newcommand{\wto}{\rightharpoonup}

	\newcommand{\ST}{\text{ ST}}
	\newcommand{\RT}{\text{ RT}}
	\newcommand{\VS}{\text{ VS}}
	\newcommand{\STr}{\mathring{\text{ ST}}}
	\newcommand{\RTr}{\mathring{\text{ RT}}}
	\newcommand{\VSr}{\mathring{\text{ VS}}}
	\newcommand{\ZBr}{\mathring{ZB}}
	\newcommand{\Zr}{\mathring{Z}}
	
	\newcommand{\bd}[1]{\mathbf{#1}} 
	\newcommand{\RR}{\mathcal{R}}      
	\newcommand{\Om}{\Omega}
	\newcommand{\om}{\omega}
	\newcommand{\q}{\quad}
	\newcommand{\p}{\partial}
	\newcommand{\dd}{\mathfrak{D}}

	\newcommand{\nab}{\nabla}
	\newcommand{\Jr}{\mathring{J}}
	
	\newcommand{\lap}{\Delta}
	\newcommand{\vp}{\varphi}
	
	\newcommand{\no}{\nonumber}

	\newcommand{\eql}{\stackrel{L}{=}}
	\newcommand{\cnab}{\overline{\nab}}
	
	\newcommand{\jp}{\lee\TP\ree}
	
	\newcommand{\dx}{\,\mathrm{d}x}
	\newcommand{\dy}{\,\mathrm{d}y}
	\newcommand{\dz}{\,\mathrm{d}z}
	\newcommand{\dt}{\,\mathrm{d}t}
	\newcommand{\dtau}{\,\mathrm{d}\tau}

	\newcommand{\lee}{\langle}
	\newcommand{\ree}{\rangle}

	\newcommand{\kk}{\kappa}
	\newcommand{\Er}{\mathring{E}}
	
	\newcommand{\Kr}{\mathring{K}}

	\newcommand{\rr}{\mathfrak{R}}
	\newcommand{\Rr}{\mathring{\mathcal{R}}}
	\newcommand{\cc}{\mathfrak{C}}
	\newcommand{\ccr}{\mathring{\mathfrak{C}}}
	\newcommand{\ddr}{\mathring{\mathfrak{D}}}
	
	\newcommand{\ccb}{\mathfrak{B}}
	\newcommand{\ccbr}{\mathring{\mathfrak{B}}}

	\newcommand{\sss}{\mathcal{S}}
	\newcommand{\h}{\mathcal{H}}
	\newcommand{\WW}{\mathcal{W}}
	\newcommand{\VV}{\mathbf{V}}
	\newcommand{\BB}{\mathbf{B}}
	\newcommand{\PP}{\mathbf{P}}
	
	\newcommand{\FS}{\mathfrak{S}}

	\newcommand{\EW}{\widetilde{E}}

	\newcommand{\fs}{\mathfrak{s}}
	\newcommand{\BS}{\mathbf{S}}

	\newcommand{\FF}{\mathbf{F}}
	
	\newcommand{\WWr}{\mathring{\mathcal{W}}}
	\newcommand{\VVr}{\mathring{\mathbf{V}}}
	\newcommand{\BBr}{\mathring{\mathbf{B}}}
	\newcommand{\QQr}{\mathring{\mathbf{Q}}}
	
	\newcommand{\BSr}{\mathring{\mathbf{S}}}

	\newcommand{\fr}{\mathring{f}}

	\newcommand{\FFr}{\mathring{\mathbf{F}}}
	\newcommand{\Ur}{\mathring{U}}
	
	\newcommand{\Ir}{\mathring{I}}

	\newcommand{\MH}{\mathcal{M}}
	\newcommand{\NH}{\mathcal{N}}

	\newcommand{\io}{\int_{\Omega}}
	\newcommand{\iop}{\int_{\Omega^+}}
	\newcommand{\iom}{\int_{\Omega^-}}
	\newcommand{\iopm}{\int_{\Omega^\pm}}
	\newcommand{\is}{\int_{\Sigma}}

	\newcommand{\ddt}{\frac{\mathrm{d}}{\mathrm{d}t}}
	\numberwithin{equation}{section}
	
	\newcommand{\eps}{\varepsilon}
	\providecommand{\jump}[1]{\left\llbracket #1 \right\rrbracket }
	\providecommand{\len}[1]{\lee #1 \ree }
	\providecommand{\ino}[1]{\left\| #1 \right\| }
	\providecommand{\bno}[1]{\left| #1 \right| }
	\usepackage{xcolor}

	\newcommand{\lam}{\lambda}
	
	\newcommand{\sh}{\sharp}

	\newcommand{\pk}{\widetilde{\varphi}}
	\newcommand{\pr}{\mathring{\varphi}}
	\newcommand{\pd}{\dot{\varphi}}
	
	\newcommand{\pkr}{\mathring{\pk}}
	\newcommand{\vr}{\mathring{v}}
	\newcommand{\rhor}{\mathring{\rho}}
	
	\newcommand{\psr}{\mathring{\psi}}
	\newcommand{\psd}{\dot{\psi}}

	\newcommand{\nnr}{{[n+1]}}
	\newcommand{\nnn}{{[n]}}
	\newcommand{\nnl}{{[n-1]}}
	\newcommand{\nnll}{{[n-2]}}	
	
	\newcommand{\mmm}{{[m]}}

	\newcommand{\npr}{\mathring{N}}

	\newcommand{\NN}{\mathbf{N}}

	\newcommand{\Npr}{\mathring{\NN}}
	\newcommand{\Npd}{\dot{\NN}}
	
	\newcommand{\QQ}{\mathbf{Q}}
	\newcommand{\vb}{\bar{v}}
	\newcommand{\wb}{\bar{w}}
	\newcommand{\vbr}{\mathring{\vb}}
	\newcommand{\ff}{\mathcal{F}}
	\newcommand{\ffp}{\mathcal{F}_p}
	\newcommand{\ffpm}{\mathcal{F}_p^\pm}
	\newcommand{\ffr}{\mathring{\mathbf{f}}}
	\newcommand{\ffpr}{\mathring{\mathcal{F}}_p}
	\newcommand{\ffpmr}{\mathring{\mathcal{F}}_p^\pm}

	\newcommand{\gapm}{{\gamma,\pm}}
	
	\newcommand{\qr}{\mathring{q}}
	\newcommand{\prr}{\mathring{p}}
	\newcommand{\sr}{\mathring{S}}
	\newcommand{\es}{{\varepsilon, \sigma}}

	\newcommand{\pp}{\p^{\varphi}}
	\newcommand{\nabp}{\nab^{\varphi}}
	\newcommand{\nabper}{\nab^{\varphi,\perp}}

	\newcommand{\bp}{(b\cdot\nab^{\varphi})}
	\newcommand{\bpl}{(b^-\cdot\nab^{\varphi})}
	
	\newcommand{\cb}{\Lambda_{\eps,b}}
	\newcommand{\cbi}{\Lambda_{\eps,b_0}}
	\newcommand{\bc}{\bar{b}}
	\newcommand{\bcp}{(\bar{b}\cdot\cnab)}
	\newcommand{\bcpu}{(\bar{b}^+\cdot\cnab)}

	\newcommand{\mb}{\bm{b}}

	\newcommand{\mbc}{\bar{\bm{b}}}

	\newcommand{\mbr}{\mathring{\bm{b}}}
	
	\newcommand{\mbcr}{\mathring{\bar{\bm{b}}}}
	\newcommand{\br}{\mathring{b}}

	\newcommand{\bpm}{(b^\pm\cdot\nab^{\varphi})}

	\newcommand{\mbpr}{(\mbr\cdot\nab^{\pr})}
	\newcommand{\bpmr}{(\br^\pm\cdot\nab^{\pr})}
	\newcommand{\mbpmr}{(\mbr^\pm\cdot\nab^{\pr})}
	
	\newcommand{\ppr}{\p^{\pr}}
	
	\newcommand{\nabpr}{\nab^{\pr}}
	
	\newcommand{\nabpkr}{\nab^{\pkr}}
	\newcommand{\nabpp}{\nab^{\varphi_0}}
	\newcommand{\lapi}{\lap^{\vp_0}}
	\newcommand{\lapp}{\lap^{\vp}}

	\newcommand{\Dtb}{\overline{D_t}}
	\newcommand{\Dtbu}{\overline{D_t^+}}
	\newcommand{\Dtbl}{\overline{D_t^-}}
	\newcommand{\Dtbpm}{\overline{D_t^\pm}}

	\newcommand{\Dtp}{D_t^{\varphi}}
	
	\newcommand{\Dtpl}{D_t^{\vp,-}}
	\newcommand{\Dtpm}{D_t^{\vp\pm}}

	\newcommand{\Dtpmr}{D_t^{\pr\pm}}
	
	\newcommand{\Dtpr}{D_t^{\pr}}
	\newcommand{\Dtpkr}{D_t^{\pkr}}

	\newcommand{\dvt}{\,\mathrm{d}\mathcal{V}_t}
	\newcommand{\dvr}{\,\mathrm{d}\mathring{\mathcal{V}}_t}

\begin{document}
\bibliographystyle{plain}
\title{\textbf{Well-posedness and Incompressible Limit of Current-Vortex Sheets with Surface Tension in Compressible Ideal MHD}}
\author{
{\sc Junyan Zhang}\thanks{School of Mathematical Sciences, University of Science and Technology of China. 96 Jinzhai Road, Hefei, Anhui 230026, China. Email: \texttt{yx3x@ustc.edu.cn}.}
 }
\date{\today}
\maketitle

\setcounter{tocdepth}{1}

\begin{abstract}
	Current-vortex sheet is one of the characteristic discontinuities in ideal compressible magnetohydrodynamics (MHD). The motion of current-vortex sheets is described by a free-interface problem of two-phase MHD flows with magnetic fields tangential to the interface. This model has been widely used in both solar physics and controlled nuclear fusion.  This paper is the first part of the two-paper sequence, which aims to present a comprehensive study for compressible current-vortex sheets with or without surface tension. In this paper, we prove the local well-posedness and the incompressible limit of current-vortex sheets with surface tension. The key observation is a hidden structure of Lorentz force in the vorticity analysis which motivates us to establish the uniform estimates in anisotropic-type Sobolev spaces with weights of Mach number determined by the number of tangential derivatives. Besides, our framework of iteration and approximation to prove the local existence of vortex-sheet problems does not rely on Nash-Moser iteration. Furthermore, the local existence of current-vortex sheets without surface tension can be proved by taking zero-surface-tension limit under certain stability conditions, which is established in \cite{Zhang2023CMHDVS2} (the second part of the two-paper sequence).
\end{abstract}



\tableofcontents

\section{Introduction}
The equations of compressible ideal magnetohydrodynamics (MHD) in $\R^d~(d=2,3)$ can be written in the following form
\begin{equation}
\begin{cases}
\varrho D_t u=B\cdot\nabla B-\nabla Q,~~Q:=P+\frac{1}{2}|B|^2, \\
D_t\varrho+\varrho\nab\cdot u=0,\\
D_t B=B\cdot\nabla u-B\nab\cdot u,\\
\nab\cdot B=0,\\
D_t \mathfrak{s}=0.
\end{cases}\label{CMHD}
\end{equation}
Here $\nabla:=(\p_{x_1},\cdots,\p_{x_d})$ is the standard spatial derivative and $\nab\cdot X:=\p_{x_i}X^i$ is the divergence of the vector field $X$. Throughout this paper, we use Einstein summation convenction, that is, repeated indices represent taking summation over these indices. $D_t:=\p_t+u\cdot\nabla$ is the material derivative. The fluid velocity, the magnetic field, the fluid density, the fluid pressure and the entropy are denoted by $u=(u_1,\cdots,u_d)$, $B=(B_1,\cdots,B_d)$, $\varrho$, $P$ and $\mathfrak{s}$ respectively. The quantity $Q:=P+\frac12|B|^2$ is the total pressure. Note that the fourth equation in \eqref{CMHD} is just an initial constraint instead of an independent equation. The last equation of \eqref{CMHD} is derived from the equation of total energy and Gibbs relation and we refer to \cite[Ch. 4.3]{MHDphy} for more details. To close system \eqref{CMHD}, we need to introduce the equation of state
\begin{equation}\label{eos}
P=P(\varrho,\mathfrak{s})\text{ satisfying }\frac{\p P}{\p\varrho}>0.
\end{equation} We also need to assume $\varrho\geq\bar{\rho_0}>0$ for some constant $\bar{\rho_0}>0$, which together with $\frac{\p P}{\p\varrho}>0$ guarantees the hyperbolicity of system \eqref{CMHD}.  For detailed requirement on the equation of state, we refer to Section \ref{sect eos}.

\subsection{Mathematical formulation of current-vortex sheets}

Let $H>10$ be a given real number, $x=(x_1,\cdots,x_d)$ and $x':=(x_1,\cdots,x_{d-1})$ and the space dimension $d=2,3$. We define the regions $\Om^+(t):=\{x\in\T^{d-1}\times\R:\psi(t,x')<x_d<H\}$, $\Om^-(t):=\{x\in\T^{d-1}\times\R:-H<x_d<\psi(t,x')\}$ and the moving interface $\Sigma(t):=\{x\in\T^{d-1}\times\R:x_d=\psi(t,x')\}$ between $\Om^+(t)$ and $\Om^-(t)$. We assume $U^{\pm}=(u^\pm,B^\pm,P^\pm,\mathfrak{s}^\pm)^\top$ to be a smooth solution to \eqref{CMHD} in $\Om^\pm(t)$ respectively. We say $\Sigma(t)$ is a \textit{current-vortex sheet} (or an \textit{MHD tangential discontinuity}) if the following conditions are satisfied:
\begin{equation}\label{CMHDjump}
\jump{Q}=\sigma\mathcal{H},~~B^\pm \cdot N=0,~~\p_t\psi=u^\pm \cdot N~~\text{ on }\Sigma(t),
\end{equation}where $N:=(-\p_1\psi,\cdots,-\p_{d-1}\psi, 1)^\top$ is the normal vector to $\Sigma(t)$ (pointing towards $\Om^+(t)$), $\sigma\geq 0$ is the constant coefficient of surface tension and the quantity $\mathcal{H}:=\cnab\cdot\left(\frac{\cnab\psi}{\sqrt{1+|\cnab\psi|^2}}\right)$ is twice the mean curvature of $\Sigma(t)$ with $\cnab=(\p_1,\cdots,\p_{d-1})$. The jump of a function $f$ on $\Sigma(t)$ is denoted by $\jump{f}:=f^+|_{\Sigma(t)}-f^-|_{\Sigma(t)}$ with $f^\pm:=f|_{\Om^\pm(t)}$. The first condition shows that the jump of total pressure is balanced by surface tension. The second condition shows that both plasmas are perfect conductors. The third condition shows that there is no mass flow across the interface and thus the two plasmas are physically contact and mutually impermeable. These conditions on $\Sigma(t)$ are given by the Rankine-Hugoniot conditions for ideal compressible MHD when the magnetic fields are tangential to the interface, and we refer to Trakhinin-Wang \cite[Appendix A]{TW2021MHDCDST} for detailed derivation. Besides, we impose the slip boundary conditions on the rigid boundaries $\Sigma^\pm:=\T^{d-1}\times\{\pm H\}$
\begin{equation}\label{CMHDslip}
u^\pm_d=B^\pm_d=0~~\text{ on }\Sigma^\pm.
\end{equation}

\begin{rmk} [Initial constraints for the magnetic field] The conditions $\nab\cdot B^\pm=0$ in $\Om^\pm(t)$, $B^\pm\cdot N|_{\Sigma(t)}=0$ and $B^\pm_d=0$ on $\Sigma^\pm$ are constraints for initial data so that system \eqref{CMHD} with jump conditions \eqref{CMHDjump} is not over-determined. One can show that $D_t^\pm(\frac{1}{\rho^\pm}\nab\cdot B^\pm)=0$ in $\Om^\pm(t)$ and $D_t^\pm(\frac{B^\pm}{\rho^\pm}\cdot N)=0$ on $\Sigma(t)$ and $\Sigma^\pm$ with $D_t^\pm:=\p_t+u^\pm\cdot\nab$. Thus, the initial constraints can propagate within the lifespan of solutions if initially hold.
\end{rmk}

To make the initial-boundary-valued problem \eqref{CMHD}-\eqref{CMHDslip} solvable, we have to require the initial data to satisfy certain compatibility conditions. Let $(u_0^\pm,B_0^\pm,\varrho_0^\pm,\mathfrak{s}_0^\pm,\psi_0):=(u^\pm,B^\pm,\varrho^\pm,\mathfrak{s}^\pm,\psi)|_{t=0}$ be the initial data of system \eqref{CMHD}-\eqref{CMHDjump}. We say the initial data satisfies the compatibility condition up to $m$-th order $(m\in\N)$ if
\begin{equation}\label{CMHDcomp}
\begin{aligned}
(D_t^\pm)^j \jump{Q}|_{t=0}=\sigma(D_t^\pm)^j\mathcal{H}|_{t=0}~~\text{ on }\Sigma(0),~~0\leq j\leq m,\\
(D_t^\pm)^j \p_t\psi|_{t=0}=(D_t^\pm)^j(u^\pm\cdot N)|_{t=0}~~\text{ on }\Sigma(0),~~0\leq j\leq m,\\
\p_t^j u_d^\pm=0~~\text{ on }\Sigma^\pm,~~0\leq j\leq m.
\end{aligned}
\end{equation} With these compatibility conditions, one can show that the magnetic fields also satisfy (cf. \cite[Section 4.1]{Trakhinin2008CMHDVS})
\[
(D_t^\pm)^j(B^\pm\cdot N)|_{t=0}=0~~\text{ on }\Sigma(0)\text{ and }\Sigma^\pm,~~0\leq j\leq m.
\]We also note that the fulfillment of the first condition implicitly requires the fulfillment of the second one. 

For $T>0$, we denote $\Om_T^\pm:=\bigcup\limits_{0\leq t\leq T}\{t\}\times\Om^\pm(t)$ and $\Sigma_T:=\bigcup\limits_{0\leq t\leq T}\{t\}\times\Sigma(t)$. We consider the Cauchy problem of \eqref{CMHD}: Given initial data $(u_0^\pm,B_0^\pm,\varrho_0^\pm,\mathfrak{s}_0^\pm,\psi_0)$ satisfying the compatibility conditions \eqref{CMHDcomp} up to certain order, the vortex-sheet condition $|\jump{u_0\cdot\tau}|_{\Sigma}>0$ for any vector $\tau$ tangential to $\Sigma(0)$, the constraints $\nab\cdot B_0^\pm=0$ in $\Om^\pm(0)$, $(B_0^\pm\cdot N)|_{\Sigma(0)}=0$ and $B_{0d}^\pm|_{\Sigma^\pm}=0$, we want to study the well-posedness and the incompressible limit of the following system for the case $\sigma>0$ in this paper.  The zero-surface-tension limit under suitable stability conditions on $\Sigma_T$ and further improvement of the incompressible limit are discussed in the second part of the two-paper sequence \cite{Zhang2023CMHDVS2}.
\begin{equation}\label{CMHD1}
\begin{cases}
\varrho^\pm (\p_t+u^\pm\cdot\nab) u^\pm-B^\pm\cdot\nabla B^\pm+\nabla Q^\pm=0,~~Q^\pm:=P^\pm+\frac{1}{2}|B^\pm|^2 &~~\text{ in }\Om_T^\pm,\\
(\p_t+u^\pm\cdot\nab)\varrho^\pm+\varrho^\pm\nab\cdot u^\pm=0&~~\text{ in }\Om_T^\pm,\\
(\p_t+u^\pm\cdot\nab) B^\pm=B^\pm\cdot\nabla u^\pm-B^\pm\nab\cdot u^\pm&~~\text{ in }\Om_T^\pm,\\
\nab\cdot B^\pm=0&~~\text{ in }\Om_T^\pm,\\
(\p_t+u^\pm\cdot\nab) \fs^\pm=0&~~\text{ in }\Om_T^\pm,\\
P^\pm=P^\pm(\varrho^\pm,\mathfrak{s}^\pm),~~\frac{\p P^\pm}{\p\varrho^\pm}>0,~~\varrho^\pm\geq\overline{\rho}_0>0&~~\text{ in }\overline{\Om_T^\pm},\\
\jump{Q}=\sigma\cnab\cdot\left(\frac{\cnab\psi}{\sqrt{1+|\cnab\psi|^2}}\right)&~~\text{ on }\Sigma_T,\\
B^\pm\cdot N=0&~~\text{ on }\Sigma_T,\\
\p_t\psi=u^\pm\cdot N&~~\text{ on }\Sigma_T,\\
u_d^\pm=B_d^\pm=0&~~\text{ on }[0,T]\times\Sigma^\pm,\\
(u^\pm,B^\pm,\varrho^\pm,\mathfrak{s}^\pm)|_{t=0}=(u_0^\pm,B_0^\pm,\varrho_0^\pm,\mathfrak{s}_0^\pm)~~\text{ in }\Om^\pm(0),\quad\quad \psi|_{t=0}=\psi_0 &~~\text{ on }\Sigma(0).
\end{cases}
\end{equation}

System \eqref{CMHD1}, as a hyperbolic conservation law, admits a conserved $L^2$ energy
\[
\begin{aligned}
E_0(t):=\sum_{\pm}\frac12\int_{\Om^\pm(t)} \varrho^\pm|u^\pm|^2+|B^\pm|^2+ 2 \mathfrak{P}(\varrho^\pm,\fs^\pm) + \varrho^\pm |\fs^\pm|^2\dx + \sigma \text{ Area}(\Sigma(t))
\end{aligned}
\]where $\mathfrak{P}(\varrho^\pm,\fs^\pm)=\int_{\bar{\rho}_0}^{\varrho^\pm}\frac{P^\pm(z,\fs^\pm)}{z^2}\dz$. See Section \ref{sect L2} for proof.

\subsection{Reformulation in flattened domains}
\subsubsection{Flattening the fluid domains}
We shall convert \eqref{CMHD1} into a PDE system defined in fixed domains $\Om^\pm := \T^{d-1}\times\{0<\pm x_d<H\}.$ One way to achieve this is to use the Lagrangian coordinates, but it would bring lots of unnecessary technical difficulties when analyzing the surface tension. Here, we consider a family of diffeomorphisms $\Phi(t, \cdot): \Om^\pm\to \Om^\pm(t)$ characterized by the moving interface. In particular, let 
\begin{equation}
\Phi(t,x',x_d) = \left(x', \varphi(t,x_d)\right), \label{change of variable}
\end{equation} 
where
\begin{align}
 \varphi (t,x) = x_d+\chi(x_d)\psi (t,x') \label{change of variable vp}
\end{align}and $\chi\in C_c^\infty ([-H,H])$ is a smooth cut-off function satisfying the following bounds:
\begin{equation}
{\|\chi'\|_{L^\infty(\R)}\leq \frac{1}{\|\psi_0\|_{\infty}+20},\quad \sum_{j=1}^8 \|\chi^{(j)}\|_{L^\infty(\R)}\leq C, }\quad\chi=1\,\,\,\, \text{on}\,\, (-1, 1) \label{chi}
\end{equation}
for some generic constant $C>0$. We assume $|\psi_0|_{L^{\infty}(\T^2)}\leq 1$. One can prove that there exists some $T_0>0$ such that $\sup\limits_{[0,T_0]}|\psi(t,\cdot)|_{L^{\infty}(\T^2)}<10<H$, the free interface is still a graph  within the time interval $[0,T_0]$ and
$$\p_d \varphi(t, x', x_d) = 1+\chi'(x_d)\psi(t,x')=1-\frac{1}{20}\times 10 \geq \frac12 ,\quad t\in[0,T_0],
$$
which ensures that $\Phi(t)$ is a diffeomorphism in $[0,T_0]$.  

Based on this, we introduce the following variables
\begin{align}
v^\pm(t, x) = u^\pm(t, \Phi(t,x)),\quad b^\pm(t,x)=B^\pm(t,\Phi(t,x)),\quad \rho^\pm(t,x) = \varrho^\pm(t, \Phi(t,x)),\notag\\
S^\pm(t, x) = \mathfrak{s}^\pm(t, \Phi(t,x)), \quad q^\pm(t,x) = Q^\pm(t, \Phi(t,x)), \quad p^\pm(t,x)=P(t,\Phi(t,x)),
\end{align}
which represent the velocity fields, the magnetic fields, the densities, the entropy functions, the total pressure functions and the fluid pressure functions defined in the fixed domains $\Omega^\pm$ respectively. Also, we introduce the differential operators
\begin{align}
\nab^{\vp}=(\pp_1,\cdots,\pp_d),~~ &~\pp_a = \p_a -\frac{\p_a \varphi}{\p_d\varphi}\p_d,~~a=t, 1,\cdots,d-1;~~\p_d^{\vp}= \frac{1}{\p_d \varphi} \p_d.\label{nabp 3}
\end{align}
Moreover, setting the tangential gradient operator and the tangential derivatives as
\[
\cnab := (\p_1, \cdots,\p_{d-1}),~~\TP_i:=\p_i,~~i=1,\cdots,d-1,
\]
 then the boundary conditions \eqref{CMHDjump} on the free interface $\Sigma(t)$ are turned into
\begin{align}
\jump{q} =\sigma\h(\psi):=\sigma \cnab \cdot \left( \frac{\cnab \psi}{\sqrt{1+|\cnab\psi|^2}}\right) &\quad \text{on}\,\,[0,T]\times \Sigma, \label{jump BC}\\
\p_t \psi =v^\pm\cdot N, \quad N = (-\TP_1\psi, -\TP_2\psi, 1)^\top&\quad  \text{on}\,\,[0,T]\times \Sigma, \label{kinematic BC}\\
b^\pm\cdot N=0&\quad  \text{on}\,\,[0,T]\times \Sigma, \label{magnetic BC}
\end{align} 
where $\Sigma =\T^{d-1}\times\{x_d=0\}$.

 Let $\Dtpm:=\p_t^\vp + v^\pm\cdot \nab^\vp$. Then system \eqref{CMHD1} is converted into
\begin{equation}\label{CMHD2}
\begin{cases}
\rho^\pm \Dtpm v^\pm -\bpm b^\pm+\nabp q^\pm=0,~~q^\pm=p^\pm+\frac12|b^\pm|^2&~~\text{ in }[0,T]\times \Omega^\pm,\\
\Dtpm\rho^\pm+\rho^\pm \nabp\cdot v^\pm=0 &~~\text{ in }[0,T]\times \Omega^\pm,\\
p^\pm=p^\pm(\rho^\pm,S^\pm),~~\frac{\p p^\pm}{\p\rho^\pm}>0,~\rho^\pm\geq \bar{\rho_0}>0 &~~\text{ in }[0,T]\times \Omega^\pm, \\
\Dtpm b^\pm-\bpm v^\pm+b^\pm\nabp\cdot v^\pm=0&~~\text{ in }[0,T]\times \Omega^\pm,\\
\nabp\cdot b^\pm=0&~~\text{ in }[0,T]\times \Omega^\pm,\\
\Dtpm S^\pm=0&~~\text{ in }[0,T]\times \Omega^\pm,\\
\jump{q}=\sigma\cnab \cdot \left( \frac{\cnab \psi}{\sqrt{1+|\cnab\psi|^2}}\right) &~~\text{ on }[0,T]\times\Sigma, \\
\p_t \psi = v^\pm\cdot N &~~\text{ on }[0,T]\times\Sigma,\\
b^\pm\cdot N=0&~~\text{ on }[0,T]\times\Sigma,\\
v_d^\pm=b_d^\pm=0&~~\text{ on }[0,T]\times\Sigma^\pm,\\
{(v^\pm,b^\pm,\rho^\pm,S^\pm,\psi)|_{t=0}=(v_0^\pm, b_0^\pm, \rho_0^\pm, S_0^\pm,\psi_0)}.
\end{cases}
\end{equation}

Invoking \eqref{nabp 3}, we can alternatively write the material derivative $\Dtp$ as
\begin{equation}
\Dtpm = \p_t + \vb^\pm\cdot\cnab+\frac{1}{\p_d \varphi} (v^\pm\cdot \NN-\p_t\varphi)\p_d, \label{Dt alternate}
\end{equation}
where $\vb^\pm:=(v_1^\pm,\cdots,v_{d-1}^\pm)^\top$ is the horizontal components of the fluid velocity, 
$\vb^\pm \cdot \cnab := \sum\limits_{j=1}^{d-1}v_j^\pm\p_j$, and $\NN:= (-\p_1\varphi, \cdots,-\p_{d-1} \varphi, 1)^\top$ is the extension of the normal vector $N$ into $\Om^\pm$. 
This formulation will be helpful for us to define the linearized material derivative when using Picard iteration to construct the solution.

\subsubsection{On the equation of state}\label{sect eos}
\paragraph*{Parametrization and requirement of the equation of state.} We assume the fluids in $\Om^+$ and $\Om^-$ satisfy the same equation of state. Specifically, we parametrize the equation of state to be $\rho_\lam(p,S):=\rho(p/\lam^2,S)$ where $\lam>0$ is proportional to the sound speed $c_s:=\sqrt{\p_p\rho}$ and $\rho$ is a $C^8$ function in its arguments satisfying $\frac{\p \rho}{\p p}>0$ as well as the non-degeneracy condition $\rho\geq \bar{\rho_0}>0$ in $\overline{\Om}$ for some constant $\bar{\rho_0}$. By chain rule, it is straightforward to see 
	\begin{align}\label{eos ineq 1}
	0<\frac{\p }{\p p}\rho_\lam(p,S)\leq C\lam^{-2}.
	\end{align}
	and 
	\begin{align}\label{eos ineq 3}
	|(\p_p)^k \rho_\lam(p,S)|\leq C\lam^{-2k},\q\q |(\p_S)^k \rho_\lam(p, S)|\leq C,\q\q 1\leq k\leq 8,
	\end{align} for some $C>0.$  For example, a polytropic gas satisfies the above assumptions whose the equation of state is parametrized  in terms of $\lam>0$: 
	\begin{equation}\label{eos11}
	p_{\lam}(\rho,S)=\lam^2\left(\rho^{\gamma}\exp(S/C_V)-1\right),~~~\gamma> 1,~~C_V>0.
	\end{equation}

\paragraph*{The formulation used in this manuscript.}		
For sake of clean notations,  we would introduce the quantity $\ff^\pm:= \log \rho^\pm$ to replace $\rho$ and introduce the parameter $\eps:=1/\lam$ to replace $\lam$  in the continuity equation, that is, $\ff_\eps(p,S):=\log\rho_{\frac{1}{\eps}}(p, S)$. Since $\frac{\p p^\pm}{\p\rho^\pm}>0$ and $\rho^\pm>0 $ imply $\frac{\p\ff^\pm}{\p p^\pm}=\frac{1}{\rho^\pm}\frac{\p \rho^\pm}{\p p^\pm}>0$,  then the continuity equation is equivalent to 
\begin{equation} \label{continuity eq f}
\frac{\p\ff_\eps^\pm}{\p p^\pm}( p^\pm,S^\pm) \Dtpm p^\pm  +\nabp\cdot v^\pm=0. 
\end{equation}\eqref{eos ineq 1}-\eqref{eos ineq 3} lead to the following inequalities: There exists a constant $A>0$ such that
		\begin{align} \label{pp property}
			0<\frac{\p \ff_\eps}{\p p}( p,S)\leq&~ A\eps^2,\\
			|\p_p^k \ff_\eps(p,S)|\leq A\eps^{2k},~~~|\p_S^k \ff_\eps(p, S)|\leq&~A,\q\q 1\leq k\leq 8.
		\end{align}
In what follows, we slightly abuse the terminology and call $\lam$ the sound speed and call $\eps$ the Mach number. When discussing the incompressible limit ($\lam\gg1$ or equivalently $\eps\ll 1$), we sometimes write $\ffpm:=\frac{\p\ff_\eps^\pm}{\p p} (p^\pm,S^\pm)=\eps^2$ for simplicity.

\subsection{History and background}
\subsubsection{An overview of previous results}
There have been a lot of studies about free-boundary problems in ideal MHD, of which the original models in physics are mainly three types: plasma-vacuum interface model, current-vortex sheets and MHD contact discontinuities. The plasma-vacuum problem is related to plasma confinement problems \cite[Chap. 4]{MHDphy} in laboratory plasma physics, which describes the motion of \textit{one isolated} perfectly conducting fluid in an electro-magentic field confined in a vacuum region (in which there is another vacuum magnetic field satisfying the pre-Maxwell system).  When the vacuum magnetic fields are neglected, the plasma-vacuum model is reduced the free-boundary problem of one-phase MHD flows and we refer to \cite{HaoLuo2014priori,LuoZhang2019MHDST,GuWang2016LWP,GuLuoZhang2021MHDST,GuLuoZhang2021MHD0ST,HaoYang2023MHDST} for local well-posedness (LWP) theory in incompressible ideal MHD. It should be noted that, when the surface tension is neglected, the Rayleigh-Taylor sign condition $-\nab_N Q|_{\Sigma(t)}\geq c_0>0$ should be added as an initial constraint for LWP analogous to Euler equations \cite{Ebin1987} and we refer to Hao-Luo \cite{HaoLuo2018ill} for the proof. For the full plasma-vacuum model without surface tension in incompressible ideal MHD, we refer to \cite{Guaxi1,Guaxi2,SWZ2017MHDLWP,LiuXin2023MHDPV}. As for the compressible case, in a series of works \cite{Secchi2013CMHDLWP,TW2020MHDLWP,TW2021MHDSTLWP,TW2022MHDPVST}, Secchi, Trakhinin and Wang used Nash-Moser iteration to construct the solution due to the derivative loss in the linearized problems. Very recently, Lindblad and the author \cite{Zhang2021CMHD} proved the LWP and a continuation criterion for the one-phase free-boundary problem in compressible ideal MHD without surface tension, which gave the first result about the energy estimates without loss of regularity. 

A vortex sheet is an interface between two ``impermeable" fluids across which there is a tangential discontinuity in fluid velocity. For incompressible inviscid fluids without surface tension, vortex sheets tend to be violently unstable, which exhibit the so-called Kelvin-Helmholtz instability. There have been numerous mathematical studies, especially for 2D irrotational flows, and we refer to \cite{Ebin1988, Wu2004VS} and references therein. On the other hand, surface tension is expected to ``suppress" the Kelvin-Helmholtz instability and we refer to \cite{AM2007VS,CCS2007,SZ3}. When the compressibility is taken into account, we shall consider not only the motion of the interface of discontinuities but also its interaction with the wave propagation in the interior. Let $\mathbf{j}=\varrho(u\cdot N-\p_t\psi)$ be the mass transfer flux. In view of hyperbolic conservation laws, strong discontinuities can be classified into shock waves $(\mathbf{j}\neq 0,\jump{\varrho}\neq 0)$ and characteristic (contact) discontinuities $(\mathbf{j}=0)$. For compressible Euler equations, contact discontinuities are classified to be vortex sheets ($\jump{u_\tau}\neq\bd{0}$) and entropy waves ($\jump{u}=\vec{0},~\jump{\varrho},\jump{\mathfrak{s}}\neq 0$). The existence and the structural stability of multi-dimensional shocks for compressible Euler equations was proved by Majda \cite{Majdashock1,Majdashock2} (see also Blokhin \cite{Blokhinshock}) provided that the uniform Kreiss-Lopatinski\u{\i} condition \cite{lopatinskii} is satisfied. Since compressible vortex sheets are characteristic discontinuities (the uniform Kreiss-Lopatinski\u{\i} condition is never satisfied), there is a potential loss of normal derivatives for compressible vortex sheets, which makes the proof of existence more difficult. For 3D Euler equations, compressible vortex sheets are always violently unstable \cite{Miles1, Miles2, SyrovatskiiEuler} which exhibit an analogue of Kelvin-Helmholtz instability; whereas for 2D Euler equations, Coulombel-Secchi \cite{Secchi2004CVS,Secchi2008CVS} proved the existence of ``supersonic" vortex sheets when the Mach number for the rectilinear background solution $(\pm\underline{v},\underline{\rho})$ exceeds $\sqrt{2}$ and the linear instability when the Mach number is lower than $\sqrt{2}$. Similarly as the incompressible case, surface tension again prevents such violent instability and we refer to Stevens \cite{Stevens2016CVS} for the proof of structural stability.

As for MHD, there are three types of characteristic discontinuities: current-vortex sheets ($\mathbf{j}=0,~B^\pm\cdot N|_{\Sigma(t)}=0$), MHD contact discontinuities ($\mathbf{j}=0,~B^\pm\cdot N|_{\Sigma(t)}\neq0$) and Alfv\'en (rotational) discontinuities ($\mathbf{j}\neq0,~\jump{\varrho}=0$). The Rankine-Hugoniot conditions for current-vortex sheets and MHD contact discontinuities (cf. \cite[Chap. 4.5]{MHDphy} and \cite[Appendix A]{TW2021MHDCDST}) are 
\begin{itemize}
\item (Current-vortex sheets/Tangential discontinuities) $\jump{Q}=\sigma\mathcal{H},~~B^\pm \cdot N=0,~~\p_t\psi=u^\pm \cdot N$  on  $\Sigma(t)$.
\item (MHD contact discontinuities)  $\jump{P}=\sigma\mathcal{H},~~\jump{u}=\jump{B}=\vec{0},~~B^\pm \cdot N\neq0,~~\p_t\psi=u^\pm \cdot N$  on  $\Sigma(t)$.
\end{itemize}

MHD contact discontinuities usually arise from astrophysical plasmas \cite{MHDphy}, where the magnetic fields typically originate in a rotating object, such as a star or a dynamo operating inside, and intersect the surface of discontinuity. An example is the photosphere of the sun. In contrast, current-vortex sheets require the magnetic fields to be tangential to the interface. An example in laboratory plasma physics is that the discontinuities confine a high-density plasma by a lower-density one, which is isolated thermally from an outer rigid wall. In particular, when the plasma is liquid metal, the effect of surface tension cannot be neglected \cite{MHDSTphy2}. In astrophysics, a generally accepted model for compressible current-vortex sheets is the heliopause \cite{CVSphy} (in some sense, the ``boundary" of the solar system\footnote{On August 25, 2012, Voyager 1 flew beyond the heliopause and entered interstellar space. At the time, it was at a distance about 122 A.U. (around 18 billion kilometers) from the sun. On November 5, 2018, Voyager 2 also traversed the heliopause.}) that separates the interstellar plasma from the solar wind plasma.  The night-side magnetopause of the earth is also considered to be current-vortex sheets. 

For MHD contact discontinuities, the transversality of magnetic fields could enhance the regularity of the free interface and avoid the possible normal derivative loss in the interior. We refer to Morando-Trakhinin-Trebeschi \cite{MTT2018MHDCD} for the 2D case under Rayleigh-Taylor sign condition $N\cdot\nab\jump{Q}|_{\Sigma(t)}\geq c_0>0$, Trakhinin-Wang \cite{TW2021MHDCDST} for the case with nonzero surface tension, and Wang-Xin \cite{WangXinMHDCD} for both 2D and 3D cases without surface tension or Rayleigh-Taylor sign condition. In other words, Wang-Xin \cite{WangXinMHDCD} showed that transversal magnetic fields across the interface could suppress the Rayleigh-Taylor instability.

As for current-vortex sheets, Kelvin-Helmholtz instability can also be suppressed, but, unlike the transversal magnetic fields in MHD contact discontinuities, the tangential magnetic fields must satisfy certain constraints. For 3D incompressible ideal MHD, Syrovatski\u{\i} \cite{SyrovatskiiMHD} introduced a stability condition by using normal mode analysis:
\begin{equation}\label{syrov}
\varrho^+|B^+\times\jump{u}|^2+\varrho^-|B^-\times\jump{u}|^2<(\varrho^++\varrho^-)|B^+\times B^-|^2,
\end{equation} which corresponds to the transition to violent instability, that is, ill-posedness of the linearized problem. Coulombel-Morando-Secchi-Trebeschi \cite{CMST2012MHDVS} proved the a priori estimate for the nonlinear problem under a more restrictive condition
\begin{equation}\label{syrov2}
\max\left\{\left|\frac{B^+}{\sqrt{\varrho^+}}\times\jump{u}\right|,\left|\frac{B^-}{\sqrt{\varrho^-}}\times\jump{u}\right|\right\}<\left|\frac{B^+}{\sqrt{\varrho^+}}\times \frac{B^-}{\sqrt{\varrho^-}}\right|.
\end{equation}Sun-Wang-Zhang \cite{SWZ2015MHDLWP} proved local well-posedness of the nonlinear problem under the original Syrovatski\u{\i} condition \eqref{syrov} by adapting the framework of Shatah-Zeng \cite{SZ3}. Very recently, Liu-Xin \cite{LiuXin2023MHDVS} gave a comprehensive study for both $\sigma>0$ and $\sigma=0$ cases (see also Li-Li \cite{sb}).

For compressible current-vortex sheets without surface tension, it is still unknown if there is any \textit{necessary and sufficient condition} for the linear (neutral) stability. Trakhinin \cite{Trakhinin2005CMHDVS} raised a sufficient condition for the problem linearized around a background planar current-vortex sheet $(\hat{v}^\pm, \hat{b}^\pm, \hat{\rho}^\pm, \hat{S}^\pm)$ in flattened domains $\Om^\pm$, which reads
\begin{equation}\label{3D stable}
\max\left\{|\hat{b}^-\times\jump{\hat{v}}|\sqrt{\hat{\rho}^+\left(1+({c_A^+}/{c_s^+})^2\right)},|\hat{b}^+\times\jump{\hat{v}}|\sqrt{\hat{\rho}^-\left(1+({c_A^-}/{c_s^-})^2\right)}\right\}<|\hat{b}^+\times \hat{b}^-|.
\end{equation}where $c_A^\pm:=|\hat{b}^\pm|/\sqrt{\hat{\rho}^\pm}$ represents the Alfv\'en speed and $c_s^\pm:=\sqrt{\p \hat{p}^\pm/\p\hat{\rho}^\pm}$ represents the sound speed. If we formally take the incompressible limit $\hat{\rho}^\pm\to 1$ and $c_s^\pm\to +\infty$, then the above inequality exactly converges to \eqref{syrov2} used in \cite{CMST2012MHDVS}, and it is easy to see \eqref{syrov2} implies \eqref{syrov}. Under \eqref{3D stable}, Chen-Wang \cite{ChenWangCMHDVS} and Trakhinin \cite{Trakhinin2008CMHDVS} proved the well-posedness for the 3D problem without surface tension and see also \cite{WangYu2013CMHDVS, MSTY2023CMHDVS} for the 2D case without surface tension.  \textbf{When the surface tension is taken into account, it is expected to drop the extra assumptions to establish the well-posedness of compressible current-vortex sheets, but so far there is no available result.} Besides, the local existence results were established by using Nash-Moser iteration in all these previous works which leads to an unavoidable loss of regularity from initial data to solution.

Apart from the local existence, the singular limits for both free-surface ideal MHD flows and compressible vortex sheets are far less developed. Ohno-Shirota \cite{OS1998MHDill} showed that the linearized problem in a fixed domain with magnetic fields tangential to the boundary is ill-posed in standard Sobolev spaces $H^l(l\geq 2)$, but the corresponding incompressible problem is well-posed in standard Sobolev spaces \cite{GuWang2016LWP,SWZ2015MHDLWP,SWZ2017MHDLWP,LiuXin2023MHDVS,LiuXin2023MHDPV}. The anisotropic Sobolev spaces defined in Section \ref{sect anisotropic}, first introduced by Chen \cite{ChenSX}, have been adopted in previous works about ideal compressible MHD \cite{1991MHDfirst,Secchi1995,Secchi1996,Trakhinin2008CMHDVS,ChenWangCMHDVS,Secchi2013CMHDLWP,TW2020MHDLWP,TW2021MHDSTLWP}. In other words, there is no explanation for the mismatch of the function spaces for local existence yet. Besides, it is also unclear about the comparison between the stabilization mechanism brought by surface tension and the one brought by certain magnetic fields when the plasma is compressible. These questions should be answered by rigorously justifying the incompressible limit and the zero-surface-tension limit. In particular, the existing literature about the incompressible limit of free-boundary problems in inviscid fluids is only avaliable for the one-phase problems \cite{LL2018priori, Luo2018CWW, DL2019limit, Zhang2020CRMHD, Zhang2021elasto, LuoZhang2022CWWST, GuWang2023LWP}. {\bf The low Mach number limit of inviscid vortex sheets remains completely open.}

\subsubsection{Our goals}
We aim to give a comprehensive study for the local-in-time solution to current-vortex sheets in ideal MHD and particularly give affirmative answers to the abovementioned questions. Specifically, in this paper, we prove well-posedness and incompressible limit of current-vortex sheets with surface tension, namely system \eqref{CMHD2}, in both 2D and 3D. In the second part of the two-paper sequence, we will prove the zero-surface-tension limit of system \eqref{CMHD2} under certain stability conditions in 3D and 2D respectively; besides, we will also improve the incompressible limit result such that the uniform boundedness (with respect to Mach number) of high-order time derivatives can be dropped, which is a rather nontrivial improvement and relies on a new framework to prove the uniform estimates.

To our knowledge, this is \textit{the first result} about the incompressible limit of compressible vortex sheets and free-boundary MHD. The incompressible limit also ties our result to the suppression effect on Kelvin-Helmholtz instability brought by either surface tension or suitable magnetic fields.

\subsection{Main results}

\subsubsection{Anisotropic Sobolev spaces}\label{sect anisotropic}

Following the notations in \cite{WZ2023CMHDlimit}, we first define the anisotropic Sobolev space $H_*^m(\Omega^\pm)$ for $m\in\N$ and $\Om^\pm=\T^{d-1}\times\{0<\pm x_d<H\}$. Let $\omega=\omega(x_d)=(H^2-x_d^2)x_d^2$ be a smooth function\footnote{The choice of $\omega(x_d)$ is not unique, as we just need $\omega(x_d)$ vanishes on $\Sigma\cup\Sigma^\pm$ and is comparable to the distance function near the interface and the boundaries.} on $[-H,H]$.Then we define $H_*^m(\Omega^\pm)$ for $m\in\N^*$ as follows
\[
H_*^m(\Omega^\pm):=\left\{f\in L^2(\Omega^\pm)\bigg| (\omega\p_d)^{\alpha_{d+1}}\p_1^{\alpha_1}\cdots\p_d^{\alpha_d} f\in L^2(\Omega^\pm),~~\forall \alpha \text{ with } \sum_{j=1}^{d-1}\alpha_j +2\alpha_d+\alpha_{d+1}\leq m\right\},
\]equipped with the norm
\begin{equation}\label{anisotropic1}
\|f\|_{H_*^m(\Omega^\pm)}^2:=\sum_{\sum\limits_{j=1}^{d-1}\alpha_j +2\alpha_d+\alpha_{d+1}\leq m}\|(\omega\p_d)^{\alpha_{d+1}}\p_1^{\alpha_1}\cdots\p_d^{\alpha_d} f\|_{L^2(\Omega)}^2.
\end{equation} For any multi-index $\alpha:=(\alpha_0,\alpha_1,\cdots,\alpha_{d},\alpha_{d+1})\in\N^{d+2}$, we define
\[
\p_*^\alpha:=\p_t^{\alpha_0}(\omega\p_d)^{\alpha_{d+1}}\p_1^{\alpha_1}\cdots\p_d^{\alpha_d},~~\lee \alpha\ree:=\sum_{j=0}^{d-1}\alpha_j +2\alpha_d+\alpha_{d+1},
\]and define the \textbf{space-time anisotropic Sobolev norm} $\|\cdot\|_{m,*,\pm}$ to be
\begin{equation}\label{anisotropic2}
\|f\|_{m,*,\pm}^2:=\sum_{\lee\alpha\ree\leq m}\|\p_*^\alpha f\|_{L^2(\Omega^\pm)}^2=\sum_{\alpha_0\leq m}\|\p_t^{\alpha_0}f\|_{H_*^{m-\alpha_0}(\Omega^\pm)}^2.
\end{equation}

We also write the interior Sobolev norm to be $\|f\|_{s,\pm}:= \|f(t,\cdot)\|_{H^s(\Omega^\pm)}$ for any function $f(t,x)\text{ on }[0,T]\times\Omega^\pm$ and denote the boundary Sobolev norm to be $|f|_{s}:= |f(t,\cdot)|_{H^s(\Sigma)}$ for any function $f(t,x')\text{ on }[0,T]\times\Sigma$. 

From now on, we assume the dimension $d=3$, that is, $\Om^\pm=\T^2\times\{0<\pm x_3<H\} $, $\Sigma^\pm=\T^2\times\{x_3=\pm H\} $ and $\Sigma=\T^2\times\{x_3=0\}$. We will see the 2D case follows in the same manner up to slight modifications in the vorticity analysis and we refer to Section \ref{sect 2D curl} for details. Invoking \eqref{continuity eq f} and writing $\ff_p^\pm:=\frac{\p\ff^\pm}{\p p^\pm}$, system \eqref{CMHD2} is equivalent to
{\small\begin{equation}\label{CMHD0}
\begin{cases}
\rho^\pm \Dtpm v^\pm -\bpm b^\pm+\nabp q^\pm=0,~~q^\pm=p^\pm+\frac12|b^\pm|^2&~~\text{ in }[0,T]\times \Omega^\pm,\\
\ff_p^{\pm}\Dtpm p^\pm+\nabp\cdot v^\pm=0 &~~\text{ in }[0,T]\times \Omega^\pm,\\
p^\pm=p^\pm(\rho^\pm,S^\pm),~~\ff^\pm=\log \rho^\pm,~~\ff_p^\pm>0,~\rho^\pm\geq \bar{\rho_0}>0 &~~\text{ in }[0,T]\times \Omega^\pm, \\
\Dtpm b^\pm-\bpm v^\pm+b^\pm\nabp\cdot v^\pm=0&~~\text{ in }[0,T]\times \Omega^\pm,\\
\nabp\cdot b^\pm=0&~~\text{ in }[0,T]\times \Omega^\pm,\\
\Dtpm S^\pm=0&~~\text{ in }[0,T]\times \Omega^\pm,\\
\jump{q}=\sigma\cnab \cdot \left( \frac{\cnab \psi}{\sqrt{1+|\cnab\psi|^2}}\right) &~~\text{ on }[0,T]\times\Sigma, \\
\p_t \psi = v^\pm\cdot N &~~\text{ on }[0,T]\times\Sigma,\\
b^\pm\cdot N=0&~~\text{ on }[0,T]\times\Sigma,\\
v_d^\pm=b_d^\pm=0&~~\text{ on }[0,T]\times\Sigma^\pm,\\
{(v^\pm,b^\pm,\rho^\pm,S^\pm,\psi)|_{t=0}=(v_0^\pm, b_0^\pm, \rho_0^\pm, S_0^\pm,\psi_0)}.
\end{cases}
\end{equation}}

Since the material derivatives are tangential to the boundary, that is, $\Dtpm=\Dtbpm:=\p_t+\vb^\pm\cdot\cnab$ on $\Sigma$ and $\Sigma^\pm$, the compatibility conditions \eqref{CMHDcomp} for initial data up to $m$-th order $(m\in\N)$ are now written as:
\begin{equation}\label{comp cond}
\begin{aligned}
\jump{\p_t^j q}|_{t=0}=\sigma\p_t^j\mathcal{H}|_{t=0},\q\p_t^{j+1}\psi|_{t=0}=\p_t^{j}(v^\pm\cdot N)|_{t=0}~~\text{ on }\Sigma,~~0\leq j\leq m,\\
\p_t^j v_d^\pm|_{t=0}=0~~\text{ on }\Sigma^\pm,~~0\leq j\leq m.
\end{aligned}
\end{equation}Under \eqref{comp cond}, one can prove that $\p_t^{j}(b^\pm\cdot N)|_{t=0}=0$ is also satisfied on $\Sigma$ and $\Sigma^\pm$ for $0\leq j\leq m$ and we refer to Trakhinin \cite[Section 4]{Trakhinin2008CMHDVS} for details.

\subsubsection{Main result 1: Well-posedness and uniform estimates in Mach number}

The first result shows the local well-posedness and the energy estimates of \eqref{CMHD0} for each fixed $\sigma>0$.
\begin{thm}[\textbf{Well-posedness and uniform estimates for fixed $\sigma>0$}]\label{thm STLWP}
Fix the constant $\sigma>0$. Let $\mathbf{U}_0^\pm:=(v_0^\pm, b_0^\pm, \rho_0 ^\pm,S_0^\pm)^\top \in H_*^8(\Om^\pm)$ and $\psi_0\in H^{9.5}(\Sigma)$ be the initial data of \eqref{CMHD0} satisfying 
\begin{itemize}
\item the compatibility conditions \eqref{comp cond} up to 7-th order;
\item the constraints $\nab^{\vp_0}\cdot b_0^\pm=0$ in $\Om^\pm$, $b^\pm\cdot N|_{\{t=0\}\times(\Sigma\cup\Sigma^\pm)}=0$ ;
\item $|\jump{\vb_0}|>0$ on $\Sigma$, $|\psi_0|_{L^{\infty}(\Sigma)}\leq 1$, and $E(0)\le M$ for some constant $M>0$.
\end{itemize}Then there exists $T_\sigma>0$ depending only on $M$ and $\sigma$, such that \eqref{CMHD0} admits a unique solution $(v^\pm(t),b^\pm(t),\rho^\pm(t),S^\pm(t),\psi(t))$ that verifies the energy estimate
		    \begin{equation}
		    	\sup_{t\in[0,T]}E(t) \le C(\sigma^{-1})P(E(0))
		    \end{equation}and $\sup\limits_{t\in[0,T_\sigma]}|\psi(t)|<10<H$, where $P(\cdots)$ is a generic polynomial in its arguments. The energy $E(t)$ is defined to be

	        \begin{equation}
	        		\begin{aligned}\label{energy lwp}
	        			E(t):=&~E_4(t) + E_5(t) + E_6(t) + E_7(t) + E_8(t),\\
	        			E_{4+l}(t):=&\sum_\pm\sum_{\lee\alpha\ree=2l}\sum_{k=0}^{4-l}\left\|\left(\eps^{2l}\TT^{\alpha}\p_t^{k}\left(v^\pm, b^\pm,S^\pm,(\ffpm)^{\frac{(k+\alpha_0-l-3)_+}{2}}p^\pm\right)\right)\right\|^2_{4-k-l,\pm}+\sum_{k=0}^{4+l}\left|\sqrt{\sigma}\eps^{2l}\p_t^{k}\psi\right|^2_{5+l-k}\quad 0\leq l\leq 4,
	        		\end{aligned}
	        \end{equation}
            where $k_+:=\max\{k,0\}$ for $k\in\R$ and we denote $\TT^{\alpha}:=(\omega(x_3)\p_3)^{\alpha_4}\p_t^{\alpha_0}\p_1^{\alpha_1}\p_2^{\alpha_2}$ to be a high-order tangential derivative for the multi-index $\alpha=(\alpha_0,\alpha_1,\alpha_2,0,\alpha_4)$ with length (for the anisotropic Sobolev spaces) $\lee \alpha\ree=\alpha_0+\alpha_1+\alpha_2+2\times0+\alpha_4$.	The quantity $\eps$ is the parameter defined in Section \ref{sect eos}. Moreover, the $H^{9.5}(\Sigma)$-regularity of $\psi$ can be recovered in the sense that 
\begin{align}
\sum_{l=0}^4\sum_{k=0}^{3+l}\left|\sigma\eps^{2l}\p_t^{k}\psi\right|^2_{5.5+l-k}\leq P(E(t)),\quad \forall t\in[0,T_\sigma].
\end{align}

	\end{thm} 

\begin{rmk}[Correction of $E_4(t)$]
		The norm $\|p^\pm\|_{4,\pm}^2$ in $E_4(t)$ defined by \eqref{energy lwp} should be replaced by $\|(\ffpm)^{\frac12} p^\pm\|_{0,\pm}^2+\|\nab p^\pm\|_{3,\pm}^2$ because we do not have $L^2$ estimates of $p^\pm$ without $\ffpm$-weight. We still write $\|p^\pm\|_{4,\pm}^2$ as above for simplicity of notations.
		\end{rmk}		

\begin{rmk}[Weights of Mach number of $p^\pm$]
		In \eqref{energy lwp}, the weight of Mach number of $p$ is slightly different from $(v,b,S)$, but such difference only occurs when $\TT^\alpha$ are full time derivatives and $k=4-l$. In fact, due to $k\leq 4-l$ and $\alpha_0\leq \lee\alpha\ree=2l$, we know $(k+\alpha_0-l-3)_+$ is always equal to zero unless $\alpha_0=2l$ and $k=4-l$ simultanously hold.
		\end{rmk}	

\begin{rmk}[Relations with anisotropic Sobolev space]
		The energy functional $E(t)$ above is considered as a variant of $\|\cdot\|_{8,*,\pm}$ norm at time $t>0$. For different multi-index $\alpha$, we set suitable weights of Mach number according to the number of tangential derivatives that appear in $\p_*^{\alpha}$, such that the energy estimates for the modified norms are uniform in $\eps$.
		\end{rmk}

\begin{rmk}[Nonlinear structural stability]
		System \eqref{CMHD0} is studied in a bounded domain $\T^2\times(-H,H)$. Indeed, our proof also applies to the case of an unbounded domain, such as $\T^2\times\R_\pm,~\R^2\times\R_\pm$, for non-localised initial data $\mathbf{U}_0^\pm$ satisfying $(\mathbf{U}_0^\pm-\underline{\mathbf{U}}^\pm,\psi_0)\in H_*^8(\Om)\times H^{9.5}(\Sigma)$ where $\underline{\mathbf{U}}^\pm$ represents a given piecewise-smooth background solution of planar current-vortex sheet $(\underline{v}_1^\pm,\underline{v}_2^\pm,0,\underline{b}_1^\pm,\underline{b}_2^\pm,0,\underline{p}^\pm,\underline{S}^\pm)^\top$ in $\Om^\pm$. The result corresponding to this initial data exactly justifies \textit{the existence and the local-in-time nonlinear structural stability} of the piecewise-smooth planar current-vortex sheet $\underline{\mathbf{U}}^\pm$.
		\end{rmk}		

\subsubsection{Main result 2: The incompressible limit}

Next we are concerned with the incompressible limit.  For any fixed $\sigma>0$, the energy estimates obtained in Theorem \ref{thm STLWP} are already uniform in $\eps$. Also, $\|\p_t (v,b,S)\|_{3}+|\psi_t|_{4.5}$ is uniformly bounded in $\eps$. Thus, using compactness argument, we can prove the incompressible limit for current-vortex sheets with surface tension. Specifically, the motion of incompressible current-vortex sheets with surface tension are characterised by the equations of $(\xi^\sigma, w^{\pm,\sigma}, h^{\pm,\sigma})$ with incompressible initial data $(\xi_0^\sigma,w_0^{\pm,\sigma},h_0^{\pm,\sigma})$ and a transport equation of $\mathfrak{S}^{\pm,\sigma}$:
\begin{equation} \label{IMHDs}
\begin{cases}
\rr^{\pm,\sigma}(\p_t+w^{\pm,\sigma}\cdot\nab^{\Xi^\sigma})w^{\pm,\sigma}- (h^{\pm,\sigma}\cdot\nab^{\Xi^\sigma})h^{\pm,\sigma}+\nab^{\Xi^\sigma} \Pi^{\pm,\sigma}=0&~~~ \text{in}~[0,T]\times \Omega,\\
\nab^{\Xi^\sigma}\cdot w^{\pm,\sigma}=0&~~~ \text{in}~[0,T]\times \Omega,\\
(\p_t+w^{\pm,\sigma}\cdot\nab^{\Xi^\sigma}) h^{\pm,\sigma}=(h^{\pm,\sigma}\cdot\nab^{\Xi^\sigma})w^{\pm,\sigma}&~~~ \text{in}~[0,T]\times \Omega,\\
\nab^{\Xi^\sigma}\cdot h^{\pm,\sigma}=0&~~~ \text{in}~[0,T]\times \Omega,\\
(\p_t+w^{\pm,\sigma}\cdot\nab^{\Xi^\sigma})\mathfrak{S}^{\pm,\sigma}=0&~~~ \text{in}~[0,T]\times \Omega,\\
\jump{\Pi^{\sigma}}=\sigma\cnab \cdot \left( \frac{\cnab \xi^{\sigma}}{\sqrt{1+|\cnab\xi^{\sigma}|^2}}\right) &~~~\text{on}~[0,T]\times\Sigma,\\
\p_t \xi^{\sigma} = w^{\pm,\sigma}\cdot N^{\sigma} &~~~\text{on}~[0,T]\times\Sigma,\\
h^{\pm,\sigma}\cdot N^{\sigma}=0&~~~\text{on}~[0,T]\times\Sigma,\\
w_3^\pm=h_3^\pm=0&~~\text{ on }[0,T]\times\Sigma^\pm,\\
(w^{\pm,\sigma},h^{\pm,\sigma},\mathfrak{S}^{\pm,\sigma},\xi^{\sigma})|_{t=0}=(w_0^{\pm,\sigma},h_0^{\pm,\sigma},\mathfrak{S}_0^{\pm,\sigma}, \xi_0^{\sigma}), 
\end{cases}
\end{equation}where $\Xi^{\sigma}(t,x) = x_3+\chi(x_3) \xi^\sigma(t,x')$ is the extension of $\xi^\sigma$ in $\Omega$ and $ N^\sigma:=(-\TP_1\xi^\sigma, -\TP_2\xi^\sigma, 1)^\top$. The quantity $\Pi^\pm:=\bar{\Pi}^\pm+\frac12|h^\pm|^2$ represent the total pressure functions for the incompressible equations with $\bar{\Pi}^\pm$ the fluid pressure functions. The quantity $\rr^\pm$ satisfies the evolution equation $(\p_t+w^{\pm,\sigma}\cdot\nab^{\Xi^\sigma})\rr^{\pm,\sigma}=0$ with initial data $\rr_0^{\pm,\sigma}:=\rho^{\pm,\sigma}(0,\mathfrak{S}_0^{\pm,\sigma})$. 

Denoting $(\psi^\es, v^{\pm,\es}, b^{\pm,\es}, \rho^{\pm,\es}, S^{\pm,\es})$ to be the solution of \eqref{CMHD0} indexed by $\sigma$ and $\eps$, we prove that $(\psi^\es, v^{\pm,\es}, b^{\pm,\es}, \rho^{\pm,\es}, S^{\pm,\es})$ converges to $(\xi^\sigma,w^{\pm,\sigma},h^{\pm,\sigma}, \rr^{\pm,\sigma}, \mathfrak{S}^{\pm,\sigma})$ as $\eps \rightarrow 0$ provided the convergence of initial data. 

\begin{thm}[\textbf{Incompressible limit for fixed $\sigma>0$}] \label{thm CMHDlimit1}
Fix $\sigma>0$. Let $(\psi_0^\es, v_0^{\pm,\es}, b_0^{\pm,\es}, \rho_0^{\pm,\es}, S_0^{\pm,\es})$ be the initial data of \eqref{CMHD0} for each fixed $(\eps, \sigma)\in \R^+\times \R^+$, satisfying 
\begin{enumerate}
\item [a.] The sequence of initial data $(\psi_0^\es, v_0^{\pm,\es}, b_0^{\pm,\es},\rho_0^{\pm,\es}, S_0^{\pm,\es}) \in H^{9.5}(\Sigma)\times H_*^8(\Omega^\pm)\times H_*^8(\Omega^\pm)\times H_*^8(\Omega^\pm)\times H^8(\Omega^\pm)$ satisfies the compatibility conditions \eqref{comp cond} up to 7-th order, and {$|\psi_0^\es|_{L^{\infty}} \leq 1$}. 
\item [b.] $(\psi_0^\es, v_0^{\pm,\es}, b_0^{\pm,\es}, S_0^{\pm,\es}) \to (\xi_0^\sigma, w_0^{\pm,\sigma}, h_0^{\pm,\sigma}, \mathfrak{S}_0^{\pm,\sigma})$ in $H^{5.5}(\Sigma) \times H^4(\Omega^\pm)\times H^4(\Omega^\pm)\times H^4(\Omega^\pm)$ as $\eps\to 0$. 
\item [c.] The incompressible initial data satisfies $|\jump{\wb_0^\sigma}|>0$ on $\Sigma$, the constraints $\nab^{\xi_0^\sigma}\cdot h_0^{\pm,\sigma}=0$ in $\Om^\pm$, $h^{\pm,\sigma}\cdot  N^\sigma|_{\{t=0\}\times\Sigma}=0$.
\end{enumerate} 
Then it holds that 
\begin{align}
(\psi^\es, v^{\pm,\es}, b^{\pm,\es}, S^{\pm,\es})\to(\xi^\sigma, w^{\pm,\sigma}, h^{\pm,\sigma},\mathfrak{S}^{\pm,\sigma}),
\end{align} weakly-* in $L^\infty([0,T_\sigma]; H^{5.5}(\Sigma)\times(H^{4}(\Om^\pm))^3)$ and strongly in  $C([0,T_\sigma]; H_{\text{loc}}^{5.5-\delta}(\Sigma)\times(H_{\text{loc}}^{4-\delta}(\Om^\pm))^3)$ after possibly passing to a subsequence, where $T_\sigma$ is the time obtained in Theorem \ref{thm STLWP}.
\end{thm}

\begin{rmk}[The ``compatibility conditions" for the incompressible problem]
For the incompressible problem, there is no need to require the so-called ``compatibility conditions" for the initial data. The convergence of compressible data automatically implies the fulfillment of time-differentiated kinematic boundary conditions and the time-differentiated slip conditions at $t=0$. The time-differentiated jump conditions can also be easily fulfilled by adjusting the boundary values of $\Pi^\pm$, as the pressure function $\Pi$ is NOT uniquely determined by the other variables for the incompressible problem.
\end{rmk}

\noindent\textbf{List of Notations: }
In the rest of this paper, we sometimes write $\TT^k$ to represent a tangential derivative $\TT^{\alpha}$ in $\Om^\pm$ with order $\lee\alpha\ree =k$ when we do not need to specify what the derivative $\TT^{\alpha}$ contains. We also list all the notations used in this manuscript.
\begin{itemize}
\item $\Omega^\pm:=\T^{d-1}\times\{0<\pm x_d<H\}$, $\Sigma:=\T^{d-1}\times\{x_d=0\}$ and $\Sigma^\pm:=\T^{d-1}\times \{x_d=\pm H\}$, $d=2,3$.
\item $\|\cdot\|_{s,\pm}$:  We denote $\|f\|_{s,\pm}:= \|f(t,\cdot)\|_{H^s(\Omega^\pm)}$ for any function $f(t,x)\text{ on }[0,T]\times\Omega^\pm$.
\item $|\cdot|_{s}$:  We denote $|f|_{s}:= |f(t,\cdot)|_{H^s(\Sigma)}$ for any function $f(t,x')\text{ on }[0,T]\times\Sigma$.
\item $\|\cdot\|_{m,*}$: For any function $f(t,x)\text{ on }[0,T]\times\Omega$, $\|f\|_{m,*,\pm}^2:= \sum\limits_{\lee \alpha\ree\leq m}\|\p_*^\alpha f(t,\cdot)\|_{0,\pm}^2$ denotes the $m$-th order space-time anisotropic Sobolev norm of $f$.
\item $P(\cdots)$:  A generic polynomial with positive coefficients in its arguments;
\item $[T,f]g:=T(fg)-fT(g)$, and $[T,f,g]:=T(fg)-T(f)g-fT(g)$, where $T$ denotes a differential operator and $f,g$ are arbitrary functions.
\item $\TP$: $\TP=\p_1,\cdots,\p_{d-1}$ denotes the spatial tangential derivative.
\item $A\eql B$: $A$ is equal to $B$ plus some lower-order terms that are easily controlled.
\end{itemize}

\paragraph*{Acknowledgement.} The author would like to thank Prof. Zhouping Xin and Prof. Chenyun Luo for helpful discussions when he visited The Chinese University of Hong Kong during May 2023. The author would also like to thank Prof. Paolo Secchi for sharing his idea about the trace theorem for anisotropic Sobolev spaces.

\section{Strategy of the proof}\label{stat}

Before going to the detailed proof, we would like to briefly introduce the strategies to tackle this complicated problem, including key observations in the uniform a priori estimates and the design of the approximate problem to avoid Nash-Moser iteration. Moreover, we will make comparison between the compressible problem and the incompressible problem, between the Lagrangian coordinates and the ``flattened coordinates" among the vortex sheet problem, the one-phase problem and the MHD contact discontinuity.

\subsection{A hidden structure of Lorentz force in vorticity analysis}\label{stat fixed}
Let us discuss the uniform (in Mach number) estimates for the original current-vortex sheet system \eqref{CMHD0}. For sake of clean notations, we omit the superscripts $\pm$ unless we analyze the terms on the free interface. The entropy is easy to control thanks to $\Dtp S=0$, so it suffices to analyze the relations between $(v,b)$ and $q:=p+\frac12|b|^2$. Take the $H^4$ estimates as an example. Using div-curl decomposition (Lemma \ref{hodgeTT}),
\begin{align}\label{divcurl1}
\forall s\geq 1,~\|X\|_{s}^2\lesssim C(|\psi|_{s},|\cnab\psi|_{W^{1,\infty}})\left(\|X\|_0^2+\|\nabp\cdot X\|_{s-1}^2+\|\nabp\times X\|_{s-1}^2+\|\TP^s X\|_{0}^2\right),
\end{align}
 we shall prove the $H^3$-estimates for divergence, curl and the boundary term in order to control $\|v,b\|_{H^4}$. The divergence part is reduced to the tangential derivatives $\|\ffp \Dtp p\|_3$. Any normal derivative falling on $p$ or $q=p+\frac12|b|^2$ can be reduced to a tangential derivative by invoking the momentum equation. However, more observations and techniques are needed to control the vorticity and the contribution of the free interface.

To control the curl part, we take $\nabp\times$ in the momentum equation and invoke the evolution equation of $b$ to get
\begin{equation}\label{bad1}
\ddt\io \rho|\p^3(\nabp\times v)|^2+|\p^3(\nab\times b)|^2\dx= -\io \p^3\nab\times(b(\nabp\cdot v))\cdot\p^3(\nabp\times b)\dx+ \text{ controllable terms},
\end{equation}where we find that there is a normal derivative loss in the term $\p^3\nabp\times(b(\nab\cdot v))$. Indeed, invoking $\nabp\cdot v=-\ffp \Dtp p$, commuting $\nabp$ with $\Dtp$ and inserting the momentum equation $-\nabp p=\rho \Dtp v+b\times(\nabp\times b)$, we find a hidden structure of the Lorentz force $b\times(\nabp\times b)$ that eliminates the normal derivative in the curl operator:
\[
\ffp b\times(\p^3\nab \Dtp)=-\ffp \rho b\times(\p^3(\Dtp)^2 v)-\ffp  b\times(b\times \p^3 \Dtp(\nabp\times b))+\text{ lower order terms},
\] in which the second term contributes to an energy term $-\frac12\ddt\io \ffp |b\times(\p^3\nabp\times b)|^2\dx$ plus controllable remainder terms. Thus, the vorticity analysis for compressible ideal MHD motivates us to \textbf{trade one normal derivative (in curl) for two tangential derivatives together with square weights of Mach number, namely $\eps^2 (\Dtp)^2$.} Furthermore, it can be seen that the anisotropic Sobolev spaces defined in Section \ref{sect anisotropic} should be the appropriate function spaces to study compressible ideal MHD with magnetic fields tangential to the boundary. This structure was observed by the author and Wang in the recent preparatory work \cite{WZ2023CMHDlimit} for ideal MHD flows in a fixed domain and gives a definitive explanation on the ``mismatch" of fucntion spaces for the well-posedness of incompressible MHD ($H^m$) and compressible MHD ($H_*^{2m}$): the ``anisotropic part", namely the part containing more than $m$ derivatives, must have weight $\eps^{2}$ or higher power which converges to zero when taking the incompressible limit. The 2D case can also be similarly treated.

Following the above argument, all normal derivatives are reduced to tangential derivatives, and the tangential estimates are expected to be parallel to the proof of $L^2$ energy conservation, which will be analyzed in the next subsection. Now, a remaining task is to determine the weights of Mach number assigned on $v,b,p$. One thing we already know from the momentum equation is that $\nabp (q+\frac12|b|^2)\sim (b\cdot\nabp)b-\Dtp v$, which suggests that $\p_t^k\nabp p$ should share the same weights of Mach number as $\p_t^{k+1}v$. Apart from this, we recall that the $L^2$ energy conservation shows that $v,b,\sqrt{\ffp} q,S\in L^2(\Om)$, which suggest that $\p_t^k(v,b,S)$ should share the same weights of Mach number as $\eps \p_t^kp$ when doing tangential estimates.

Thus, we can conclude our reduction scheme as follows
\begin{itemize}
\item [a.] Using div-curl analysis to reduce any normal derivatives on $v,b$. In this process, we have $(\nabp\cdot v,\nabp\cdot b)\to \eps^2\Dtp q$ and $(\nabp\times v,\nabp\times b)\to \eps^2(\Dtp)^2 v$.
\item [b.] Using the momentum equation to reduce $\nab p$ to $\TT (v,b)$ and $\nabp (\frac12|b|^2)$ (this term should be further reduced via div-curl analysis), where $\TT$ can be any one of the tangential derivatives $\p_t,\TP_1,\TP_2,\omega(x_3)\p_3$..
\item [c.] Tangential estimates: When estimating $E_{4+l}(t)$ (defined in \eqref{energy lwp}), $\TT^{\gamma}(v,b)$ is controlled together with $\sqrt{\ffp}\TT^{\gamma}p$ in the estimates of full tangential derivatives, i.e., when $\lee\gamma\ree=4+l$.
\end{itemize}

Based on the above three properties, we design the energy functional $E(t)$ in \eqref{energy lwp} and we expect to establish uniform-in-$\eps$ estimates for this energy functional. 

\begin{rmk}
It should be noted that the above reduction scheme is substantially different from the one in the author's previous work \cite{Zhang2021CMHD} about one-phase free-surface MHD without surface tension, in which the normal derivatives are not reduced via the div-curl analysis, and the hidden structure in vorticity analysis is not observed. The energy estimates obtained via the method in \cite{Zhang2021CMHD} are \textit{never uniform in Mach number}. The precise reasons are referred to \cite[Section 2.2]{WZ2023CMHDlimit}. 
\end{rmk}

\subsection{Analysis of the free interface}\label{stat AGU}
After the above reduction of normal derivatives, we need to control $\|\eps^{2l}\TT^{\alpha}\TT^{\beta}\p_t^k(v,b,\sqrt{\ffp} p,S)\|_{0}^2$ where $\TT^{\alpha}=(\omega(x_3)\p_3)^{\alpha_4}\p_t^{\alpha_0}\p_1^{\alpha_1}\p_2^{\alpha_2}$ and $\alpha,\beta,k,l$ satisfy
\begin{equation}\label{TT index}\lee\alpha\ree=2l,~\lee\beta\ree=4-l-k,~0\leq k\leq 4-l,~0\leq l\leq4~\text{ and }~\beta_0=0.\end{equation} In fact, the $\eps^{2l}\TT^\alpha$-part comes from the vorticity analysis for $E_{4+l}$ and the $\TT^{\beta}\p_t^k$-part comes from the interior tangential derivatives in div-curl inequality \eqref{divcurl1}.

When commuting $\TT^\gamma$ with $\nabp$, the commutator $[\TT^\gamma,\p_i^\vp]f$ contains the term $(\p_3\vp)^{-1}\TT^\gamma\p_i\vp\p_3 f$ whose $L^2(\Om)$-norm is controlled by $|\TT^\gamma\cnab\psi|_0$. However, the regularity of $\psi$ obtained in $\TT^\gamma$-estimate is $|\sqrt{\sigma}\TT^\gamma\cnab\psi|_0$, which is $\sigma$-dependent. To get rid of such dependence, we introduce the Alinhac good unknown method \cite{Alinhac1989good} which reveals that the ``essential" leading order term in $\TT^\gamma(\nabp f)$ is not simply $\nabp(\TT^\gamma f)$, but the covariant derivative of the ``Alinhac good unknown" $\FF$. Namely, the Alinhac good unknown for a function $f$ with respect to $\TT^\gamma$ is defined by $\FF^\gamma:=\TT^\gamma f-\TT^\gamma\varphi\p_3^\vp f$ and satisfies
\begin{equation}\label{AGU good intro}
\TT^{\gamma}\nab_i^\vp f=\nabp_i\FF^\gamma+\cc_i^\gamma(f),~~\TT^\gamma \Dtp f=\Dtp\FF^\gamma+\dd^\gamma(f),
\end{equation}where $\|\cc_i^\gamma(f)\|_0$ and $\|\dd^\gamma(f)\|_0$ can be directly controlled. Therefore, we can reformulate the $\TT^\gamma$-differentiated current-vortex sheets system \eqref{CMHD0} in terms of $\VV^\gapm,\BB^\gapm,\PP^\gapm,\QQ^\gapm, \BS^\gapm$ (the Alinhac good unknowns of $v^\pm,b^\pm,p^\pm,q^\pm,S^\pm$ in $\Om^\pm$) and reduce the $\TT^\gamma$-estimate of $(v^\pm,b^\pm,p^\pm,q^\pm,S^\pm)$ to the $L^2$-estimate of good unknowns, which is similar to the $L^2$ energy conservation and exhibits several important structures when proving the uniform-in-$\eps$ estimates.

Dropping the superscript $\gamma$ for convenience and applying $L^2$ estimates to the good unknowns, we get the following equality which includes four major terms
\begin{equation}\label{TT energy 1}
\begin{aligned}
\sum_\pm\ddt\frac12\int_{\Om^\pm} \rho^\pm|\VV^\pm|^2+|\BB^\pm|^2+\ffpm|\PP^\pm|^2\dvt=\ST+\RT+\VS+\sum_\pm(Z^\pm + ZB^\pm)+\cdots
\end{aligned}
\end{equation}where $\dvt:=\p_3\vp\dx$. These four major terms are
\begin{align}
\ST:=&~\eps^{4l}\is\TT^\gamma(\sigma\h)\p_t\TT^\gamma\psi\dx',~~\RT:=-\eps^{4l}\is \jump{\p_3q}\TT^\gamma\psi\TT^\gamma\p_t\psi\dx',\\
\VS:=&~\eps^{4l}\is \TT^\gamma q^- (\jump{\vb}\cdot\cnab)\TT^\gamma\psi\dx',\\
ZB^\pm:=&~\mp\eps^{4l}\is\TT^\gamma q^\pm [\TT^\gamma, v_i, N_i]\dx',\quad Z^\pm:=-\eps^{4l}\iopm \TT^\gamma q^\pm [\TT^\gamma, \p_3v_i, \NN_i]\dvt.
\end{align} On the interface $\Sigma$, the weight function $\omega(x_3)=0$, so it remains to consider $\TT^\gamma=\p_t^{k+\alpha_0}\TP^{4-l-k+(\alpha_1+\alpha_2)}=\p_t^{k+\alpha_0}\TP^{4+l-(k+\alpha_0)}$. For simplicity of notations, we replace $k+\alpha_0$ by $k$. One can directly show that the term$\ST$ gives the $\sqrt{\sigma}\eps^{2l}$-weighted boundary regularity in $E(t)$. The term RT is supposed to give us boundary regularity $|\eps^{2l}\p_t^{k}\psi|_{4+l-k}^2$ without $\sigma$-weight provided the Rayleigh-Taylor sign condition $\jump{\p_3 q}\geq c_0>0$. However, in the presence of surface tension, we cannot impose the Rayleigh-Taylor sign condition. Thus, we have to use the $\sqrt{\sigma}$-weighted boundary energy, contributed by surface tension, to control RT.

\subsubsection{A crucial term for vortex-sheet problems}\label{stat VSST}
Let us consider the term$\VS$ that exhibits an essential difficulty in the study of vortex sheets. 
\begin{equation}\label{intro VS}
\begin{aligned}
\VS:=\eps^{4l}\is \p_t^{k}\TP^{4+l-k} q^- (\jump{\vb}\cdot\cnab)\p_t^{k}\TP^{4+l-k}\psi\dx'.
\end{aligned}
\end{equation}

The difficulty is that we only have a jump condition for $\jump{q}$ but no conditions for $q^\pm$ individually. Thus, when $0\leq k\leq 3+l$, we integrate $\TP^{1/2}$ by parts and control $q^\pm$ by using Lemma \ref{nctrace}
\[
\VS\leq |\eps^{2l}\p_t^{k}\TP^{3.5+l-k} q^-|_0|\eps^{2l}(\jump{\vb}\cdot\cnab)\p_t^{k}\TP^{4+l-k} \psi|_{1/2}\leq \|\eps^{2l}\p_t^{k}\TP^{4+l-k}  q^-\|_{0,-}^{1/2}\|\eps^{2l}\p_t^{k}\TP^{3+l-k} \p_3 q^-\|_{0,-}^{1/2}|\vb|_{2}|\eps^{2l}\p_t^k\psi|_{5.5+l-k}.
\]
This indicates us to seek for the control of $|\eps^{2l}\p_t^k\psi|_{5.5+l-k}$ for $0\leq k\leq 3+l$, which is exactly given by the surface tension. Indeed, the jump condition $\h(\psi)=\sigma^{-1}\jump{q}$ and the ellipticity of the mean curvature operator indicates that we can control $|\eps^{2l}\p_t^k\psi|_{5.5+l-k}$ by $|\sigma^{-1}\eps^{2l}\p_t^k\jump{q}|_{3.5+l-k}$ plus lower-order terms. Thus, \textbf{surface tension significantly enhances the regularity of the free interface such that VS is directly controlled.}

\begin{rmk}[\textbf{Comparison with one-phase problems and MHD contact discontinuities}]
The above estimate of VS term is not uniform in $\sigma$ as the elliptic estimate is completely contributed by surface tension. This corresponds to the fact that one cannot take the vanishing surface tension limit of vortex sheets for Euler equations as they are usually violently unstable (except the 2D supersonic case \cite{Secchi2004CVS,Secchi2008CVS}). In the absence of surface tension, the term VS loses control even if the Rayleigh-Taylor sign condition holds because the Rayleigh-Taylor sign condition $N\cdot\nab \jump{Q}|_{\Sigma(t)}\geq c_0> 0$ only gives the energy of $|\eps^{2l}\p_t^k\psi|_{4+l-k}$ which is 1.5-order lower than the desired regularity. \textit{For one-phase problems, the term VS does not appear} because everything in $\Om^-$ is assumed to be vanishing, so the Rayleigh-Taylor sign condition is usually enough to guarantee the well-posedness \cite{TW2020MHDLWP, Zhang2021CMHD}. \textit{For MHD contact discontinuities, the jump condition $\jump{v}=\vec{0}$ also eliminates the term VS} and the transversality of magnetic fields automatically give the bound for $|\p_t^k\psi|_{4-k}$ (cf. Wang-Xin \cite{WangXinMHDCD}). However, the term VS must appear in the vortex sheet problems due to $|\jump{\vb}|>0$ on $\Sigma$. Thus, the appearance of the term VS shows an essential difference from one-phase flow problems and MHD contact discontinuities. 
\end{rmk}

\begin{rmk}[Treatment of full time derivatives] It should be noted that when the tangential derivatives are the full time derivatives $\eps^{2l}\p_t^{4+l}$, the above analysis is no longer valid as we cannot integrate by part $\p_t^{1/2}$. Instead, one has to replace one $\p_t$ by $\Dtpl$ and repeatedly use the Gauss-Green formula, the symmetric structure, the continuity equation. In fact, this is the most difficult step in the proof of uniform estimates and we refer to Step 2 in Section \ref{sect E8ttbdry} for those rather technical computations.
\end{rmk}

\begin{rmk}[Comparison with the Lagrangian setting]\label{Lagrangian}
In the author's previous paper \cite{Zhang2021CMHD} about the one-phase MHD without surface tension under the setting of Lagrangian coordinates, the ``modified Alinhac good unknowns" were introduced to avoid the derivative loss in these commutators, that is, lots of modification terms were added to $\FF$ such that the corresponding $\cc(f)$ is $L^2$-controllable. Those modification terms are necessary when using Lagrangian coordinates but are redundant in the setting of this paper when the free interface is a graph. The precise reason is that, in the Lagrangian setting, the boundary regularity we obtain from tangential estimates has the form $|\TP^r\eta\cdot N|_0^2$ where $\eta$ represents the flow map of $v$, which is not enough to control the top-order derivatives of the co-factor matrix $A:=[\p\eta]^{-1}$ and the Eulerian normal vector $N=\TP\eta\times\TP\eta$. In contrast, the setting in this paper allows us to \textit{explicitly express the Eulerian normal vector, the surface tension, the boundary energy in terms of }$\cnab\psi$, and we can also \textit{explicitly} write the normal derivative of the ``non-characteristic variables" $(q, v\cdot\NN)$ in terms of tangential derivatives of the other quantities.
\end{rmk}

\subsubsection{A cancellation structure for the incompressible limit}\label{stat cancel}

So far, it remains to control the term $Z^\pm+ZB^\pm$:
\begin{equation}\label{TT cancel}
Z^\pm+ZB^\pm=\mp\eps^{4l}\is\TT^\gamma q^\pm [\TT^\gamma, v_i, N_i]\dx'-\eps^{4l}\iopm \TT^\gamma q^\pm [\TT^\gamma, \p_3v_i, \NN_i]\dvt.
\end{equation} We only obtain the regularity for $\sqrt{\ffpm}\TT^\gamma q^\pm$ in tangential estimates, but the first term contains $\TT^\gamma q^\pm$ without $\eps$-weights. When $\TT^\gamma$ contains at least one spatial derivative ($\gamma_0<\lee\gamma\ree$), one can invoke the momentum equation to replace $\TT_i q~(i=1,2,4)$ by tangential derivatives of $v,b$, as only the full time derivatives of $q^\pm$ require one more $\eps$-weight. However, there may be a loss of $\eps$-weight in this term when $\TT^\gamma$ only contains time derivatives, e.g., in $\eps^{2l}\p_t^{4+l}$-estimates for $0\leq l\leq 4$. To get rid of this, there is a cancellation structure that is observed by comparing the concrete forms of the commutators. Using Gauss-Green formula and integrating by parts in $\p_t$, it is easy to see that the leading-order part is
\begin{equation}\label{TTcancel0}
\begin{aligned}
ZB^\pm+Z^\pm\overset{\p_t}{=}&~\eps^{4l}\ddt\iopm\p_3\p_t^{3+l}  q^\pm\, \p_t^{3+l} v^\pm\cdot\p_t\NN\dx\\
&-\eps^{4l}\iopm \p_3\p_t^{3+l} q^\pm\p_t( \p_t^{3+l}v^\pm\cdot\p_t\NN)\dx-\eps^{4l}\iopm\p_t^{3+l} q^\pm \p_t(\p_t^{3+l}v^\pm\cdot\p_3\p_t\NN)\dx+\cdots
\end{aligned}
\end{equation}where the first term can be controlled by using Young's inequality after integrating in time $t$ and the other two terms can be directly controlled uniformlly in $\eps$ because the full time derivatives of $q$ no longer appear. Hence, the problematic terms in \eqref{TT cancel} are controlled uniformly in $\eps$.


\subsection{Our method to solve the compressible vortex-sheet problem}\label{stat LWP}
For equations of free-surface inviscid fluids, there is a loss of one tangential derivative in $\psi$ arising from the analogues of ST and RT terms when doing iteration. Besides, due to the presence of surface tension and compressibility, one has to control the full time derivatives of $v,b,p,S$ which only belong to $L^2(\Om^\pm)$ and their boundary regularity is unknown due to the failure of trace lemma. The delicate cancellation structures for the original nonlinear problem \eqref{CMHD0} no longer exist for the linearized problem. {\bf Therefore, we shall enhance the regularity of $\psi$ in both tangential spatial variables $x'$ and the time variable $t$.} 

There are mainly two methods to prove the existence in previous related works
\begin{enumerate}
\item \textbf{Nash-Moser iteration}. For the linearized problem, the order of the regularity loss is a fixed number, so one can use Nash-Moser iteration to prove the local existence of $C^\infty$ solution or solution in Sobolev spaces with a loss of regularity from initial data to solution. See \cite{Secchi2008CVS, Trakhinin2008CMHDVS, Secchi2013CMHDLWP, TW2020MHDLWP, TW2021MHDSTLWP, TW2022MHDPVST}.	
\item \textbf{Tangential smoothing}. This method was widely used in the study of free-surface inviscid fluids by using Lagrangian coordinates \cite{CCS2007, GuWang2016LWP, LuoZhang2020CWW, Zhang2020CRMHD, Zhang2021elasto, GuLuoZhang2021MHDST}. In \cite{LuoZhang2022CWWST}, Luo and the author introduced the tangential smoothing scheme for Euler equations in the ``flattened coordinates" when the free surface is a graph.
\end{enumerate} 

In this paper,  the method to do tangential regularization is different: the constraint $b\cdot N|_{\Sigma}=0$ no longer propagates from the initial data after doing tangential smoothing on $N$ (via convolution as in \cite{LuoZhang2022CWWST}).

\subsubsection{Our new design of the approximate problem}
In this paper, we introduce the following approximation scheme, namely the nonlinear approximate problem \eqref{CMHD0kk} (indexed by $\kk>0$) for \eqref{CMHD0}, by adding two regularization terms to the jump condition for $q$ as below:
\begin{equation}\label{CMHDjumpkk}
\jump{q}=\sigma\h - \kk (1-\TL)^2\psi - \kk (1-\TL)\p_t\psi,
\end{equation}where $\TL:=\TP_1^2+\TP_2^2$ is the tangential Laplacian operator on $\Sigma$. These two regularization terms help us to get $\sqrt{\kk}$-weighted enhanced regularity for both $\psi$ and $\psi_t$ which is enough for us to compensate the loss of derivatives in the Picard iteration process. {\bf One of the advantages of this approximation is that it does not change the structure of MHD equations and so the constraint $b\cdot N|_{\Sigma}=0$ still propagates for the nonlinear approximate system.}

By using this new regularization, we can solve the nonlinear approximate problem for each fixed $\kk>0$ and we refer to Section \ref{sect linear kkeq} for detailed construction of the linearization and the iteration scheme. It should also be noted that, when treating the linearized equation in Picard iteration, the constraint $b\cdot N|_{\Sigma}=0$ can be recovered at the end of each iteration step by modifying the normal component of the solution to the linearized system and we refer to step 3 in Section \ref{sect linear kkeq} and also Section \ref{sect recoverkk} for details. As for the uniform-in-$\kk$ estimates for the nonlinear approximate problem, the appearance of these two regularization terms will not introduce any uncontrollable terms with the help of some delicate technical modifications. In particular, the term VS remains the same as \eqref{intro VS}, and the elliptic estimate for $\jump{q}$ in Section \ref{stat VSST} is still vaild and uniform in $\kk$ for the approximate problem. Hence, the local existence of \eqref{CMHD0} is proven after passing the limit $\kk\to0$. The bounds obtained in both the iteration process (for fixed $\kk$) and the uniform-in-$\kk$ nonlinear estimates have no loss of regularity and are uniform in Mach number. 

When the free surface is a graph, the approximation scheme provides a method to prove the local existence without loss of regularity (and also the incompressible limit) for a large class of free-boundary problems in inviscid fluids (not only restricted within the study of Euler equations), especially the vortex-sheet problems with surface tension. Furthermore, we believe that taking zero-surface-tension limit seems to be an alternative way, other than Nash-Moser iteration, to prove the local existence of compressible vortex-sheet problem under certain stability conditions. This is presented in the second paper of this two-paper sequence \cite{Zhang2023CMHDVS2}.

\begin{rmk}
We choose the ``flattened coordinate" because, as mentioned in Remark \ref{Lagrangian}, we can explicitly introduce the regularized equations in terms of $\psi$. However, it should also be noted that the design of the linearized problem and the Picard iteration process in the ``flattened coordinate" is much more difficult than in the Lagrangian coordinate because one has to ``define" the free surface in each step of the iteration, whereas the free surface is not explicitly computed and the flow map $\eta$ is completely determined by the velocity in Lagrangian coordinates.
\end{rmk}

\section{Uniform estimates of the nonlinear approximate system}\label{sect uniformkk}

Now we introduce the approximate system of \eqref{CMHD0} indexed by $\kk>0$.
\begin{equation}\label{CMHD0kk}
\begin{cases}
\rho^\pm \Dtpm v^\pm -\bpm b^\pm+\nabp q^\pm=0,~~q^\pm=p^\pm+\frac12|b^\pm|^2&~~\text{ in }[0,T]\times \Omega^\pm,\\
\ff_p\Dtpm p^\pm+\nabp\cdot v^\pm=0 &~~\text{ in }[0,T]\times \Omega^\pm,\\
p^\pm=p^\pm(\rho^\pm,S^\pm),~~\ff^\pm=\log \rho^\pm,~~\ff_p^\pm>0,~\rho^\pm\geq \bar{\rho_0}>0 &~~\text{ in }[0,T]\times \Omega^\pm, \\
\Dtpm b^\pm-\bpm v^\pm+b^\pm\nabp\cdot v^\pm=0&~~\text{ in }[0,T]\times \Omega^\pm,\\
\nabp\cdot b^\pm=0&~~\text{ in }[0,T]\times \Omega^\pm,\\
\Dtpm S^\pm=0&~~\text{ in }[0,T]\times \Omega^\pm,\\
\jump{q}=\sigma\cnab \cdot \left( \frac{\cnab \psi}{\sqrt{1+|\cnab\psi|^2}}\right) - \kk (1-\TL)^2\psi - \kk (1-\TL)\p_t\psi&~~\text{ on }[0,T]\times\Sigma, \\
\p_t \psi = v^\pm\cdot N &~~\text{ on }[0,T]\times\Sigma,\\
b^\pm\cdot N=0&~~\text{ on }[0,T]\times\Sigma,\\
v_3^\pm=b_3^\pm=0&~~\text{ on }[0,T]\times\Sigma^\pm,\\
{(v^\pm,b^\pm,\rho^\pm,S^\pm,\psi)|_{t=0}=(v_0^{\kk,\pm}, b_0^{\kk,\pm}, \rho_0^{\kk,\pm}, S_0^{\kk,\pm},\psi_0^{\kk})}.
\end{cases}
\end{equation}Note that this system is not over-determined: the continuity equation, the evolution equation of $b^\pm$ and the kinematic boundary condition stay unchanged, so one can still prove $\nabp\cdot b^\pm=0$, $b^\pm\cdot N|_{\Sigma}=0$ and $b_3^\pm|_{\Sigma^\pm}=0$ all propagate from the initial data.

The energy functional associated with system \eqref{CMHD0kk} is defined by	        
\begin{equation}
	        	\begin{aligned}\label{energy lwp kk}
	        			E^\kk(t)&:= E_4^{\kk}(t) + E_5^{\kk}(t) + E_6^{\kk}(t) + E_7^{\kk}(t) + E_8^{\kk}(t)\\
	        			E_{4+l}^{\kk}(t)&:=\sum_\pm\sum_{\lee\alpha\ree=2l}\sum_{k=0}^{4-l}\left\|\left(\eps^{2l}\TT^{\alpha}\p_t^{k}\left(v^\pm, b^\pm,S^\pm,(\ffpm)^{\frac{(k+\alpha_0-l-3)_+}{2}}p^\pm\right)\right)\right\|^2_{4-k-l,\pm}\\
 &\quad+\sum_{k=0}^{4+l}\left|\sqrt{\sigma}\eps^{2l}\p_t^{k}\psi\right|^2_{5+k-l}+\left|\sqrt{\kk}\eps^{2l}\p_t^{k}\psi\right|_{6+k-l}^2+\int_0^t \left|\sqrt{\kk}\eps^{2l}\p_t^{k+1}\psi(\tau)\right|_{5+k-l}^2 \dtau,
			\end{aligned}
\end{equation}
            where $0\leq l\leq 4$ and we denote $\TT^{\alpha}:=(\omega(x_3)\p_3)^{\alpha_4}\p_t^{\alpha_0}\p_1^{\alpha_1}\p_2^{\alpha_2}$ to be a tangential derivative for the multi-index $\alpha=(\alpha_0,\alpha_1,\alpha_2,0,\alpha_4)$ with length $\lee \alpha\ree=\alpha_0+\alpha_1+\alpha_2+2\times0+\alpha_4$. The quantity $(k+\alpha_0-l-3)_+=1$ only when $\alpha_0=2l$ and $k=4-l$ and it is equal to 0 otherwise.

We aim to establish the a priori estimates of system \eqref{CMHD0kk} that is uniform in $\kk>0$, which allows us taking the limit $\kk\to 0_+$ to construct the local-in-time solution to the original system \eqref{CMHD0} for fixed $\sigma>0$. Spefically, we want to prove the following proposition
\begin{prop}\label{prop Ekk}
There exists some $T_{\sigma}>0$ independent of $\kk,\eps$ such that
\begin{equation}
\sup\limits_{0\leq t\leq T_{\sigma}} E^\kk(t)\leq C(\sigma^{-1})P(E^\kk(0)).
\end{equation}
\end{prop}
\begin{rmk}
The initial data of the approximate system \eqref{CMHD0kk} is not the same as the initial data of the original system \eqref{CMHD0} because of the different compatibility conditions. The compatibility conditions (up to 7-th order) for system \eqref{CMHD0kk} are
\begin{equation}\label{comp cond kk}
\begin{aligned}
\jump{\p_t^j q}\big|_{t=0}=\p_t^j\left(\sigma\mathcal{H}- \kk (1-\TL)^2\psi - \kk (1-\TL)\p_t\psi\right)\big|_{t=0}~~\text{ on }\Sigma,~~0\leq j\leq 7,\\
\p_t^{j+1}\psi|_{t=0}=\p_t^{j}(v^\pm\cdot N)|_{t=0}~~\text{ on }\Sigma,~~0\leq j\leq 7,\\
\p_t^j v_3^\pm|_{t=0}=0~~\text{ on }\Sigma^\pm,~~0\leq j\leq 7.
\end{aligned}
\end{equation} In Appendix \ref{sect comp cond}, we construct the initial data of \eqref{CMHD0kk} satisfying the compatibility conditions \eqref{comp cond kk} that is uniformly bounded in $\kk$ and converges to a given initial data of \eqref{CMHD0} satisfying the compatibility conditions \eqref{comp cond} up to 7-th order.
\end{rmk}

\subsection{$L^2$ energy conservation}\label{sect L2}
\begin{prop}\label{prop L2kk}
The approximate system \eqref{CMHD0kk} admits the following conserved quantity: Let \begin{equation}\label{L2kk}
\begin{aligned}
E_0^\kk(t):=&\sum_{\pm}\frac12\iopm \rho^\pm|v^\pm|^2+|b^\pm|^2+ 2 \mathfrak{P}(\rho^\pm,S^\pm) + \rho^\pm |S^\pm|^2\dvt \\
& + \frac12\is \sigma \sqrt{1+|\cnab\psi|^2} + \kk |(1-\TL)\psi|^2 \dx' + \int_0^t\is \kk |\jp \psi_t|^2\dx'\dtau.
\end{aligned}
\end{equation} Then $\ddt E_0^{\kk}(t)=0$ with in the lifespan of the solution to \eqref{CMHD0kk}. Here $\jp:=\sqrt{1-\TL}$, that is, $\widehat{\jp f}(\xi)=\sqrt{1+|\xi|^2}\hat{f}(\xi)$ in $\T^2$ and $\dvt:=\p_3\vp\dx$.
\end{prop}

\begin{proof}
The proof of $L^2$ estimate is straightforward. Taking $L^2(\Om^\pm)$-inner product of $v$ and the first equation in \eqref{CMHD0kk} and using Reynolds transport formula \eqref{transport thm nonlinear}, we get
\begin{equation}\label{L2 kk}
\begin{aligned}
&\sum_\pm\ddt\frac12\iopm \rho^\pm |v^\pm|^2\dvt=\sum_\pm\iopm (\rho^\pm \Dtpm v^\pm)\cdot v^\pm\dvt\\
=& \is \jump{q} \p_t \psi\dx'+\sum_\pm \iopm p^\pm (\nabp\cdot v^\pm)\dvt-\iopm \bpm v^\pm\cdot b^\pm \dvt + \iopm \frac12|b^\pm|^2 (\nabp\cdot v^\pm)\dvt,
\end{aligned}
\end{equation} where the integral on $\Sigma^\pm$ vanishes thanks to the slip conditions. Let $\mathfrak{P}(\rho^\pm,S^\pm)=\int_{\bar{\rho}_0}^{\rho^\pm}\frac{p^\pm(z,S^\pm)}{z^2}\dz$. Then the first integral above together with $\Dtpm S^\pm=0$ gives
\begin{align*}
\iopm p^\pm (\nabp\cdot v^\pm)\dvt = -\iopm \frac{p^\pm}{(\rho^\pm)^2} \Dtpm\rho^\pm\dvt = -\ddt\iopm \rho^\pm\mathfrak{P}(\rho^\pm)\dvt.
\end{align*} The boundary term gives $\sqrt{\sigma}$-weighted and $\sqrt{\kk}$-weighted regularity of $\psi$ and $\psi_t$. One has
\begin{align*}
\is \jump{q} \p_t \psi\dx' =-\ddt\frac12\is \sigma \sqrt{1+|\cnab\psi|^2} + \kk |(1-\TL)\psi|^2 \dx'-\is \kk |\jp \psi_t|^2\dx'.
\end{align*} Then we insert the evolution equation of $b^\pm$ in the third term in \eqref{L2 kk} to get the energy of $b^\pm$.
\begin{align*}
-\iopm \bpm v^\pm\cdot b^\pm \dvt =- \iopm \Dtpm b^\pm \cdot b^\pm\dvt - \iopm |b^\pm|^2 (\nabp\cdot v^\pm)\dvt\\
=-\ddt\frac12\iopm |b^\pm|^2\dvt +\frac12\iopm |b^\pm|^2 (\nabp\cdot v^\pm)\dvt-\iopm|b^\pm|^2 (\nabp\cdot v^\pm)\dvt,
\end{align*} where the last two terms exactly cancels with the last term in \eqref{L2 kk}. Finally, $\Dtpm S^\pm=0$ and the Reynolds transport theorem shows that $\ddt\frac12\iopm \rho^\pm|S^\pm|^2\dvt=0$. Therefore, we conclude that system \eqref{CMHD0kk} admits the following conserved quantity
\begin{equation}
\begin{aligned}
E_0^\kk(t):=&\sum_{\pm}\frac12\iopm \rho^\pm|v^\pm|^2+|b^\pm|^2+ 2 \mathfrak{P}(\rho^\pm,S^\pm) + \rho^\pm |S^\pm|^2\dvt \\
& + \frac12\is \sigma \sqrt{1+|\cnab\psi|^2} + \kk |(1-\TL)\psi|^2 \dx' + \int_0^t\is \kk |\jp \psi_t|^2\dx'\dtau,
\end{aligned}
\end{equation}which can also be inherited to the original current-vortex sheet system \eqref{CMHD0} after taking $\kk\to 0_+$.\end{proof}

\subsection{Reformulations in Alinhac good unknowns}

Let $\TT^\gamma:=(\omega(x_3)\p_3)^{\gamma_4}\p_t^{\gamma_0}\p_1^{\gamma_1}\p_2^{\gamma_2}$ be a tangential derivative with $\lee\gamma\ree=\gamma_0+\gamma_1+\gamma_2+\gamma_4$. We define the Alinhac good unknown of a given function $f$ with respect to $\TT^\gamma$ by $\FF^\gamma:=\TT^\gamma f-\TT^\gamma\varphi\p_3^\vp f$. The good unknown $\FF$ satisfies
\begin{equation}\label{AGU good}
\TT^{\gamma}\nab_i^\vp f=\nabp_i\FF^\gamma+\cc_i^\gamma(f),~~\TT^\gamma \Dtp f=\Dtp\FF^\gamma+\dd^\gamma(f),
\end{equation}where the commutators $\cc^\gamma_i(f)$ and $\dd^\gamma(f)$ are defined by
\begin{align}\label{AGU comm Ci}
\mathfrak{C}_i^\gamma(f) =&~(\pp_3\pp_i f)\TT^\gamma\vp
+\left[ \TT^\gamma, \frac{\NN_i}{\p_3 \vp}, \p_3 f\right]+\p_3 f \left[ \TT^\gamma, \NN_i, \frac{1}{\p_3 \vp}\right] +\NN_i \p_3 f\left[\TT^{\gamma-\gamma'}, \frac{1}{(\p_3 \vp)^2}\right] \TT^{\gamma'} \p_3  \vp \nonumber\\
&+\frac{\NN_i}{\p_3  \vp} [\TT^\gamma, \p_3] f - \frac{\NN_i}{(\p_3 \vp)^2}\p_3f [\TT^\gamma, \p_3] \vp,\quad i=1,2,3,
\end{align} and 
\begin{align} \label{AGU comm D}
\mathfrak{D}^\gamma(f) =&~ (\Dtp \pp_3 f)\TT^\gamma\vp+[\TT^\gamma, \vb]\cdot \TP f + \left[\TT^\gamma, \frac{1}{\p_3\vp}(v\cdot \NN-\p_t\varphi), \p_3 f\right]+\left[\TT^\gamma, v\cdot \NN-\p_t\varphi, \frac{1}{\p_3\vp}\right]\p_3 f\nonumber\\ 
&+\frac{1}{\p_3\vp} [\TT^\gamma, v]\cdot \NN \p_3 f-(v\cdot \NN-\p_t\varphi)\p_3 f\left[ \TT^{\gamma-\gamma'}, \frac{1}{(\p_3 \vp)^2}\right]\TT^{\gamma'} \p_3 \vp\nonumber\\
&+\frac{1}{\p_3\vp}(v\cdot \NN-\p_t\varphi) [\TT^\gamma, \p_3] f+ (v\cdot \NN-\p_t\varphi) \frac{\p_3 f}{(\p_3\vp)^2}[\TT^\gamma, \p_3] \vp
\end{align}
with $\len{\gamma'}=1$. Here $\NN:=(-\TP_1\vp,-\TP_2\vp,1)^\top$ is the extension of normal vector $N$ in $\Om^\pm$. The third term on the right side of \eqref{AGU comm Ci} is zero when $i=3$ because $\NN_3=1$ is a constant.

Under the setting of anisotropic Sobolev spaces, we also need to carefully treat the terms generated by commuting $\TT^\gamma$ with $\bp$ and inserting the Alinhac good unknown $\FF^\gamma$. Now, we shall rewrite the directional derivative to be $\bp=\bc\cdot\cnab+(\p_3\vp)^{-1}(b\cdot \NN)\p_3$. Similarly as in \eqref{AGU comm Ci}, we have
\begin{equation}\label{AGUB good}
\TT^\gamma(\bp f)=\bp\FF^\gamma + \ccb^\gamma(f)
\end{equation}where the commutator  $\ccb^\gamma(f)$ is defined by
\begin{align}\label{AGU comm B}
\ccb^\gamma(f) =&~(\bp\pp_3f)\TT^\gamma\vp
+\left[ \TT^\gamma, \frac{b\cdot\NN}{\p_3 \vp}, \p_3 f\right]+\p_3 f \left[ \TT^\gamma, b\cdot\NN, \frac{1}{\p_3 \vp}\right] +(b\cdot\NN) \p_3 f\left[\TT^{\gamma-\gamma'}, \frac{1}{(\p_3 \vp)^2}\right] \TT^{\gamma'} \p_3  \vp \nonumber\\
&+[\TT^\gamma,\bc_i]\TP_i f +\TT^\gamma b_3\,\pp_3 f  +\frac{b\cdot\NN}{\p_3  \vp} [\TT^\gamma, \p_3] f - \frac{b\cdot\NN}{(\p_3 \vp)^2}\p_3f [\TT^\gamma, \p_3] \vp.
\end{align}

Therefore, we can reformulate the $\TT^\gamma$-differentiated current-vortex sheets system \eqref{CMHD0kk} in terms of $\VV^\gapm$, $\BB^\gapm$, $\PP^\gapm$, $\BS^\gapm$ (the Alinhac good unknowns of $v^\pm,b^\pm,p^\pm,S^\pm$ in $\Om^\pm$) as follows
\begin{align}
\label{goodv} \rho^\pm\Dtpm\VV^{\gapm}-\bpm \BB^{\gapm} +\nabp \QQ^\gapm = \RR_v^{\gapm} - \cc^\gamma(q^\pm) + \ccb^\gamma(b^\pm) ~~&\text{ in }[0,T]\times\Om^\pm,\\
\label{goodp} \ffp\Dtpm\PP^{\gamma,\pm}+\nabp\cdot \VV^{\gamma,\pm} =\RR_p^{\gamma,\pm}-\cc_i^\gamma(v_i^\pm)~~&\text{ in }[0,T]\times\Om^\pm,\\
\label{goodb} \Dtpm \BB^{\gapm} - \bpm \VV^\gapm +b^\pm\nabp\cdot\VV^\gapm = \RR_b^\gapm - \ccb^\gamma(v^\pm)+b^\pm\cc_i^\gamma(v_i^\pm)~~&\text{ in }[0,T]\times\Om^\pm,\\
\label{goodd} \nabp\cdot b^\pm=0~~&\text{ in }[0,T]\times\Om^\pm,\\
\label{goods} \Dtpm\BS^{\pm,\alpha}=\dd^{\gamma}(S^\pm)~~&\text{ in }[0,T]\times\Om^\pm,
\end{align}with boundary conditions
\begin{align}
\label{goodbdc}\jump{\QQ^\gamma}=\sigma\TT^{\gamma}\h-\kk\TT^{\gamma}(1-\TL)^2\psi-\kk\TT^{\gamma}(1-\TL)\p_t \psi-\jump{\p_3 q}\TT^\gamma\psi~~&\text{ on }[0,T]\times\Sigma,\\
\label{goodkbc}\VV^\gapm\cdot N=\p_t\TT^{\gamma}\psi+\vb^\pm\cdot\cnab\TT^\gamma\psi-\WW^\gapm~&\text{ on }[0,T]\times\Sigma,\\
\label{goodbnc}b^\pm\cdot N=0~&\text{ on }[0,T]\times\Sigma,\\
b_3^\pm=v_3^\pm=\BB_3^\pm=\VV_3^\pm=0~&\text{ on }[0,T]\times\Sigma^\pm,
\end{align}where $\RR_v,\RR_p,\RR_b$ terms consist of the following commutators
\begin{align}
\label{goodrv} \RR_v^{\gamma,\pm}:=&-[\TT^\gamma,\rho^\pm]\Dtpm v^\pm-\rho^\pm\dd^\gamma(v^\pm)\\
\label{goodrp} \RR_p^{\gamma,\pm}:=&-[\TT^\gamma,\ffpm]\Dtpm p^\pm -\ffpm \dd^\gamma(p^\pm)\\
\label{goodrb} \RR_b^{\gamma,\pm}:=&-[\TT^\gamma, b^\pm](\nabp\cdot v^\pm)-\dd^\gamma(b^\pm),
\end{align} and the boundary term $\WW^\gapm$ is 
\begin{align}\label{goodww}
\WW^\gapm:=(\p_3 v^\pm\cdot N)\TT^\gamma\psi+[\TT^\gamma,N_i,v_i^\pm],
\end{align} 

Note that $\om(x_3)=0$ on $\Sigma\cup\Sigma^\pm$, so all boundary conditions are vanishing when $\gamma_4>0$. Thus, $\TT^\gamma$ can be written as $\p_t^{k+\alpha_0}\TP^{(4+l)-(k+\alpha_0)}$ on $\Sigma$. We can replace $k+\alpha_0$ by $k$ ($0\leq k\leq 4+l$) in the boundary energy terms. In the rest of Section \ref{sect uniformkk}, we aim to prove the following tangential estimates
\begin{prop}[Tangential estimates for the approximate system]\label{prop EkkTT}
For fixed $l\in\{0,1,2,3,4\}$ and any $\delta\in(0,1)$, the following uniform-in-$(\kk,\eps)$ energy inequalities hold:
\begin{align}
&\sum_\pm\sum_{\lee\alpha\ree=2l}\sum_{\substack{0\leq k\leq 4-l \\ k+\alpha_0<4+l}}\left\|\left(\eps^{2l}\TP^{4-k-l}\TT^{\alpha}\p_t^{k}(v^\pm, b^\pm,S^\pm,p^\pm)\right)\right\|^2_{0,\pm}\no\\
&+\sum_{k=0}^{3+l}\left|\sqrt{\sigma}\eps^{2l}\p_t^{k}\psi\right|^2_{5-k-l}+\left|\sqrt{\kk}\eps^{2l}\p_t^{k}\psi\right|_{6+k-l}^2+\int_0^t \left|\sqrt{\kk}\eps^{2l}\p_t^{k+1}\psi(\tau)\right|_{5+k-l}^2 \dtau\no\\
\lesssim&~\delta E_{4+l}^{\kk}(t)+\sum_{k=0}^{3+l}\bno{\eps^{2l}\p_t^{k}\psi(0)}_{5.5+l-k}^2 + P\left(\sigma^{-1}, \sum_{j=0}^lE_{4+j}^{\kk}(0)\right)+P\left(\sum_{j=0}^lE_{4+j}^{\kk}(t)\right)\int_0^tP\left(\sigma^{-1},\sum_{j=0}^lE_{4+j}^{\kk}(\tau)\right)\dtau
\end{align}and
\begin{align}
&\sum_\pm\sum_{k=0}^{4-l}\left\|\left(\eps^{2l}\p_t^{4+l}(v^\pm, b^\pm,S^\pm,(\ffp)^{\frac12}p^\pm)\right)\right\|^2_{4-k-l,\pm}+\left|\sqrt{\sigma}\eps^{2l}\p_t^{4+l}\psi\right|^2_{1}+\left|\sqrt{\kk}\eps^{2l}\p_t^{4+l}\psi\right|_{2}^2+\int_0^t \left|\sqrt{\kk}\eps^{2l}\p_t^{5+l}\psi(\tau)\right|_{1}^2 \dtau\no\\
\lesssim&~\delta E_{4+l}^{\kk}(t)+\bno{\eps^{2l}\p_t^{3+l}\psi(0)}_{2.5}^2+ P\left(\sigma^{-1},\sum_{j=0}^lE_{4+j}^{\kk}(0)\right)+ P\left(\sum_{j=0}^lE_{4+j}^{\kk}(t)\right)\int_0^tP\left(\sigma^{-1},\sum_{j=0}^lE_{4+j}^{\kk}(\tau)\right)\dtau.
\end{align} Here the first inequality represents the case when there are at least one spatial tangential derivatives and the second inequality represents the case of full time derivatives. Moreover, the term $\bno{\eps^{2l}\p_t^{k}\psi(0)}_{5.5+l-k}^2$ on the right side does not appear when $\kk=0$.
\end{prop}

\subsection{Tangential estimates: full spatial derivatives}\label{sect ETT}

We first study the case when all tangential derivatives are spatial derivatives $\TP_1$ and $\TP_2$, namely $\gamma_0=\gamma_4=0$ in $\TT^\gamma:=(\omega(x_3)\p_3)^{\gamma_4}\p_t^{\gamma_0}\p_1^{\gamma_1}\p_2^{\gamma_2}$. In view of the definition of $E(t)$ and the div-curl decomposition, we need to prove the $L^2$ estimates for the $\eps^{2l}\TP^{4+l}$-differentiated system $(0\leq l\leq 4)$. We now consider the case $l=0$, that is, the $\TP^4$-estimate for the approximate system \eqref{CMHD0kk} and aim to prove the following estimate
\begin{prop}\label{prop ETT}
Fix $l\in\{0,1,2,3,4\}$. For the tangential derivative $\TT^{\gamma}=\TP^{4+l},~(\gamma_0+\gamma_4=0,~\gamma_1+\gamma_2=4+l)$, the $\eps^{2l}\TP^{4+l}$-differentiated approximate system admits the following uniform-in-$(\kk,\eps)$ estimate: For any $0<\delta< 1$
\begin{equation}
\begin{aligned}
&\ino{\eps^{2l}\left(\VV^{\gamma,\pm},\BB^{\gamma,\pm},\BS^{\gamma,\pm},\sqrt{\ffpm}\PP^{\gamma,\pm}\right)(t)}_0^2+\bno{\sqrt{\sigma}\eps^{2l}\TP^{4+l}\psi(t)}_1^2+\bno{\sqrt{\kk}\eps^{2l}\TP^{4+l}\psi(t)}_2^2+\int_0^t\bno{\sqrt{\kk}\eps^{2l}\TP^{4+l}\p_t\psi(\tau)}_1^2\dtau\\
\lesssim&~\delta E_{4+l}^\kk(t)+\bno{\eps^{2l}\psi_0}_{5.5+l}^2+\sum_{j=0}^l\int_0^t P(\sigma^{-1},E_{4+j}^\kk(\tau))\dtau,~~~0\leq l\leq 4.
\end{aligned}
\end{equation}
\end{prop}

\subsubsection{The case $l=0$: $\TP^4$-estimates}\label{sect E4TT}
As stated in Section \ref{stat AGU}, we introduce the Alinhac good unknowns for $\TT^\gamma=\TP^4$ and drop the script $\gamma$ for simplicity of notations
\[
\VV^\pm:=\TP^4v^\pm-\TP^4\vp\pp_3 v^\pm,~\BB^\pm:=\TP^4b^\pm-\TP^4\vp\pp_3 b^\pm,~\PP^\pm:=\TP^4p^\pm-\TP^4\vp\pp_3 p^\pm,~\QQ^\pm:=\TP^4q^\pm-\TP^4\vp\pp_3 q^\pm.
\] Note that we have
\[
\QQ^\pm=\PP^\pm+b\cdot\BB^\pm+\underbrace{\sum_{k=1}^3 c_k \TP^{k} b^\pm\cdot\TP^{4-k} b^\pm}_{=:\RR_q^{\gamma,\pm}}
\]for some constants $c_k\in\N^*$. 

\subsubsection*{Step 1: Interior energy structure.} We test the equation \eqref{goodv} by $\VV^\pm$ in $\Om^\pm$ and integrate by parts to get one boundary term and several interior terms
\begin{equation}\label{E4TT0}
\begin{aligned}
&\frac12\ddt\iopm\rho^\pm|\VV^\pm|^2\dvt=\iopm \rho^\pm\Dtpm \VV^\pm\cdot\VV^\pm\dvt\\
=&-\iopm \BB^\pm\cdot \bpm\VV^\pm\dvt +  \iopm b^\pm\cdot\BB^\pm (\nabp\cdot\VV^\pm)\dvt+ \iopm \PP^\pm (\nabp\cdot\VV^\pm)\dvt\\
&\pm\is\QQ^\pm(\VV^\pm\cdot N)\dx' +\underbrace{\iopm \VV^\pm\cdot(\RR_v^\pm-\cc(q^\pm)+ \ccb^\gamma(b^\pm) )\dvt}_{=:R_1^\pm}+\iopm \RR_q^\pm (\nabp\cdot\VV^\pm)\dvt.
\end{aligned}
\end{equation}

Invoking the equation \eqref{goodb} for the evolution of $\BB$ in the first integral above, the energy of $\BB^\pm$ is produced.
\begin{equation}\label{E4TT B}
\begin{aligned}
&-\iopm \BB^\pm \cdot\bpm\VV^\pm\dvt \\
=&-\frac12\ddt\iopm|\BB^\pm|^2\dvt \underbrace{-\frac12\iopm (\nabp\cdot v^\pm)|\BB^\pm|^2\dvt  +\iopm \BB^\pm\cdot(\RR_b^\pm-\ccb^\gamma(v^\pm))\dvt}_{R_2^\pm}\\
& - \iopm (\BB^\pm\cdot b^\pm) (\nabp\cdot\VV^\pm)\dvt-\iopm(\BB^\pm\cdot b^\pm) \cc_i(v_i^\pm)\dvt,
\end{aligned}
\end{equation}where the first term in the last line is cancelled with the second integral in \eqref{E4TT0}, and the analysis of the second term in the last line will be postponed.

The third term in \eqref{E4TT0} produces the energy of $(\ffpm)^{\frac12} \PP^\pm$ with the help of equation \eqref{goodp}. 
\begin{equation}\label{E4TT P}
\begin{aligned}
&\iopm \PP^\pm (\nabp\cdot\VV^\pm)\dvt\\
=&-\frac12\ddt\iopm\ffpm(\PP^\pm)^2\dvt \underbrace{-\frac12\iopm (\Dtpm\ffpm+\ffpm\nabp\cdot v^\pm)|\PP^\pm|^2\dvt + \iopm \PP^\pm \RR_p^\pm\dvt}_{R_3^\pm} -\iopm \PP^\pm \cc_i(v_i^\pm)\dvt.
\end{aligned}
\end{equation}

The last term in \eqref{E4TT0} can be controlled by inserting again the continuity equation and integrating $\Dtpm$ by parts. We have
\begin{equation}
\begin{aligned}
&\iopm \RR_q^\pm (\nabp\cdot\VV^\pm)\dvt = -\iopm \ffpm \RR_q^\pm  \Dtpm \PP^\pm\dvt+\iopm\RR_q^\pm\RR_p^\pm \dvt -\iopm\RR_q^\pm \cc_i(v_i^\pm)\dvt\\
=&-\ddt\iopm\left( \sqrt{\ffpm}\RR_q^\pm\right)\left(\sqrt{\ffpm}\PP^\pm\right)\dvt + \iopm \left( \sqrt{\ffpm}\Dtpm\RR_q^\pm\right)\left(\sqrt{\ffpm}\PP^\pm\right)\dvt+\iopm\RR_q^\pm\RR_p^\pm \dvt\\
&-\iopm\RR_q^\pm \cc_i(v_i^\pm)\dvt,
\end{aligned}
\end{equation}where the first term on the right side is controlled under time integral by
\[
\delta\left\|\sqrt{\ffpm}\PP^\pm(t)\right\|_0^2+P(E^\kk_4(0))+\int_0^t P(E^\kk_4(\tau))\dtau,~~\forall 0<\delta\ll1
\]and the second term, the third term on the right side can be both controlled by $P(E_4(t))$ via direct computation because $\RR_q$ only contains 3-rd order tangential derivative of $b$. 

The entropy is directly bounded by testing the transport equation of $\BS^\pm$ with $\BS^{\pm}$ itself
\begin{equation}\label{E4TT S}
\ddt\frac12\iopm\rho^\pm(\BS^\pm)^2\dvt=\iopm\rho^\pm\dd(S^\pm)\,\BS^\pm\dvt\leq \|\BS^\pm\|_0\|\rho^\pm\|_{L^{\infty}}\,\sqrt{E_4^{\kk}(t)}.
\end{equation}

The remainder terms are controlled by direct computation. For the commutator $\cc,\dd,\ccb$, we have 
$$\|\cc(f^\pm)\|_{0,\pm}\lesssim C(|\psi|_4)\|f^\pm\|_{4,\pm},~~\|\dd(f^\pm)\|_{0,\pm}\lesssim C(|\psi|_4,|\p_t\psi|_{L^\infty}|)\|f^\pm\|_{4,\pm},~~\|\ccb(f^\pm)\|_{0,\pm}\lesssim C(|\psi|_4)\|f^\pm\|_{4,\pm}\|b^\pm\|_{4,\pm}$$ 
when $\TT^\gamma=\TP^4$ by straightforward computation. Note that the initial data is \textit{well-prepared} in the sense that $\p_t v|_{t=0}=O(1)$ with respect to Mach number, so there is no loss of $\eps$-weight in $\RR_v$ term. We have 
\begin{align}
R_1^\pm+R_2^\pm+R_3^\pm\leq P(E_4^{\kk}(t)).
\end{align}

\subsubsection*{Step 2: The boundary regularity contributed by surface tension.}
We denote $Z^\pm:=-\iopm (\PP^\pm+b^\pm\cdot\BB^\pm+\RR_q^\pm)\cc_i(v_i^\pm)\dvt=-\iopm\QQ^\pm\cc_i(v_i^\pm)\dvt$ to be the remaining interior terms presented above which should be controlled together with some boundary terms involving $\WW^\pm$. Now we analyze the boundary integral in \eqref{E4TT0}. The sum of two boundary integrals can be written as
\begin{equation}\label{E4TTbdry}
\begin{aligned}
&\is \QQ^+ (\VV^+\cdot N)\dx' - \is \QQ^- (\VV^-\cdot N)\dx' \\
=&\is\TP^4\jump{q}\TP^4\p_t\psi\dx' + \is \TP^4\jump{q} (\vb^+\cdot\cnab)\TP^4\psi\dx'+\is \TP^4 q^- (\jump{\vb}\cdot\cnab)\TP^4\psi \dx'\\
&-\is \jump{\p_3 q}\TP^4\psi\p_t\TP^4\psi\dx'-\is \p_3q^+\TP^4\psi (\vb^+\cdot\cnab)\TP^4\psi\dx'+\is \p_3q^-\TP^4\psi (\vb^-\cdot\cnab)\TP^4\psi\dx'\\
&-\is \QQ^+\WW^+\dx'+\is \QQ^-\WW^-\dx'\\
=:&\ST+\ST'+\VS+\RT+\RT^++\RT^-+ZB^++ZB^-.
\end{aligned}
\end{equation} We will see that the term ST gives the $\sqrt{\sigma}$-weighted boundary regularity (contributed by surface tension) and the $\sqrt{\kk}$-weighted boundary regularity (contributed by the two regularization terms) which help us control the terms$\ST'$, VS, RT,$\RT^\pm$. The terms $ZB^\pm$ will be controlled together with $Z^\pm$ by using Gauss-Green formula. Do note that the slip conditions imply $\VV_3^\pm=\BB_3^\pm=\TP\psi=0$ on $\Sigma^\pm$, which eliminates all boundary integrals on $\Sigma^\pm$.

Inserting the jump condition $\jump{q}=\sigma\h-\kk(1-\TL)^2\psi-\kk(1-\TL)\psi_t$ into the term ST, we get
\begin{equation}
\begin{aligned}
\ST=&~\sigma\is\TP^4\cnab\cdot\left(\frac{\cnab\psi}{\sqrt{1+|\cnab\psi|^2}}\right)\TP^4\p_t\psi\dx'-\frac12\ddt\is\left|\sqrt{\kk}\jp^2\TP^4\psi\right|^2\dx'-\is\left|\sqrt{\kk}\TP^4\jp\psi_t\right|^2\dx'.
\end{aligned}
\end{equation} Integrating by parts in the mean curvature term and using $\TP(|N|^{-1})=\frac{\cnab\psi\cdot\cnab\TP\psi}{|N|^3},~~~|N|=\sqrt{1+|\cnab\psi|^2},$ we get 
\begin{equation}\label{STTT40}
\begin{aligned}
&\sigma\is\TP^4\cnab\cdot\left(\frac{\cnab\psi}{\sqrt{1+|\cnab\psi|^2}}\right)\TP^4\p_t\psi\dx'=-\frac{\sigma}{2}\ddt\is\frac{|\TP^4\cnab\psi|^2}{\sqrt{1+|\cnab\psi|^2}}-\frac{|\cnab\psi\cdot\TP^4\cnab\psi|^2}{{\sqrt{1+|\cnab\psi|^2}}^3}\dx'\\
&\q\underbrace{-\sigma\is\left(\left[\TP^3,\frac{1}{|N|}\right]\TP\cnab_i\psi+\left[\TP^3,\frac{1}{|N|^3}\right](\cnab_k\psi\cdot\TP\cnab_k\psi\cnab_i\psi)-\frac{1}{|N|^3}[\TP^3,\cnab_i\psi\cnab_k\psi]\TP\cnab_k\psi\right)\cdot\p_t\cnab_i\TP^4\psi\dx'}_{=:\ST_1^R}\\
&\q+\underbrace{\frac{\sigma}{2}\is \p_t\left(\frac{1}{|N|}\right) \bno{\TP^4\cnab\psi}^2 -\p_t\left(\frac{1}{|N|^3}\right) \bno{\cnab\psi\cdot\cnab\TP^4\psi}^2\dx'}_{=:\ST_2^R}.
\end{aligned}
\end{equation}

The control of$\ST_1^R,\ST_2^R$ is straightforward which has been analyzed in the author's previous paper \cite[(4.77)-(4.78)]{LuoZhang2022CWWST}, so we only record the result here
\[
\ST_1^R+\ST_2^R\lesssim P(|\cnab\psi|_{L^{\infty}})|\cnab\psi|_{W^{1,\infty}}\bno{\sqrt{\sigma}\TP^4\cnab\psi}_0\bno{\sqrt{\sigma}\p_t\TP^4\psi}_0\leq P(E_4^\kk(t)).
\]

Using Cauchy's inequality \begin{equation}\label{cauchy}
\forall \mathbf{a}\in\R^2,~~\frac{|\mathbf{a}|^2}{\sqrt{1+|\cnab\psi|^2}}-\frac{|\cnab\psi\cdot\mathbf{a}|^2}{{\sqrt{1+|\cnab\psi|^2}}^3}\geq \frac{|\mathbf{a}|^2}{{\sqrt{1+|\cnab\psi|^2}}^3},
\end{equation} we obtain the $\sqrt{\sigma}$-weighted boundary regularity
\begin{equation}\label{E4TT ST}
\begin{aligned}
&\int_0^t\ST\dtau+\frac{\sigma}{2}\is\frac{|\cnab\TP^4\psi|^2}{{\sqrt{1+|\cnab\psi|^2}}^3}\dx'+\is\left|\sqrt{\kk}\jp^2\TP^4\psi\right|^2\dx'+\int_0^t\is\left|\sqrt{\kk}\TP^4\jp\psi_t\right|^2\dx'\dtau\\
\leq&\int_0^t \ST_1^R+\ST_2^R\dx' \leq \int_0^t P(E_4^\kk(\tau))\dtau.
\end{aligned}
\end{equation} So far, we already obtain the boundary regularity $\sqrt{\sigma}\psi\in H^5(\Sigma),~\sqrt{\kk}\psi\in H^6(\Sigma)$ and $\sqrt{\kk}\psi_t\in L_t^2H_{x'}^5([0,T]\times\Sigma)$. Using this, we can easily control$\ST'$ term in \eqref{E4TTbdry}. Invoking again the boundary condition for $\jump{q}$, we get
\begin{equation}
\ST'=\is\sigma\h (\vb^+\cdot\cnab)\TP^4\psi\dx'-\kk\is(1-\TL)^2\TP^4\psi~(\vb^+\cdot\cnab)\TP^4\psi\dx'-\kk\is (1-\TL)\TP^4\psi_t (\vb^+\cdot\cnab)\TP^4\psi\dx'.
\end{equation} Integrating by parts $1-\TL$ in the second term and $\jp=\sqrt{1-\TL}$ in the third term above, we can easily use the $\sqrt{\kk}$-weighted energy to control the last two terms.
\begin{equation}
\begin{aligned}
&-\kk\is(1-\TL)^2\TP^4\psi~(\vb^+\cdot\cnab)\TP^4\psi\dx'\\
=&-\kk\is\left( (1-\TL)\TP^4\psi \right)(\vb^+\cdot\cnab)(1-\TL)\TP^4\psi\dx'-\kk\is\left((1-\TL)\TP^4\psi\right) ~[1-\TL,\vb^+\cdot\cnab]\TP^4\psi\dx',
\end{aligned}
\end{equation} where the first term is controlled by $|\vb^+|_{W^{1,\infty}}|\sqrt{\kk}(1-\TL)\TP^4\psi|_0^2$ after integrating $\vb^+\cdot\cnab$ by parts and using the symmetry, and the second term is directly controlled by $|\vb^+|_{W^{2,\infty}}|\sqrt{\kk}(1-\TL)\TP^4\psi|_0|\sqrt{\kk}\TP^4\psi|_2$. Similarly, we have for any $\delta\in(0,1)$
\begin{equation}
\begin{aligned}
&-\kk\int_0^t\is (1-\TL)\TP^4\psi_t (\vb^+\cdot\cnab)\TP^4\psi\dx'\dtau=-\kk\int_0^t\is\jp\TP^4\psi_t~\jp( (\vb^+\cdot\cnab)\TP^4\psi)\dx'\dtau\\
\leq &~ \delta\bno{\sqrt{\kk}\jp\TP^4\psi_t}_{L_t^2L_{x'}^2}^2 +\frac{1}{4\delta}\int_0^t\is |\vb^+|_{W^{1,\infty}}^2 \bno{\sqrt{\kk}\TP^4\psi}_2^2\dx'\dtau \leq \delta E_4^\kk(t) +\int_0^t P(E_4^\kk(\tau))\dtau.
\end{aligned}
\end{equation} Picking $\delta>0$ sufficiently small, the $\delta$-term can be absorbed by $E_4^\kk(t)$. The first term in${\ST}'$ is controlled in the same way if we integrating $\cnab\cdot$ by parts. Here we only list the result and refer the details to \cite[(4.87)-(4.89)]{LuoZhang2022CWWST}
\begin{equation}\label{E4TT ST'}
\is\sigma\h (\vb^+\cdot\cnab)\TP^4\psi\dx'\leq P(|\cnab\psi|_{W^{1,\infty}})|\vb^+|_{W^{1,\infty}}\bno{\sqrt{\sigma}\cnab\TP^4\psi}_0^2\leq P(E_4^\kk(t)).
\end{equation}

Next we control the terms RT and$\RT^\pm$ in \eqref{E4TTbdry}. Note that we do not have the Rayleight-Taylor sign condition $\jump{\p_3 q}|_{\Sigma}\geq c_0>0$, so we have to use the $\sqrt{\sigma}$-weighted energy to control these terms, we have
\begin{equation}
\RT\leq |\p_3 q|_{L^{\infty}}|\psi|_4|\psi_t|_4\leq \sigma^{-1} P(E_4^{\kk}(t)).
\end{equation}Similarly, integrating $\vb^\pm\cdot\nab$ by parts in$\RT^\pm$ and using symmetry, the terms$\RT^\pm$ can be directly controlled
\begin{equation}\label{E4TT RT}
\RT^\pm\leq |\vb^\pm\p_3 q|_{W^{1,\infty}}|\psi|_4^2\leq \sigma^{-1} P(E_4^{\kk}(t)).
\end{equation}

\subsubsection*{Step 3: The crucial term for vortex sheets problem.}

Now we study the term VS in \eqref{E4TTbdry} which appears to be the most problematic term for the vortex sheets problem. Note that we do not have any boundary condition for $q^\pm$ individually. Thus, we may alternatively integrate $\TP^{1/2}$ by parts and use \eqref{q IBP} to control VS.
\begin{equation}
\begin{aligned}
\VS=\is\TP^4 q^- (\jump{\vb}\cdot\cnab)\TP^4\psi \dx'\lesssim\|\TP^4 q^-\|_{0,-}^{\frac12}\|\p_3\TP^3 q^-\|_{0,-}^{\frac12}\|\vb^\pm\|_{2,\pm}|\cnab\TP^4\psi|_{1/2} \leq P(E_4^\kk(t))|\psi|_{5.5},
\end{aligned}
\end{equation}where we have used the Kato-Ponce inequality (cf. Lemma \ref{KatoPonce}) for $s=1/2,~p_1=2,~p_2=\infty,~q_1=q_2=4$ and Sobolev embedding $H^{1/2}(\T^2)\hookrightarrow L^4(\T^2)$. Now we need to control $|\psi|_{5.5}$ via the jump condition of $\jump{q}$. Without the $\kk$-regularization terms, we may use the ellipticity of the mean curvature operator to control $|\psi|_{5.5}$ by $\sigma^{-1}|\jump{q}|_{3.5}$. Now, we can still prove an analogous result for the $\kk$-regularized jump condition. 
\begin{lem}[Elliptic estimate for the free interface]\label{qelliptic}
For any $s\geq 0.5$ and $\kk>0$, we have the uniform-in-$\kk$ estimate $$|\psi|_{s+1.5}\leq |\psi_0|_{s+1.5}+\sigma^{-1}\left( P(|\cnab\psi|_{L^{\infty}})|\cnab\psi|_{W^{1,\infty}}|\TP\psi|_{s-0.5}+|\jump{q}|_{s-0.5}\right).$$
\end{lem}Moreover, when $\kk=0$, $|\psi_0|_{s+1.5}$ is not needed
\begin{align}
|\sigma\psi|_{s+1.5}\leq~P(|\cnab\psi|_{L^{\infty}})|\cnab\psi|_{W^{1,\infty}}|\sigma\TP\psi|_{s-0.5}+|\jump{q}|_{s-0.5}.
\end{align}
\begin{proof}
We take $\jp^{s+0.5}$ in the jump condition to get
\[
-\jp^{s+0.5}\jump{q}=-\sigma\jp^{s+0.5}\cnab\cdot\left(\frac{\cnab\psi}{\sqrt{1+|\cnab\psi|^2}}\right)+\kk(1-\TL)^2\jp^{s+0.5}\psi +\kk(1-\TL)\jp^{s+0.5}\psi_t.
\] Testing this equation with $\jp^{s+0.5}\psi$ in $L^2(\Sigma)$, we get
\[
-\is \jp^{s+0.5}\jump{q}~\jp^{s+0.5}\psi\dx'\leq |\jp^{s-0.5} \jump{q}|_0|\jp^{s+1.5}\psi|_0.
\] For the right side, we can mimic the treatment of ST term to obtain the boundary regularity. The two regularization terms can be directly controlled
\[
\is\kk(1-\TL)^2\jp^{s+0.5}\psi~\jp^{s+0.5}\psi\dx'=\is\kk(1-\TL)\jp^{s+0.5}\psi~(1-\TL)\jp^{s+0.5}\psi\dx'=\bno{\sqrt{\kk}\psi}_{s+2.5}^2,
\]
\[
\is\kk(1-\TL)\jp^{s+0.5}\psi_t~\jp^{s+0.5}\psi\dx'\overset{\jp}{=}\ddt\bno{\sqrt{\kk}\psi}_{s+1.5}^2.
\]
The term involving surface tension is controlled as follows
\begin{align*}
&-\sigma\is\jp^{s+0.5}\cnab\cdot\left(\frac{\cnab\psi}{\sqrt{1+|\cnab\psi|^2}}\right)\jp^{s+0.5}\psi\dx'=\sigma\is\jp^{s+0.5}\left(\frac{\cnab\psi}{\sqrt{1+|\cnab\psi|^2}}\right)\cdot\jp^{s+0.5}\cnab\psi\dx'\\
=&~\sigma \is \frac{|\jp^{s+0.5}\cnab\psi|^2}{\sqrt{1+|\cnab\psi|^2}}-\frac{|\cnab\psi\cdot\jp^{s+0.5}\cnab\psi|^2}{{\sqrt{1+|\cnab\psi|^2}}^3}\dx'\\
&+\sigma\is\left(\left[\jp^{s-0.5},\frac{1}{|N|}\right]\jp\cnab_i\psi+\left[\jp^{s-0.5},\frac{1}{|N|^3}\right](\cnab_k\psi\cdot\jp\cnab_k\psi\cnab_i\psi)-\frac{1}{|N|^3}[\jp^{s-0.5},\cnab\psi]\jp\cnab_i\psi\right)\cdot\cnab_i\jp^{s+0.5}\psi\dx'
\end{align*}
Using Kato-Ponce commutator estimate (cf. \eqref{commutator} in Lemma \ref{KatoPonce}), the commutators in the last line of the above identity are controlled by $P(|\cnab\psi|_{L^{\infty}})|\cnab\psi|_{W^{1,\infty}}|\TP\psi|_{s-0.5}.$ Using again Cauchy's inequality \eqref{cauchy}, we conclude the elliptic estimate by
\begin{align*}
\sigma|\psi|_{s+1.5}^2+\kk\bno{\psi}_{s+2.5}^2+\kk\ddt\bno{\psi}_{s+1.5}^2\leq~\left(P(|\cnab\psi|_{L^{\infty}})|\cnab\psi|_{W^{1,\infty}}|\sigma\TP\psi|_{s-0.5}+|\jump{q}|_{s-0.5}\right)|\psi|_{s+1.5}.
\end{align*} In particular, Lemma \ref{parabolic} suggests that we have
\[
|\psi|_{s+1.5}\leq |\psi_0|_{s+1.5}+\sigma^{-1}\left( P(|\cnab\psi|_{L^{\infty}})|\cnab\psi|_{W^{1,\infty}}|\sigma\TP\psi|_{s-0.5}+|\jump{q}|_{s-0.5}\right).
\]
Moreover, when $\kk=0$, $|\psi_0|_{s+1.5}$ no longer appears as we do not need Lemma \ref{parabolic}
\begin{align}
|\sigma\psi|_{s+1.5}\leq~P(|\cnab\psi|_{L^{\infty}})|\cnab\psi|_{W^{1,\infty}}|\sigma\TP\psi|_{s-0.5}+|\jump{q}|_{s-0.5}.
\end{align} 
\end{proof}

Now we can easily obtain the control for the problematic term VS by setting $s=4$ in Lemma \ref{qelliptic}
\begin{equation}\label{E4TT VS}
\VS\lesssim |\psi_0|_{5.5}+\sigma^{-1}P(E_4^\kk(t)).
\end{equation} 

\subsubsection*{Step 4: A cancellation structure for the incompressible limit.}

It remains to control the term $Z^\pm$ and $ZB^\pm$. In $\TP^4$-estimates, each of these terms can be directly controlled. However, in the control of $E_8^{\kk}(t)$ and the control of full time derivatives, there will be extra technical difficulties due to the loss of Mach number or the anisotropy of the function spaces. Thus, we would like to present a robust approach to control these terms. We take $Z^-+ZB^-$ as an example and the ``+" case is controlled in the same way by reversing the sign when integrating by parts. Recall that $\QQ^-=\TP^4 q^--\TP^4\vp\pp_3 q^-$, so we have
\begin{equation}
\begin{aligned}
ZB^-=&\is \TP^4q^-(\p_3v^-\cdot N)\TP^4\psi\dx'-\is\TP^4\psi\p_3 q^-(\p_3v^-\cdot N)\TP^4\psi\dx'+\sum_{k=1}^{3}\is\binom{4}{k}\QQ^-\TP^{4-k}v^-\cdot\TP^k N\dx'.
\end{aligned}
\end{equation} The first two terms in $ZB^-$ can be directly controlled
\begin{align*}
&\is \TP^4q^-(\p_3v^-\cdot N)\TP^4\psi\dx'-\is\TP^4\psi\p_3 q^-(\p_3v^-\cdot N)\TP^4\psi\dx'\\
\leq&~\left(|\TP^{7/2}q^-|_0|\psi|_{4.5}+|\psi|_4^2\right)|\p_3 v\cdot N|_{1.5}\leq \left(\|q^-\|_{4,-}|\psi|_{4.5}+|\psi|_4^2\right)\|\p v^-\|_{2,-}|\psi|_{2.5}\leq  P(\sigma^{-1},E_4^{\kk}(t)).
\end{align*}

The last term in $ZB^-$ is controlled together with $Z^-:=-\iom\QQ^-\cc_i(v_i^-)\dvt$. Recall that
\begin{align*}
\mathfrak{C}_i(v_i^-) =&~(\pp_3\pp_i v_i^-)\TP^4\vp
-\left[ \TP^4, \frac{\p_i  \vp}{\p_3\vp}, \p_3 v_i^-\right]- \p_3 v_i^-\left[\TP^4, \p_i\vp,\frac{1}{\p_3\vp} \right]+\p_i \vp \p_3 v_i^-\left[\TP^3, \frac{1}{(\p_3 \vp)^2}\right] \TP \p_3  \vp,~i=1,2 \\
\end{align*}and 
\begin{align*}
\mathfrak{C}_3(f) =&~(\pp_3)^2 v_3^-\TP^4\vp + \left[\TP^4, \frac{1}{\p_3\vp}, \p_3 v_3^-\right] - \p_3 v_3^-\left[\TP^3, \frac{1}{(\p_3\vp)^2}\right] \TP \p_3 \vp. \\
\end{align*} Note that $\NN_i=-\p_i\vp$ for $i=1,2$, so we have
\begin{align*}
&-\left[ \TP^4, \frac{\p_i  \vp}{\p_3\vp}, \p_3 v_i^-\right]=\left[ \TP^4, \frac{\NN_i}{\p_3\vp}, \p_3 v_i^-\right]=\sum_{k=1}^3\binom{4}{k}\TP^{k}\NN_i\pp_3\TP^{4-k} v_i^- - \binom{4}{k}\left[\TP^k,\frac{1}{\p_3\vp}\right]\p_i\vp\p_3\TP^{4-k}v_i^-,
\end{align*}where the contribution of the first term above gives us (using Gauss-Green formula)
\begin{equation}
\begin{aligned}
&ZB^--\sum_{k=1}^3\iom\binom{4}{k}\QQ^-\TP^{k}\NN_i\pp_3\TP^{4-k} v_i^-\dvt\\
=&\sum_{k=1}^3\binom{4}{k}\left(\iom\p_3\QQ^-\TP^{k}\NN_i\TP^{4-k} v_i^-\dx+\iom\QQ^-\TP^{k}\NN_i\p_3\TP^{4-k} v_i^-\dx -\iom\QQ^-\TP^{k}\NN_i\p_3\TP^{4-k} v_i^-\dx\right)\\
=&\sum_{k=1}^3\binom{4}{k}\iom\p_3\QQ^-\TP^{k}\NN_i\TP^{4-k} v_i^-\dx.
\end{aligned}
\end{equation} Now invoking $\QQ^-=\TP^4 q^- - \TP^4\vp \pp_3 q^-$ and integrating one $\TP$ by parts, we find that
\begin{equation}
\sum_{k=1}^3\binom{4}{k}\iom\p_3\QQ^-\TP^{k}\NN_i\TP^{4-k} v_i^-\dx\lesssim (\|\TP^3\p_3 q^-\|_{0,-} +|\psi|_4\|\p_3 q^-\|_{L^{\infty}(\Om^-)})|\psi|_{4}\|v_i^-\|_{4,-}.
\end{equation} Among other terms in $\cc_i (v_i^-)$, we shall focus on the case when there are 4 derivatives falling on $v_i^-$ and $\vp$, and the control of these terms (lised below) appears to be easier.
\begin{equation}
\begin{aligned}
&-\iom\QQ^- \TP^4\vp \pp_3(\nabp\cdot v^-)\dvt\quad\text{from the first term in }\cc_i(v_i^-)\\
&4\sum_{i=1}^3\iom\QQ^-\p_3\TP\vp~\pp_3\TP^3 v^-\cdot\NN\dx\quad\text{from the second term in }\cc_i(v_i^-)\text{ when }\TP^3\text{ falls on }\p_3v_i^-.
\end{aligned}
\end{equation} Note that $\pp_3 v^-\cdot N=\nabp\cdot v^--\cnab\cdot\vb^-$, we have 
\begin{equation}
-\iom\QQ^- \TP^4\vp \pp_3(\nabp\cdot v)\dvt\lesssim \ino{\sqrt{\ffp^-}\QQ^-}_{0,-}\ino{\sqrt{\ffp^-}\p_3D_t^{\vp,-} p^-}_{0,-}\bno{\psi}_4,
\end{equation}and 
\begin{equation}
\begin{aligned}
4\iom\QQ^-\p_3\TP\vp~\pp_3\TP^3 v^-\cdot\NN\dx\overset{L}{=}&~4\iom\QQ^-\p_3\TP\vp~\TP^3 (\nabp\cdot v^-)\dx-4\iom\QQ^-\p_3\TP\vp~\TP^3(\cnab\cdot \vb^-)\dx\\
\lesssim&~|\TP\psi|_{L^{\infty}}\left(\ino{\sqrt{\ffp^-}\QQ^-}_{0,-}\ino{\sqrt{\ffp^-}\TP^3D_t^{\vp, -}p^-}_{0,-}+\|\cnab\TP^3\QQ^-\|_{0,-}\|\TP^4 v^-\|_{0,-}\right).
\end{aligned}
\end{equation}  Thus, combining the estimates in the above four steps, we conclude the $\TP^4$-estimate by: For the tangential derivative $\TT^\gamma=\TP^4~(\gamma_0=\gamma_4=0,~\gamma_1+\gamma_2=4)$ and for any $0<\delta< 1$, we have
\begin{equation}
\begin{aligned}
&\ino{\left(\VV^{\gamma,\pm},\BB^{\gamma,\pm},\BS^{\gamma,\pm}, \sqrt{\ffpm}\PP^{\gamma,\pm}\right)(t)}_0^2+\bno{\sqrt{\sigma}\eps^{2l}\TP^{4+l}\cnab\psi(t)}_0^2+\bno{\sqrt{\kk}\eps^{2l}\TP^{4+l}\psi(t)}_2^2+\int_0^t\bno{\sqrt{\kk}\eps^{2l}\TP^{4+l}\p_t\psi(\tau)}_1^2\dtau\\
\lesssim&~\delta E_{4+l}^\kk(t)+\bno{\eps^{2l}\psi_0}_{5.5+l}^2+\sum_{j=0}^l\int_0^t P(\sigma^{-1},E_{4+j}^\kk(\tau))\dtau,~~~0\leq l\leq 4.
\end{aligned}
\end{equation}

\begin{rmk}
It should be noted that we only have the $L^2$ control of $\VV,\BB,\BS$ and $(\ffp)^{\frac12} \PP$ in the tangential estimates, but the term $\QQ$ without $\ffp$-weight does appear in tangential estimates. When $\TT^{\gamma}$ contains at least one spatial derivative, that is, $\gamma_0<\len{\gamma}$, one can invoke the momentum equation to replace $\TT q$ by $\Dtp v$ and $\bp b$ to avoid the loss of Mach number. This also suggests that we can actually control $\|\PP\|_0$ instead of only $\|\ffp^{1/2}\PP\|_0$ when there is at least one spatial derivatives in $\TT^\gamma$. However, when $\TT^\gamma$ only consists of time derivatives, we cannot do such substitution any longer. Thus, we have to use the above cancellation structure between $ZB$ and $Z$ to control these two terms together.
\end{rmk}

\subsubsection{The case $l>0$: No loss of regularity or weights of Mach number}\label{sect E8TT}

Next we consider the tangential estimates for $\eps$-weighted spatial derivatives, namely $\eps^{2l}\TP^{4+l}$ for $1\leq l\leq 4$. The proof is parallel to the case $\TT^\gamma=\TP^4$, but we have to check the following aspects
\begin{itemize}
\item [a.] We have to guarantee that there is no loss of $\ffp$-weight in various commutators, especially those involving $q$.
\item [b.] When $l=4$, we only have tangential regularity for 8 derivatives. Due to the anisotropy of $H_*^8$, we have to put extra efforts to reduce the terms involving the derivative $\TP^7\p_3$.
\end{itemize}
We only show the detailed modifications for the case $l=4$, that is, the $\eps^8\TP^8$-estimate. When $1\leq l\leq 3$, similar modifications can be made in the same way. 

\subsubsection*{Commutators of type $\eps^8[\TP^8,f]\TT g$ for $\TT=\TP$ or $\Dtp$}
This type of commutator includes the following terms
\begin{align*}
-[\TT^\gamma,\rho]\Dtp v\text{ in }\RR_v,~~-[\TT^\gamma,\ffp]\Dtp p\text{ in }\RR_p,~~[\TT^\gamma, \vb]\cdot\TP f\text{ and }\pp_3 f[\TT^\gamma,v]\cdot \NN \text{ in }\dd^\gamma(f)
\end{align*}It is controlled directly by expanding the commutator. We have
\begin{align*}
\eps^8[\TP^8,f]\TT g=&~(\eps^8\TP^8 f)\TT g+8(\eps^6\TP^7 f)(\eps^2\TP \TT g)+28(\eps^6\TP^6f)(\eps^2\TP^2 \TT g)+56(\eps^6\TP^5f)(\eps^2\TP^3\TT g)\\
&+70(\eps^2\TP^4 f)(\eps^6\TP^4\TT g)+56(\eps^2\TP^3f)(\eps^6\TP^5\TT g)+28(\eps^2\TP^2f)(\eps^6\TP^6\TT  g)+8(\TP f)(\eps^8\TP^7\TT g),
\end{align*}whose $L^2(\Om)$ norm is controlled by
\begin{align*}
&\|\eps^8\TP^8 f\|_{0}\|\TT g\|_{L^{\infty}}+8\eps^2\|\eps^6\TP^7 f\|_0\|\TP \TT g\|_{L^{\infty}}+28\|\eps^6\TP^6 f\|_{L^6}\|\eps^2\TP^2\TT g\|_{L^3}+56\|\eps^6\TP^5 f\|_{L^6}\|\eps^2\TP^3\TT g\|_{L^3}\\
+&~70\|\eps^2\TP^4 f\|_{L^3}\|\eps^6\TP^4\TT g\|_{L^6}+56\|\eps^6\TP^3 f\|_{L^3}\|\eps^2\TP^5\TT g\|_{L^6}+28\|\eps^6\TP^6 g\|_{L^6}\|\eps^2\TP^2\TT f\|_{L^3}+8\eps^2\|\eps^6\TP^7 g\|_0\|\TP \TT f\|_{L^{\infty}}\\
\lesssim&~(1+\eps^2)\left(\sqrt{E_8^\kk(t)E_4^\kk(t})+\sqrt{E_7^\kk(t)E_4^\kk(t)}+\sqrt{E_7^\kk(t)E_5^\kk(t)}\right),
\end{align*}where we use the Sobolev embedding $H^1\hookrightarrow L^6$ and $H^1\hookrightarrow H^{1/2}\hookrightarrow L^3$ in 3D. In 2D case, we can replace $(L^6,L^3)$ by $(L^4,L^4)$ and use Lady\v{z}enskaya's inequality $\|f\|_{L^4}^2\leq \|f\|_{L^2}\|\p f\|_{L^2}\leq \|f\|_1^2$ to obtain the same bound.

Besides, we also have to treat the term $-[\TT^\gamma, b^\pm](\nabp\cdot v^\pm)$ in $\RR_b$. In fact, invoking the continuity equation converts $-(\nabp\cdot v^\pm)$ to $\ffpm \Dtpm p^\pm$ which again has the form $\eps^8[\TP^8,f]\TT g$.

\subsubsection*{Commutator $\dd(f)$ for $f=v,~p,~b,~S$}
Among all terms in \eqref{AGU comm D}, we need to further analyze the third term, that is, the commutator $\eps^{8}\left[\TP^8, \frac{1}{\p_3\vp}(v\cdot \NN-\p_t\varphi), \p_3 f\right]$ for $f=v,b,p$. The problem is the same as above, that is, $\TP^7$ may fall on $\p_3 f$ which is not directly controllable. Again, we notice that there is only one $\TP$ falling on $\frac{1}{\p_3\vp}(v\cdot \NN-\p_t\varphi)$ and $(v\cdot \NN-\p_t\varphi)|_{\Sigma}=0$, so we can use the same method (as in the control of $\eps^8[\TP^8,b]\cdot\nabp f$) to control this commutator.

\subsubsection*{Commutator $\cc(q)$}
The problematic term is $-8(\p_3\vp)^{-1}(\TP\NN_i)(\TP^7\p_3 q)$ arising from $[\TT^{\gamma}, \NN_i/\p_3\vp, \p_3 q]$. To control this term, we can invoke the third component of the momentum equation to convert $\p_3 q$ to tangential derivatives of other quantities
\[
-\p_3 q=(\p_3 \vp)\left(\rho \Dtp v_3 -\bp b_3\right),
\]where $\Dtp=\p_t+\vb\cdot\cnab+(\p_3\vp)^{-1}(v\cdot\NN-\p_t\vp)\p_3$ and $\bp=\bc\cdot\cnab+(\p_3\vp)^{-1}(b\cdot\NN)\p_3$ are both tangential derivatives. Also, there is no loss of weight of Mach number  in this term because one can always replace $\TP q$ by $\Dtp v$ and $\bp b$.

\subsubsection*{Commutators $\ccb(f)$ for $f=b,v$}
We just need to put extra effort on the term $8\TP((\p_3\vp)^{-1}(b\cdot \NN))\TP^7\p_3 f$ arising in $[\TP^8,(\p_3\vp)^{-1}(b\cdot \NN)]\p_3 f$ because the length of the multi-index in $\TP^7\p_3$ exceeds 8 when $|x_3|\lesssim 1$. (Recall that the weight function $\omega(x_3)$ is comparable to $|x_3|$ when $x_3\lesssim 1$ and is comparable to 1 when $|x_3|\gg 1$.) In this case, we notice that $b\cdot\NN|_{\Sigma}=0$, and thus its interior value can be expressed via the fundamental theorem of calculus
\[
(\p_3\vp)^{-1}(b\cdot \NN)(x',x_3)=0+\int_0^{x_3}\p_3\left((\p_3\vp)^{-1}(b\cdot \NN)(x',\xi_3)\right)\mathrm{d}\xi_3,
\]whose $L^{\infty}(\Om)$ norm is controlled by $C\omega(x_3)\|\p_3(b\cdot\NN)\|_{L^{\infty}(\Om)}$.

\subsubsection*{Commutator $\cc_i(v_i)$}
The problematic term is $-8(\p_3\vp)^{-1}(\TP\NN_i)(\TP^7\p_3 v_i)$ arising from $[\TT^{\gamma}, \NN_i/\p_3\vp, \p_3 v_i]$. In fact, this term may not be controlled independently, but its contribution only appears in $-\io\QQ \cc_i(v_i)\dvt$ which has been analyzed in step 4 of Section \ref{sect E4TT}. Specifically, its contribution in the term $Z$, after combining it with $ZB$ term, is
\[
8\eps^{16}\io\p_3(\TP^8 q-\TP^8\vp\pp_3 q)~\TP\NN_i~\TP^7 v_i\dx,
\]which is controlled by $(\|\eps^8\TP^7\p_3 q\|_{0}+|\eps^8\TP^8\psi|_0\|\p q\|_{L^{\infty}})|\TP\psi|_{W^{1,\infty}}\|\eps^8 \TP^8 v\|_0$ after integrating one $\TP$ by parts. Then we convert $\p_3 q$ to tangential derivatives of other quantities via the momentum equation, which has been presented in the control of $\cc(q)$.

Based on the above analysis, we can follow the same method as in $\TP^4$-estimate to prove the following inequality for $\eps^{2l}\TP^{4+l}$-estimates $(1\leq l\leq 4)$ for the nonlinear $\kk$-approximate problem \eqref{CMHD0kk}: For any $0<\delta< 1$ and fixed $l\in\{1,2,3,4\}$. 
\begin{equation}
\begin{aligned}
&\ino{\eps^{2l}\left(\VV^{\gamma,\pm},\BB^{\gamma,\pm},\BS^{\gamma,\pm},\sqrt{\ffpm} \PP^{\gamma,\pm}\right)(t)}_{0,\pm}^2+\bno{\sqrt{\sigma}\eps^{2l}\TP^{4+l}\cnab\psi(t)}_0^2+\bno{\sqrt{\kk}\eps^{2l}\TP^{4+l}\psi(t)}_2^2+\int_0^t\bno{\sqrt{\kk}\eps^{2l}\TP^{4+l}\p_t\psi(\tau)}_1^2\dtau\\
\lesssim&~\delta E_{4+l}^\kk(t)+\bno{\eps^{2l}\psi_0}_{5.5+l}^2+\sum_{j=0}^l\int_0^t P(\sigma^{-1},E_{4+j}^{\kk}(\tau))\dtau,
\end{aligned}
\end{equation}where $\left(\VV^{\gamma,\pm},\BB^{\gamma,\pm},\BS^{\gamma,\pm}, \PP^{\gamma,\pm}\right)$ represent that  Alinhac good unknowns of $(v^\pm,b^\pm,S^\pm,p^\pm)$ with respect to $\TP^{4+l}$.

\subsection{Tangential estimates: full time derivatives}\label{sect E8tt}

Now we control the full time derivatives, that is, the $\eps^{2l}\p_t^{4+l}$ estimates for $0\leq l \leq 4$. We will take the most difficult case $l=4$ for an example, that is, the $\eps^8\p_t^8$-estimate. The other cases $(0\leq l\leq 3)$ can be treated in the same way.  

\subsubsection{Replacing one time derivative by a material derivative}
Following the analysis in Section \ref{sect E4TT} and Section \ref{sect E8TT}, we expect to control the following norms
\[
\ino{\eps^8\left(\VV^\pm,\BB^\pm,\sqrt{\ffpm}\PP^\pm,\BS^\pm\right)}_{0,\pm}^2+\bno{\eps^8\sqrt{\sigma}\p_t^8\psi}_1^2+\bno{\eps^8\sqrt{\kk}\p_t^8\psi}_2^2+\bno{\eps^8\sqrt{\kk}\p_t^9\psi}_{L_t^2H_{x'}^1}^2,
\]which further gives the control of $\ino{\eps^8\p_t^8(v^\pm,b^\pm,\sqrt{\ffpm} p^\pm, S^\pm)}_0^2$. However, there are several extra difficulties that may make our previous method invalid.
\begin{itemize}
\item [a.] We cannot substitute $\p q$ by $\TT(v,b)$ because there is no spatial derivative.
\item [b.] $\p_t^{4+l} p$ has weight $\sqrt{\ffpm}\eps^{2l}=O(\eps^{1+2l})$ instead of $\eps^{2l}$. There might be a loss of $\eps$-weight.
\item [c.] $\sqrt{\ffpm}\eps^{2l}\p_t^{4+l} q$ only has $L^2(\Om)$ regularity, so the trace lemma is no longer valid.
\item [d.] We cannot integrate by parts for ``half-order time derivative" $\p_t^{1/2}$. Thus, the control of VS term will be rather different.
\end{itemize} 

To overcome the abovementioned difficulties, especially (c) and (d) in the control of the crucial boundary term VS, we would like to replace the full-time derivative $\p_t^{4+l}$ by $\Dtpl\p_t^{3+l}$ where $\Dtpl=\p_t+\vb^-\cdot\cnab+(\p_3\vp)^{-1}(v^-\cdot\NN-\p_t\vp)\p_3$ and $v^-|_{\Om^+}$ is defined to be the Sobolev extension of $v^-$ in $\Om^+$. We aim to prove the following estimates.
\begin{prop}\label{prop Ett}
Fix $l\in\{0,1,2,3,4\}$. We have the following uniform-in-$(\kk,\eps)$ estimate for any $0<\delta< 1$
\begin{equation}
\begin{aligned}
&\ino{\eps^{2l}\left(\VV^{*,\gamma,\pm},\BB^{*,\gamma,\pm},\BS^{*,\gamma,\pm}, (\ffpm)^{1/2}\PP^{*,\gamma,\pm}\right)(t)}_0^2+\bno{\sqrt{\sigma}\eps^{2l}\Dtpl\p_t^{3+l}\cnab\psi(t)}_0^2\\
&+\bno{\sqrt{\kk}\eps^{2l}\Dtpl\p_t^{3+l}\psi(t)}_2^2+\int_0^t\bno{\sqrt{\kk}\eps^{2l}\Dtpl\p_t^{3+l}\p_t\psi(\tau)}_1^2\dtau\\
\lesssim&~\delta E_{4+l}^\kk(t)+\sum_{j=0}^lP(E_{4+j}^\kk(0))+\int_0^t P(\sigma^{-1},E_{4+j}^\kk(\tau))\dtau,~~~0\leq l\leq 4,
\end{aligned}
\end{equation}where $\VV^{*,\gamma,\pm},\BB^{*,\gamma,\pm},\BS^{*,\gamma,\pm}, (\ffpm)^{1/2}\PP^{*,\gamma,\pm}$ represent the Alinhac good unknowns of $v^\pm,b^\pm,S^\pm,p^\pm$ with respect to $\Dtpl\p_t^{3+l}$ respectively, that is, $\FF^{*,\gamma,\pm}=\Dtpl\p_t^{3+l}f^\pm-(\Dtpl\p_t^{3+l}\vp)\pp_3 f^\pm$.
\end{prop}

For the case $l=4$, we introduce the Alinhac good unknowns with respect to $\Dtpl\p_t^7$
\[
(\VV^{*,\pm},\BB^{*,\pm},\PP^{*,\pm},\QQ^{*,\pm},\BS^{*,\pm}):=\Dtpl\p_t^7 (v^\pm,b^\pm,p^\pm,q^\pm,S^\pm)-(\Dtpl\p_t^7 \vp)\pp_3 (v^\pm,b^\pm,p^\pm,q^\pm,S^\pm).
\] They satisfy
\[
\Dtpl\p_t^7\pp_i f^\pm=\pp_i\FF^{*,\pm}+\cc_i^{*}(f^\pm),~~\Dtpl\p_t^7\Dtpl f^\pm=\Dtpl\FF^{*,\pm}+\dd_i^{*}(f^\pm),\]where $\cc^*(f),\dd^*(f)$ are defined in the same way as \eqref{AGU comm Ci}-\eqref{AGU comm B} by replacing $\TT^{\gamma}$ with $\Dtpl\p_t^7$. The boundary conditions of these good unknowns are
\begin{align}
\label{goodbcd*}\jump{\QQ^*}=\sigma\Dtbl\p_t^7\h-\kk\Dtbl\p_t^7(1-\TL)^2\psi-\kk\Dtbl\p_t^7(1-\TL)\p_t\psi - \jump{\p_3 q}\Dtbl\p_t^7\psi\\
\label{goodkbc*}\VV^{*,\pm}\cdot N=\p_t\Dtbl\p_t^7\psi+(\vb^\pm\cdot\cnab)\Dtbl\p_t^7 \psi-\Dtbpm \vb^-\cdot\cnab\p_t^7\psi-\WW^{*,\pm},
\end{align}with\begin{align}\label{goodW*} \WW^{*\pm}=(\p_3 v^\pm\cdot N)\Dtbl\p_t^7\psi + [\Dtpl\p_t^7, N_i, v_i^\pm ] ,\end{align} where we use the fact that $\Dtpm|_{\Sigma}=\Dtbpm=\p_t+\vb^\pm\cdot\cnab$. Note that $\Dtbl$ does not directly commute with $\p_t$ or $\p_i$, so there is an extra term $-\Dtbpm \vb^-\cdot\cnab\p_t^7\psi$ in the expression of $\VV^{\pm,*}\cdot N$.

\subsubsection{Analysis of the interior commutators}
Since we replaced $\p_t^8$ with $\Dtpl\p_t^7$ and $\Dtpl$ does not directly commute with $\p_3$, we need to further analyze the commutators $\cc_i(f)$ for $f=q$ and $v_i$, $\dd(f)$ for $f=v,b,p,S$ and $\ccb(f)$ for $f=b,v$. The problematic thing is that $\p_3$ may fall on $(\p_3\vp)^{-1}(v^-\cdot\NN-\p_t\vp)$ (in $\Dtpl$) and produce a normal derivative without a weight function that vanishes on $\Sigma$, which may introduce a second-order derivative in the setting of anisotropic Sobolev space. 

This problem does not appear in $\dd(f)$, as we find that such commutator has the form $(\p_3\vp)^{-1}(v\cdot\NN-\p_t\vp)[\Dtpl\p_t^7,\p_3]f$ which already includes a weight $(v\cdot\NN-\p_t\vp)$ that vanishes on $\Sigma$. In $\cc_i(f)$, according \eqref{AGU comm Ci}, we need to further analyze the term $\NN_i(\p_3\vp)^{-1}[\Dtpl\p_t^7,\p_3]f$ for $f=q,v_i$. Using $\Dtpl=\p_t+\vb^-\cdot\cnab+(\p_3\vp)^{-1}(v^-\cdot\NN-\p_t\vp)\p_3$, we have
\begin{align*}
&\NN_i(\p_3\vp)^{-1}[\Dtpl\p_t^7,\p_3]f=\NN_i(\p_3\vp)^{-1}[\Dtpl,\p_3]\p_t^7f\\
=&-\NN_i(\p_3\vp)^{-1}\p_3\vb^-\cdot\cnab\p_t^7 f +\NN_i\p_3\left((\p_3\vp)^{-1}(v^-\cdot\NN-\p_t\vp)\right)\pp_3\p_t^7 f.
\end{align*} The first term above can be directly controlled in $L^2$ because only tangential derivative falls on $\p_t^7 f$. For the second term, we can invoke the momentum equation and the continuity equation to convert this normal derivative to a tangential derivative.
\begin{itemize}
\item When $f=q$, we use $-\pp_3 q = \rho\Dtp v_3 -\bp b_3$.
\item When $f=v_i$, using $\nabp\cdot v=\cnab\cdot \vb+\pp_3 v\cdot\NN$, the continuity equation becomes $\pp_3 v\cdot\NN=-\eps^2\Dtp p - \cnab\cdot \vb$. Thus we have $\p_t^7\pp_3 v\cdot\NN=-\p_t^7(\eps^2\Dtp p + \cnab\cdot \vb)+[\p_t^7,\NN]\cdot\p_3 v$ in which both terms can be directly controlled in $\|\cdot\|_{8,*}$ norm.
\end{itemize}
As for $\ccb(f)$, the problematic terms are $(b\cdot\NN)(\p_3\vp)^{-1}\p_3\vb^-\cdot\cnab\p_t^7 f +(b\cdot\NN)\p_3\left((\p_3\vp)^{-1}(v^-\cdot\NN-\p_t\vp)\right)\pp_3\p_t^7 f.$ Again, we can use $(b\cdot\NN)|_{\Sigma}=0$ to create a weight function such that $(b\cdot\NN)\pp_3$ is actually a tangential derivative.

Also note that there is no loss of Mach number even if $\p_t^8 p$ requires one more $\eps$-weight. In fact, the only term in the commutators $\cc,\dd$ that contains  $\p_t^8 p$ is $\RR_p$, but there is an extra weight $\ffp=O(\eps^2)$ multiplying with it. Therefore, we can follow the same strategy as in Section \ref{sect E4TT} and Section \ref{sect E8TT} to analyze the interior part and obtain the following identity 
\begin{equation}\label{E8ttbdry}
\begin{aligned}
&\sum_\pm\ddt\frac{\eps^{16}}{2}\iopm \rho^\pm|\VV^{*,\pm}|^2+|\BB^{*,\pm}|^2+\ffpm(\PP^{*,\pm})^2+ \rho^\pm (\BS^{*,\pm})^2 \dvt\\
=&\ST^*+{\ST^{*}}'+\VS^*+\RT^*+\sum_{\pm}\RT^{*,\pm}+ZB^{*,\pm}+Z^{*,\pm}+R_{\Sigma}^{*,\pm}+R_{\Om}^{*,\pm},
\end{aligned}
\end{equation}where
\begin{align}
\label{def STtt8} \ST^*:=&~\eps^{16}\is \Dtbl\p_t^7\jump{q}\,\p_t\Dtbl\p_t^7\psi\dx',\\
\label{def STtt8'} {\ST^*}':=&~\eps^{16}\is \Dtbl\p_t^7\jump{q}\,(\vb^+\cdot\cnab)\Dtbl\p_t^7\psi\dx',\\
\label{def VStt8}\VS^*:=&~\eps^{16}\is\Dtbl\p_t^7 q^-\,(\jump{\vb}\cdot\cnab)\Dtbl\p_t^7\psi\dx',\\
\label{def RTtt8}\RT^*:=&-\eps^{16}\is\jump{\p_3 q}\Dtbl\p_t^7\psi\,\p_t\Dtbl\p_t^7\psi\dx',\\
\label{def RTtt8'}\RT^{*,\pm}:=&\mp\eps^{16}\is \p_3 q^\pm\,\Dtbl\p_t^7\psi\,(\vb^\pm\cdot\cnab)\Dtbl\p_t^7\psi\dx',\\
\label{def Rbtt8} R_{\Sigma}^{*,\pm}:=&\pm\eps^{16}\is\QQ^{*,\pm}\Dtbpm\vb^-\cdot\cnab\p_t^7\psi\dx',\\
\label{def ZBtt8}ZB^{*,\pm}:=&\mp\eps^{16}\is \QQ^{*,\pm}\WW^{*,\pm}\dx',~~Z^{*,\pm}=-\iopm\eps^{16}\QQ^{\pm,*}\cc_i^*(v_i^\pm)\dvt,
\end{align}and $R_{\Om}^{*,\pm}$ represents the controllable terms in the interior containing the analogues of $R_1^\pm,R_2^\pm,R_3^\pm$. Specifically, we have
\begin{align}
\eps^{-16}R_{\Om}^{*,\pm}=&\iopm\VV^{*,\pm}\cdot(\RR_v^{*,\pm}-\cc^*(q^\pm)+\ccb^*(b^\pm) )\dvt+\iopm \RR_q^{*,\pm}(\nabp\cdot\VV^{*,\pm})\dvt+\iopm \BB^{*,\pm}\cdot(\RR_b^{*,\pm}-\ccb^*(v))\dvt+\iopm \PP^{*,\pm}\RR_p^{*,\pm}\dvt\no\\
&-\frac12\iopm (\nabp\cdot v^\pm)|\BB^{*,\pm}|^2\dvt-\frac12\iopm (\Dtpm\ffpm+\ffpm\nabp\cdot v^\pm)|\PP^{*,\pm}|^2\dvt+\iopm\rho^\pm\dd^*(S^\pm)\,\BS^{*,\pm}\dvt
\end{align}
where 
\begin{align*}
 \RR_v^{*,\pm}:=&-[\Dtpl\p_t^7,\rho^\pm]\Dtpm v^\pm-\rho^\pm\dd^*(v^\pm),\q \RR_p^{*,\pm}:=-[\Dtpl\p_t^7,\ffpm]\Dtpm p^\pm -\ffpm \dd^*(p^\pm),\\
\RR_b^{*,\pm}:=&-[\Dtpl\p_t^7,b^\pm](\nabp\cdot v^\pm)-\dd^*(b^\pm),\q \RR_q^{*,\pm}:=~\QQ^{*,\pm}-\PP^{*,\pm}-b^\pm\cdot\BB^{*,\pm}.
\end{align*}These terms can be directly controlled in the same way as presented in Section \ref{sect E4TT}, so we omit the details \begin{align} \int_0^tR_{\Om}^{*,\pm}\dtau\lesssim P(E^{\kk}(0))+\int_0^t P(E_4^{\kk}(t))E_8^{\kk}(t)\dtau. \end{align}

\subsubsection{Analysis of the boundary integrals}\label{sect E8ttbdry}
Similarly as in Section \ref{sect E4TT}, we can decompose the control of these terms in the following steps.

\subsubsection*{Step 1: Boundary regularity of full time derivatives given by surface tension.}
Invoking the boundary condition \eqref{goodbcd*} for $\jump{\QQ^*}$, the term ST becomes
\begin{equation}\label{STtt80}
\begin{aligned}
\ST^*=&~\sigma\eps^{16}\is\Dtbl\p_t^7\cnab\cdot\left(\frac{\cnab\psi}{\sqrt{1+|\cnab\psi|^2}}\right)\p_t\Dtbl\p_t^7\psi\dx'\\
&-\kk\eps^{16}\is \Dtbl\p_t^7(1-\TL)^2\psi\,\p_t\Dtbl\p_t^7\psi\dx'-\kk\eps^{16}\is\Dtbl\p_t^7(1-\TL)\p_t\psi\,\p_t\Dtbl\p_t^7\psi\dx'\\
=:&\ST_0^*+\ST_{1,\kk}^{*}+\ST_{2,\kk}^{*}.
\end{aligned}
\end{equation} Commuting $\cnab\cdot$ with $\Dtbl$ and integrating $\cnab\cdot$ by parts in the mean curvature term, we get an analogous energy term contributed by surface tension as in Section \ref{sect E4TT}
\begin{equation}\label{STtt801}
\begin{aligned}
\ST_0^*=&-\frac{\sigma\eps^{16}}{2}\ddt\is\frac{|\Dtbl\p_t^7\cnab\psi|^2}{\sqrt{1+|\cnab\psi|^2}}-\frac{|\cnab\psi\cdot\Dtbl\p_t^7\cnab\psi|^2}{{\sqrt{1+|\cnab\psi|^2}}^3}\dx'+\underbrace{\sigma\eps^{16}\is \TP_i \vb^-_j\TP_j\p_t^7(\TP_i\psi/|N|)\p_t\Dtbl\p_t^7\psi\dx'}_{\ST_0^{*,R}}\\
&\underbrace{-\sigma\eps^{16}\is \frac{\Dtbl\p_t^7\TP_i\psi}{|N|}\cdot\p_t(\TP_i\vb^-_j\TP_j\p_t^7\psi)\dx'-\sigma\eps^{16}\is\frac{\cnab\psi\cdot\Dtbl\p_t^7\TP_i\psi}{|N|^3}\cnab\psi\,\p_t(\TP_i\vb^-_j\TP_j\p_t^7\psi)\dx'}_{=:\ST_1^{*,R}}\\
&\underbrace{-\sigma\eps^{16}\is\left(\left[\Dtbl\p_t^6,\frac{1}{|N|}\right]\p_t\cnab\psi+\left[\Dtbl\p_t^6,\frac{1}{|N|^3}\right]((\cnab\psi\cdot\p_t\cnab\psi)\cnab\psi)-\frac{1}{|N|^3}[\Dtbl\p_t^6,\cnab\psi]\p_t\cnab\psi\right)\cdot\p_t\cnab\Dtbl\p_t^7\psi\dx'}_{=:\ST_2^{*,R}}\\
&+\underbrace{\frac{\sigma\eps^{16}}{2}\is \p_t\left(\frac{1}{|N|}\right) \bno{\Dtbl\p_t^7\cnab\psi}^2 -\p_t\left(\frac{1}{|N|^3}\right) \bno{\cnab\psi\cdot\Dtbl\p_t^7\cnab\psi}^2\dx'}_{=:\ST_3^{*,R}}
\end{aligned}
\end{equation}
 The first line above together with the inequality \eqref{cauchy} gives the $\sqrt{\sigma}$-weighted boundary regularity as in step 2 in Section \ref{sect E4TT}. The term$\ST_1^{*,R}$ is generated by commuting $\Dtbl$ with $\cnab$ (the one falling on $\p_t^7\psi$) and is directly controlled by the energy. The term$\ST_3^{*,R}$ is controlled in the same way as$\ST_2^R$ in step 2 of Section \ref{sect ETT}. The term$\ST_2^{*,R}$ is controlled as in \cite[Section 4.6]{LuoZhang2022CWWST}. The term$\ST_0^{*,R}$ is controlled similarly as $\ST_2^{*,R}$. Thus, we conclude their estimates by 
\begin{equation}\label{STtt804}
\int_0^t\ST_0^{*,R}+\ST_1^{*,R}+\ST_2^{*,R}+\ST_3^{*,R}\dtau\lesssim P(E^\kk(0))+\int_0^tP(E^\kk(\tau))\dtau.
\end{equation}

Next we analyze the terms$\ST^*_{1,\kk},\ST^*_{2,\kk}$ involving the $\kk$-regularization terms. Note that we have to commute $\Dtbl$ with $1-\TL$ or $\jp=\sqrt{1-\TL}$ when deriving the $\sqrt{\kk}$-weighte terms. Integrating $1-\TL$ by parts in$\ST^*_{1,\kk}$
\begin{equation}\label{STtt8kk01}
\begin{aligned}
&\ST_{1,\kk}^*=-\ddt\frac{1}{2}\bno{\sqrt{\kk}\eps^8\Dtbl\p_t^7(1-\TL)\psi}_0^2\\
&\underbrace{-\kk\eps^{16}\is[\Dtbl,1-\TL]\left(\p_t^7(1-\TL)\psi\right)\,\p_t\Dtbl\p_t^7\psi\dx'}_{\ST^{*,R}_{11,\kk}}\underbrace{-\kk\eps^{16}\is\Dtbl\p_t^7(1-\TL)\psi\,\p_t\left([1-\TL,\Dtbl]\p_t^7\psi\right)\dx'}_{\ST^{*,R}_{12,\kk}}.
\end{aligned}
\end{equation}On $\Sigma$, the material derivative $\Dtpl=\Dtbl=\p_t+\vb^-\cdot\cnab$, so the commutator is
\[
[\Dtbl,1-\TL]f=[\TL,\vb^-\cdot\cnab]f=\TL\vb^-\cdot\cnab f+2\TP_i\vb^-_j\TP_j\TP_i f.
\] Then$\ST^{*,R}_{11,\kk}$ is controlled under time integral by integrating $\TP_j$ by parts in the second term
\begin{align}
\int_0^t\ST^{*,R}_{11,\kk}\dtau\eql&-\kk\eps^{16}\int_0^t\is\TL \vb^-_j\TP_j\left(\p_t^7(1-\TL)\psi\right)\,\p_t\Dtbl\p_t^7\psi\dx'\dtau\no\\
&+2\kk\eps^{16}\int_0^t\is\TP_i\vb^-_j\TP_i \left(\p_t^7(1-\TL)\psi\right)\,\TP_j\p_t\Dtbl\p_t^7\psi\dx'\dtau\no\\
\lesssim&~\delta\bno{\sqrt{\kk}\eps^8\Dtbl\p_t^8\psi}_1^2 +\frac{1}{4\delta} \int_0^t\bno{\sqrt{\kk}\eps^8\TP\p_t^7\psi}_2^2|\TP\vb^-|_{W^{1,\infty}}^2\dtau\leq \delta E_8^\kk(t)+\int_0^t P(E_8^\kk(\tau),E_4^{\kk}(\tau))\dtau. \label{STtt8kk02}
\end{align} The control of$\ST^{*,R}_{12,\kk}$ is easier because there is no term containing 9 time derivatives of $\psi$. It is directly controlled by using the $\sqrt{\kk}$-weighted boundary energy obtained above.
\[
\ST^{*,R}_{12,\kk}\lesssim\bno{\sqrt{\kk}\eps^8\Dtbl\p_t^7(1-\TL)\psi}_0\left(\bno{\sqrt{\kk}\eps^8\p_t^8\psi}_2+\bno{\sqrt{\kk}\eps^8\p_t^7\psi}_2 \right)\sqrt{E_4^{\kk}(t)}\leq E_8^{\kk}(t)\sqrt{E_4^{\kk}(t)}.
\] The control of$\ST_{2,\kk}^{*}$ is similar to$\ST_{1,\kk}^{*}$. Using $\jp^2=1-\TL$, we have
\begin{equation}\label{STtt8kk03}
\begin{aligned}
\ST_{2,\kk}^*=&-\is \bno{\sqrt{\kk}\eps^8\Dtpl\p_t^8\jp\psi}^2\dx'+\kk\eps^{16}\is[\Dtbl,\TP_i]\left(\p_t^8\TP_i\psi\right)\,\p_t\Dtbl\p_t^7\psi\dx'+\kk\eps^{16}\is\Dtbl\p_t^8\TP_i\psi\,\left([\TP_i\p_t,\Dtbl]\p_t^7\psi\right)\dx'\\
=:&-\is \bno{\sqrt{\kk}\eps^8\Dtpl\p_t^8\jp\psi}^2\dx'+\ST^{*,R}_{21,\kk}+\ST^{*,R}_{22,\kk},
\end{aligned}
\end{equation}where we use the concrete form of the commutators $[\Dtbl,\TP_i]=-\TP_i \vb_j^-\TP_jf,~~[\TP_i\p_t,\Dtbl]f=\p_t\left(\TP_i \vb_j^-\TP_j\Dtbl f\right)+\p_t \vb_j^-\TP_j\TP_i f$ to get estimates similar to$\ST^{*,R}_{11,\kk}$ and$\ST^{*,R}_{12,\kk}$  
\[
\int_0^t\ST^{*,R}_{21,\kk}+\ST^{*,R}_{22,\kk}\dtau \lesssim~\delta E_8^\kk(t)+\int_0^t P(E_8^\kk(\tau),E_4^{\kk}(\tau))\dtau.
\]

Hence, the control of$\ST^*$ in \eqref{E8ttbdry} is concluded by
\begin{equation}\label{STtt8}
\begin{aligned}
&\int_0^t\ST\dtau+\frac{1}{2}\is\frac{|\sqrt{\sigma}\eps^8\Dtbl\p_t^7\cnab\psi|^2}{{\sqrt{1+|\cnab\psi|^2}}^3}\dx'+\is\left|\sqrt{\kk}\eps^8\Dtbl\p_t^7(1-\TL)\psi\right|^2\dx'+\int_0^t\is\left|\sqrt{\kk}\eps^8\Dtbl\p_t^8\jp\psi\right|^2\dx'\dtau\\
\leq&~\delta E_8^\kk(t) +  P(E^\kk(0))+\int_0^tP(E^\kk(\tau))\dtau.
\end{aligned}
\end{equation}

The term${\ST^*}'$ is controlled in the same way as$\ST^*$ by replacing $\p_t\Dtbl\p_t^7\psi$ with $(\vb^+\cdot\cnab)\Dtbl\p_t^7\psi$. We no longer get energy terms, but we can integrate $(\vb^+\cdot\cnab)$ by parts and use symmetry and the above boundary regularity to control them. Invoking the jump condition, we have
\begin{align}
{\ST^*}'=&~\sigma\eps^{16}\Dtbl\p_t^7\h\,(\vb^+\cdot\cnab)\Dtbl\p_t^7\psi\dx'\no\\
& -\kk\eps^{16}\is\Dtbl\p_t^7(1-\TL)^2\psi\,(\vb^+\cdot\cnab)\Dtbl\p_t^7\psi\dx'-\kk\eps^{16}\is\Dtbl\p_t^7(1-\TL)\p_t\psi\,(\vb^+\cdot\cnab)\Dtbl\p_t^7\psi\dx'\no\\
=:&{\ST_0^*}'+{\ST_{1,\kk}^*}'+{\ST_{2,\kk}^*}'.
\end{align}Following the analysis \eqref{STtt801}-\eqref{STtt804}, the first term is controlled thanks to the boundary regularity and symmetric structure after integrating  $(\vb^+\cdot\cnab)$ by parts.
\begin{align}
{\ST_0^*}'\eql \frac{1}{2}\sigma\eps^{16}\is (\cnab\cdot\vb^+)\left(\frac{|\Dtbl\p_t^7\psi|^2}{|N|}-\frac{|\cnab\psi\cdot\cnab\Dtbl\p_t^7\psi|^2}{|N|^3}\right)\dx'\leq P(|\cnab\psi|_{L^{\infty}})|\vb^+|_{W^{1,\infty}}\bno{\sqrt{\sigma}\eps^8\cnab\Dtbl\p_t^7\psi}_0^2.
\end{align}
Similarly, we can use the symmetric structure to control${\ST_{1,\kk}^*}'+{\ST_{2,\kk}^*}'$. We only check the commutators arising in the control of${\ST_{1,\kk}^*}'$ as an example.
\begin{align}
{\ST_{1,\kk}^{*,R}}':=&-\kk\eps^{16}\is[\Dtbl,1-\TL]\left(\p_t^7(1-\TL)\psi\right)\,(\vb^+\cdot\cnab)\left(\Dtbl\p_t^7\psi\right)\dx'\no\\
&-\kk\eps^{16}\is\left(\Dtbl\p_t^7(1-\TL)\psi\right)\,(\vb^+\cdot\cnab)\left([1-\TL,\Dtbl]\p_t^7\psi\right)\dx'\no\\
&-\kk\eps^{16}\is\Dtbl\left(\p_t^7(1-\TL)\psi\right)\,[1-\TL,(\vb^+\cdot\cnab)]\left(\Dtbl\p_t^7\psi\right)\dx'\no\\
=:&{\ST_{11,\kk}^{*,R}}'+{\ST_{12,\kk}^{*,R}}'+{\ST_{13,\kk}^{*,R}}'.
\end{align} The control of ${\ST_{11,\kk}^{*,R}}'+{\ST_{12,\kk}^{*,R}}'$ is similar to ${\ST_{11,\kk}^{*,R}}+{\ST_{12,\kk}^{*,R}}$. We have
\begin{align}
{\ST_{11,\kk}^{*,R}}'\lesssim&~|\vb^-|_{W^{2,\infty}}|\vb^+|_{L^{\infty}}\bno{\sqrt{\kk}\eps^8\p_t^7\psi}_3\left(\bno{\sqrt{\kk}\eps^8\p_t^7\psi}_3+\bno{\sqrt{\kk}\eps^8\p_t^8\psi}_2\right)\lesssim E_4^{\kk}(t)E_8^{\kk}(t),
\end{align}and
\begin{align}
{\ST_{12,\kk}^{*,R}}'\lesssim~|\vb^-|_{W^{2,\infty}}|\vb^+|_{L^{\infty}}\left(\bno{\sqrt{\kk}\eps^8\p_t^8\psi}_2+\bno{\sqrt{\kk}\eps^8\p_t^7\psi}_3\right)^2\lesssim  E_4^{\kk}(t)E_8^{\kk}(t).
\end{align} The extra term${\ST_{13,\kk}^{*,R}}'$ is also directly controlled
\begin{align}
{\ST_{13,\kk}^{*,R}}'\lesssim&~|\vb^+|_{W^{2,\infty}}|\vb^-|_{L^{\infty}}\left(\bno{\sqrt{\kk}\eps^8\p_t^8\psi}_2+\bno{\sqrt{\kk}\eps^8\p_t^7\psi}_3\right)^2\lesssim E_4^{\kk}(t)E_8^{\kk}(t).
\end{align} Thus we have
\begin{align}
{\ST_{1,\kk}^*}'\lesssim \frac12\is(\cnab\cdot\vb^+)\bno{\sqrt{\kk}\eps^8\Dtbl\p_t^7(1-\TL)\psi}^2 +E_4^{\kk}(t)E_8^{\kk}(t).
\end{align}Similarly, we have
\begin{align}
\int_0^t{\ST_{2,\kk}^{*}}'\dtau\lesssim\delta\bno{\sqrt{\kk}\eps^8 \Dtbl\p_t^8\jp\psi}_0^2 + \int_0^t |\vb^\pm|_{W^{1,\infty}}^2\left(\bno{\sqrt{\kk}\eps^8\p_t^7\psi}_3^2+\bno{\sqrt{\kk}\eps^8\p_t^8\psi}_2^2\right)\dtau.
\end{align}Hence, we have the estimate of ${\ST^*}'$:
\begin{align}\label{STtt8'}
\int_0^t{\ST^{*}}'\dtau\lesssim\delta E_8^{\kk}(t)+\int_0^tE_8^{\kk}(\tau)E_4^{\kk}(\tau)\dtau.
\end{align} What's more, we can also control the remainder term  $R_{\Sigma}^{*,\pm}:=\pm\eps^{16}\is\QQ^{*,\pm}\Dtbpm\vb^-\cdot\cnab\p_t^7\psi\dx'.$ Indeed, we use Gauss-Green formula to write it to be an interior intergral, insert the expressions of $\QQ^{*,\pm}$ and integrate by parts $\Dtbl$ to get
\begin{align}
&\int_0^tR_{\Sigma}^{*,\pm}\dtau\eql -\eps^{16}\int_0^t\iopm\p_3\QQ^{*,\pm}\,\Dtbpm\vb^-\cdot\cnab\p_t^7\vp\dx\dtau\no\\
\eql&~\eps^{16}\int_0^t\iopm\p_3(\p_t^7 q^\pm-\p_t^7\vp\p_3 q^\pm)\Dtbpm\vb^-\cdot\Dtpl\cnab\p_t^7\vp\dx\dtau-\eps^{16}\int_0^t\iopm\p_3(\p_t^7 q^\pm-\p_t^7\vp\p_3 q^\pm)\Dtbpm\vb^-\cdot\cnab\p_t^7\vp\dx.
\end{align}Using the reduction for $\p_3 q$ again, we can control the above integral by
\begin{align}
\int_0^tR_{\Sigma}^{*,\pm}\dtau\lesssim \delta\|\eps^8\p_3\p_t^7 q^\pm\|_{0,\pm}^2+\int_0^t P(E_4^\kk(\tau))E_8^{\kk}(\tau)\dtau\lesssim\delta E_8^{\kk}(t)+\int_0^t P(E_4^\kk(\tau))E_8^{\kk}(\tau)\dtau.
\end{align}

\subsubsection*{Step 2: Control of VS term.}

Now we start to analyze the most difficult boundary term 
\begin{align}\label{VStt80}
\VS^*:=~\eps^{16}\is\Dtbl\p_t^7 q^-\,(\jump{\vb}\cdot\cnab)\Dtbl\p_t^7\psi\dx'.
\end{align} Note that there is no spatial derivative $\TP$ in$\VS^*$, so we cannot integrate $\TP^{1/2}$ by parts as in step 3 in Section \ref{sect E4TT}. To overcome this difficulty, we try to rewrite the term $\Dtbl\p_t^7\psi$ by invoking the kinematic boundary condition
\begin{align*}
\Dtbl\p_t^7\psi=&~\p_t^8\psi+\vb^-\cdot\cnab\p_t^7\psi=\p_t^7(v^-\cdot N)-v^-\cdot\p_t^7 N=\p_t^7 v^-\cdot N+[\p_t^{7}, N_i, v_i^-],
\end{align*}and thus 
\begin{align}
\VS^*=&~\eps^{16}\is\Dtbl \p_t^7 q^-\,(\jump{\vb}\cdot\cnab)\p_t^7 v^-\cdot N\dx' \no\\
&+\eps^{16}\is\Dtbl \p_t^7 q^-\,\p_t^7 v^-\cdot (\jump{\vb}\cdot\cnab)N\dx'+\eps^{16}\is\Dtbl \p_t^7 q^-\,(\jump{\vb}\cdot\cnab)[\p_t^{7}, N_i, v_i^-]\dx' \no\\
=:&\VS^*_0+\VS^{*,ZB}_1+\VS^{*,ZB}_2 \label{VStt81}
\end{align}

Using divergence theorem, we convert$\VS^*_0$ to an interior integral in $\Om^-$
\begin{align}
\VS_0^*=&~\eps^{16}\iom\Dtpl\p_t^7 q^-\,\nabp\cdot\left((\jump{\vb}\cdot\cnab)\p_t^7 v^-\right)\dvt+\eps^{16}\iom\pp_i\Dtpl\p_t^7 q^-\,(\jump{\vb}\cdot\cnab)\p_t^7 v^-_i\dvt =:\VS^*_{01}+\VS^*_{02}, \label{VStt82}
\end{align}where $\jump{\vb}=\vb^+-\vb^-$ is defined via Sobolev extension in $\Om^-$. In$\VS^*_{01}$, we want to commute $\nabp\cdot$ with $(\jump{\vb}\cdot\cnab)$ in order to get a similar cancellation structure as in $ZB+Z$. The commutator is 
\begin{align*}
[\pp_i,\jump{\vb}\cdot\cnab]f=&~\pp_i\jump{\vb}\cdot\cnab f-(\jump{\vb}\cdot\cnab)\NN_i\pp_3 f + \NN_i \frac{(\jump{\vb}\cdot\cnab)\p_3\vp}{\p_3 \vp}\pp_3 f,~~i=1,2,\\
[\pp_3,\jump{\vb}\cdot\cnab]f=&~\pp_3\jump{\vb}\cdot\cnab f+\frac{(\jump{\vb}\cdot\cnab)\p_3\vp}{\p_3 \vp}\pp_3 f.
\end{align*} Commuting $\nabp\cdot$ with $(\jump{\vb}\cdot\cnab)$, we get
\begin{align}
\VS^*_{01}=&~\eps^{16}\iom\Dtpl\p_t^7 q^-\,(\jump{\vb}\cdot\cnab)\left(\nabp\cdot\p_t^7 v^-\right)\dvt-\eps^{16}\iom\Dtpl\p_t^7 q^-\,\p_3\p_t^7 v^-\cdot (\jump{\vb}\cdot\cnab)\NN\dx\no\\
&+\eps^{16}\iom\Dtpl\p_t^7 q^-\,\pp_i\jump{\vb}\cdot\cnab\p_t^7 v^-_i\dvt+\eps^{16}\iom\Dtpl\p_t^7 q^-\, (\jump{\vb}\cdot\cnab)\p_3\vp\,\pp_3\p_t^7 v^-\cdot\NN\dx\no\\
=:&\VS^*_{011}+\VS^{*,Z}_{011}+\VS^{*,R}_{011}+\VS^{*,R}_{012} \label{VStt8011}
\end{align} 

Next we introduce $\FF^\sh:=\p_t^7 f-\p_t^7\psi\pp_3 f$ to be the Alinhac good unknown of $f$ with respect to $\p_t^7$ in order to commute $\nabp$ with $\p_t^7$. Namely, we have \[\p_t^7\pp_i f=\pp_i\FF^{\sh}+\cc_i^{\sh}(f),~~ \p_t^7\Dtp f=\Dtp\FF^{\sh}+\dd^{\sh}(f),~~ \p_t^7\bp f=\bp\FF^{\sh}+\ccb^{\sh}(f)\]where $\cc^\sh,\dd^\sh,\ccb^\sh$ are defined in the same way as \eqref{AGU comm Ci}-\eqref{AGU comm B} with $\TT^\gamma=\p_t^7$. With this formulation, we have 
\begin{align*}
\nabp\cdot \p_t^7 v^-=\nabp\cdot\VV^{\sh,-}+\pp_i(\p_t^7\vp\pp_3 v_i^-)=\p_t^7(\nabp\cdot v^-)-\cc_i^\sh(v_i^-)+\pp_i(\p_t^7\vp\pp_3 v_i^-).
\end{align*}Now we insert the good unknowns in$\VS^*_{011}$ to get
\begin{align}
\VS_{011}^*=&~\eps^{16}\iom\Dtpl\p_t^7 q^-\, (\jump{\vb}\cdot\cnab)\p_t^7(\nabp\cdot v^-)\dvt\underbrace{-\eps^{16}\iom\Dtpl\p_t^7 q^-\,(\jump{\vb}\cdot\cnab)\left(\cc_i^\sh(v_i^-)-\pp_i(\p_t^7\vp\pp_3 v_i^-)\right)\dvt}_{\VS^{*,Z}_{012}}\no\\
=&-\eps^{16}\iom\ffpm\Dtpl\p_t^7 p^-\,(\jump{\vb}\cdot\cnab) \p_t^7\Dtpl p^-\dvt+\eps^{16}\iom\Dtpl\p_t^7 (\frac12|b^-|^2)\, (\jump{\vb}\cdot\cnab)\p_t^7(\nabp\cdot v^-)\dvt +\VS^{*,Z}_{012}\no \\
=:&\VS^{*}_{0111}+\VS^{*,B}_{0111}+\VS^{*,Z}_{012}. \label{VStt8012}
\end{align} By the definition of $\PP^{\sh,-}$
\[
\Dtpl\p_t^7 p^-=\Dtpl\PP^{\sh,-}+\Dtpl(\p_t^7\vp\pp_3 p^-),~~\p_t^7\Dtpl p^-=\Dtpl\PP^{\sh}+\dd^{\sh}(p^-).
\]Then we integrate $(\jump{\vb}\cdot\cnab)$ by parts and use symmetry to find
\begin{align}
\VS_{0111}^*=&-\frac{1}{2}\iom(\cnab\cdot\jump{\vb})(\sqrt{\ffpm}\eps^{8}\Dtpl\PP^{\sh,-})^2\dvt \no\\
&\underbrace{+\eps^{16}\iom\ffpm\Dtpl\PP^{\sh,-}\left((\cnab\cdot\jump{\vb})\Dtpl(\p_t^7\vp\pp_3 p^-)-(\jump{\vb}\cdot\cnab)\left(\dd^{\sh}(p^-)-\Dtpl(\p_t^7\vp\pp_3 p^-)\right)\right)\dvt}_{\VS^{*,R}_{0111}}\no\\
&\underbrace{-\eps^{16}\iom\ffpm\Dtpl(\p_t^7\vp\pp_3 p^-)\,(\jump{\vb}\cdot\cnab)\dd^{\sh}(p^-)\dvt}_{\VS^{*,R}_{0112}}, \label{VStt8013}
\end{align}where the first term on the right side is controlled by $\ino{(\ffpm)^{\frac12}\eps^8\PP^{*,-}}_0^2\|\cnab \jump{\vb}\|_{L^{\infty}}^2$. Next we adapt the analysis for $Z^\pm+ZB^\pm$ term to the control of $\VS_1^{*,ZB}+\VS_{011}^{*,Z}$ and $\VS_2^{*,ZB}+\VS_{012}^{*,Z}$. Using Gauss-Green formula, integrating $\Dtpl$ by parts under time integral and invoking the momentum equation, we get
\begin{align}
\int_0^t\VS_1^{*,ZB}+\VS_{011}^{*,Z}=&~\eps^{16}\iom(\Dtpl\pp_3\p_t^7 q^-)\,(\jump{\vb}\cdot\cnab)\NN\cdot \p_t^7 v^-\dvt\dtau\no\\
&+\int_0^t\underbrace{\eps^{16}\iom[\pp_3,\Dtpl](\p_t^7 q^-)\,(\jump{\vb}\cdot\cnab)\NN\cdot \p_t^7 v^-+\Dtpl\p_t^7 q^-\,\pp_3\left((\jump{\vb}\cdot\cnab)\NN\right)\cdot \p_t^7 v^-\dvt}_{=:\VS_{1}^{*,ZR}}\dtau\no\\
\eql&-\eps^{16}\int_0^t\iom(\pp_3\p_t^7 q^-)\,\Dtpl\left((\jump{\vb}\cdot\cnab)\NN\cdot \p_t^7 v^-\right)+\VS_{1}^{*,ZR}\dvt\dtau+\eps^{16}\iom(\pp_3\p_t^7 q^-)\,(\jump{\vb}\cdot\cnab)\NN\cdot \p_t^7 v^-\dvt\bigg|^t_0\no\\
 \lesssim&~ \delta E_8^{\kk}(t)+P(E_4^{\kk}(0))E_8^{\kk}(0)+\int_0^t P(E_4^{\kk}(\tau))E_8^{\kk}(\tau)\dtau,~~~\forall \delta\in(0,1). \label{VStt8015}
\end{align} 

For$\VS_2^{*,ZB}+\VS_{012}^{*,Z}$, we recall that the term $\cc_i^\sh(v_i^-)$ in$\VS_{012}^{*,Z}$ includes a term $[\p_t^7,\NN_i/\p_3\vp,v_i^-]$ which also appears in$\VS_2^{*,ZB}$. Thus we can again use the Gauss-Green formula to analyze this term. In fact, the commutator in$\VS_{012}^{*,Z}$ can be written as:
\begin{align*}
\cc_i^\sh(v_i^-)-\pp_i(\p_t^7\vp\pp_3 v_i^-)=\frac{1}{\p_3\vp}\left[\p_t^7,\NN_i,\p_3 v_i^-\right]+\cc^{\sh,R}_i(v_i^-)
\end{align*}where the $L^2(\Om^\pm)$ norm of $(\jump{\vb}\cdot\cnab)\cc^{\sh,R}_i(v_i^-)$ is directly controlled by $P(E^\kk(t))$ Then
\begin{align}
\VS_2^{*,ZB}+\VS_{012}^{*,Z}=&\eps^{16}\iom \pp_3\left(\Dtpl\p_t^7 q^- \right)(\jump{\vb}\cdot\cnab)[\p_t^7,\NN_i,v_i^-]\dx\underbrace{-\eps^{16}\iom \Dtpl\p_t^7 q^-\,(\jump{\vb}\cdot\cnab)\cc^{\sh,R}_i(v_i^-)\dvt}_{\VS_{2}^{*,ZR}},\label{VStt8016}
\end{align}where the first term on the right side is again controlled by integrating $\Dtpl$ by parts under time integral. We omit the details and just list the result
\[
\int_0^t\iom \eps^{16}\pp_3\left(\Dtpl\p_t^7 q^- \right)(\jump{\vb}\cdot\cnab)[\p_t^7,\NN_i,v_i^-]\dx\dtau\leq\delta E_8^{\kk}(t)+P(E_4^{\kk}(0))E_8^{\kk}(0)+\int_0^t P(E_4^{\kk}(\tau))E_8^{\kk}(\tau)\dtau,~~~\forall \delta\in(0,1).
\]

Now the term$\VS_1^*$ is controlled except for those remainder terms$\VS^{*,R}_{011}$,$\VS^{*,R}_{012}$,$\VS^{*,B}_{0111}$,$\VS^{*,R}_{0111}$,$\VS^{*,R}_{0112}$, $\VS_{1}^{*,ZR}$ and$\VS_{2}^{*,ZR}$. In fact, apart from$\VS^{*,B}_{0111}$, the other remainder terms can be directly controlled by counting the number of derivatives and invoking the reduction for $\pp_3 \p_t^7v^-\cdot\NN$ and $\pp_3\p_t^7 q^-$. There is no loss of Mach number in these remainder terms. In fact, when $\p_t^8 p^-$ appears in the remainder terms, either we have $\eps^{16}\ffpm$-weight to control it directly, or we can integrate by parts $\Dtpl$ and $(\jump{\vb}\cdot\cnab)$ under time integral to move one time derivative to $v_i^-$. Besides, the control of $\p_t^7\vp,\p_t^8\vp$ depends on the boundary regularity contributed by surface tension and so depends on $\sigma^{-1}$. Therefore, we can conclude the estimates of $\VS_1^*$ by
\begin{equation}\label{VStt801}
\VS_{01}^*+\VS^{*,ZB}_1+\VS^{*,ZB}_2 \leq \VS^{*,B}_{0111}+\delta E_8^{\kk}(t)+P(E_4^{\kk}(0))E_8^{\kk}(0)+\int_0^t P(\sigma^{-1}, E^{\kk}(\tau))E_8^{\kk}(\tau)\dtau~~~\forall \delta\in(0,1).
\end{equation}

Next we control$\VS_{02}^*=\eps^{16}\iom\pp_i\Dtpl\p_t^7 q^-\,(\jump{\vb}\cdot\cnab)\p_t^7 v^-_i\dvt$. First, we commute $\Dtpl$ with $\pp_i$ to get
\begin{align}
\VS_{02}^*=&~\eps^{16}\iom\Dtpl\pp_i\p_t^7 q^-\,(\jump{\vb}\cdot\cnab)\p_t^7 v^-_i\dvt+\eps^{16}\iom\pp_iv_j^-\pp_j\p_t^7 q^-\,(\jump{\vb}\cdot\cnab)\p_t^7 v^-_i\dvt\no\\
=:&\VS_{021}^{*}+\VS_{021}^{*,R}. \label{VStt8020}
\end{align} In the first term, we integrate by parts $\Dtpl$ under time integral and commute $\Dtpl$ with $(\jump{\vb}\cdot\cnab)$ to get
\begin{align}
\int_0^t\VS_{021}^*\dtau\eql&-\eps^{16}\int_0^t\iom\pp_i\p_t^7 q^-\,(\jump{\vb}\cdot\cnab)\Dtpl\p_t^7 v^-_i\dvt\dtau+\eps^{16}\iom\pp_i\p_t^7 q^-\,(\jump{\vb}\cdot\cnab)\p_t^7 v^-_i\dvt\bigg|^t_0\no \\
&-\eps^{16}\int_0^t\iom\pp_i\p_t^7 q^-\,[\Dtpl , (\jump{\vb}\cdot\cnab)]\p_t^7 v^-_i\dvt\dtau\no \\
=:&\int_0^t\VS_{0211}^{*}\dtau+\VS_{022}^{*,R}+\int_0^t\VS_{023}^{*,R}\dtau.  \label{VStt8021}
\end{align} Next we insert the good unknowns $\QQ^{\sh,-}$ and $\VV^{\sh,-}$ and invoke again the momentum equation  to get
\begin{align}
\VS_{0211}^{*}=&~\eps^{16}\iom\rho^-\Dtpl \VV^{\sh,-}\cdot(\jump{\vb}\cdot\cnab)\Dtpl\VV^{\sh,-}\dvt-\eps^{16}\iom\bpl\BB^{\sh,-}\cdot(\jump{\vb}\cdot\cnab)\Dtpl\VV^{\sh,-}\dvt\no\\
&+\eps^{16}\iom\left(\cc^{\sh}(q^-)-\RR_v^{\sh,-}-\ccb^\sh(b)\right)\cdot(\jump{\vb}\cdot\cnab)\Dtpl\VV^{\sh,-}\dvt+\eps^{16}\iom\pp_i(\p_t^7\vp\pp_3 q^-)\,(\jump{\vb}\cdot\cnab)\Dtpl\p_t^7 v^-_i\dvt\no\\
=:&~\eps^{16}\iom\rho^-\Dtpl \VV^{\sh,-}\,(\jump{\vb}\cdot\cnab)\Dtpl\VV^{\sh,-}\dvt+\VS_{0211}^{*,B}+\VS_{0211}^{*,R}+\VS_{0212}^{*,R},  \label{VStt8022}
\end{align} where the first term is again controlled by integrating by parts in $(\jump{\vb}\cdot\cnab)$ and using symmetry
\begin{align}
\eps^{16}\iom\rho^-\Dtpl \VV^{\sh,-}\,(\jump{\vb}\cdot\cnab)\Dtpl\VV^{\sh,-}\dvt=\frac{\eps^{16}}{2}\iom\left(\cnab\cdot(\rho^-\jump{\vb})\right)\bno{\Dtpl\VV^{\sh,-}}^2\dvt\leq P(E_4^{\kk}(t))E_8^{\kk}(t).
\end{align} Next we wish to combine$\VS^{*,B}_{0211}$ with$\VS^{*,B}_{0111}:=\eps^{16}\iom\Dtpl\p_t^7 (\frac12|b^-|^2)\, (\jump{\vb}\cdot\cnab)\p_t^7(\nabp\cdot v^-)\dvt$ to get a cancellation structure. In$\VS^{*,B}_{0111}$, we invoke the evolution equation $\Dtpl b_j^-=\bpl v^- - b^-(\nabp\cdot v^-)$ to get 
\begin{align}
\VS_{0111}^{*,B}\overset{L}{=}&-\eps^{16}\iom\Dtpl\BB^{\sh,-}_j\,(\jump{\vb}\cdot\cnab)\Dtpl \BB_j^{\sh,-}\dvt+\eps^{16}\iom\Dtpl\BB^{\sh,-}_j\,(\jump{\vb}\cdot\cnab)\p_t^7(\bpl v_j^-)\dvt\no\\
&-\eps^{16}\iom\Dtpl\BB^{\sh,-}_j\,(\jump{\vb}\cdot\cnab)\dd^\sh(b_j^-)\dvt+\underbrace{\eps^{16}\iom\Dtpl\BB^{\sh,-}_j\,\left[b_j,(\jump{\vb}\cdot\cnab)\p_t^7\right](\nabp\cdot v^-)\dvt}_{\VS^{*,BR}_{0111}},  \label{VStt8023}
\end{align} where the first term on the right side is again controlled by integrating by parts in $(\jump{\vb}\cdot\cnab)$ and using symmetry, and the third term on the right side is controlled directly after inserting the expression of $\dd^{\sh}(b)$. We denote $$\VS_{0112}^{*,B}:=\eps^{16}\iom\Dtpl\BB^{\sh,-}_j\,(\jump{\vb}\cdot\cnab)\p_t^7(\bpl v_j^-)\dvt$$ to be the second term on the right side above. Inserting the good unknown $\VV^{\sh,-}$, the term$\VS_{0112}^{*,B}$ is equal to
\begin{align}
&\eps^{16}\iom\Dtpl\BB^{\sh,-}_i\,(\jump{\vb}\cdot\cnab)\left(\bpl\VV_i^{\sh,-}\right)\dvt+\underbrace{\eps^{16}\iom\Dtpl\BB^{\sh,-}_i\,(\jump{\vb}\cdot\cnab)\left([\p_t^7, b_j^-]\pp_j v_i^-+b_j^-\cc_j^{\sh}(v_i^-)\right)\dvt}_{\VS^{*,BR}_{0112}}\no \\
=&~\eps^{16}\iom\Dtpl\BB^{\sh,-}_i\,\bpl\left((\jump{\vb}\cdot\cnab)\VV_i^{\sh,-}\right)\dvt-\eps^{16}\iom\Dtpl\BB^{\sh,-}_i\,\left[\bpl,(\jump{\vb}\cdot\cnab)\right]\VV_i^{\sh,-}\dvt+\VS^{*,BR}_{0112}\no\\
=:&\VS_{0113}^{*,B}+\VS^{*,BR}_{0113}+\VS^{*,BR}_{0112}.  \label{VStt8024}
\end{align} 
Now we can integrate by parts $\bpl$ and then $\Dtpl$ in$\VS_{0113}^{*,B}$ in order to produce the cancellation with$\VS_{0211}^{*,B}$. Under time integral, $\int_0^t\VS_{0113}^{*,B}\dtau$ is equal to
\begin{align}
&\int_0^t\underbrace{\eps^{16}\iom\bpl\BB^{\sh,-}_i\,(\jump{\vb}\cdot\cnab)\Dtpl\VV_i^{\sh,-}\dvt}_{=-\VS_{0211}^{*,B}}\dtau+\eps^{16}\iom\bpl\BB^{\sh,-}_i\,(\jump{\vb}\cdot\cnab)\VV_i^{\sh,-}\dvt\bigg|^t_0 \no\\
&+\eps^{16}\int_0^t\iom[\bpl,\Dtpl]\BB^{\sh,-}_i\,(\jump{\vb}\cdot\cnab)\VV_i^{\sh,-}\dvt\dtau+\eps^{16}\int_0^t\iom\bpl\BB^{\sh,-}_i\,[\Dtpl,(\jump{\vb}\cdot\cnab)]\VV_i^{\sh,-}\dvt\dtau\no\\
=:&-\int_0^t\VS_{0211}^{*,B}\dtau+\VS^{*,BR}_{0211}+\int_0^t\VS^{*,BR}_{0212}+\VS^{*,BR}_{0213}\dtau. \label{VStt8025}
\end{align} Note that $[\Dtpl,\bpl]=-(\nabp\cdot v^-)\bpl f$ and when we commute $(\jump{\vb}\cdot\cnab)$ with either $\Dtpl$ or $\bpl,$, no normal derivative will be generated because the weight functions in front of $\p_3$ (namely, $b^-\cdot\NN$ and $(v^-\cdot\NN-\p_t\vp)$) are still vanishing on the interface $\Sigma$ after taking $(\jump{\vb}\cdot\cnab)$. Therefore, the commutators above are all controllable in $\|\cdot\|_{8,*,-}$ norm and no loss of Mach number occurs. The following remainder terms are controlled directly
\begin{align}
&\VS_{022}^{*,R}+\VS^{*,BR}_{0211}+\int_0^t\VS_{021}^{*,R}+\VS_{023}^{*,R}+\VS^{*,BR}_{0111}+\VS^{*,BR}_{0112}+\VS^{*,BR}_{0212}+\VS^{*,BR}_{0213}\dtau \no\\
\leq&~ \delta E_8^{\kk}(t) +P(E^{\kk}(0))+\int_0^t P(E^{\kk}(\tau))\dtau.  \label{VStt8026}
\end{align} In the terms$\VS_{0211}^{*,R}+\VS_{0212}^{*,R}$, we can integrate $(\jump{\vb}\cdot\cnab)$ by parts to get to get the desired control thanks to the $\sqrt{\sigma}$-weighted boundary regularity of $\psi$
\begin{align}
\VS_{0211}^{*,R}+\VS_{0212}^{*,R}\lesssim_{\sigma^{-1}} \left(|\eps^{8}\p_t^7\psi|_2+|\eps^{8}\p_t^8\psi|_1\right)\|\eps^{8}\Dtpl\p_t^7 v\|_{0,-}P(E_4^\kk(t)).  \label{VStt8027}
\end{align}
Thus, the control of $\VS_{02}^*$ term is concluded by
\begin{align} \label{VStt802}
\int_0^t\VS_{02}^{*}+\VS_{0111}^{*,B}\dtau\leq\delta E_8^{\kk}(t)+P(E^{\kk}(0)) +\int_0^t P(\sigma^{-1}, E^{\kk}(\tau))\dtau.
\end{align} Finally, combining \eqref{VStt80}, \eqref{VStt81}, \eqref{VStt801} and \eqref{VStt802}, we get the estimate of$\VS^*$ term
\begin{align}\label{VStt8}
\int_0^t\VS^{*}\dtau\leq\delta E_8^{\kk}(t) +P(E^{\kk}(0))+\int_0^t P(\sigma^{-1}, E^{\kk}(\tau))\dtau.
\end{align}

\subsubsection*{Step 3: Control of RT term.}

In step 3, we control the terms$\RT^*$ and $\RT^{*,\pm}$ defined in \eqref{def RTtt8}-\eqref{def RTtt8'}, The latter one can be directly controlled by using symmetry
\begin{align}\label{RTtt8'}
\RT^{*,\pm}=\mp\frac12\is(\cnab\cdot(\p_3 q^\pm\,v^\pm))\bno{\Dtbl\p_t^7\psi}^2\dx'\leq\,\sigma^{-1}E_4^{\kk}(t)E_8^{\kk}(t).
\end{align} The term$\RT^*=-\eps^{16}\is\jump{\p_3 q}\Dtbl\p_t^7\psi\,\p_t\Dtbl\p_t^7\psi\dx'$ cannot be controlled in the same way as in the estimates of spatial derivatives because we do not have $L^2(\Sigma)$-control for $\p_t\Dtbl\p_t^7\psi$ without $\kk$-weight nor can we integrate by parts $\p_t^{1/2}$. To overcome this difficulty, we need to invoke the kinematic boundary condition to reduce the number of time derivatives. We have
\[
\Dtbl\p_t^7\psi=\p_t^7v^-\cdot N +[\p_t^7,v^-,N],~~\p_t\Dtbl\p_t^7\psi=\p_t^8v^-\cdot N+8\p_t^7 v^-\cdot\p_t N+\text{ lower order terms}.
\] Plugging it to$\RT^*$, we find
\begin{align}
\RT^*\eql-\eps^{16}\is\jump{\p_3 q}\p_t^7v^-\cdot N\,\p_t^8 v^-\cdot N\dx'-8\eps^{16}\is\jump{\p_3 q}\p_t^7v^-\cdot N\,\p_t^7v^-\cdot\p_t N\dx'=:\RT^*_1+\RT^*_2.
\end{align} The term$\RT^*_2$ can be controlled by using Gauss-Green formula
\begin{align}
\RT^*_2\eql&~-8\eps^{16}\iom\jump{\p_3 q}(\pp_3\p_t^7 v^-\cdot \NN)(\p_t^7v^-\cdot \p_t\NN)\dvt - 8\eps^{16}\iom\jump{\p_3 q}(\p_t^7 v^-\cdot  \NN)(\p_3\p_t^7v^-\cdot  \p_t\NN)\dx,
\end{align}where $\jump{\p_3 q}$ is defined via Sobolev extension.  The first term above is directly controlled after invoking the reduction $\pp_3\p_t^7v^-\cdot \NN\eql -\p_t^7(\eps^2\Dtpl p^- + \cnab\cdot \vb^-)$. For the second term, it suffices to integrate $\p_t$ by parts under time integral
\begin{align}
&- 8\eps^{16}\int_0^t\iom\jump{\p_3 q}(\p_t^7 v^-\cdot  \NN)(\p_3\p_t^7v^-\cdot  \p_t\NN)\dx\dtau\no\\
\eql& -8 \eps^{16}\iom\jump{\p_3 q}(\p_t^7 v^-\cdot  \NN)(\p_3\p_t^6v^-\cdot  \p_t\NN)\dx\bigg|^t_0 + 8\int_0^t\iom \jump{\p_3 q}(\p_t^8v^-\cdot\NN)(\p_3\p_t^6v^-\cdot  \p_t\NN)\dx\no\\
\lesssim&~\delta\|\eps^8\p_3\p_t^6 v^-\|_{0,-}^2+P(E^{\kk}(0))+\int_0^t P(E_4^{\kk}(\tau))E_8^{\kk}(\tau)\dtau
\end{align}

Using the same trick as above, the term$\RT^*_1$ is directly controlled by repeated invoking $\pp_3\p_t^7v^-\cdot \NN\eql -\p_t^7(\eps^2\Dtpl p^- + \cnab\cdot \vb^-)$
\begin{align}
\int_0^t\RT^*_1\dtau\overset{\p_t,~L}{===}&\int_0^t\iom\jump{\p_3 q}\left((\pp_3\p_t^7 v^-\cdot  \NN)(\p_t^8v^-\cdot  \NN)-\p_t(\p_t^7v^-\cdot\NN)(\pp_3\p_t^7v^-\cdot  \NN)\right)\dvt\dtau\no\\
&-  \eps^{16}\iom\jump{\p_3 q}(\p_t^7v^-\cdot\NN)(\pp_3\p_t^7v^-\cdot  \NN)\dvt\bigg|^t_0\no\\
\lesssim&~\delta E_8^{\kk}(t)+P(E^{\kk}(0))+\int_0^t P(E_4^{\kk}(\tau))E_8^{\kk}(\tau)\dtau.
\end{align} Hence, we conclude the estimate of$\RT^*$ by
\begin{align}\label{RTtt8}
\int_0^t\RT^*\dtau\lesssim\delta E_8^{\kk}(t)+P(E^{\kk}(0))+\int_0^t P(E_4^{\kk}(\tau))E_8^{\kk}(\tau)\dtau.
\end{align}

\subsubsection*{Step 4: The cancellation structure between $ZB^*$ and $Z^*$.}
Now we control the term $ZB^{*,\pm}+Z^{*,\pm}$. Note that we cannot integrate by parts $\TP^{1/2}$ due to the lack of spatial derivatives. First, $ZB^{*,\pm}$ can be written as
\begin{align}\label{ZBtt80}
ZB^{*,\pm}=&\mp\eps^{16}\is\Dtbl\p_t^7 q^\pm(\p_3 v^\pm\cdot N)\Dtbl\p_t^7\psi\dx' \pm \eps^{16}\is\Dtbl\p_t^7\psi\,\p_3 q^\pm(\p_3 v^\pm\cdot N)\Dtbl\p_t^7\psi\dx' \no\\
&\mp\eps^{16}\is\QQ^{*,\pm}\,\left[\Dtpl\p_t^7,N_i,v_i^\pm\right]\dx'\no\\
=:&~ZB^{*,R,\pm}_1+ZB^{*,R,\pm}_2 +ZB^{*,\pm}_0.
\end{align} The second term on the right side can be directly controlled. We have
\begin{align}
ZB^{*,R,\pm}_2\leq \left|\Dtpl\p_t^7\psi\right|_0^2 \bno{\p_3 q^\pm\,(\p_3v^\pm\cdot N)}_{L^{\infty}}\leq P(\sigma^{-1}, E_4^{\kk}(t))E_8^{\kk}(t).
\end{align}For the first term, using again $\Dtbl\p_t^7\psi=\p_t^7 v\cdot N+$lower order terms, Gauss-Green formula and integrating by parts in $\Dtpl$, we get
\begin{equation}\label{ZBtt801}
\begin{aligned}
\int_0^tZB^{*,R,\pm}_1\dtau\eql&~\eps^{16}\iopm(\p_3v^\pm\cdot\NN)\pp_3\p_t^7q^\pm\,\p_t^7v^\pm\cdot \NN\dvt\bigg|^t_0+\eps^{16}\int_0^t\iopm(\p_3v^\pm\cdot\NN)\left([\pp_3,\Dtpl]\p_t^7q^\pm\right)\,\p_t^7v^\pm\cdot \NN\dvt\dtau\\
&+\eps^{16}\int_0^t\iopm(\p_3v^\pm\cdot\NN)\left(\pp_3\p_t^7q^\pm\,\Dtpl\p_t^7v^\pm\cdot \NN+\Dtpl\p_t^7q^\pm\,\pp_3\p_t^7v^\pm\cdot \NN\right)\dvt +\text{l.o.t}
\end{aligned}
\end{equation} Now we can invoke the reduction for $\pp_3 q$ and $\pp_3v\cdot\NN$ to convert $\pp_3$ to a tangential derivative. Note that the continuity equation above produces an extra $\ffp=O(\eps^2)$ weight, so there is no loss of Mach number when $\Dtpl\p_t^7 q$ appears. When $\Dtpl\p_t^7 q$ is multiplied by $\p_t^7\cnab\cdot\vb$, we can further integrate by parts in $\p_t$ and then in $\cnab\cdot$ to move one time derivative to $v$. Hence, $ZB^{*,R,\pm}_1$ is controlled in $\|\cdot\|_{8,*}$ norm without loss of $\eps$-weights
\begin{align}
\int_0^tZB^{*,R,\pm}_1\dtau \lesssim~\delta E_8^{\kk}(t)+P(E^{\kk}(0))+\int_0^t P(E^{\kk}(\tau))\dtau.
\end{align}

Next we will see again the cancellation structure in $ZB_0^{*,\pm}+Z^{*,\pm}$. From \eqref{AGU comm Ci}, we find it suffices to further analyze the following two terms
\begin{align}
\left[\Dtpl\p_t^7,\frac{\NN_i}{\p_3\vp},\p_3 v_i\right]\eql&~ (\p_3\vp)^{-1}\left[\Dtpl\p_t^7,\NN_i,\p_3 v_i\right] - \pp_3\p_t\vp\,\pp_3\Dtpl\p_t^6 v\cdot\NN,\\
(\p_3\vp)^{-1}\NN\cdot[\Dtpl\p_t^7,\p_3]v= &~\left(\pp_3\vb^-\cdot\cnab\right)\p_t^7v\cdot\NN+\p_3\left((\p_3\vp)^{-1}(v^-\cdot\NN-\p_t\vp)\right)\pp_3\p_t^7v\cdot\NN.
\end{align} Thus, we find that, apart from the term $(\p_3\vp)^{-1}\left[\Dtpl\p_t^7,\NN_i,\p_3 v_i\right]$, all the other terms in $\cc_i^*(v_i)$ include either a tangential derivative falling on the leading order term or the term $\pp_3v\cdot\NN$ (possibly with some derivatives) such that $\ffp\Dtp p$ and $\cnab\cdot\vb$ are produced by invoking the continuity equation. Thus, when $\QQ^*$ is multiplied with these terms, its contribution in $Z^{*,\pm}$ can be directly controlled without any loss of weights of Mach number. 

It now remains to control $ZB^{*,\pm}_0+Z^{*,\pm}$ with $Z_0^{*,\pm}:=\eps^{16}\iopm \QQ^{*,\pm}(\p_3\vp)^{-1}\left[\Dtpl\p_t^7,\NN_i,\p_3 v_i\right]\dvt$. Using $\dvt=\p_3\vp\dx$ and Gauss-Green formula, we have
\begin{align}
&ZB^{*,\pm}_0+Z_0^{*,\pm}=\mp\eps^{16}\is\QQ^{*,\pm}\,\left[\Dtpl\p_t^7,N_i,v_i^\pm\right]\dx'+\eps^{16}\iopm \QQ^{*,\pm}\left[\Dtpl\p_t^7,\NN_i,\p_3 v_i^\pm\right]\dx\no\\
\eql&-\sum_{j=0}^1\sum_{k=1}^6\eps^{16}\binom{7}{k}\iopm \pp_3\QQ^{*,\pm}(\Dtpl)^j\p_t^k v_i^\pm\,(\Dtpl)^{1-j}\p_t^{6-k} \NN_i\dvt.
\end{align} Inserting the concrete form of $\QQ^{*,\pm}$, integrating by parts in $\Dtpl$ and invoking the momentum equation, we have 
\begin{align}
\int_0^tZB^{*,\pm}_0+Z_0^{*,\pm}\dtau\eql &\sum_{j=0}^1\sum_{k=1}^6\eps^{16}\binom{7}{k}\int_0^t\iopm \pp_3\QQ^{\sh,\pm}\Dtpl\left((\Dtpl)^{j}\p_t^k v_i^\pm\cdot,(\Dtpl)^{1-j}\p_t^{6-k} \NN\right)\dvt\dtau\no\\
&+\eps^{16}\binom{7}{k}\int_0^t\iopm \pp_3\QQ^{\sh,\pm}\left((\Dtpl)^{j}\p_t^k v_i^\pm\cdot,(\Dtpl)^{1-j}\p_t^{6-k} \NN\right)\dvt\bigg|^t_0\no\\
\lesssim&~\delta\|\pp_3\QQ^{\sh,\pm}\|_0^2+P(E^{\kk}(0))+\int_0^t P(\sigma^{-1},E^{\kk}(\tau))\dtau,~~~\forall \delta\in(0,1).
\end{align}Combining this with the control of remainder terms and commutators, we can easily obtain that
\begin{align}
\int_0^tZB^{*,\pm}+Z^{*,\pm}\dtau\lesssim\delta E_8^{\kk}(0)+P(E^{\kk}(0))+\int_0^t P(\sigma^{-1}, E^{\kk}(\tau))\dtau,~~~\forall \delta\in(0,1).
\end{align}

\subsection{Tangential estimates:  general cases and summary}\label{sect ETTG}
Let $\TT^\gamma=(\omega(x_3)\p_3)^{\gamma_4}\p_t^{\gamma_0}\TP_1^{\gamma_1}\TP_2^{\gamma_2}$ be a tangential derivative with length of the multi-index $\len{\gamma}:=\gamma_0+\gamma_1+\gamma_2+2\times 0+\gamma_4$. Section \ref{sect E4TT}-Section \ref{sect E8tt} are devoted to the control of full spatial derivatives ($\gamma_1+\gamma_2=\len{\gamma}$) and full time derivatives ($\gamma_0=\len{\gamma}$). Now we analyze how to handle the general case.

\subsubsection*{Space-time mixed derivatives: $\gamma_0>0$ and $\gamma_1+\gamma_2>0$}
Let us temporarily assume $\gamma_4=0$. In this case, the tangential derivatives that we need to consider have the form $\TP^{4-l-k}\p_t^k\TT^{\alpha}$ with $\len{\alpha}=2l,~\alpha_4=0$ and weights of Mach number $\eps^{2l}$. That is, we need to consider $\eps^{2l}\p_t^{k+\alpha_0}\TP^{4+l-k-\alpha_0}$-estimates. Following the previous paper \cite{LuoZhang2022CWWST} by Luo and the author, the control of space-time mixed tangential derivatives ($0<k+\alpha_0<4+l$) is the same as the control of purely spatial tangential derivatives. In particular, compared with the one-phase fluid problem \cite{LuoZhang2022CWWST}, we only need one spatial derivative to do integration by parts in order for the control of the extra problematic term 
$$\VS:=\eps^{4l}\is\p_t^{k+\alpha_0}\TP^{4+l-k-\alpha_0} q^-\,(\jump{\vb}\cdot\cnab)\p_t^{k+\alpha_0}\TP^{4+l-k-\alpha_0}\psi\dx'$$
in which we need to integrate by parts $\TP^{1/2}$ and seek for the control of $\eps^{2l}\bno{\p_t^{k+\alpha_0}\TP^{4+l-k-\alpha_0}}_{1.5}$. Mimicing the proof of Lemma \ref{qelliptic}, we can show that (replacing $k+\alpha_0$ by $k$)
\begin{lem}[Elliptic estimate for the time derivatives of the free interface]\label{qelliptic0}
Fix $l\in\{0,1,2,3,4\}$. For $0<k<4+l$, we have the following uniform-in-$(\eps,\kk)$ inequality, in which the first term on the right side disappears when $\kk=0$.
\begin{align*}\bno{\eps^{2l}\p_t^{k}\psi}_{5.5+l-k}\leq&~ \bno{\eps^{2l}\p_t^{k}\psi(0)}_{5.5+l-k}+\sigma^{-1}\left|\eps^{2l}\p_t^{k}\jump{q}\right|_{3.5+l-k}\\
&+ P\left(\sigma^{-1},|\cnab\psi|_{L^{\infty}},\sum_{j=0}^l E_{4+j}^{\kk}(t)\right)\left(\bno{\eps^{2l}\p_t^{k}\psi}_{4.5+l-k}+\bno{\eps^{2l}\p_t^{k-1}\psi}_{5.5+l-k}\right).
\end{align*}
\end{lem}

\subsubsection*{Weighted normal derivatives: $\gamma_4>0$}
In the most general case, $\TT^\gamma$ may contain weighted normal derivative $\omega(x_3)\p_3$, so we have to analyze the commutator involving $[\TT^\gamma,\p_3]$ in $\cc(f)$ and $\dd(f)$ defined in \eqref{AGU comm Ci}-\eqref{AGU comm B}. The problematic thing is that $\p_3$ may fall on $\omega(x_3)$ which converts a ``tangential" derivative $\omega(x_3)\p_3$ (a first-order derivative) to a normal derivative $\p_3$ (considered to be second-order under the setting of anisotropic Sobolev spaces). Such terms in $\dd(f)$ are $$(\p_3\vp)^{-1}(v\cdot\NN-\p_t\vp)[\TT^\gamma,\p_3]f+(v\cdot\NN-\p_t\vp)\frac{\p_3 f}{(\p_3\vp)^2}[\TT^\gamma,\p_3]\vp.$$ They can be directly controlled because an extra weight $(v\cdot\NN-\p_t\vp)$, which vanishes on $\Sigma$, is automatically generated to compensate the possible loss of weight function. As for $\cc(f)$, we notice that the terms involving $[\TT^\gamma,\p_3]$ can be written to be
\[
\frac{\NN_i}{\p_3\vp}[\TT^\gamma,\p_3]f-\frac{\NN_i}{\p_3 \vp}\pp_3 f[\TT^\gamma,\p_3]\vp,~~f=q\text{ or }v_i.
\] The second term above is easy to control because $\vp(t,x)=x_3+\chi(x_3)\psi(t,x')$ implies the $C^{\infty}$-regularity of $\vp$ in $x_3$ direction. For the first term, it may generate a term $\TT^{\beta}\p_3 f\NN_i$ with $\beta_i=\gamma_i(i=0,1,2),~\beta_4=\gamma_4-1$, whose $L^2(\Om)$ norm may be not directly bounded. Luckily, for $f= q$ or $v_i$, we can again invoke the momentum equation or the continuity equation to reduce $-\pp_3 q$ and $\pp_3 v\cdot \NN$ to tangential derivatives $\rho\Dtp v-\bp b$ and $-\ffp\Dtp p -\cnab\cdot\vb$ respectively. Therefore, there is no extra loss of derivative in the commutators $\cc(f)$ and $\dd(f)$ when $\gamma_4>0$.

\subsubsection*{Summary of tangential estimates}
Finally, we need to recover the estimates of $\TT^\gamma(v,b,S,\sqrt{\ffp} p)$ from the $L^2$-estimates of their Alinhac good unknowns. By definition, we have
\begin{align*}
\ino{\TT^{\gamma} f^\pm}_{0,\pm}^2\leq \ino{\FF^{\gamma,\pm}}_{0,\pm}^2 + \bno{\TT^\gamma\psi}_0^2\|\pp_3 f^\pm\|_{L^{\infty}(\Om^\pm)}^2,
\end{align*}in which $\ino{\FF^{\gamma,\pm}}_{0,\pm}$ and $\bno{\TT^\gamma\psi}_0$ have been controlled by $\delta E^{\kk}(t)+\int_0^t P(\sigma^{-1},E^{\kk}(\tau))\dtau$. When $\TT^\gamma$ contains at least one spatial derivative, we can use $-\TT q\sim \Dtp v +\bp b$ to get the control of $\TT q$ instead of $\sqrt{\ffp}\TT q$. For the full time derivatives, we use $\Dtpl=\p_t+(\vb^-\cdot\cnab)+(\p_3\vp)^{-1}(v^-\cdot\NN-\p_t\vp)\p_3$ to convert the $\eps^{2l}\p_t^{4+l}$-estimate to $\eps^{2l}\Dtpl\p_t^{3+l}$-estimate, $\eps^{2l}\TP\p_t^{3+l}$-estimate and $\eps^{2l}(\omega\p_3)\p_t^{3+l}$-estimate, in the second part of which the norm $|\eps^{2l}\p_t^{3+l}\psi(0)|_{2.5}$ is needed to control the VS term. Also, since $\omega(x_3)=0$ on the interface, $\TT^\gamma$ can be expressed as $\TP^{4+l-k}\p_t^k$ for $0\leq k\leq 4+l,~0\leq l\leq 4$. Hence, we establish the desired uniform-in-$(\kk,\eps)$ tangential estimates as in Proposition \ref{prop EkkTT}.

\subsection{Div-Curl analysis and reduction of pressure}\label{sect divcurl}
 The tangential derivatives of the variables $(v,b,p)$ are analyzed in Section \ref{sect ETT}-Section \ref{sect ETTG}. Here we show the reduction of normal derivatives of pressure and the analysis for the divergence and vorticity. We use the div-curl decomposition (cf. Lemma \ref{hodgeTT}) such that the normal derivatives of $(v,b)$ are controlled via their divergence and curl parts. For $0\leq l\leq 3,~0\leq k\leq 3-l,~\len{\alpha}=2l,~\alpha_3=0$, we have
\begin{align}
\ino{\eps^{2l}\p_t^k\TT^\alpha (v^\pm,b^\pm)}_{4-k-l,\pm}^2\leq C\bigg(&\ino{\eps^{2l}\p_t^k\TT^\alpha (v^\pm,b^\pm)}_{0,\pm}^2+\ino{\eps^{2l}\nabp\cdot \p_t^k\TT^\alpha (v^\pm,b^\pm)}_{3-k-l,\pm}^2 \no\\
&+\ino{\eps^{2l}\nabp\times\p_t^k\TT^\alpha (v^\pm,b^\pm)}_{3-k-l,\pm}^2+\ino{\eps^{2l}\TP^{4-k-l}\p_t^k\TT^\alpha (v^\pm,b^\pm)}_{0,\pm}^2\bigg)\label{divcurlE}
\end{align} with $$C=C\left(\sum_{j=0}^l\sum_{k=0}^{3+j}|\eps^{2j}\p_t^{j}\psi|_{4+l-j}^2,|\cnab\psi|_{W^{1,\infty}}\right)>0$$ a positive continuious function linear in $|\eps^{2j}\p_t^{j}\psi|_{4+l-j}^2$.
The conclusion for the div-curl analysis is 
\begin{prop}\label{prop divcurlkk}
Fix $l\in\{0,1,2,3\}$. For any $0\leq k\leq l-1$, any multi-index $\alpha$ satisfying $\len{\alpha}=2l$ and any constant $\delta\in(0,1)$, we can prove the following estimates for the curl part
\begin{equation}
\begin{aligned}
&\ino{\eps^{2l}\nabp\times\p_t^k\TT^\alpha v^\pm}_{3-k-l,\pm}^2+\ino{\eps^{2l}\nabp\times\p_t^k\TT^\alpha b^\pm}_{3-k-l,\pm}^2\\
\lesssim &~\delta E_{4+l}^{\kk}(t)+P\left(\sigma^{-1},\sum_{j=0}^{l} E_{4+j}^{\kk}(0)\right)+P(E_4^{\kk}(t))\int_0^t P\left(\sigma^{-1},\sum_{j=0}^l E_{4+j}^{\kk}(\tau)\right) + E_{4+l+1}^{\kk}(\tau)\dtau,
\end{aligned}
\end{equation} and for the divergence part
\begin{equation}
\begin{aligned}
&\ino{\eps^{2l}\nabp\cdot \p_t^k\TT^\alpha v^\pm}_{3-k-l,\pm}^2 +\ino{\eps^{2l}\nabp\cdot \p_t^k\TT^\alpha v^\pm}_{3-k-l,\pm}^2\\
\lesssim &~\delta E_{4+l}^{\kk}(t)+P\left(\sigma^{-1},\sum_{j=0}^{l} E_{4+j}^{\kk}(0)\right)+P(E_4^{\kk}(t))\int_0^t P\left(\sigma^{-1},\sum_{j=0}^l E_{4+j}^{\kk}(\tau)\right) \dtau.
\end{aligned}
\end{equation}
\end{prop}

\subsubsection{Reduction of pressure and divergence}\label{sect divE4}
Let us start with $l=0$. The spatial derivative of $q$ is controlled by invoking the momentum equation:
\begin{align}
-\p_3 q=&~(\p_3 \vp)\left(\rho\Dtp v_3 - \bp b_3\right);\\
-\TP_i q=&-(\p_3\vp)^{-1}\TP_i\vp\,\p_3 q+\rho\Dtp v_i - \bp b_i,~~i=1,2.
\end{align} Let $\TT$ be $\p_t$ or $\TP$ or $\omega(x_3)\p_3$. Then we have
\begin{align}
\|\p_t^k \p_3 q\|_{3-k}\lesssim&~ \|\p_t^k(\rho\TT v_3)\|_{3-k}+\|\p_t^k(b\TT b_3)\|_{3-k}\\
\|\p_t^k \TP_i q\|_{3-k}\lesssim&~ \|\p_t^k (\TP_i\vp\p_3 q)\|_{3-k}+\|\p_t^k(\rho\TT v_i)\|_{3-k}+\|\p_t^k(b\TT b_i)\|_{3-k},
\end{align}in which the leading order terms are $\|\p_t^k\TT(v,b)\|_{3-k}$ and $|\p_t^k \psi|_{4-k}$. This shows that we can convert the control of spatial derivative of $q$ to \textit{tangential estimates} of $v$ and $b$. 

Next we turn to the div-curl analysis for $v,b$. Let us first analyze $E_4(t)$. For $0\leq k\leq 3$, we have
\begin{align}\label{divcurlE4}
\|\p_t^k(v,b)\|_{4-k}^2\leq C(|\psi|_{4-k},|\cnab\psi|_{W^{1,\infty}})\left(\|\p_t^k (v,b)\|_0^2+\|\nabp\cdot \p_t^k(v,b)\|_{3-k}^2+\|\nabp\times\p_t^k (v,b)\|_{3-k}^2+\|\TP^{4-k}\p_t^k (v,b)\|_0^2\right).
\end{align}

For the divergence, we can directly invoke the continuity equation to convert $\nabp\cdot v$ to time derivative of $p$ together with square weights of Mach number. When $k=0$, we have
\begin{align}
\|\nabp\cdot v\|_3^2=\|\ffp \Dtp p\|_3^2,
\end{align}which is further reduced to the tangential derivatives of $v$ and $b$ by using the above reduction of $q$. Note that the magnetic tension term $\frac12|b|^2$ in the total pressure $q$ does not involve extra normal derivatives thanks to $\TT(\frac12 |b|^2)=b\cdot\TT b$. Taking $\p_t$ in the continuity equation, we have
\[
\nabp\cdot\p_t^k v\eql -\ffp \p_t^k\Dtp p +(\p_3\vp)^{-1}\TP\p_t^k\vp\cdot\p_3 v,
\]which gives
\begin{align}
\|\nabp\cdot \p_t^kv^\pm\|_{3-k,\pm}^2\lesssim C(\|v^\pm\|_{W^{1,\infty}(\Om^\pm)})\left(\ino{\ffpm \p_t^k\TT p^\pm}_{3-k,\pm}^2+\bno{\p_t^k\psi}_{4-k}^2\right)+ P(E_4^{\kk}(t))\int_0^t P(E_4^{\kk}(\tau))\dtau.
\end{align}Again, this can be reduced to tangential derivatives of $v,b$ until there is no spatial derivative falling on $p$.
As for the divergence of magnetic fields, we can invoke the div-free constraint to convert it to lower order terms. Namely, using $\nabp\cdot b=0$, we have
\[
\nabp\cdot\p_t^k b\eql \underbrace{\p_t^k(\nabp\cdot b)}_{=0} +(\p_3\vp)^{-1}\TP\p_t^k\vp\cdot\p_3 b
\]and thus
\begin{align}
\|\nabp\cdot \p_t^kb^\pm\|_{3-k,\pm}^2\lesssim C(\|b^\pm\|_{W^{1,\infty}(\Om^\pm)})\bno{\p_t^k\psi}_{4-k}^2+P(E_4^{\kk}(t))\int_0^t P(E_4^{\kk}(\tau))\dtau.
\end{align}The term $\bno{\p_t^k\psi}_{4-k}^2$ has been controlled in tangential estimates of $E_4^{\kk}(t)$. Combining the result of tangential estimates in Proposition \ref{prop EkkTT}, the control of divergence of time derivatives is concluded by
\begin{equation}
\begin{aligned}
&\ino{\nabp\cdot \p_t^k(v^\pm,b^\pm)}_{3-k,\pm}^2 \lesssim C(\|v^\pm\|_{W^{1,\infty}(\Om^\pm)})\ino{\ffpm \p_t^k\TT p^\pm}_{3-k,\pm}^2 + C(\|v^\pm, b^\pm\|_{W^{1,\infty}(\Om^\pm)})\bno{\p_t^k\psi}_{4-k}^2\\
\lesssim &~C(\|v^\pm\|_{W^{1,\infty}(\Om^\pm)})\ino{\ffpm \p_t^k\TT p^\pm}_{3-k,\pm}^2  +\delta E_4^{\kk}(t)+P(E_4^{\kk}(0))+ P(E_4^{\kk}(t))\int_0^t P(E_4^{\kk}(\tau))\dtau,
\end{aligned}
\end{equation}where the term involving $p^\pm$ can be further reduced to $\TT(v^\pm,b^\pm)$ when $3-k>0$ so that one can further apply the div-curl analysis to it.

\subsubsection{Vorticity analysis for $E_4$}\label{sect curlE4}
Taking $\nabp\times$ in the momentum equation of $v$ and the evolution equation of $b$, we get the evolution equation for the vorticity $\nabp\times v$ and the current density $\nabp\times b$
\begin{align}
\label{curlveq} \rho\Dtp(\nabp\times v)- \bp(\nabp\times b) =&~-(\nabp\rho)\times (\Dtp v) - \rho(\nabp v_j)\times(\pp_j v) + (\nabp b_j)\times(\pp_j b),\\
\label{curlbeq} \Dtp(\nabp\times b)-\bp(\nabp\times v)-b\times\nabp(\nabp\cdot v) =&-(\nabp\times b)(\nabp\cdot v)-(\nabp v_j)\times(\pp_j b) + (\nabp b_j)\times(\pp_j v),
\end{align}and taking $\p^3$ gives 
\begin{align}
\rho\Dtp(\p^3\nabp\times v)- \bp(\p^3\nabp\times b) =&~RK_v,\\
\Dtp(\p^3\nabp\times b)-\bp(\p^3\nabp\times v)-b\times\p^3\nabp(\nabp\cdot v) =&~RK_b,
\end{align}where 
\begin{align*}
RK_v:=&-[\p^3,\rho\Dtp](\nabp\times v)+[\p^3,\bp](\nabp\times b)+\p^3(\text{right side of }\eqref{curlveq}),\\
RK_b:=&-[\p^3,\Dtp](\nabp\times b)+[\p^3,\bp](\nabp \times v)+\p^3(\text{right side of }\eqref{curlbeq}).
\end{align*}
Direct computation shows that the highest-order terms in $RK_v,~RK_b$ only have 4 spatial derivatives and do not contain time derivative of $q$. Therefore, we can prove the $H^3$-control of the voriticity and current density by standard energy estimates.
\begin{equation}\label{divcurlK1}
\begin{aligned}
&\frac12\ddt\iopm\rho^\pm\bno{\p^3(\nabp\times v^\pm)}^2\dvt\\
=&-\iopm(\p^3\nabp\times b^\pm)\cdot\Dtpm(\p^3\nabp\times b^\pm)\dvt+\underbrace{\iopm (\p^3\nabp\times b^\pm)\cdot\left(b^\pm\times(\p^3\nabp(\nabp\cdot v^\pm))\right)\dvt}_{=:K_1^\pm}\\
&+\underbrace{\iopm RK_v^\pm\cdot(\p^3\nabp\times v^\pm)\dvt}_{=:L_1^\pm}+\underbrace{\iopm(\p^3\nabp\times b^\pm)\cdot RK_b^\pm\dvt}_{L_2^\pm} ,
\end{aligned}
\end{equation}where $L_1^\pm,~L_2^\pm$ are directly controlled by
$
L_1^\pm+L_2^\pm\leq P(\|(v^\pm,b^\pm)\|_{4,\pm},|\psi|_4).
$ It remains to analyze the term $K_1^\pm$ in which there is a key observation for the energy structure of compressible MHD system. We invoke the continuity equation $\nabp\cdot v^\pm=\ffpm\Dtpm p^\pm$ and commute $\Dtpm$ with $\nabp$ to get
\[
\nabp(\nabp\cdot v^\pm)=-\ffpm\Dtpm\nabp p^\pm+\ffpm (\nabp v_j^\pm)(\pp_j p^\pm).
\] Next, we rewrite the momentum equation to be $\rho^\pm\Dtpm v^\pm + b^\pm\times(\nabp\times b^\pm)=-\nabp p^\pm$ and plug it into the highest-order term $-\ffpm\Dtpm\nabp p^\pm$ to get
\begin{align*}
-\ffpm\Dtpm\nabp p^\pm=&~\ffpm\Dtpm(\rho^\pm\Dtpm v^\pm)+\ffpm\Dtpm(b^\pm\times(\nabp\times b^\pm))\\
=&~\ffpm\rho^\pm (\Dtpm)^2 v^\pm+\ffpm\Dtpm(b^\pm\times(\nabp\times b^\pm))+\ffpm(\Dtpm \rho^\pm)(\Dtpm v^\pm).
\end{align*}  Thus, the term $K_1^\pm$ becomes
\begin{equation}
\begin{aligned}
K_1^\pm=&\iopm (\p^3\nabp\times b^\pm)\cdot\left(b^\pm\times( \ffpm \rho^\pm \p^3(\Dtpm)^2 v^\pm)\right)\dvt-\iopm \ffpm \left(b^\pm\times(\p^3\nabp\times b^\pm)\right)\cdot\Dtpm\left(b^\pm\times(\p^3\nabp\times b^\pm)\right)\dvt\\
&+\iopm(\p^3\nabp\times b^\pm)\cdot RK_p^\pm\dvt,
\end{aligned}
\end{equation}where 
\begin{align*}
RK_p^\pm:=&~\ffpm \p^3\left((\nabp v_j^\pm)(\pp_j p^\pm)+(\Dtpm \rho^\pm)(\Dtpm v^\pm)\right)+[\p^3,\ffpm \rho^\pm](\Dtpm)^2 v^\pm\\
&+\ffpm[\p^3,\Dtpm](b^\pm\times(\nabp\times b^\pm))+\ffpm\Dtpm\left([\p^3,b^\pm\times](\nabp\times b^\pm)\right)
\end{align*}consists of $\leq 4$ derivatives in each term and its contribution can be directly controlled
\begin{align}
L_3^\pm:=\iopm(\p^3\nabp\times b^\pm)\cdot RK_p^\pm\dvt\leq P(\|b^\pm,v^\pm,\ffpm p^\pm\|_{4,\pm},\|\ffpm\Dtpm(v^\pm,b^\pm,p^\pm)\|_{3,\pm})
\end{align}Note that the second term on the right side of $K_1^\pm$ is obtained by using the vector identity $\bd{a}\cdot(\bd{b}\times\bd{c})=-\bd{c}\cdot(\bd{b}\times\bd{a})$:
\[
(\p^3\nabp\times b^\pm)\cdot\left(b^\pm\times\Dtpm\left(b^\pm\times(\p^3\nabp\times b^\pm)\right)\right)=-\Dtpm\left(b^\pm\times(\p^3\nabp\times b^\pm)\right)\cdot\left(b^\pm\times(\p^3\nabp\times b^\pm)\right).
\]Therefore, we have
\begin{equation}
\begin{aligned}
&\frac12\ddt\iopm\rho^\pm\bno{\p^3(\nabp\times v^\pm)}^2\dvt\\
=&-\iopm(\p^3\nabp\times b^\pm)\cdot\Dtpm(\p^3\nabp\times b^\pm)\dvt-\iopm \ffpm \left(b^\pm\times(\p^3\nabp\times b^\pm)\right)\cdot\Dtpm\left(b^\pm\times(\p^3\nabp\times b^\pm)\right)\dvt\\
&+\underbrace{\iopm (\p^3\nabp\times b^\pm)\cdot\left(b^\pm\times( \ffp \rho^\pm \p^3(\Dtpm)^2 v^\pm)\right)\dvt}_{K^\pm}+L_1^\pm+L_2^\pm+L_3^\pm\\
=&-\frac12\ddt\iopm\bno{\p^3(\nabp\times b^\pm)}^2+\ffpm\bno{b^\pm\times\p^3(\nabp\times b^\pm)}^2\dvt \\
&+\frac12\iopm (\nabp\cdot v^\pm)\left(\bno{\p^3(\nabp\times b^\pm)}^2+\ffpm\bno{b^\pm\times\p^3(\nabp\times b^\pm)}^2\right)\dvt  + K^\pm+L_1^\pm+L_2^\pm+L_3^\pm,
\end{aligned}
\end{equation}which further gives the control of vorticity and current density simultaneously
\begin{equation}\label{divcurlE4TT}
\begin{aligned}
&\frac12\ddt\iopm\rho^\pm\bno{\p^3(\nabp\times v^\pm)}^2+\bno{\p^3(\nabp\times b^\pm)}^2+\ffpm\bno{b^\pm\times\p^3(\nabp\times b^\pm)}^2\dvt\\
\lesssim&~P(E_4^{\kk}(t)) + P(|\psi|_{4})\|b^\pm\|_{4,\pm}\|b^\pm\rho^\pm\|_{L^{\infty}(\Om^\pm)}\ino{\eps^2(\Dtpm)^2v^\pm}_3\leq P(E_4^{\kk}(t)) + E_5^{\kk}(t).
\end{aligned}
\end{equation}Hence, the vorticity analysis for compressible ideal MHD cannot be closed in standard Sobolev space because of the term $ \eps^2 \p^3(\Dtpm)^2 v^\pm$ in $K^\pm$. Instead, the appearance of this term indicates us to \textbf{trade one normal derivative (in the curl operator) for two tangential derivatives $(\Dtpm)^2$ together with square weights of Mach number $\eps^2$. Besides, the normal derivative part involving $\p^3\Dtpm(\nabp\times b^\pm)$ contributes to the energy of current density thanks to the special structure of Lorentz force $-b^\pm\times(\nabp\times b^\pm)$. This is exactly the motivation for us to define the energy functional $E(t)$ under the setting the anisotropic Sobolev spaces instead of standard Sobolev spaces.} 

Similarly, the curl estimates for the time derivatives (in $E_4(t)$) can be proven in the same way by replacing $\p^3$ with $\p^{3-k}\p_t^k~(1\leq k\leq 3)$. We omit the details and list the conclusion
\begin{equation}
\begin{aligned}
&\frac12\ddt\iopm\rho^\pm\bno{\p^{3-k}\p_t^k(\nabp\times v^\pm)}^2+\bno{\p^{3-k}\p_t^k(\nabp\times b^\pm)}^2+\ffpm\bno{b^\pm\times\p^{3-k}\p_t^k(\nabp\times b^\pm)}^2\dvt\\
\lesssim&~P(E_4^{\kk}(t)) +\ino{\eps^2\p_t^k(\Dtpm)^2v^\pm}_{3-k,\pm}^2 \leq P(E_4^{\kk}(t)) + E_5^{\kk}(t).
\end{aligned}
\end{equation}Finally, we need to commute $\p_t^k$ with $\nabp\times$ when $k\geq 1$.  We have
\[
(\nabp\times\p_t^k v)_i\eql \p_t^k(\nabp\times v)_i + \epsilon_{ijl}(\p_3\vp)^{-1}(\TP_j\p_t^k\vp)(\p_3 v_l),
\]where $\epsilon_{ijl}$ is the sign of permutation $(ijl)\in S_3$. This gives
\begin{align}
\|\nabp\times \p_t^kv^\pm\|_{3-k,\pm}^2\lesssim C(\|v^\pm\|_{W^{1,\infty}(\Om^\pm)})\left(\ino{\p_t^k(\nabp\times v^\pm)}_{3-k,\pm}^2+\bno{\p_t^k\psi}_{4-k}^2\right)+P(E_4^{\kk}(t))\int_0^t P(E_4^{\kk}(\tau))\dtau,
\end{align}where both leading order terms have been controlled in tangential estimates of $E_4^{\kk}(t)$. The same result holds for $b^\pm$. Using the result of tangential estimates of $E_4^{\kk}(t)$, we have: for any $k\in\{0,1,2,3\}$ and any $\delta\in(0,1)$
\begin{equation}
\begin{aligned}
\ino{\nabp\times\p_t^k v^\pm}_{3-k,\pm}^2+\ino{\nabp\times\p_t^k b^\pm}_{3-k,\pm}^2\lesssim ~\delta E_4^{\kk}(t)+P(\sigma^{-1},E_4^{\kk}(0))+P(E_4^{\kk}(t))\int_0^t P(\sigma^{-1},E_4^{\kk}(\tau)) + E_5^{\kk}(\tau)\dtau.
\end{aligned}
\end{equation}

\subsubsection{Further div-curl analysis for $E_5\sim E_7$}\label{sect divcurlE5}
The vorticity analysis for $E_4(t)$ requires the control of $\ino{\eps^2\p_t^k(\Dtpm)^2v^\pm}_{3-k}^2$ for $0\leq k\leq 3$. When $0\leq k\leq 2$, there are still normal derivatives in this term. Thus, we shall do further div-curl analysis on $\ino{\eps^2\p_t^k(\Dtpm)^2v^\pm}_{3-k}^2$ for $0\leq k\leq 2$. Let $\TT^\alpha=\p_t^{\alpha_0}\TP_1^{\alpha_1}\TP_2^{\alpha_2}(\omega(x_3)\p_3)^{\alpha_4}$ with $\len{\alpha}=2$. The divergence part can be reduced in the same way as in Section \ref{sect divE4}. We take $\p_t^k\TT^\alpha$ in the continuity equation to get  
\[
\nabp\cdot\p_t^k \TT^\alpha v\eql -\eps^2 \p_t^k\TT^\alpha\Dtp p +(\p_3\vp)^{-1}\TP\p_t^k\TT^\alpha\vp\cdot\p_3 v,
\]which gives
\begin{align}
\|\eps^2\nabp\cdot \p_t^k\TT^\alpha v^\pm\|_{2-k}^2\lesssim C(\|v\|_{W^{1,\infty}})\left(\ino{\eps^4 \p_t^k\TT^\alpha \TT p^\pm}_{2-k,\pm}^2+\bno{\eps^2\p_t^k\TT^\alpha\psi}_{2-k}^2\right)+P(E_4^{\kk}(t))\int_0^t P(E_4^{\kk}(\tau))E_5^{\kk}(\tau)\dtau.
\end{align}
\begin{rmk}
The term generated when commuting $\TT^\alpha$ with $\nabp$ is actually of lower order. One can check that (see also \cite[(3.24)-(3.25)]{WZ2023CMHDlimit})
\[
[(\omega\p_3)^m,\p_3]f=\underbrace{(\omega\p_3)^m\p_3 f- \p_3((\omega\p_3)^m f)}_{\text{both are (m+1)-th order terms}}=\sum_{k\leq m-1}c_{m,k}(\omega\p_3)^{k}\p_3 f=\sum_{k\leq m-1}d_{m,k}\p_3(\omega\p_3)^{k} f
\] for some smooth functions $c_{m,k}, d_{m,k}$ depending on $m,k$ and the derivatives (up to order $m$) of $\omega$, and the right side only contains $\leq m$-th order terms.
\end{rmk}
Similarly, using $\nabp\cdot b=0$, we have $\nabp\cdot\p_t^k\TT^\alpha b\eql \underbrace{\p_t^k\TT^\alpha(\nabp\cdot b)}_{=0} +(\p_3\vp)^{-1}\TP\p_t^k\TT^\alpha\vp\cdot\p_3 b$ and thus
\begin{align}
\|\eps^2\nabp\cdot \p_t^kb^\pm\|_{2-k,\pm}^2\lesssim C(\|b^\pm\|_{W^{1,\infty}(\Om^\pm)})\bno{\eps^2\p_t^k\TT^\alpha\psi}_{2-k}^2+P(E_4^{\kk}(t))\int_0^t P(E_4^{\kk}(\tau))E_5^{\kk}(\tau)\dtau.
\end{align}The control of divergence part in the analysis of $E_5^{\kk}(t)$ is concluded by the following energy inequality. For any $k\in\{0,1,2\}$, any multi-index $\alpha$ with $\len{\alpha}=2$ and any $\delta\in(0,1)$
\begin{equation}
\begin{aligned}
&\ino{\eps^2\nabp\cdot \p_t^k\TT^\alpha(v^\pm,b^\pm)}_{2-k,\pm}^2 \lesssim C(\|v^\pm\|_{W^{1,\infty}(\Om^\pm)})\ino{\eps^4 \p_t^k\TT^\alpha\TT p^\pm}_{2-k,\pm}^2 + C(\|v^\pm, b^\pm\|_{W^{1,\infty}(\Om^\pm)})\bno{\p_t^k\TT^\alpha\psi}_{3-k}^2\\
\lesssim &~C(\|v^\pm\|_{W^{1,\infty}(\Om^\pm)})\ino{\eps^4 \p_t^k\TT^\alpha\TT p^\pm}_{2-k,\pm}^2  +\delta E_5^{\kk}(t)+P(E_4^{\kk}(0),E_5^{\kk}(0))+ P(E_4^{\kk}(t))\int_0^t P(E_4^{\kk}(\tau),E_5^{\kk}(\tau))\dtau,
\end{aligned}
\end{equation}where the term involving $\TT p^\pm$ can be further reduced to $\TT(v^\pm,b^\pm)$ when $2-k>0$ so that one can further apply the div-curl analysis to it.

As for the curl part, we can still mimic the proof in Section \ref{sect curlE4} to get the control of $\ino{\eps^2\p_t^k\TT^\alpha(\nabp\times (v,b))}_{2-k}$ for $0\leq k\leq 2$ and $\len{\alpha}=2$ with $\alpha_3=0$
\begin{equation}
\begin{aligned}
&\frac12\ddt\iopm\rho^\pm\bno{\eps^2\p^{2-k}\p_t^k\TT^\alpha(\nabp\times v^\pm)}^2+\bno{\eps^2\p^{2-k}\p_t^k\TT^\alpha(\nabp\times b^\pm)}^2+\ffpm\bno{\eps^2b^\pm\times\p^{2-k}\p_t^k\TT^\alpha(\nabp\times b^\pm)}^2\dvt\\
\lesssim&~P(E_4^{\kk}(t), E_5^{\kk}(t)) +\ino{\eps^4\p_t^k\TT^\alpha(\Dtpm)^2v^\pm}_{2-k}^2 \leq P(E_4^{\kk}(t), E_5^{\kk}(t)) + E_6^{\kk}(t).
\end{aligned}
\end{equation} Then we commute $\p^{2-k}\p_t^k \TT^\alpha$ with $\nabp\times$ to get: for any $k\in\{0,1,2\}$, any multi-index $\alpha$ with $\len{\alpha}=2$ and $\alpha_3=0$, and any $\delta\in(0,1)$
\begin{equation}
\begin{aligned}
&\ino{\nabp\times\p_t^k\TT^\alpha v^\pm}_{2-k,\pm}^2+\ino{\nabp\times\p_t^k\TT^\alpha b^\pm}_{2-k,\pm}^2\\
\lesssim&~P(E_4^{\kk}(0), E_5^{\kk}(0))+\int_0^t P(E_4^{\kk}(\tau), E_5^{\kk}(\tau)) + E_6^{\kk}(\tau)\dtau+P(\|v^\pm,b^\pm\|_{W^{1,\infty}(\Om^\pm)})\bno{\eps^2\p_t^k\TT^\alpha\psi}_{2-k}^2\\
\lesssim &~\delta E_5^{\kk}(t)+P(\sigma^{-1},E_4^{\kk}(0), E_5^{\kk}(0))+P(E_4^{\kk}(t))\int_0^t P(\sigma^{-1},E_4^{\kk}(\tau), E_5^{\kk}(\tau)) + E_6^{\kk}(\tau)\dtau,
\end{aligned}
\end{equation} where we use the result of tangential estimates to control $\bno{\eps^2\p_t^k\TT^\alpha\psi}_{2-k}$.
When $k\leq 1$ in the above energy estimate, we shall continue to apply the div-curl analysis to $\ino{\eps^4\p_t^k\TT^\alpha(\Dtpm)^2v^\pm}_{2-k}^2$. 

For $E_6^{\kk}$ and $E_7^{\kk}$, we have analogous div-curl inequalities. For $l=2,3$, we continue to analyze the divergence and the curl according to \eqref{divcurlE}. Similarly as above, we have the following estimates for any $k\in\{0,1\}$, any multi-index $\alpha$ with $\len{\alpha}=4,~\alpha_3=0$ and any $\delta\in(0,1)$
\begin{equation}
\begin{aligned}
&\ino{\eps^4\nabp\times\p_t^k\TT^\alpha v^\pm}_{1-k,\pm}^2+\ino{\eps^4\nabp\times\p_t^k\TT^\alpha b^\pm}_{1-k,\pm}^2\\
\lesssim &~\delta E_6^{\kk}(t)+P\left(\sigma^{-1},\sum_{l=0}^2 E_{4+l}^{\kk}(0)\right)+P(E_4^{\kk}(t))\int_0^t P\left(\sigma^{-1},\sum_{l=0}^2 E_{4+l}^{\kk}(\tau)\right) + E_7^{\kk}(\tau)\dtau,
\end{aligned}
\end{equation} For any multi-index $\alpha$ with $\len{\alpha}=6$ and $\alpha_3=0$ and any $\delta\in(0,1)$, we have
\begin{equation}
\begin{aligned}
&\ino{\eps^6\nabp\times\TT^\alpha v^\pm}_{0,\pm}^2+\ino{\eps^6\nabp\times\TT^\alpha b^\pm}_{0,\pm}^2\\
\lesssim &~\delta E_7^{\kk}(t)+P\left(\sigma^{-1},\sum_{l=0}^3 E_{4+l}^{\kk}(0)\right)+P(E_4^{\kk}(t))\int_0^t P\left(\sigma^{-1},\sum_{l=0}^3 E_{4+l}^{\kk}(\tau)\right) + E_8^{\kk}(\tau)\dtau.
\end{aligned}
\end{equation} 

The control of divergence part for $ E_6^{\kk}(t),~ E_7^{\kk}(t)$ also follows the same way as $ E_4^{\kk}(t),  E_5^{\kk}(t)$. For any $k\in\{0,1\}$, any multi-index $\alpha$ with $\len{\alpha}=4,~\alpha_3=0$, we have
\begin{equation}
\begin{aligned}
&\ino{\eps^4\nabp\cdot \p_t^k\TT^\alpha(v^\pm,b^\pm)}_{1-k,\pm}^2 \lesssim C(\|v^\pm\|_{W^{1,\infty}(\Om^\pm)})\ino{\eps^6 \p_t^k\TT^\alpha\TT p^\pm}_{1-k,\pm}^2 + C(\|v^\pm, b^\pm\|_{W^{1,\infty}(\Om^\pm)})\bno{\p_t^k\TT^\alpha\psi}_{2-k}^2\\
\lesssim &~\delta E_6^{\kk}(t)+P\left(\sigma^{-1},\sum_{l=0}^2 E_{4+l}^{\kk}(0)\right)+P(E_4^{\kk}(t))\int_0^t P\left(\sigma^{-1},\sum_{l=0}^2 E_{4+l}^{\kk}(\tau)\right)\dtau.
\end{aligned}
\end{equation} For any multi-index $\alpha$ with $\len{\alpha}=6,~\alpha_3=0$, we have
\begin{equation}
\begin{aligned}
&\ino{\eps^6\nabp\cdot \p_t^k\TT^\alpha(v^\pm,b^\pm)}_{0,\pm}^2 \lesssim C(\|v^\pm\|_{W^{1,\infty}(\Om^\pm)})\ino{\eps^8 \p_t^k\TT^\alpha\TT p^\pm}_{0,\pm}^2 + C(\|v^\pm, b^\pm\|_{W^{1,\infty}(\Om^\pm)})\bno{\eps^6\TT^\alpha\psi}_{1}^2\\
\lesssim &~\delta E_7^{\kk}(t)+P\left(\sigma^{-1},\sum_{l=0}^3 E_{4+l}^{\kk}(0)\right)+P(E_4^{\kk}(t))\int_0^t P\left(\sigma^{-1},\sum_{l=0}^3 E_{4+l}^{\kk}(\tau)\right) \dtau,
\end{aligned}
\end{equation}where the term $\ino{\eps^8 \TT^\alpha\TT p^\pm}_{0,\pm}^2$ does not appear because it has been included in tangential estimates for $E_7^{\kk}(t)$.

\subsubsection{Modifications for the 2D case}\label{sect 2D curl}
In the case of 2D, the equations of vorticity $\nabper\cdot v$ and current density $\nabper\cdot b$ are
\begin{align}
 \rho\Dtp(\nabper\cdot v)- \bp(\nabper\cdot b) =&-(\nabper\rho)\cdot (\Dtp v) - \rho(\nabper v_j)\cdot(\nabp_j v) + (\nabper b_j)\cdot(\nabp_j b),\\
\Dtp(\nabper\cdot b)-\bp(\nabper\cdot v)-b\cdot\nabper(\nabp\cdot v) =&-(\nabper\cdot b)(\nabp\cdot v)-(\nabper v_j)\cdot(\nabp_j b) + (\nabper b_j)\cdot(\nabp_j v),
\end{align}which has the same structure as \eqref{curlveq}-\eqref{curlbeq}. Thus, we expect to adopt the strategy in Section \ref{sect divcurl} to prove the div-curl estimates. The only slight difference is the structure of Lorentz force. Let us take the $\p^3$-estimate of $\nabper\cdot (v,b)$ for an example. In this case, the problematic term (in the analogue of $K_1^\pm$ in \eqref{divcurlK1}) becomes
\[
{K_1^\pm}'=\iopm (\p^3\nabper\cdot b^\pm)\,\left(b^\pm\cdot\nabper(\p^3\nabp\cdot v^\pm)\right)\dvt.
\]Again, we invoke the continuity equation, commute $\nabp$ with $\Dtpm$ to get
\begin{align*}
b^\pm\cdot\nabper(\p^3\nabp\cdot v^\pm)\eql&~ \eps^2 (b_1^\pm\p^3\pp_2\Dtpm p^\pm-b_2^\pm\p^3\pp_1\Dtpm p^\pm)\eql   \eps^2 (b_1^\pm\p^3\Dtpm (\pp_2p^\pm)-b_2^\pm\p^3\Dtpm (\pp_1p^\pm)).
\end{align*} Then we plug the momentum equation
\begin{align*}
-\pp_1 p=\rho\Dtp v_1-b_1\pp_1b_1-b_2\pp_2b_1+b_1\pp_1b_1 +b_2\pp_1 b_2=\rho\Dtp v_1+b_2(\nabper\cdot b)\\
-\pp_2 p=\rho\Dtp v_2-b_1\pp_1b_2-b_2\pp_2b_2+b_1\pp_2b_1 +b_2\pp_2b_2=\rho\Dtp v_2-b_1(\nabper\cdot b)
\end{align*} 
to get
\begin{align*}
b^\pm\cdot\nabper(\p^3\nabp\cdot v^\pm)\eql&~\ffpm\rho^\pm(b^{\pm,\perp}\cdot\p^3(\Dtpm)^2 v^\pm) - \ffpm\left((b_1^\pm)^2+(b_2^\pm)^2\right)\p^3\Dtpm(\nabper\cdot b),~~b^\perp:=(-b_2,b_1).
\end{align*} Thus, the term ${K_1^\pm}'$ can be controlled in a similar manner as in Section \ref{sect divcurl}
\begin{align*}
{K_1^\pm}'\eql &\iopm (\p^3\nabper\cdot b^\pm)\,\left(\ffpm\rho^\pm b^{\pm,\perp}\cdot\p^3(\Dtpm)^2 v^\pm\right)\dvt\\
&-\iopm \ffpm|b^\pm|^2(\p^3\nabper\cdot b^\pm)\,(\Dtpm\p^3\nabper\cdot b^\pm)\dvt\\
=&\iopm (\p^3\nabper\cdot b^\pm)\,\left(\ffpm\rho^\pm b^{\pm,\perp}\cdot\p^3(\Dtpm)^2 v^\pm\right)\dvt -\frac12\iopm(\nabp\cdot v^\pm)\ffpm\bno{b^\pm}^2\bno{\p^3\nabper\cdot b^\pm}^2\dvt\\
&-\frac12\ddt\iopm\ffpm\bno{b^\pm}^2\bno{\p^3\nabper\cdot b^\pm}^2\dvt.
\end{align*}Hence, the curl estimate \eqref{divcurlE4TT} should be modified to be
\begin{align}
\frac12\ddt\iopm\rho^\pm\bno{\p^3(\nabper\cdot  v^\pm)}^2+(1+\ffpm|b^\pm|^2)\bno{\p^3(\nabper \cdot b^\pm)}^2\dvt\lesssim P(\EW_4(t)) + \EW_5(t).
\end{align}

\subsection{Uniform estimates for the nonlinear approximate system}
\subsubsection{Control of the entropy}
It remains to control the full (anisotropic) Sobolev norms of the entropy functions $S^\pm$. This can be easily proven thanks to $\Dtpm S^\pm=0$. In the control of $E_{4+l}^{\kk}(t)$ for fixed $0\leq l \leq 4$, we need to take the derivative $\p_*^{\alpha}:=\p^{4-l-k}\p_t^k\TT^{\gamma}=\p_3^{\gamma_3'}(\omega\p_3)^{\gamma_4}\p_t^{k+\gamma_0}\TP_1^{\gamma_1+\gamma_1'}\TP_2^{\gamma_2+\gamma_2'}$ with $\gamma_0+\gamma_1+\gamma_2+\gamma_4=2l$, $\gamma_1'+\gamma_2'+\gamma_3'=4-k-l$ and $0\leq k\leq 4-l$ and also multiply the weight $\eps^{2l}$. Then we can  introduce the Alinhac good unknown $\BS^\p$ with respect to this general derivative $\p_*^\alpha$ by 
\[
\BS^{\p,\pm}:=\p_*^\alpha S^\pm -\p_*^\alpha \vp\pp_3 S^\pm,
\]which satisfies the evolution equation $\Dtpm \BS^{\p,\pm}=\dd^\p(S^\pm)$ in $\Om^\pm$ where $\dd^{\p}(S^\pm)$ is defined by \eqref{AGU comm B} after replacing $\TT^\gamma$ with $\p_*^\alpha$. We will get
\begin{align}
&\left\|\eps^{2l}\p_t^k\TT^{\gamma} S^\pm\right\|_{4-k-l,\pm}^2\lesssim \ino{\eps^{2l}\BS^{\p,\pm}}_{0,\pm}^2+\bno{\p_*^\alpha\psi}_{0}^2\|\p_3 S^\pm\|_{L^{\infty}(\Om^\pm)}^2\no\\
\lesssim&~ \delta E_{4+l}^{\kk}(t)+P\left(\sigma^{-1},\sum_{j=0}^l E_{4+j}^{\kk}(0)\right)+E_4^{\kk}(t)\int_0^t P\left(\sigma^{-1},\sum_{j=0}^l E_{4+j}^{\kk}(\tau)\right) \dtau, \label{prop Skk}
\end{align}where we again invoke the estimate of $\bno{\p_*^\alpha\psi}_{0}$ that has been proven in Section \ref{sect ETT}-Section \ref{sect ETTG}.

\subsubsection{Uniform-in-$\kk$ estimates for the nonlinear approximate system}
Summarizing Proposition \ref{prop L2kk} ($L^2$-energy conservation), Proposition \ref{prop EkkTT} (tangential estimates), Proposition \ref{prop divcurlkk} (div-curl estimates) and \eqref{prop Skk} (entropy estimates), we conclude the estimates of the energy functional $E^\kk(t)$ for the nonlinear approximate system \eqref{CMHD0kk} by
\begin{equation}
E^{\kk}(t)\lesssim \delta E^\kk(t)+ P(E^\kk(0))+ P(E^\kk(t))\int_0^t P(\sigma^{-1}, E^\kk (\tau))\dtau,~~\forall \delta\in(0,1)
\end{equation} Thus, choosing $\delta$ suitably small such that $\delta E^\kk(t)$ can be absorbed by the left side and then using Gronwall-type argument, we find that there exists a time $T_\sigma>0$ that depends on $\sigma$ and the initial data and is independent of $\kk$ and $\eps$, such that 
\begin{align}
\sup\limits_{0\leq t\leq T_\sigma} E^\kk(t)\leq C(\sigma^{-1})P(E^\kk(0)),
\end{align} which is exactly the conclusion of Proposition \ref{prop Ekk}.

\section{Well-posedness of the nonlinear approximate system}\label{sect LWPkk}
We already prove the uniform-in-$\kk$ estimates for the nonlinear approximate problem \eqref{CMHD0kk}. If we can prove the well-posedness of \eqref{CMHD0kk} for each fixed $\kk>0$, then the uniform estimates allow us to take the limit $\kk\to 0_+$ and prove the local existence of system \eqref{CMHD0} for the compressible current-vortex sheets with surface tension. Since there is no loss of regularity in the estimate of $E^\kk(t)$, we would like to use Picard iteration to construct the solution to \eqref{CMHD0kk} for each fixed $\kk$.

\subsection{Definition of the linearized problem and the iteration scheme}\label{sect linear kkeq}
We now state the process of Picard iteration.

\textbf{Step 1: Start from a constant state.} We start with $\psi^{[-1]}=\psi^{[0]}=0$ and $(v^{[0],\pm},b^{[0],\pm},\rho^{[0],\pm},S^{[0],\pm})=(\vec{0},\vec{0},\underline{\rho}^\pm,0)$ for some constants $\underline{\rho}^\pm\geq \bar{\rho_0}$.

\textbf{Step 2: Define the linearized system.} For any $n\geq 0, n\in\N$, given $\{(v^{[k],\pm},b^{[k],\pm},\rho^{[k],\pm},S^{[k],\pm},\psi^{[k]})\}_{k\leq n}$, we define $(v^{[n+1],\pm},b^{[n+1],\pm},q^{[n+1],\pm},S^{[n+1],\pm},\psi^{[n+1]})$ by the following linear system with variable coefficients depending on the basic state $(v^{[n],\pm},b^{[n],\pm},q^{[n],\pm},S^{[n],\pm},\psi^{[n]},\psi^{[n-1]})$
\begin{equation}\label{CMHD0llkk}
\begin{cases}
\rho^{\nnn,\pm} D_t^{\vp^\nnn,\pm} v^{\nnr,\pm} - (\mb^{\nnn,\pm}\cdot\nab^{\vp^\nnn,\pm}) b^{\nnr,\pm}+\nab^{\vp^\nnn,\pm} q^{\nnr,\pm}=0&~~\text{ in }[0,T]\times \Omega^\pm,\\
(\ffpm)^\nnn D_t^{\vp^{[n]},\pm} q^{\nnr,\pm} - (\ffpm)^\nnn D_t^{\vp^{[n]},\pm} b^{\nnr,\pm}\cdot \mb^{[n],\pm} +\nab^{\vp^\nnn,\pm}\cdot v^{\nnr,\pm}=0 &~~\text{ in }[0,T]\times \Omega^\pm,\\
D_t^{\vp^{[n]},\pm} b^{\nnr,\pm}-(\mb^{\nnn,\pm}\cdot\nab^{\vp^\nnn,\pm}) v^{\nnr,\pm}+\mb^{[n],\pm}\nab^{\vp^\nnn,\pm}\cdot v^{\nnr,\pm}=0&~~\text{ in }[0,T]\times \Omega^\pm,\\
D_t^{\vp^{[n]},\pm} S^{\nnr,\pm}=0&~~\text{ in }[0,T]\times \Omega^\pm,\\
\jump{q^{\nnr}}=\sigma\h(\psi^{[n]})- \kk (1-\TL)^2\psi^{\nnr} - \kk (1-\TL)\p_t\psi^{\nnr}&~~\text{ on }[0,T]\times\Sigma, \\
\p_t \psi^{[n+1]} = v^{[n+1],\pm}\cdot N^{[n]} &~~\text{ on }[0,T]\times\Sigma,\\
v_3^{[n+1],\pm}=0,~~&~~\text{ on }[0,T]\times\Sigma^\pm,\\
(v^{\nnr,\pm},b^{\nnr,\pm},q^{\nnr,\pm},S^{\nnr,\pm},\psi^{\pm})|_{t=0}=(v_0^{\kk,\pm}, b_0^{\kk,\pm}, q_0^{\kk,\pm}, S_0^{\kk,\pm},\psi_0^{\kk}),
\end{cases}
\end{equation}where $\mb_i^{[n],\pm}:=b_i^{[n],\pm}$ for $i=1,2$ and $\mb_3^{[n],\pm}$ is defined by
\begin{align}\label{modified b}
\mb_3^{[n],\pm}:=b_3^{[n],\pm}+\mathfrak{R}_T^\pm\left(b_1^{[n],\pm}\TP_1\psi^{[n]}+b_2^{[n],\pm}\TP_2\psi^{[n]}-b_3^{[n],\pm}\right)\big|_{\Sigma}
\end{align} where $\mathfrak{R}_T^\pm$ is the lifting operator defined in Lemma \ref{trace}. The initial data $(v_0^{\kk,\pm}, b_0^{\kk,\pm}, \rho_0^{\kk,\pm}, S_0^{\kk,\pm},\psi_0^{\kk})$ is the same as \eqref{CMHD0kk}. 

In \eqref{CMHD0llkk}, the basic state $(v^{[n],\pm},\mb^{[n],\pm},\rho^{[n],\pm},S^{[n],\pm},\psi^{[n]},\psi^{[n-1]})$ satisfies:
\begin{enumerate}
\item (The hyperbolicity assumption) $\rho^{\nnn,\pm}>0$ is determined by the equation of state \eqref{eos11} where $p^{\nnn,\pm}$ is defined by $p^{\nnn,\pm}:=q^{\nnn,\pm}-\frac12|b^{\nnn,\pm}|^2$. Then define $\ff^\nnn=\log \rho^\nnn,~~\ffp^{\nnn,\pm}:=\frac{\p\ff^{\nnn,\pm}}{\p p}(p^{\nnn,\pm},S^{\nnn,\pm})>0.$
\item (Tangential magnetic fields) $\mb^{[n],\pm}\cdot N^{\nnn}=0\text{ on }\Sigma,$ and $\mb^{[n],\pm}_3=0$ on $\Sigma^\pm$.
\item (Linearized material derivatives and covariant derivatives)
\begin{align}
D_t^{\vp^\nnn,\pm}:=&~\p_t+\vb^\nnn\cdot\cnab+\frac{1}{\p_3\vp^\nnn}(v^\nnn\cdot\NN^\nnl-\p_t\vp^\nnn)\p_3,\\
\p_t^{\vp^\nnn}:=&~\p_t-\frac{\p_t\vp^\nnn}{\p_3\vp^\nnn}\p_3,~\nab_a^{\vp^\nnn}=\p_a^{\vp^\nnn}:=\p_a-\frac{\p_a\vp^\nnn}{\p_3\vp^\nnn}\p_3,~~a=1,2,\quad \nab_a^{\vp^\nnn}=\p_3^{\vp^\nnn}:=\frac{1}{\p_3\vp^\nnn}\p_3
\end{align}where $N^\nnn:=(-\p_1\psi^\nnn,-\p_2\psi^\nnn,1)^\top$ and $\NN^\nnn$ is the extension of $N^\nnn$ with $\vp^{\nnn}=x_3+\chi(x_3)\psi^{\nnn}(t,x')$. 
\end{enumerate}

\textbf{Step 3: Define the remaining variables $p,\rho$ and the modified magnetic field $\mb$.}  After solving the linear problem \eqref{CMHD0llkk}, we define $p^{\nnr,\pm} = q^{\nnr,\pm}-\frac12|b^{\nnr,\pm}|^2$ and use the equation of state $p^{\nnr}=p^{\nnr}(\rho^\nnr,S^\nnr)$ to determine the density $\rho^\nnr>0$. We shall also define the ``modified magnteic fields" $\mb^{[n+1],\pm}$ as follows in order to guarantee $\mb^{[n+1],\pm}\cdot N^{\nnr}=0$ on $\Sigma$ and $\Sigma^\pm$:
\begin{align}
\mb_1^{[n+1],\pm}=b_1^{\nnr,\pm},\quad \mb_2^{[n+1],\pm}=b_2^{\nnr,\pm},\no\\
\mb_3^{[n+1],\pm}=b_3^{[n+1],\pm}+\mathfrak{R}_T^\pm\left(b_1^{[n+1],\pm}\TP_1\psi^{[n+1]}+b_2^{[n+1],\pm}\TP_2\psi^{[n+1]}-b_3^{[n+1],\pm}\right)\big|_{\Sigma}.
\end{align}

\begin{rmk}[The boundary constraint of magnetic fields] The modified basic state $\mbr$ is necessary here, because the quantity $b^\nnr$ solved from \eqref{CMHD0llkk} may not be tangential to $\Sigma$ and so integrating $\bp$ by parts produces uncontrollable boundary terms. When taking the limit $n\to\infty$, we can show that the limit function $b^{[\infty]}$ also satisfies the constraint $b^{[\infty]}\cdot N^{[\infty]}|_{\Sigma}=0$ which indicates $\mb_3^{[\infty]}=b_3^{[\infty]}$ and so recover the nonlinear approximate system \eqref{CMHD0kk}. We refer to Section \ref{sect recoverkk} for details.
\end{rmk}

\begin{rmk}[The divergence constraint of magnetic fields]Notice that the divergence-free condition for $b^\pm$ no longer propagates from the initial data for the linear problem, but we will show that the contribution of the divergence of part of $b^\pm$ is still controllable and does not introduce extra substantial difficulty. After solving the nonlinear problem \eqref{CMHD0kk} for each fixed $\kk>0$, $\nabp\cdot b^\pm=0$ in \eqref{CMHD0kk} is automatically recovered from the initial constraint $\nabp\cdot b_0^{\kk,\pm}=0$.
\end{rmk}

For simplicity of notations, given any $n\in\N$, we denote $(v^{[n+1],\pm},b^{[n+1],\pm},q^{[n+1],\pm},p^{[n+1],\pm},\rho^{[n+1],\pm},S^{[n+1],\pm},\psi^{[n+1]})$, \\ $(v^{[n],\pm},b^{[n],\pm},\mb^{[n],\pm},q^{[n],\pm},\rho^{[n],\pm}, p^{[n],\pm}, S^{[n],\pm},\psi^{[n]})$, $\psi^\nnl$ respectively by $(v^\pm,b^\pm,q^\pm,p^\pm,\rho^\pm,S^\pm)$, $(\vr^\pm,\br^\pm,\mbr^\pm,\qr^\pm,\rhor^\pm,\prr^\pm,\sr^\pm,\psr)$ and $\psd$. Also, we denote $D_t^{\vp^\nnn,\pm}$ and $\p_i^{\vp^\nnn},~\nab_i^{\vp^\nnn}$ by $\Dtpmr$ and $\ppr_i,~\nabpr$. Thus, the linear problem above becomes
\begin{equation}\label{CMHDllkk}
\begin{cases}
\rhor \Dtpmr v^{\pm} - \mbpr b^{\pm}+\nabpr q^{\pm}=0&~~\text{ in }[0,T]\times \Omega^\pm,\\
\ffpmr \Dtpmr q^{\pm} - \ffpmr \Dtpmr b^{\pm}\cdot \mbr^{\pm} +\nabpr\cdot v^{\pm}=0 &~~\text{ in }[0,T]\times \Omega^\pm,\\
\Dtpmr b^{\pm}- (\mbr^\pm\cdot\nabpr) v^{\pm}+\mbr^\pm\nabpr\cdot v^{\pm}=0&~~\text{ in }[0,T]\times \Omega^\pm,\\
\Dtpmr S^{\pm}=0&~~\text{ in }[0,T]\times \Omega^\pm,\\
\jump{q}=\sigma\h(\psr)- \kk (1-\TL)^2\psi - \kk (1-\TL)\p_t\psi&~~\text{ on }[0,T]\times\Sigma, \\
\p_t \psi= v^{\pm}\cdot \npr &~~\text{ on }[0,T]\times\Sigma,\\
v_3^{\pm}=0&~~\text{ on }[0,T]\times\Sigma^\pm,\\
(v^{\pm},b^{\pm},q^{\pm},S^{\pm},\psi)|_{t=0}=(v_0^{\kk,\pm}, b_0^{\kk,\pm}, q_0^{\kk,\pm}, S_0^{\kk,\pm},\psi_0^{\kk}),
\end{cases}
\end{equation}where $\Dtpmr=\p_t+\vbr\cdot\cnab+\frac{1}{\p_3\pr}(\vr\cdot\Npd-\p_t\pr)\p_3$ and $\h(\psr)=\cnab\cdot(\cnab\psr/|\npr|).$

In Section \ref{sect linear kkeq lwp}, we prove the well-posedness of the linearized problem by using Galerkin approximation. Then we prove the high-order energy estimates for the linearized problem in Section \ref{sect uniformll} and prove the strong convergence of the Picard approximate sequence in Section \ref{sect picard}. Finally, in Section \ref{sect recoverkk}, we verify the limit system \eqref{CMHDllkklimit} is exactly the nonlinear approximate problem \eqref{CMHD0kk}.

\subsection{Well-posedness of the linearized approximate problem}\label{sect linear kkeq lwp}

In this section, we prove the well-posedness of the linearized problem \eqref{CMHDllkk}. We assume the basic state $(\vr,\br,\qr,\rhor,\prr,\sr,\psr)$ and $\psd$ satisfy the following bounds: There exists some $\Kr_0>0$ and a time $T_\kk>0$ (depending on $\kk>0$) such that
\begin{equation}\label{energy llkkrr}
\begin{aligned}
\sup_{0\leq t\leq T_\kk}\sum_{l=0}^4\bigg(\sum_\pm\sum_{\lee\alpha\ree=2l}\sum_{k=0}^{4-l}\left\|\left(\eps^{2l}\TT^{\alpha}\p_t^{k}(\vr^\pm, \br^\pm, \sr^\pm,(\ffpmr)^{\frac{(k+\alpha_0-l-3)_+}{2}}\qr^\pm)\right)\right\|^2_{4-k-l,\pm}&\\
 +\sum_{k=0}^{4+l}\left|\sqrt{\kk}\eps^{2l}\p_t^{k}\psr\right|^2_{6+l-k}+\int_0^t\left|\sqrt{\kk}\eps^{2l}\p_t^{5+l}\psr\right|^2_{1} \dtau\bigg)&\leq \Kr_0,
\end{aligned}
\end{equation}
            where $\TT^{\alpha}:=(\omega(x_3)\p_3)^{\alpha_4}\p_t^{\alpha_0}\p_1^{\alpha_1}\p_2^{\alpha_2}$ with the multi-index $\alpha=(\alpha_0,\alpha_1,\alpha_2,0,\alpha_4),~\lee \alpha\ree=\alpha_0+\alpha_1+\alpha_2+2\times0+\alpha_4$. Moreover, we have
\begin{equation}\label{energy llkkrr2}
\forall 0\leq T\leq T_\kk,\quad\int_0^T\ino{\eps^{2l}\TT^{\alpha}\p_t^{k}\mbr^\pm}_{4-k-l,\pm}^2\dt\leq C(\Kr_0).
\end{equation}
\begin{rmk}\label{mbrbound}
The $L_t^2$-type bound of $\mbr$ is obtained by using the second part of Lemma \ref{trace} and the $\sqrt{\kk}$-weighted enhanced regularity for the free interface. Indeed, the modification term $\mathfrak{R}_T^\pm\left(\br_1^{\pm}\TP_1\psr+\br_2^{\pm}\TP_2\psr-\br_3^\pm\right)\big|_{\Sigma}$ has \textit{vanishing initial value} thank to $b_0^{\kk,\pm}\cdot N_0^{\kk}=0$ on $\Sigma$. Thus, one can extend this function to $(-\infty,T]\times\Om^\pm$ and then apply the trace lemma for  anisotropic Sobolev spaces (cf. Trakhinin-Wang \cite[Lemma 3.4]{TW2020MHDLWP} or Lemma \ref{trace} in this paper) to show that
\begin{align*}
&\int_0^T\ino{\TT^{\alpha}\p_t^{k}(\br^\pm-\mbr^\pm)}_{4-k-l,\pm}^2\dt\leq \ino{\br^\pm-\mbr^\pm}_{8,*,T,\pm}^2\lesssim \bno{\br_1^{\pm}\TP_1\psr+\br_2^{\pm}\TP_2\psr-\br_3^\pm}_{7,T}^2 \\
\lesssim&~ \ino{\br_1^{\pm}\TP_1\pr+\br_2^{\pm}\TP_2\pr-\br_3^\pm}_{8,*,T,\pm}^2=\int_0^T\ino{\br_1^{\pm}\TP_1\pr+\br_2^{\pm}\TP_2\pr-\br_3^\pm}_{8,*,\pm}^2\dt\lesssim T\Kr_0,~~\forall T\in[0,T_\kk],
\end{align*}where $\|\cdot\|_{m,*,T,\pm},|\cdot|_{m,T}$ norms are defined in Appendix \ref{sect lemma}.
Notice that this $\sqrt{\kk}$-weighted enhanced regularity is necessary here, otherwise we lose the control of $|\TP\psi(t)|_{8}$ and a loss of tangential derivative occurs as in lots of previous works  \cite{ChenWangCMHDVS,Trakhinin2008CMHDVS,TW2020MHDLWP,TW2021MHDSTLWP} and references therein.
\end{rmk}

We aim to prove the following proposition.
\begin{prop}\label{prop CMHDllkk lwp}
Fix $\kk>0$. Under the hypothesis \eqref{energy llkkrr}-\eqref{energy llkkrr2}, there exists a time $t_{\kk}>0$ depending on $\kk,\Kr_0$ and the initial data such that the linearized problem \eqref{CMHDllkk} admits a unique solution $$(q^\pm, v^\pm,b^\pm,S^\pm)\in L^2(0,t_\kk; L^2(\Om^\pm)),~~\psi\in L^2(0,t_\kk,H^2(\Sigma))\text{ with }\p_t\psi\in L^2(0,t_\kk,H^1(\Sigma)).$$
\end{prop}

\subsubsection{Verification of the characteristic boundary of constant multiplicity}\label{sect char bdry}
First, let us verify that system \eqref{CMHDllkk} is a first-order linear symmetric hyperbolic system with  boundary conditions being characteristic, maximally dissipative and of constant multiplicity. We can write the linearized system \eqref{CMHDllkk} into a symmetric hyperbolic system of $U^\pm:=(q^\pm,v_1^\pm,v_2^\pm,v_3^\pm,b_1^\pm,b_2^\pm,b_3^\pm,S^\pm)^\top\in\R^8$: 
\begin{equation}\label{CMHDllhh}
\begin{aligned}
L^\pm U^\pm:=A_0(\Ur^\pm)\p_t U^\pm+A_1(\Ur^\pm)\p_1U^\pm+A_2(\Ur^\pm)\p_2 U^\pm+A_3(\Ur^\pm)\p_3 U^\pm &= \mathbf{0} \quad\text{ in }\Om^\pm\\
B(U^+,U^-,\psi):=\begin{bmatrix} U_4^+ - U_2^+\TP_1\psr - U_3^+\TP_2\psr-\p_t\psi \\ U_4^- - U_2^-\TP_1\psr - U_3^-\TP_2\psr-\p_t\psi \\ U_1^+-U_1^- +\kk(1-\TL)^2\psi +\kk (1-\TL)\p_t \psi \end{bmatrix}&=\mathbf{g} \quad\text{ on }\Sigma,\\
B^{\pm}_H(U^+,U^-):=\begin{bmatrix} U_4^+ \\ U_4^- \end{bmatrix}&=\mathbf{0} \quad\text{ on }\Sigma^\pm,
\end{aligned}
\end{equation}
where $\mathbf{g}:= (0,0,\sigma\mathcal{H}(\psr))^\top$ and the coefficient matrices are
\begin{align*}
A_0(\Ur):=\begin{bmatrix}
\ffpr&\vec{0}^\top& -\ffpr \mbr^\top & 0 \\
\vec{0}&\rhor \bm{I}_3 & \bm{O}_3 & \vec{0}\\
-\ffpr \mbr & \bm{O}_3 & \bm{I}_3+\ffpr \mbr\otimes \mbr &\vec{0}\\
0&\vec{0}^\top&\vec{0}^\top&1
\end{bmatrix},~
A_i(\Ur):=\begin{bmatrix}
\ffpr \vbr_i&\vec{e_i}^\top& -\ffpr\vbr_i \mbr^\top & 0 \\
\vec{e_i}&\rhor\vbr_i \bm{I}_3 & -\mbcr_i\bm{I}_3 & \vec{0}\\
-\ffpr\vbr_i \mbr& -\mbcr_i\bm{I}_3 & \vbr_i\bm{I}_3+\ffpr \vbr_i(\mbr\otimes \mbr) &\vec{0}\\
0&\vec{0}^\top&\vec{0}^\top&\vbr_i
\end{bmatrix}~(i=1,2),\\
A_3(\Ur):=\frac{1}{\p_3\pr}\begin{bmatrix}
\ffpr (\vr\cdot\Npd-\p_t\pr)&\Npr^\top& -\ffpr(\vr\cdot\Npd-\p_t\pr) \mbr^\top & 0 \\
\Npr&\rhor(\vr\cdot\Npd-\p_t\pr) \bm{I}_3 & -(\mbr\cdot\Npr)\bm{I}_3 & \vec{0}\\
-\ffpr(\vr\cdot\Npd-\p_t\pr) \mbr& -(\mbr\cdot\Npr)\bm{I}_3 & (\vr\cdot\Npd-\p_t\pr)\bm{I}_3+\ffpr (\vr\cdot\Npd-\p_t\pr)(\mbr\otimes \mbr) &\vec{0}\\
0&\vec{0}^\top&\vec{0}^\top&\vr\cdot\Npd-\p_t\pr
\end{bmatrix}.
\end{align*} Also notice that the matrix $A_3(\Ur)$ is equal to the following matrix on the boundary
\begin{equation}\label{bdry matrix 1}
A_3(\Ur)|_{\Sigma,\Sigma^\pm}:=\begin{bmatrix}
0&\Npr^\top& \bm{0}_4^\top& \\
\Npr& \bm{O}_3 &  \\
\bm{0}_4& & \bm{O}_4
\end{bmatrix}.
\end{equation}

In later steps, we want to apply the ``weak = strong" property and the uniqueness argument in Rauch \cite[Theorem 4, 8, 9]{Rauch1985} to the above linear system, so we shall first homogenize the boundary condition and simplify the boundary matrix \eqref{bdry matrix 1}. 
Let $q_h^\pm$ be the harmonic extension of $\pm\frac12 \sigma\h(\psr)$ in $\Om^\pm$ with $\p_3 q_h^\pm = 0$ on the fixed boundaries $\Sigma^\pm$ respectively. Then the unknowns $V^\pm:=U^\pm-(q_h^\pm,\mathbf{0}_7)^\top$ satisfy $B(V^+,V^-,\psi)=\mathbf{0}$. Next, we simplify the boundary matrix. We define $W^\pm:=\Jr^{-1}V^\pm$ and $\bm{A}_i(\Ur^\pm):={\Jr}^\top A_i(\Ur^\pm)\Jr$ where $$\Jr:=\begin{bmatrix}1& & & & & & \\ &1 &&&&&\\ &&1&&&&\\&\TP_1\pr&\TP_2\pr&1&&&\\ &&&&&&\\&&&&\mathbf{I}_4&\\&&&&&&   \end{bmatrix},$$ and so
\begin{equation}\label{bdry matrix}
\bm{A}_3(\Ur)|_{\Sigma,\Sigma^\pm}=\begin{bmatrix}
0&\mathbf{e}_3^\top& \bm{0}_4^\top& \\
\mathbf{e}_3& \bm{O}_3 &  \\
\bm{0}_4& & \bm{O}_4
\end{bmatrix},~~\mathbf{e_3}=(0,0,1)^\top.
\end{equation} Then the variables $W^\pm$ satisfy
\begin{equation}\label{CMHDllh}
\begin{aligned}
\mathbb{L}^\pm W^\pm:=\bm{A}_0(\Ur^\pm)\p_t W^\pm+\bm{A}_1(\Ur^\pm)\p_1 W^\pm+\bm{A}_2(\Ur^\pm)\p_2 W^\pm+\bm{A}_3(\Ur^\pm)\p_3 W^\pm &= \ffr^\pm \quad\text{ in }\Om^\pm\\
\mathbb{B}(W^+,W^-,\psi):=\begin{bmatrix} W_4^+ - \p_t\psi \\ W_4^-  - \p_t\psi \\ W_1^+-W_1^- +\kk(1-\TL)^2\psi +\kk (1-\TL)\p_t \psi \end{bmatrix}&=\mathbf{0} \quad \text{ on }\Sigma\\
\mathbb{B}^{\pm}_H(W^+,W^-):=\begin{bmatrix} W_4^+ \\ W_4^- \end{bmatrix}&=\mathbf{0} \quad\text{ on }\Sigma^\pm,
\end{aligned}
\end{equation}for some sufficiently regular\footnote{This is because the basic state $\psr$ has high-order Sobolev regularity as shown in \eqref{energy llkkrr}.} $\ffr^\pm$ depending only on $\psr$. 

As we can see, the rank of the $8\times8$ matrix $\bm{A}_3(\Ur)|_{\Sigma,\Sigma^\pm}$ is always equal to $2<8$, so the boundary is characteristic of constant multiplicity. Also, both $\bm{A}_3(\Ur^+)|_{\Sigma}$ and $\bm{A}_3(\Ur^-)|_{\Sigma}$ have exactly one negative eigenvalue respectively. Therefore, the correct number of boundary conditions should be 1 on either of $\Sigma^+$ or $\Sigma^-$ and the correct number of boundary conditions on $\Sigma$ should be $1\times2+1=3$, where the last ``+1" is the extra one to determine the graph function $\psi$. In \eqref{CMHDllh}, there are indeed three independent boundary conditions on $\Sigma$ and one independent boundary condition on either $\Sigma^+$ or $\Sigma^-$. The maximally dissipative condition is fulfilled because of the correct number of boundary conditions.

\subsubsection{Construction of Galerkin sequence}\label{sect Galerkin seq}
Before applying the ``weak = strong" property and the uniqueness argument in Rauch \cite{Rauch1985} to \eqref{CMHDllh}, we first prove the existence of weak solution to the linear system \eqref{CMHDllh} in $L^2([0,T]\times\Om^\pm)$ by using Galerkin's method. Since $\Om:=\T^2\times(-H,H)$ is bounded, there exists an orthonormal basis $\{\bm{e}_j\}_{j=1}^{\infty}\subset C^{\infty}(\Om)$ for $L^2(\Om)$ and $H^1(\Om)$. Given $2\leq m\in\N^*$, we introduce the Galerkin sequence $\{W^{m,\pm}(t,x),\psi^{m}(t,x')\}$ by
\begin{align}
\label{W galerkin}W_j^{m,\pm}(t,x):=&\sum_{l=1}^m W_{lj}^{m,\pm}(t)\bm{e}_l(x)\quad 1\leq j\leq 8,\\
\label{psi galerkin}\psi^{m}(t,x'):=&\sum_{l=1}^m \psi_{l}^{m}(t)\bm{e}_l(x',0).
\end{align} The Galerkin sequence is assumed to satisfy the following boundary conditions
\begin{align}
\label{kbc galerkin}\p_t\psi^m=&~W_4^{m,\pm}&\text{ on }\Sigma,\\
\label{jump galerkin}\jump{W_1^m}=&- \kk(1-\TL)^2\psi^m - \kk (1-\TL)\p_t\psi^m &\text{ on }\Sigma,\\
\label{bc galerkin} 0=&~W_4^{m,\pm}&\text{ on }\Sigma^\pm.
\end{align}

Now we introduce an ODE system as the ``truncated version" of \eqref{CMHDllhh} in $\text{Span}\{\bm{e}_1,\cdots,\bm{e}_m\}$ by testing the Galerkin sequence by a vector field $\bm{\phi}:=(\phi_1,\cdots,\phi_8)^\top$ with $$\phi_i:=\sum_{l=1}^m \phi_{il}(t)\bm{e}_l(x)\in\text{Span}\{\bm{e}_1,\cdots,\bm{e}_m\}.$$ Here and thereafter, repeated indices represent taking summation over them. Then we have
\begin{align}\label{ODE Galerkin}
\iopm \bm{A}_0^{ij}(\Ur^\pm)(\p_tW_j^{m,\pm})\phi_i\dvr + \sum_{k=1}^2 \iopm \bm{A}_k^{ij}(\Ur^\pm)(\TP_kW_j^{m,\pm})\phi_i\dvr+\iopm \bm{A}_3^{ij}(\Ur^\pm)(\p_3 W_j^{m,\pm})\phi_i\dvr=\iopm \ffr_i^\pm\phi_i\dvr
\end{align}where $\dvr:=\p_3\pr\dx$. Integrating by parts in $\TP_k$ and $\p_3$, we get
\begin{align}\label{weak Galerkin}
\iopm \bm{A}_0^{ij}(\Ur^\pm)(\p_tW_j^{m,\pm})\phi_i- \sum_{k=1}^3 W_j^{m,\pm} \p_k(\bm{A}_k^{ij}(\Ur^\pm)\phi_i) \dvr \mp \is \bm{A}_3^{ij}(\Ur^\pm) W_j^{m,\pm} \phi_i\dx'=\iopm \ffr_i^\pm\phi_i\dvr.
\end{align} Plugging the Galerkin sequence into the above identity, we get
\begin{align}
\iopm \bm{A}_0^{ij}(\Ur^\pm) \bm{e}_l(x)\phi_i (W_{lj}^{m,\pm})'(t) - \sum_{k=1}^3 \p_k(\bm{A}_k^{ij}(\Ur^\pm)\phi_i)\bm{e}_l(x) W_{lj}^{m,\pm}(t)  \dvt-\iopm \ffr_i^\pm\phi_i\dvr=\pm  \is \bm{A}_3^{ij}(\Ur^\pm) W_j^{m,\pm} \phi_i\dx'.
\end{align} Taking sum for the two parts in $\Om^\pm$, setting $\phi_i(x)=\bm{e}_i(x)$ and using the jump condition for $\jump{W_1}$ and $W_4^\pm$, we obtain a first-order linear ODE system for $\{W_{lj}^\pm(t)\}$
\begin{align}
&\sum_\pm\left(\iopm \bm{A}_0^{ij}(\Ur^\pm) \bm{e}_l(x) \bm{e}_i(x)\dvr\right)(W_{lj}^{m,\pm})'(t) - \left(\iopm \p_k(\bm{A}_k^{ij}(\Ur^\pm)\bm{e}_i(x))\bm{e}_l(x)\dvr\right)W_{lj}^{m,\pm}(t)-\iopm \ffr_i^\pm \bm{e}_i\dvr \no\\
=&\is \jump{W_1^m}\bm{e}_4(x',0)\dx'\no\\
=&-\kk\is (1-\TL)\psi^m\,(1-\TL)\bm{e}_4(x',0)\dx'-\kk \is \p_t\psi^m\bm{e}_4(x',0)\dx'-\kk \is \cnab\p_t\psi^m\cdot\cnab\bm{e}_4(x',0)\dx'. \label{ODE galerkin}
\end{align}
Since the basis $\{\bm{e}_l\}$ are smooth and the coefficients $(\Ur^{m,\pm},\psr)$ are sufficiently regular, standard ODE theory guarantees the local existence and uniqueness of the above ODE system \eqref{ODE galerkin} with initial data $$W_{lj}^{m,\pm}(0):=\iopm W_j^{m,\pm}(0,x)\bm{e}_l(x)\p_3\pr_0\dx.$$Therefore, the existence of Galerkin sequence $\{(W_j^{m,\pm},\psi^m)\}$ (defined by \eqref{weak Galerkin}, satisfying \eqref{kbc galerkin}-\eqref{bc galerkin}) is proved.

\subsubsection{Existence of the weak solution to the linearized problem}\label{sect Galerkin lwp}
In view of \eqref{weak Galerkin} and the concrete form of $\bm{A}_3(\Ur^\pm)$ and $\jump{W_1^m}$, we can define the weak solution to \eqref{CMHDllh} as below.
\begin{defn}
We say $(W^\pm,\psi)$ is a weak solution to \eqref{CMHDllh}, if $W^\pm\in L^2(0,T;L^2(\Om^\pm))$ and $\psi \in L^2(0,T;H^2(\Sigma))$ satisfies
\begin{itemize}
\item [a.] $\p_t W^\pm \in L^2(0,T; (H^{\frac52}(\Om^\pm))^*)$, $\p_t \psi\in L^2(0,T;H^1(\Sigma))$;
\item [b.] For any $\bm{\phi}=(\phi_1,\cdots,\phi_8)^\top\in L^2(0,T;H^{\frac52}(\Om))$, the following identity holds:
\begin{align}\label{linear weak sol}
&\sum_\pm\int_0^T \left[\left\langle \p_tW_j^{\pm}, \bm{A}_0^{ij}(\Ur^\pm)\phi_i\right\rangle_{\frac{5}{2},\Om^\pm} - \iopm \sum_{k=1}^3 W_j^{\pm} \p_k(\bm{A}_k^{ij}(\Ur^\pm)\phi_i) \dvr-\iopm \ffr_i^\pm\phi_i\dvr \right]\dt \\
=&-\kk\int_0^T\left[\is (1-\TL)\psi\,(1-\TL)\phi_4(x',0)\dx'-\kk \is \jp \p_t\psi\, \jp\phi_4(x',0)\dx'\right]\dt.\nonumber
\end{align}
\end{itemize}
Here $\jp=\sqrt{1-\TL}$, $X^*$ represents the dual space of a normed vector space $X$, and $\len{f,g}_{s,\Om^\pm}$ represents the pairing between $f\in (H^s(\Om^\pm))^*$ and $g\in H^s(\Om^\pm)$.
\end{defn}
The existence of weak solution is guaranteed by the uniform-in-$m$ estimates for the Galerkin sequence $\{W^{m,\pm}(t,x),\psi^{m}(t,x')\}$. We set $\bm{\phi}=W^{m,\pm}$ in $\Om^\pm$ respectively to obtain the standard $L^2$-type energy estimates thanks to the symmetric property of the coefficient matrices and the concrete form of $\bm{A}_3(\Ur)$ on the boundary
\begin{align}
&\sum_\pm\ddt\frac12\iopm (W^{m,\pm})^\top\cdot \bm{A}_0(\Ur^\pm)W^{m,\pm} \dvt \no\\
=&  \sum_{\pm}\frac12\iopm  (W^{m,\pm})^\top\cdot \p_t(\bm{A}_0(\Ur^\pm))W^{m,\pm} \dvt +\frac12\iopm  (W^{m,\pm})^\top\cdot \p_k(\bm{A}_k(\Ur^\pm))U^{m,\pm}\dvt - \iopm (W^{m,\pm})^\top\cdot \ffr \dvt \no\\
&+\is \jump{W_1^m}W_4^{m,\pm}  \dx'\label{L2 Galerkin}
\end{align}where the interior term can be controlled directly by $C(\Kr_0)\|W^{m,\pm}\|_{0,\pm}^2$. For the boundary term, using  the boundary conditions \eqref{kbc galerkin}-\eqref{jump galerkin}, we get the energy bounds under time integral
\begin{align}
&\int_0^t\is \jump{W_1^m}W_4^{m,\pm} \dx'\dtau=-\int_0^t\is \left( \kk(1-\TL)^2\psi^m + \kk (1-\TL)\p_t\psi^m\right)\p_t\psi^m \dx'\dtau\no\\
=& - \frac12 \bno{\sqrt{\kk}\psi^m}_2^2\bigg|^t_0 - \int_0^t\bno{\sqrt{\kk}\p_t\psi^m}_1^2\dtau.
\end{align} 
We define
\begin{align}
\bm{N}^m(t):=\sum_\pm \ino{\left(\sqrt{\ffpmr} W_1^{m,\pm},W_2^{m,\pm},\cdots,W_8^{m,\pm}\right)}_{0,\pm}^2 + \bno{\sqrt{\kk}\psi^m}_2^2 + \int_0^t \bno{\sqrt{\kk}\p_t\psi^m}_1^2\dtau.
\end{align}
Since $\bm{A}_0(\Ur^\pm)>0$, we obtain the uniform-in-$m$ estimate for the Galerkin sequence $\{U^{m,\pm}(t,x),\psi^{m}(t,x')\}$.
\begin{align}\label{uniform galerkin}
\bm{N}^m(t)\leq \bm{N}^m(0)+\int_0^t C(\Kr_0,\kk^{-1}) \bm{N}^m(\tau)\dtau,
\end{align} and thus there exists a time $T_{\bm{N}}>0$ (depending on $\kk$ and $\bm{N}^m(0)$, independent of $m$) such that $$\sup\limits_{0\leq t\leq T_{\bm{N}}} \bm{N}^m(t)\leq  C'(\Kr_0,\kk^{-1})\bm{N}^m(0).$$Because $L^{\infty}(0,T_{\bm{N}};L^2(\Om^\pm))$ is not reflexive, we consider the weak convergence in $L^2(0,T_{\bm{N}};L^2(\Om^\pm))$. By Eberlein-\v{S}mulian theorem and the uniqueness of expansion in Galerkin basis $\{\bm{e}_l\}_{l=1}^{\infty}$, there exists a subsequence $\{W^{m_k,\pm}(t,x),\psi^{m_k}(t,x')\}_{k=1}^{\infty}$ such that
\begin{align}
\left(\sqrt{\ffpmr} W_1^{m_k,\pm},W_2^{m_k,\pm},\cdots,W_8^{m_k,\pm}\right)\wto \left(\sqrt{\ffpmr} W_1^{\pm}, W_2^{\pm},\cdots, W_8^{\pm}\right)\text{ in }L^2(0,T_{\bm{N}};L^2(\Om^\pm)),\\
\psi^{m_k}\wto \psi\text{ in }L^2(0,T_{\bm{N}};H^2(\Sigma)),\quad \p_t\psi^{m_k}\wto \p_t\psi\text{ in }L^2(0,T_{\bm{N}};H^1(\Sigma)).
\end{align} 

To prove the existence of weak solution to \eqref{CMHDllh} (and equivalently for \eqref{CMHDllhh} and \eqref{CMHDllkk}), it remains to prove $\p_t W_j^{m,\pm}$ has a weakly convergent subsequence in $L^2(0,T; (H^{\frac52}(\Om^\pm))^*)$. Since $\bm{A}_0(\Ur^\pm)$ is positive-definite (and so it is invertible), any test function $V=(V_1,\cdots,V_8)^\top\in L^2(0,T;H^{\frac52}(\Om))$ can be written as $V_j=\bm{A}_0^{ij}(\Ur^\pm)\bm{\phi}_i$ for some $\bm{\phi}\in L^2(0,T;H^{\frac52}(\Om))$. Thus, from \eqref{weak Galerkin}, we have 
\begin{align*}
&\sum_\pm\int_0^T \len{\p_t W_j^{m,\pm}(t), V_j}_{\frac52}\dt= \sum_\pm\int_0^T\langle\p_t W_j^{m,\pm}(t), \bm{A}_0^{ij}(\Ur^\pm)\phi_i\rangle_{\frac52}\dt\\
=& \sum_\pm\sum_{k=1}^3 \int_0^T\iopm W_j^{m,\pm} \p_k(\bm{A}_k^{ij}(\Ur^\pm)\phi_i) \dvr\dt +\sum_\pm\int_0^T\iopm \ffr_i^\pm\phi_i\dvr\dt\\
&\quad\quad+ \int_0^T\is \bm{A}_3^{ij}(\Ur^+) W_j^{m,+} \phi_i- \bm{A}_3^{ij}(\Ur^-) W_j^{m,-} \phi_i\dx'\dt.
\end{align*} Invoking the boundary conditions and integrating by parts, we obtain that
\begin{align*}
&\is \bm{A}_3^{ij}(\Ur^+) W_j^{m,+} \phi_i- \bm{A}_3^{ij}(\Ur^-) W_j^{m,-} \phi_i\dx'\\
=&-\kk\is (1-\TL)\psi^m\,(1-\TL)\phi_4(x',0)\dx'-\kk \is \jp \p_t\psi^m\, \jp\phi_4(x',0)\dx'\dt.
\end{align*} Therefore, we find that
\begin{align}
\left| \sum_\pm\int_0^T\langle\p_t W_j^{m,\pm}(t), \bm{A}_0^{ij}(\Ur^\pm)\phi_i\rangle_{\frac52}\dt\right|\lesssim&\sum_\pm\int_0^T\left[\sum_{k=1}^3\|W^{m,\pm}\|_{0,\pm}\|\bm{A}_k(\Ur^\pm)\|_{W^{1,\infty}(\Om^\pm)}+\|\ffr^\pm\|_{0,\pm}\right] \|\bm{\phi}(t,\cdot)\|_{H^1(\Om)}\dt\nonumber\\
&+\kk\int_0^T\left[|\psi^m(t,\cdot)|_2 + |\p_t\psi^m(t,\cdot)|_1\right]|\phi_4(t,\cdot)|_{2} \dt.\label{bdry uniform galerkin}
\end{align} 
Using trace lemma, we know $|\phi_4(t,\cdot)|_{2}\lesssim \|\bm{\phi}(t,\cdot)\|_{H^{\frac52}(\Om)}$. Setting $T=T_{\bm{N}}$, taking supremum over all $\bm{\phi}\in L^2(0,T_{\bm{N}};H^{\frac52}(\Om))$ with $\|\bm{\phi}\|_{L^2(0,T_{\bm{N}};H^{\frac52}(\Om))} \leq 1$ and combining \eqref{uniform galerkin}, we prove $\{\p_t W_j^{m,\pm}\}$ is uniformly bounded in $L^2(0,T_{\bm{N}}; (H^{\frac52}(\Om^\pm))^*)$, and so it has a weakly convergent subsequence in $L^2(0,T_{\bm{N}}; (H^{\frac52}(\Om^\pm))^*)$. In particular, the weak limit is exactly $\p_t W_j^\pm$, which can be proved by mimicing Evans \cite[Exercise 7.5]{Evans}.

The above weak limits give us a weak solution to \eqref{CMHDllh}. In fact, integrating \eqref{weak Galerkin} in the time variable and invoking the boundary conditions and the concrete form of $\bm{A}_3(\Ur^\pm)$, we obtain
\begin{align*}
&\sum_\pm\int_0^{T_{\bm{N}}} \left[\left\langle \p_tW_j^{m,\pm}, \bm{A}_0^{ij}(\Ur^\pm)\phi_i\right\rangle_{\frac52,\Om^\pm} - \iopm \sum_{k=1}^3 W_j^{m,\pm} \p_k(\bm{A}_k^{ij}(\Ur^\pm)\phi_i) \dvr-\iopm \ffr_i^\pm\phi_i\dvr \right]\dt \\
=&-\kk\int_0^{T_{\bm{N}}}\left[\is (1-\TL)\psi^m\,(1-\TL)\phi_4(x',0)\dx'-\kk \is \jp \p_t\psi^m\, \jp\phi_4(x',0)\dx'\right]\dt.
\end{align*} Setting $m\to\infty$ and using the weak convergence of $(W^m,\psi^m)$ and $\p_t(W^m,\psi^m)$, we obtain the desired identity \eqref{linear weak sol}.
\begin{rmk}
From the second line in \eqref{bdry uniform galerkin}, we find it is exactly the appearance of $\kk$-regularization terms that forces us to choose the test function in $L^2(0,T;H^{\frac52}(\Om))$ instead of $L^2(0,T;H^{1}(\Om))$ and prove $\p_t W^\pm\in L^2(0,T;(H^{\frac52}(\Om^\pm))^*)$ instead of $L^2(0,T;(H^{1}(\Om^\pm))^*)$.
\end{rmk}

\subsubsection{Uniqueness and anisotropic regularity of the linearized problem}\label{sect Galerkin unique}
According to Rauch \cite[Theorem 4]{Rauch1985}, we introduce the definition of ``strong solution" to \eqref{CMHDllh}.
\begin{defn}[Strong solution]
We say $(W^\pm,\psi)\in L^2(0,T;L^2(\Om^\pm))\times L^2(0,T;H^2(\Sigma))$ is a strong solution to \eqref{CMHDllh} if there exist a sequence of sufficiently smooth\footnote{In Rauch \cite{Rauch1985} or Lax-Phillips \cite{Lax1960LWP}, the regularity is assumed to be $C^1$, but in fact it can be $C^\infty$ according the proof in those two papers. $C^1$ is required just for the fulfillment of the Gauss-Green formula.} functions $(W_n^{\pm},\psi_n^\pm)\in C^\infty([0,T]\times\overline{\Om^\pm})\times C^\infty([0,T]\times\Sigma) $ such that $W_n^\pm\to W^\pm$, $\mathbb{L}^\pm W_n^\pm\to \ffr^\pm$ in $L^2(0,T;L^2(\Om^\pm))$  and $\psi_n\to \psi$ in $L^2(0,T;H^2(\Sigma))\cap H^1(0,T;H^1(\Sigma))$.
\end{defn}
\begin{prop}[Weak = Strong]\label{prop weak=strong}
The weak solution $(W^\pm,\psi)$ is a strong solution to \eqref{CMHDllh}, that is, there exist a sequence of sufficiently smooth functions $(W_n^{\pm},\psi_n^\pm)\in C^\infty([0,T]\times\overline{\Om^\pm})\times C^\infty([0,T]\times\Sigma) $ such that $W_n^\pm\to W^\pm$, $\mathbb{L}^\pm W_n^\pm\to \ffr^\pm$ in $L^2(0,T;L^2(\Om^\pm))$  and $\psi_n\to \psi$ in $L^2(0,T;H^2(\Sigma))\cap H^1(0,T;H^1(\Sigma))$.  
\end{prop} 
\begin{proof}[Discussion of the proof] 
 In Section \ref{sect char bdry}, we already show that \eqref{CMHDllh} (equivalently \eqref{CMHDllhh} and \eqref{CMHDllkk}) is a first-order linear symmetric hyperbolic systems with boundary characteristic of constant multiplicity and the maximally dissipative property is fulfilled. Thus, using Rauch \cite[Theorem 8]{Rauch1985}, we conclude that the weak solution to \eqref{CMHDllh} is indeed a strong solution.  Below, we briefly sketch the proof of \cite[Theorem 8]{Rauch1985}, which reveals how to construct the smooth approximation $(W_n^{\pm},\psi_n^\pm)$. This will be needed in the proof of uniqueness.
 
Since this is a boundary-value problem, we cannot directly regularize the weak solution $(W,\psi)$ by using the 3D convolution mollifier. We also note that the tangential smoothing (as in Lax-Phillips \cite{Lax1960LWP}) is too restrictive as it requires rank $\bm A_3(\Ur^\pm)$ to be constant {\it near the boundary $\Sigma$} (not only on $\Sigma$). 

{\bf Step 1: Modified tangential smoothing.} To overcome the difficulty as above, the first-step regularization in Rauch \cite{Rauch1985} is (taking $\Om^+$ for example)
\[
W_{(\eta)}^+(t,x) := \int_0^t\iop \zeta(\tau,y)W(t+\eta\tau,x'+\eta y',x_3e^{\eta y_3})\dy\dtau,
\]where $\zeta\in C_c^\infty(\{|t|+|x|\leq 1,x_3>0\}),~\zeta\geq 0,~\int \zeta=1$. (The modification near $t=0$ is referred to \cite[pp. 182]{Rauch1985}) Then the analysis in \cite[pp. 173-175]{Rauch1985} shows that $W_{(\eta)}^+\to W^+$ in $L^2(0,T;L^2(\Om^+))$ and $\mathbb{L}^+W_{(\eta)}^+\to \mathbb{L}^+W^+$ in $L^2(0,T;L^2(\Om^+))$. Moreover, $W_{(\eta)}^+$ has infinite-order differentiability in $(t,x')$ variables. The regularization of $\psi$ can be directly defined by taking the tangential smoothing because it only depends on tangential variables. The concrete form of boundary conditions is still preserved as we only mollify the tangential variables.

{\bf Step 2: Normal shift. }After step 1, we may assume $W^+\in H_*^1([0,T]\times\Om^+)$ (actually infinite-order differentiable in $(t,x')$ variables) and $\psi\in C^\infty([0,T]\times\Sigma)$. Using the concrete form of $\bm{A}_3(\Ur^+)|_{\Sigma}$, we know the non-characteristic unknowns $W_1^+,W_4^+$ belong to $H^1([0,T]\times\Om^+)$ because $\p_3(W_1^+,W_4^+)$ can be expressed by the first-order tangential derivatives of other $W_i$'s. Therefore, their traces on the boundary exist in $L^2$.

To preserve the boundary conditions when regularizing $W^\pm$ in $x_3$-direction, we must avoid the change of boundary values of the noncharacteristic unknowns $W_1^+,W_4^+$. We now write $W^+=W_c^++W_{nc}^+$ to separate the non-charateristic unknowns $W_{nc}^+:=(W_1^+,0,0,W_4^+,0,0,0,0)^\top$ and then extend $W_{nc}^+=(-\frac{\kk}{2}(1-\TL)^2\psi-\frac{\kk}{2}(1-\TL)\p_t\psi,0,0,\p_t\psi,0,0,0,0)^\top$ to the other side $\Om^-$. To preserve the boundary condition, we slightly shift $W_{nc}^+$ by setting
\[
W_{[h]}^+(t,x):=W_c^+(t,x)+\chi_h(x_3)W_{nc}^+(t,x',x_3-h),~~~h\ll 1
\]where $\chi_h(x_3)\in[0,1]$ is a smooth cut-off function near the fixed boundary $\Sigma^+$ satisfying $\chi_h(x_3)= 1$ when $0\leq x_3<H-2h$ and $\chi_h(x_3)=0$ when $H-h<x_3\leq H$. Then, we must have $W_{[h]}^+\to W^+$ in $H^1([0,T]\times\Om^+)$ and  $\mathbb{L}^+W_{[h]}^+\to \mathbb{L}^+W^+$ in $L^2([0,T]\times\Om^+)$. We can do a similar shift for $W^-$ in $\Om^-$.

{\bf Step 3: 3D regularization. } After the normal shift, we may assume $W_{4}^\pm\equiv \p_t \psi$, $W_{1}^+-W_{1}^-\equiv -\kk(1-\TL)^2\psi - \kk(1-\TL)\p_t\psi$ in $\{-h<x_3<h\}$ and $W_4^\pm\equiv 0$ in $\{H-h<\pm x_3 \leq H\}$ for some fixed $h>0$. Now, we can mollify $W^+$ by using the convolution mollifier in 3D and establish the convergence as in \cite[pp. 176]{Rauch1985}:
\[
W_{\epsilon}^+:=J_\epsilon^+*W^+,~~J_\epsilon^+(x)=\epsilon^{-3}J^+(t/\epsilon,x'/\epsilon,-x_3/\epsilon),~~0\leq J^+\in C_c^\infty(\{t+|x|\leq 1\}\cap\Om^+),~~\int_{\Om^+} J^+=1.
\]  We can do similar mollification for $W^-$ in $\Om^-$. Such regularization does not change the concrete form of  boundary conditions when $\epsilon\ll h$ because 
\begin{itemize}
\item The boundary conditions are linear and only involve $W_{nc}^\pm$, whose concrete forms are already given in a strip with width $h\gg$ the smoothing parameter $\epsilon$ (so taking convolution does not change the concrete form on the boundaries);
\item $W_c^\pm$ does not appear in the boundary conditions, so the change of their boundary values has no influence on the boundary conditions.
\end{itemize}

After these three steps of regularization (as shown in \cite[Theorem 4, 8]{Rauch1985}), the smooth approximate functions $(W_n^\pm,\psi_n)$ (converging to $(W^\pm,\psi)$ as desired) can be chosen by the diagonal argument (when passing to the limit $\epsilon,h,\eta\to 0$) and they are smooth in all variables.
\end{proof}
The uniqueness is then a corollary of Proposition \ref{prop weak=strong} as shown in Rauch \cite[Theorem 9]{Rauch1985}.
\begin{cor}[Uniqueness]
The strong solution to \eqref{CMHDllh} is unique.
\end{cor}
\begin{proof}
Since the smooth approximation $(W_n^{\pm},\psi_n)$ of the strong solution $(W^\pm,\psi)$ are given by the above smoothing procedures, we have
\begin{equation}\label{CMHDllhn}
\begin{aligned}
\mathbb{L}^\pm W_n^\pm=\bm{A}_0(\Ur^\pm)\p_t W_n^\pm+\bm{A}_1(\Ur^\pm)\p_1 W_n^\pm+\bm{A}_2(\Ur^\pm)\p_2 W_n^\pm+\bm{A}_3(\Ur^\pm)\p_3 W_n^\pm &= \ffr_n^\pm \quad\text{ in }\Om^\pm,\\
\mathbb{B}(W_n^+,W_n^-,\psi_n)=\begin{bmatrix} W_{n,4}^+ - \p_t\psi_n \\ W_{n,}^-  - \p_t\psi_n \\ W_{n,1}^+-W_{n,1}^- +\kk(1-\TL)^2\psi_n +\kk (1-\TL)\p_t \psi_n \end{bmatrix}&=\mathbf{0} \quad \text{ on }\Sigma,\\
\mathbb{B}_H(W_n^+,W_n^-)=\begin{bmatrix} W_{n,4}^+ \\ W_{n,4}^- \end{bmatrix}&=\mathbf{0} \quad\text{ on }\Sigma^\pm,
\end{aligned}
\end{equation}where $\ffr_n^\pm\to \ffr^\pm$ in $L^2(\Om^\pm)$. Invoking the concrete form of $\bm{A}_i$'s and Gr\"onwall's inequality, we deduce the energy estimate
\[
\sup_{t\in[0,t_\kk]}\sum_\pm\|W_n^\pm\|_{0,\pm}^2 +\bno{\sqrt{\kk}\psi_n}_2^2+\int_0^t\bno{\sqrt{\kk}\p_t \psi_n}_1^2\dtau \leq C(\Kr_0)\int_0^{t_\kk}\sum_\pm\|\ffr_n^\pm\|_{0,\pm}^2\dtau
\]for some $t_\kk>0$ depending on $\kk$ and $\Kr_0$. 

Now, assuming there are two such strong solutions with the same initial data, say $(W^\pm,\psi)$ and $(\tilde{W}^\pm,\tilde{\psi})$, we consider their smooth approximation $(W_n^\pm,\psi_n)$ and $(\tilde{W}_n^\pm,\tilde{\psi}_n)$ that are obtained through the same smoothing procedures. By linearity, we know $(W_n^\pm-\tilde{W}_n^\pm,\psi_n-\tilde{\psi}_n)$ satisfies system \eqref{CMHDllhn} with $\ffr_n^\pm=\bd{0}$  and $(W_n^\pm,\psi_n)$ replaced by $(W_n^\pm-\tilde{W}_n^\pm,\psi_n-\tilde{\psi}_n)$. So, the above energy estimate implies that $W_n^\pm-\tilde{W}_n^\pm\equiv\bd{0}$ and $\psi_n-\tilde{\psi}_n=0$. By definition of strong solution, we conclude that $(W^\pm,\psi)=(\tilde{W}^\pm,\tilde{\psi})$.
\end{proof}

Since the boundary is characteristic, we may not expect the full Sobolev regularity (for the interior variables) as in the case of non-characteristic boundary. Instead, we can obtain the full anisotropic regularity.
\begin{prop}\label{cor weak strong}
The strong solution to \eqref{CMHDllkk} satisfies
\begin{align}
 q^{\pm},v^{\pm},b^{\pm},S^{\pm} \in C([0,t_\kk];H_*^8(\Om^\pm))\text{ with } \p_t^kq^{\pm},\p_t^kv^{\pm},\p_t^k b^{\pm},\p_t^k S^{\pm} \in C([0,t_\kk];H_*^{8-k}(\Om^\pm)),~~k\leq 8;\\
\psi \in C([0,t_\kk];H^{10}(\Sigma))\text{ with }  \p_t^k\psi \in C([0,t_\kk]; H^{10-k}(\Sigma))~~k\leq 8;~~\p_t^9\psi \in L^2(0,t_\kk; H^{1}(\Sigma)).
\end{align} 
\end{prop}
\begin{proof}
 This property is a consequence of Secchi \cite[Theorem 2.1]{Secchi1996-2}. Since our initial data is assumed to be $U_0^\pm\in H_*^8(\Om^\pm)$, we know the solution $\p_t^kU^\pm$ also belongs to $C([0,t_\kk];H_*^{8-k}(\Om^\pm))~(0\leq k\leq 8)$ for some time $t_\kk>0$ depending on $\kk,\Kr_0$. 

The regularity of $\psi$ is obtained below. When $0\leq k\leq 7$, the $H^{10-k}(\Sigma)$ regularity of $\p_t^k\psi$ can be proved by applying the elliptic estimates to the regularized boundary condition, which is parallel to Lemma \ref{qelliptic}. Indeed, taking $\jp^{7-k}\p_t^k$ in the regularized boundary condition, we get 
$$\jp^{7-k}\p_t^7\jump{q}=\sigma\jp^{7-k}\p_t^k\h(\psr)- \kk (1-\TL)^2\jp^{7-k}\p_t^k\psi - \kk (1-\TL)\jp^{7-k}\p_t^{k+1}\psi.$$ As in the proof of Lemma \ref{qelliptic}, we test this equation with $\jp^{7-k}\p_t^{k+1}(1-\TL)\psi$ in $L^2(\Sigma)$ and integrate by parts to get
\[
\frac12\ddt\is \kk \bno{(1-\TL)^{\frac32}\jp^{7-k}\p_t^k \psi}^2 \dx' + \kk\is\bno{(1-\TL)\jp^{7-k}\p_t^{k+1}\psi}^2\dx' \leq \|q^\pm\|_{8,*,\pm}^2 + C(\Kr_0). 
\]This gives the desired regularity of  $\p_t^k\psi$ for $0\leq k\leq 7$. When $k=8, 9$, the regularity of $\p_t^k\psi$ can be obtained from $\p_t^8$-estimates of $U^\pm$, which we refer to Section \ref{sect ETTll} (by setting $\TT^\gamma=\p_t^8$) and Section \ref{sect ETTbdryll} (by setting $k=8,l=4$).
\end{proof}

\subsection{High-order uniform estimates of the linearized approximate problem}\label{sect uniformll}
To proceed the Picard iteration, we shall prove that the bounds \eqref{energy llkkrr} for the coefficients $(\Ur,\psr,\psd)$ can be preserved by the solution to \eqref{CMHDllkk}. Fix $\kk>0$, we define the energy functional for \eqref{CMHDllkk} to be
\begin{equation}\label{energy llkk}
\begin{aligned}
\Er^\kk(t):=&~\Er^\kk_4(t)+\cdots+\Er^\kk_8(t)\\
\Er^\kk_{4+l}(t):=&\sum_\pm\sum_{\lee\alpha\ree=2l}\sum_{k=0}^{4-l}\left\|\left(\eps^{2l}\TT^{\alpha}\p_t^{k}(v^\pm, b^\pm, S^\pm,(\ffpmr)^{\frac{(k+\alpha_0-l-3)_+}{2}}q^\pm)\right)\right\|^2_{4-k-l,\pm}\\
 &+\sum_{k=0}^{4+l}\left|\sqrt{\kk}\eps^{2l}\p_t^{k}\psi\right|^2_{6+l-k}+\int_0^t\left|\sqrt{\kk}\eps^{2l}\p_t^{5+l}\psi\right|^2_{1} \dtau
\end{aligned}
\end{equation}  where $\TT^{\alpha}:=(\omega(x_3)\p_3)^{\alpha_4}\p_t^{\alpha_0}\p_1^{\alpha_1}\p_2^{\alpha_2}$ with the multi-index $\alpha=(\alpha_0,\alpha_1,\alpha_2,0,\alpha_4),~\lee \alpha\ree=\alpha_0+\alpha_1+\alpha_2+2\times0+\alpha_4$. We aim to prove that
\begin{prop}\label{prop energy llkk}
There exists some $T_\kk>0$ depending on $\kk$ and $\Kr_0$, such that
\[
\sup_{0\leq t\leq T_\kk} \Er^\kk(t)\leq  C(\kk^{-1},\Kr_0)\Er^\kk(0).
\]
\end{prop}

It should be noted that, since $\kk>0$ is fixed, we now can obtain higher boundary regularity for the free interface $\psi$, which allows us to avoid some technical steps (such as the analysis in Section \ref{sect E8tt}). Now we start with div-curl analysis.

\subsubsection{Div-Curl analysis}\label{sect divcurlll}
We start with $\Er_4(t)$. Using \eqref{divcurlTT} and the boundary conditions for $v,b$, we get
\begin{align}
\ino{v^\pm,b^\pm}_{4,\pm}^2\lesssim C(|\psr|_{4},|\cnab\psr|_{W^{1,\infty}})\left(\|(v^\pm,b^\pm)\|_{0,\pm}^2 + \|\nabpr\cdot (v^\pm,b^\pm)\|_{3,\pm}^2 +  \|\nabpr\times (v^\pm,b^\pm)\|_{3,\pm}^2 + \|\TP^4(v^\pm,b^\pm)\|_{0}^2\right).
\end{align}
\begin{rmk}
Here we cannot use the div-curl inequality \eqref{divcurlNN} to estimate the normal traces because the boundary constraint $b\cdot\npr=0$ no longer holds for the linearized problem.
\end{rmk}
The $L^2$-estimates are already proven in the uniform estimates of Galerkin sequence. The treatment of $\nabpr\cdot v$ is also the same as in Section \ref{sect E4TT}, that is, invoking the continuity equation. For $\nabpr\cdot b$, we no longer have the div-free constraint. Instead, we can take $\nabpr\cdot$ in the linearized evolution equation of $b$ to get
\begin{align}
\Dtpmr(\nabpr\cdot b^\pm)=&~(\ppr_i\mbr_j^\pm)(\ppr_j v_i^\pm) - (\nabpr\cdot \mb^\pm)(\nabpr\cdot v^\pm)+ [\Dtpmr,\nabpr\cdot]b^\pm .
\end{align}Direct calculation shows that $[\Dtpmr,\ppr_i](\cdot)=-(\ppr_i\vr_j)\ppr_j(\cdot)-(\ppr_i\p_t(\pr-\pd))\ppr_3(\cdot)$. On the other hand, the $\kk$-regularization term provides extra regularity for $\vp_t,\pr_t,\pd_t$. Thus, standard $H^3$ estimates give the control of divergence
\begin{align}
\frac12\ddt\ino{\nabpr\cdot b^\pm}_{3,\pm}^2 \lesssim C(\Kr_0,\kk^{-1})\left(\|b^\pm\|_{4,\pm}\|v^\pm\|_{4,\pm}\right)\leq  C(\Kr_0,\kk^{-1})\Er^\kk_4(t). 
\end{align}

The vorticity part is analyzed in a similar way as in Section \ref{sect divcurl}. The evolution equations are
\begin{align*}
\rhor\Dtpr(\nabpr\times v)-\mbpr(\nabpr\times b)=&~(\nabpr \rhor)\times(\Dtpr v)-(\nabpr \mbr_j)\times(\ppr_j b)\\
&-\rhor\left((\nabpr \vr_j)\times(\ppr_j v)+\nabpr(\p_t\pr-\p_t \pd)\times \ppr_3 v\right),\\
\Dtpr(\nabpr\times b)-\mbpr(\nabpr\times v)-\mbr\times\nabpr(\nabpr\cdot v)=&-(\nabpr\times \mbr)(\nabpr\cdot v)-(\nabpr \mbr_j)\times(\ppr_j v)\\
&-(\nabpr \vr_j)\times(\ppr_j b)-\nabpr(\p_t\pr-\p_t \pd)\times \ppr_3 b,
\end{align*}on the right side of which the highest-order derivative is 1 (except the mismatch term involving $\nabpr(\p_t\pr-\p_t\pd)$ which is directly bounded by $P(\Kr_0)$). Thus, we can still follow the analysis in Section \ref{sect curlE4} to get
\begin{align}
\ddt\frac12\iopm \rhor^\pm\bno{\p^3\nabpr\times v^\pm}^2+\bno{\p^3\nabpr\times b^\pm}^2\dvr \leq P(\Er^\kk_4(t),\Kr_0) +\Kr_1^\pm,
\end{align}where
\begin{align}
\Kr_1^\pm:=\iopm(\p^3\nabpr\times b^\pm)\cdot\left(\mbr^\pm\times(\p^3\nabpr(\nabpr\cdot v^\pm))\right)\dvr.
\end{align}Again, we invoke the continuity equation and the momentum equation to get
\begin{align*}
\mbr\times(\p^3\nabpr(\nabpr\cdot v))\eql\ffpr \rhor\mbr \times (\p^3(\Dtpr)^2 v) + \ffpr \mbr\times \Dtpr(\mbr\times(\p^3\nabpr\times b))
\end{align*}where we use the vector identity $(\bd{a}\times(\nabpr\times\bd{b}))_i=(\ppr_i\bd{b}_j)\bd{a}_j-\bd{a}_j\ppr_j\bd{b}_i$, and the omitted terms are directly controlled by $P(\Er^\kk_4(t),\Kr_0)$. Thus, we have
\begin{align*}
\Kr_1^\pm\eql &\iopm  \ffpmr \rhor (\p^3\nabpr\times b)\cdot\left(\mbr\times (\p^3(\Dtpr)^2 v) \right)\dvr + \iopm \ffpmr (\p^3\nabpr\times b)\cdot\left(\mbr\times\Dtpr(\mbr\times(\p^3\nabpr b))\right)\dvt\\
=&\iopm  \ffpmr \rhor (\p^3\nabpr\times b)\cdot\left(\mbr\times (\p^3(\Dtpr)^2 v) \right)\dvr-\iopm \ffpmr \Dtpr\left(\mbr\times(\p^3\nabpr \times b)\right)\cdot\left(\mbr\times (\p^3\nabpr\times b)\right)\dvt \\
\lesssim & -\frac12\ddt\iopm\ffpmr\bno{\mbr\times (\p^3\nabpr\times b)}_0^2 + P(\Kr_0)\Er^\kk_4(t) + \Er^\kk_5(t).
\end{align*}
So, we have
\begin{align}
\frac12\ddt\iopm\rhor^\pm\bno{\p^3\nabpr\times v^\pm}^2+\bno{\p^3\nabpr\times b^\pm}^2+\ffpm\bno{\mbr^\pm\times (\p^3\nabpr\times b^\pm)}_0^2 \dvr\lesssim P(\Kr_0)\Er_4^\kk(t) + \Er_5^\kk(t).
\end{align} Similarly as in Section \ref{sect curlE4} and Section \ref{sect divcurlE5}, we can prove the div-curl estimates for time-differentiated system and $\TT^\alpha$-differentiated system. For $0\leq l\leq 3,0\leq k\leq 3-l,\len{\alpha}=2l,~\alpha_3=0$, we have
\begin{align}
\ino{\eps^{2l} \p_t^k\TT^\alpha (v^\pm,b^\pm)}_{4-l-k,\pm}^2\leq C(|\psr|_{3})\bigg(& \ino{\eps^{2l}\p_t^k\TT^\alpha (v^\pm,b^\pm)}_{0,\pm}^2+\ino{\eps^{2l}\nabpr\cdot \p_t^k\TT^\alpha (v^\pm,b^\pm)}_{3-k-l,\pm}^2\no\\
&+\ino{\eps^{2l}\nabpr\times\p_t^k\TT^\alpha (v^\pm,b^\pm)}_{3-k-l,\pm}^2+\ino{\eps^{2l}\TP^{4-k-l}\p_t^k\TT^\alpha (v^\pm,b^\pm)}_{0,\pm}^2\bigg).
\end{align} Then the curl part has the following control
\begin{equation}
\begin{aligned}
&\ino{\eps^{2l}\nabpr\times\p_t^k\TT^\alpha v^\pm}_{3-l-k,\pm}^2+\ino{\eps^{2l}\nabpr\times\p_t^k\TT^\alpha b^\pm}_{3-l-k,\pm}^2 +\ino{\eps^{2l}\ffpmr \br\times(\p_t^k\TT^\alpha b^\pm)}_{3-l-k,\pm}^2\\
\lesssim &~P(\Kr_0)\left(\sum_{j=0}^l E_{4+j}^{\kk}(0)\right)+P(\Kr_0)\int_0^t \sum_{j=0}^l \Er_{4+j}^{\kk}(\tau) + \Er_{4+l+1}^{\kk}(\tau)\dtau.
\end{aligned}
\end{equation}  Similarly, the divergence part is controlled by
\begin{equation}
\begin{aligned}
&\ino{\eps^{2l}\nabpr\cdot\p_t^k\TT^\alpha v^\pm}_{3-l-k,\pm}^2+\ino{\eps^{2l}\nabpr\cdot\p_t^k\TT^\alpha b^\pm}_{3-l-k,\pm}^2\\
\lesssim &~\ino{\eps^{2l}\ffpmr\p_t^k\TT^\alpha \Dtpmr(q^\pm,b^\pm)}_{3-l-k,\pm}^2+C(\Kr_0)\left(\sum_{j=0}^l E_{4+j}^{\kk}(0)\right)+P(\Kr_0)\int_0^t \sum_{j=0}^l \Er_{4+j}^{\kk}(\tau)\dtau,
\end{aligned}
\end{equation} in which the first term will be controlled via tangential estimates.

For the pressure $q$, we still use the linearized momentum equation to convert it to tangential derivatives of $v$ and $b$. This step is exactly the same as Section \ref{sect divE4}, so we do not repeat the details here.

\subsubsection{Tangential estimates}\label{sect ETTll}

For the tangential estimates, compared with the analysis for the nonlinear uniform-in-$\kk$ estimates in Section \ref{sect ETT}-Section \ref{sect ETTG}. we find that those rather technical steps in the estimates of full time derivatives can be simplified a lot thanks to the $\sqrt{\kk}$-weighted extra regularity of the free interface. For $\TT^\gamma$-differentiated linearize system \eqref{CMHDllkk}, we introduce the corresponding Alinhac good unknown $\FFr^\gamma:=\TT^\gamma f - \TT^\gamma\pr \ppr_3 f$ which satisfies
\[
\TT^\gamma(\ppr_i f)=\ppr_i\FFr^\gamma +\ccr^\gamma_i(f),\quad\TT^\gamma(\Dtpr f)=\Dtpr\FFr^\gamma +\ddr^\gamma(f),\q \TT^\gamma(\mbpr f)=\mbpr\FFr^\gamma +\ccbr^\gamma(f)
\]where 
\begin{align}\label{AGU comm Cr}
\ccr_i^\gamma(f) =&~(\ppr_3\ppr_i f)\TT^\gamma\pr
+\left[ \TT^\gamma, \frac{\Npr_i}{\p_3 \pr}, \p_3 f\right]+\p_3 f \left[ \TT^\gamma, \Npr_i, \frac{1}{\p_3  \pr}\right] +\Npr_i \p_3 f\left[\TT^{\gamma-\gamma'}, \frac{1}{(\p_3 \pr)^2}\right] \TT^{\gamma'} \p_3   \pr \nonumber\\
&+\frac{\Npr_i}{\p_3  \pr} [\TT^\gamma, \p_3] f- \frac{\Npr_i}{(\p_3 \pr)^2}\p_3f [\TT^\gamma, \p_3] \pr,
\end{align}and 
\begin{align} \label{AGU comm Dr}
\ddr^\gamma(f) =&~ (\Dtpr \ppr_3 f)\TT^\gamma\vp+[\TT^\gamma, \vbr]\cdot \TP f + \left[\TT^\gamma, \frac{1}{\p_3\pr}(\vr\cdot \Npd-\p_t\pr), \p_3 f\right]+\left[\TT^\gamma, (\vr\cdot \Npd-\p_t\pr), \frac{1}{\p_3\pr}\right]\p_3 f\nonumber\\ 
&+\frac{1}{\p_3\vp} [\TT^\gamma, \vr]\cdot \Npd \p_3 f-(\vr\cdot \Npd-\p_t\pr)\p_3 f\left[ \TT^{\gamma-\gamma'}, \frac{1}{(\p_3 \pr)^2}\right]\TT^{\gamma'} \p_3 \pr\nonumber\\
&+\frac{1}{\p_3\pr}(\vr\cdot \Npd-\p_t\pr) [\TT^\gamma, \p_3] f+ (\vr\cdot \Npd-\p_t\pr)\frac{\p_3 f}{(\p_3\pr)^2}[\TT^\gamma, \p_3] \pr +\TT^\gamma\p_t(\pd-\pr)\ppr_3 f
\end{align} and 
\begin{align}\label{AGU comm Br}
\ccbr^\gamma(f) =&~(\mbpr\ppr_3f)\TT^\gamma\pr
+\left[ \TT^\gamma, \frac{\mbr\cdot\Npr}{\p_3 \pr}, \p_3 f\right]+\p_3 f \left[ \TT^\gamma, \mbr\cdot\Npr, \frac{1}{\p_3 \pr}\right] +(\mbr\cdot\Npr) \p_3 f\left[\TT^{\gamma-\gamma'}, \frac{1}{(\p_3 \pr)^2}\right] \TT^{\gamma'} \p_3  \pr \nonumber\\
&+[\TT^\gamma,\mbcr_i]\TP_i f +\TT^\gamma \mbr_3\,\ppr_3 f  +\frac{\mbr\cdot\Npr}{\p_3  \pr} [\TT^\gamma, \p_3] f - \frac{\mbr\cdot\Npr}{(\p_3 \pr)^2}\p_3f [\TT^\gamma, \p_3] \pr.
\end{align}
with $\len{\gamma'}=1$. Since $\Npr_3=1$, the third term in $\ccr_i^\gamma(f)$ does not appear when $i=3$. Under this setting, the $\TT^\gamma$-differentiated linearized system is reformulated as follows
\begin{align}
\label{goodlv} \rhor^\pm\Dtpmr \VVr^{\gamma,\pm}-\mbpmr\BBr^{\gamma,\pm}+\nabpr \QQr^{\gamma,\pm} =\Rr_v^{\gamma,\pm}-\ccr^\gamma(q^\pm)+\ccbr^\gamma(b^\pm) ~~&\text{ in }[0,T]\times\Om^\pm,\\
\label{goodlp} \ffpmr\Dtpmr\QQr^{\gamma,\pm}-\ffpmr\Dtpmr \BB^{\gamma,\pm}\cdot \mbr^\pm+\nabpr\cdot \VVr^{\gamma,\pm} =\Rr_p^{\gamma,\pm}-\ccr_i^\gamma(v_i^\pm)~~&\text{ in }[0,T]\times\Om^\pm,\\
\label{goodlb} \Dtpmr\BBr^\gapm-\mbpmr\VVr^\gapm+\mbr^\pm(\nabpr\cdot\VVr^\gapm)=\Rr_b^\gapm+\ccbr^\gamma(v^\pm)-\mbr^\pm\ccr_i^\gamma(v_i^\pm)~~&\text{ in }[0,T]\times\Om^\pm,\\
\label{goodls} \Dtpmr\BSr^{\pm,\alpha}=\ddr^{\gamma}(S^\pm)~~&\text{ in }[0,T]\times\Om^\pm,
\end{align}with boundary conditions
\begin{align}
\label{goodlbdc}\jump{\QQr^\gamma}=\sigma\TT^{\gamma}\h(\psr)-\kk\TT^{\gamma}(1-\TL)^2\psi-\kk\TT^{\gamma}(1-\TL)\p_t \psi-\jump{\p_3 q}\TT^\gamma\psr~~&\text{ on }[0,T]\times\Sigma,\\
\label{goodlkbc}\VVr^\gapm\cdot \npr=\p_t\TT^{\gamma}\psi+\vb^\pm\cdot\cnab\TT^\gamma\psr-\WW^\gapm~&\text{ on }[0,T]\times\Sigma,
\end{align}where $\Rr_v,\Rr_p,\Rr_b$ terms consist of the following commutators
\begin{align}
\label{goodlRv} \Rr_v^{\gamma,\pm}:=&-[\TT^\gamma,\rhor^\pm]\Dtpmr v^\pm-\rhor^\pm\ddr^\gamma(v^\pm)\\
\label{goodlRp} \Rr_p^{\gamma,\pm}:=&-[\TT^\gamma,\ffpmr]\Dtpmr q^\pm -\ffpmr \ddr^\gamma(q^\pm)  \no\\
&+ [\TT^\gamma,\ffpmr]\Dtpmr b^\pm\cdot \mbr^\pm+ \ffpmr \ddr^\gamma(b^\pm)\cdot \mbr^\pm + [\TT^\gamma, \ffpmr \mbr^\pm]\cdot \Dtpmr b^\pm\\
\label{goodlRb} \Rr_b^{\gamma,\pm}:=&~-[\TT^\gamma,\mbr^\pm](\nabpr\cdot v^\pm)-\ddr^\gamma(b^\pm),
\end{align} and the boundary term $\WW^\gapm$ is 
\begin{align}\label{goodlww}
\WWr^\gapm:=(\p_3 v^\pm\cdot \npr)\TT^\gamma\psr+[\TT^\gamma,\npr_i,v_i^\pm].
\end{align}

Given $0\leq l\leq 4$, we shall consider the tangential estimates for $\TP^{4-k-l}\p_t^k\TT^\alpha$ for $0\leq k \leq 4-l$ and $\len{\alpha}=2l,~\alpha_3=0$. Following the analysis in Section \ref{sect ETT}-Section \ref{sect ETTG}, using the linearized Reynolds transport theorem (Lemma \ref{transpt linearized}), dropping $\gamma$ for simplicity of notations, we get
\begin{align}
\ddt\frac12\iopm\rhor|\eps^{2l}\VVr^\pm|^2\dvr =& \iopm \eps^{4l} \bpmr \BBr^\pm \cdot\VVr^\pm \dvr - \iopm \eps^{4l}\VVr^\pm\cdot\nabpr\QQr^\pm \dvr \no\\
&-\iopm \eps^{4l}(\Rr_v^\pm-\ccr(q^\pm)+\ccbr^\gamma(b^\pm))\cdot\VVr^\pm\dvr \no\\
&+\frac12\iopm\eps^{4l}\left(\Dtpmr\rhor+\rhor\nabpr\cdot\vr^\pm+\rhor\ppr_3(\vbr\cdot\cnab)(\pr-\pd)\right)|\VVr^\pm|^2\dvr,
\end{align}where the last two terms can be directly controlled by $C(\Kr_0)\Er^\kk(t)$. We then analyze the first line. Integrating $\bpmr$ and $\nabpr$ by parts, using $\br\cdot\npr|_{\Sigma}=0$ and invoking the evolution equation of $\BBr$ and $\QQr$, we get
\begin{align}
& \iopm \eps^{4l}\mbpmr \BBr^\pm \cdot\VVr^\pm \dvr 
\eql & -\frac12\ddt\iopm|\eps^{2l}\BBr^\pm|^2\dvr -\frac12\ddt\iopm \ffpmr(\eps^{2l}\BBr^\pm\cdot \mbr^\pm)^2\dvt + \iopm \eps^{4l}\ffpmr(\BBr^\pm\cdot\mbr^\pm)\Dtpmr \QQr^\pm\dvr
\end{align}and
\begin{align}
 - \iopm\eps^{4l} \VVr^\pm\cdot\nabpr\QQr^\pm \dvr \eql &~ \underbrace{\pm\is\eps^{4l} (\VVr^\pm\cdot\npr)\QQr^\pm\dx'}_{=:\Ir^\pm}\underbrace{-\iopm \eps^{4l} \QQr^\pm\ccr_i(v_i^\pm)\dvr}_{=:\Zr^\pm} \no\\
&- \frac12\ddt\iopm \ffpmr (\eps^{2l}\QQr^\pm)^2 \dvr + \iopm \ffpmr \eps^{4l} \QQr^\pm\Dtpmr(\BBr^\pm\cdot\mbr^\pm)\dvr .
\end{align} Notice that
\begin{align}
 &\iopm \eps^{4l}\ffpmr(\BBr^\pm\cdot\mbr^\pm)\Dtpmr \QQr^\pm\dvr+ \iopm\eps^{4l} \ffpmr \QQr^\pm\Dtpmr(\BBr^\pm\cdot\mbr^\pm)\dvr\eql \ddt\iopm \eps^{4l}\ffpmr \QQr^\pm(\BBr^\pm\cdot\mbr^\pm)\dvr,
\end{align}we find that
\begin{align}
& \iopm \eps^{4l} \mbpmr \BBr^\pm \cdot\VVr^\pm \dvr - \iopm\eps^{4l} \VVr^\pm\cdot\nabpr\QQr^\pm \dvr\no\\
\eql&~\Ir^\pm+\Zr^\pm -\frac12\ddt\iopm |\eps^{2l}\BBr^\pm|^2\dvr -\frac12\ddt\iopm \ffpmr\bno{\eps^{2l}\left(\QQr^\pm-\BBr^\pm\cdot\mbr^\pm\right)}^2\dvr.
\end{align} Thus, we already get the energy terms for $\VVr,\BBr$ and $\QQr$, and it remains to analyze the boundary term $\Ir^\pm$. Again, following the analysis in Section \ref{sect ETT}-Section \ref{sect ETTG}, we have
\begin{align}
\Ir^+ + \Ir^- =\STr+\STr'+\VSr+\RTr+\RTr^+ + \RTr^- +\ZBr^++\ZBr^-
\end{align}where
\begin{align}
\label{def STr} \STr:=&~\eps^{4l}\is \TT^\gamma\jump{q}\,\p_t\TT^\gamma\psi\dx',\\
\label{def STr'} {\STr}':=&~\eps^{4l}\is \TT^\gamma\jump{q}\,(\vb^+\cdot\cnab)\TT^\gamma\psr\dx',\\
\label{def VSr}\VSr:=&~\eps^{4l}\is\TT^\gamma q^-\,(\jump{\vb}\cdot\cnab)\TT^\gamma\psr\dx',\\
\label{def RTr}\RTr:=&-\eps^{4l}\is\jump{\p_3 q}\TT^\gamma\psr\,\p_t\TT^\gamma\psi\dx',\\
\label{def RTr'}\RTr^{\pm}:=&\mp\eps^{4l}\is \p_3 q^\pm\,\TT^\gamma\psr\,(\vb^\pm\cdot\cnab)\TT^\gamma\psr\dx',\\
\label{def ZBr}\ZBr^{\pm}:=&\mp\eps^{4l}\is \QQr^{\pm}\WWr^{\pm}\dx',~~\Zr^{\pm}=-\iopm\eps^{4l}\QQr^{\pm}\ccr_i(v_i^\pm)\dvr.
\end{align}

\subsubsection{Analysis of the boundary integrals}\label{sect ETTbdryll}
Since the weight function $\omega(x_3)$ vanishes on $\Sigma$, we can alternatively write $\TT^\alpha=\p_t^{\alpha_0}\TP^{2l-\alpha_0}$ and $\TT^\gamma=\p_t^{k+\alpha_0}\TP^{4+l-(k+\alpha_0)}$. Replacing $k+\alpha_0$ by $k$, it suffices to analyze the case $\TT^\gamma=\p_t^k\TP^{4+l-k}$ for $0\leq k\leq 4+l,~0\leq l\leq 4$. First, there is no need to analyze $\RTr$ and $\RTr^\pm$ because they can be directly controlled by using the energy bounds \eqref{energy llkkrr} for the basic state. For the term $\STr$, the boundary regularity is given by the $\kk$-regularization terms instead of the surface tension because we do not need a uniform-in-$\kk$ estimate for the linearized problem. Using the jump conditions for $\jump{q}$ and integrating by parts, we have
\begin{align}
\int_0^t\STr\dtau\lesssim&-\bno{\sqrt{\kk}\eps^{2l}\p_t^k\TP^{4+l-k}\psi}_2^2\bigg|^t_0 - \bno{\sqrt{\kk}\eps^{2l}\p_t^{k+1}\TP^{4+l-k}\psi}_{L_t^2H_{x'}^1}^2 +\delta\bno{\sqrt{\kk}\eps^{2l}\p_t^{k+1}\TP^{4+l-k}\psi}_{L_t^2H_{x'}^1}^2+\frac{\sigma}{\kk}\int_0^t C(\Kr_0)  \dtau.
\end{align} For the term $\STr'$, we have
\begin{align}
\int_0^t{\STr}'\dtau=&~\sigma\int_0^t\is \eps^{4l} \p_t^k\TP^{4+l-k}\left(\frac{\cnab\psr}{\sqrt{1+|\cnab\psr|^2}}\right)\cdot \cnab\left((\vb^+\cdot\cnab)\p_t^{k}\TP^{4+l-k} \psr\right)\dx' \dtau\no\\
&-\kk\int_0^t\is\eps^{4l}\p_t^k\TP^{4+l-k}(1-\TL)\psi\,(1-\TL)\left((\vb^+\cdot\cnab)\p_t^{k}\TP^{4+l-k} \psr\right)\dx'\dtau \no\\
&- \kk\int_0^t\is \eps^{4l}\p_t^{k+1}\TP^{4+l-k}\jp\psi\, \jp\left((\vb^+\cdot\cnab)\p_t^{k}\TP^{4+l-k} \psr\right)\dx'\dtau \no\\
\lesssim&~\sigma C(\Kr_0,\kk^{-1})t + C(\Kr_0)\int_0^t\bno{\sqrt{\kk}\eps^{2l}\p_t^k\TP^{4+l-k}\psi}_2 \|v^+\|_{4,\pm}\dtau + \delta\bno{\sqrt{\kk}\eps^{2l}\p_t^{k+1}\TP^{4+l-k}\psi}_{L_t^2H_{x'}^1}^2 + C(\Kr_0)\int_0^t \|v^+\|_{4,\pm}^2\dtau \no\\
\lesssim&~ \sigma C(\Kr_0,\kk^{-1})t+C(\Kr_0)\int_0^t \Er_{4+l}^\kk(\tau)+\Er_{4}^\kk(\tau)\dtau.
\end{align} Here we note that the second term in $\int_0^t{\STr}'\dtau$ requires the bound for $|\p_t^{k}\TP^{4+l-k}\psi|_3$, which can be proved by the elliptic estimates as in Lemma \ref{qelliptic} (differentiating the regularized boundary condition by $\p_t^{k}\TP^{3+l-k}$ and testing the differentiated equation with $\p_t^{k}\TP^{3+l-k}(1-\TL)^2\psi$ in $L^2(\Sigma)$.) The term $\VSr$ can also be directly controlled even if $\TT^\gamma$ only contains time derivatives. When $k<4+l$, we can use the $\kk$-weighted energy to control it after integrating $\TP^{\frac12}$ by parts and using Lemma \ref{nctrace}
\begin{align}
\VSr=&\is\eps^{4l}\p_t^k\TP^{3.5+l-k} q^-\, \TP^{\frac12}\left((\jump{\vb}\cdot\cnab)\p_t^k\TP^{4+l-k}\psr\right)\dx'\no\\
\lesssim&~\|\eps^{2l}\p_t^k\TP^{4+l-k} q^-\|_{0,-}^{\frac12}\|\eps^{2l}\p_t^k\TP^{3+l-k}\p_3 q^-\|_{0,-}^{\frac12}|\vb^\pm|_{W^{\frac12,\infty}}\bno{\eps^{2l}\p_t^k\psr}_{5.5+l-k}\lesssim (\Er^\kk_4(t)+\Er^\kk_{4+l}(t))C(\Kr_0).
\end{align}
When $k=4+l$, we can first integrate $\p_t$ by parts and then integrate $\TP^{\frac12}$ by parts 
\begin{align}
&\int_0^t\VSr\dtau\eql \int_0^t\is\eps^{4l}\TP^{\frac12} (\jump{\vb}\p_t^{3+l}q^-)\, \TP^{\frac12}\p_t^{5+l}\psr\dx'\dtau +\is\eps^{4l}\p_t^{3+l} q^-\, (\jump{\vb}\cdot\cnab)\p_t^{4+l}\psr\dx'\bigg|^t_0\no\\
\lesssim&~\int_0^t\|\eps^{2l}\TP\p_t^{3+l} q^-\|_{0,-}^{\frac12}\|\eps^{2l}\p_t^{3+l}\p_3 q^-\|_{0,-}^{\frac12}|\vb^\pm|_{W^{\frac12,\infty}}\bno{\eps^{2l}\p_t^{5+l}\psr}_{0.5}\dtau\no\\
& + \delta\|\eps^{2l}\p_t^{3+l}q^-\|_{1,-}^2 + |\vb^\pm|_{L^{\infty}}^2 \bno{\eps^{2l}\p_t^4\psr}_{1}^2+C(\Kr_0)\Er^\kk(0)\no\\
\lesssim&~\delta \Er_{4+l}^\kk(t)+C(\Kr_0,\kk^{-1})\left(\Er^\kk(0)+\int_0^t \Er^\kk(\tau)\dtau\right).
\end{align}

For $\ZBr+\Zr$, the cancellation obtained in Section \ref{sect E4TT} and Section \ref{sect E8tt} still holds. Following step 4 in Section \ref{sect E8tt}, we have
\begin{align}
\ZBr^\pm+\Zr^\pm=&~\mp\is\eps^{4l}(\p_t^k\TP^{4+l-k} q^\pm -\p_t^k\TP^{4+l-k}\psr\p_3 q^\pm )(\p_3v^\pm\cdot\npr)\,\p_t^k\TP^{4+l-k}\psr\dx' \no\\
&\mp \is\eps^{4l} \QQr^\pm\left[\p_t^k\TP^{4+l-k},\npr_i,v_i^\pm\right]\dx' -\iopm\eps^{4l}\QQr^{\pm}\ccr_i(v_i^\pm)\dvr,
\end{align}where the first line is controlled in the same way as$\VSr$. Mimicing the proof in step 4 in Section \ref{sect E8tt}, we have
\begin{align}
\mp \is\eps^{4l} \QQr^\pm\left[\p_t^k\TP^{4+l-k},\npr_i,v_i^\pm\right]\dx' -\iopm\eps^{4l}\QQr^{\pm}\ccr_i(v_i^\pm)\dvr\eql\iopm \eps^{4l}\ppr_3\QQr^{\pm}\left[\p_t^k\TP^{4+l-k},\Npr_i,v_i^\pm\right]\dvr,
\end{align}whose time integral can be directly controlled by $$\delta\Er_{4+l}^\kk(t) + C(\Kr_0,\kk^{-1})\left(\Er^\kk(0)+\int_0^t \Er^\kk(\tau)\dtau\right)$$ after integrating by parts one tangential derivative in $\p_t^k\TP^{4+l-k}$. 

\subsubsection{Uniform-in-$n$ estimates for the linearized approximate system}
Summarizing the estimates obtained in Section \ref{sect divcurlll}-Section \ref{sect ETTbdryll}, we prove that for any $\delta\in(0,1)$, 
\[
\Er^\kk(t)\lesssim \delta\Er^\kk(t) + C(\Kr_0,\kk^{-1})\left(\Er^\kk(0)+\int_0^t \Er^\kk(\tau)\dtau\right).
\] Choosing $\delta>0$ suitably small such that the $\delta$-term can be absorbed by the left side and using Gr\"onwall's inequality, we find that there exists a time $T_\kk>0$ (independent of $\eps$ and $n$), such that
\[
\sup_{0\leq t\leq T_\kk}\Er^\kk(t)\leq C'(\Kr_0,\kk^{-1})\Er^\kk(0)
\]for some positive function $C'$ continuous in its arguments. Following the argument in remark \ref{mbrbound}, it is straightforward to show that
\[
\sum_\pm\sum_{l=0}^4\sum_{\lee\alpha\ree=2l}\sum_{k=0}^{4-l}\int_0^t\ino{\eps^{2l}\TT^{\alpha}\p_t^{k}\mb^\pm}_{4-k-l,\pm}^2\dtau< P(\Er^\kk(t)) \quad \forall t\in[0,T_\kk].
\]

\subsection{Picard iteration}\label{sect picard}
We already establish the local existence of the linear system \eqref{CMHD0llkk} for each $n$ and the uniform-in-$n$ estimates for the solution to \eqref{CMHD0llkk}. It suffices to prove $\{(v^{\nnn,\pm},b^{\nnn,\pm},\mb^{\nnn,\pm},q^{\nnn,\pm},\psi^{\nnn})\}$ has a strongly convergent subsequence (in certain anisotropic Sobolev norms). For a function sequence $\{f^{\nnn,\pm}\}$, we define $[f]^{\nnn,\pm}:=f^{\nnr,\pm}-f^{\nnn,\pm}$. Then we can write the linear system of $\{([v]^{\nnn,\pm},[b]^{\nnn,\pm},[q]^{\nnn,\pm},[\psi]^{\nnn})\}$ as follows
\begin{equation}\label{CMHDllkkgap}
\begin{cases}
\rho^{\nnn,\pm} D_t^{\vp^\nnn}[v]^{\nnn,\pm} - (\mb^{\nnn,\pm}\cdot\nab^{\vp^\nnn}) [b]^{\nnn,\pm}+\nab^{\vp^\nnn} [q]^{\nnn,\pm}+\nab^{[\vp]^\nnl}q^{\nnn,\pm} =-\fr_v^{\nnn,\pm} &\text{ in }[0,T]\times\Om^\pm\\
\ffp^{\nnn,\pm} D_t^{\vp^\nnn}[q]^{\nnn,\pm} - \ffp^{\nnn,\pm} D_t^{\vp^\nnn}[b]^{\nnn,\pm}\cdot\mb^{\nnn,\pm} +\nab^{\vp^\nnn}\cdot [v]^{\nnn,\pm}+\nab^{[\vp]^\nnl}\cdot v^{\nnn,\pm}  =-\fr_p^{\nnn,\pm} &\text{ in }[0,T]\times\Om^\pm\\
 D_t^{\vp^\nnn}[b]^{\nnn,\pm} - (\mb^{\nnn,\pm}\cdot\nab^{\vp^\nnn}) [v]^{\nnn,\pm}+\mb^{\nnn,\pm}(\nab^{\vp^\nnn} \cdot[v]^{\nnn,\pm}+\nab^{[\vp]^\nnl}\cdot v^{\nnn,\pm}) =-\fr_b^{\nnn,\pm} &\text{ in }[0,T]\times\Om^\pm\\
 D_t^{\vp^\nnn}[S]^{\nnn,\pm}=-\fr_S^{\nnn,\pm}&\text{ in }[0,T]\times\Om^\pm\\
\jump{[q]^{\nnn,\pm}} = \sigma (\h(\psi^\nnn)-\h(\psi^\nnl)) - \kk(1-\TL)^2[\psi]^\nnn - \kk(1-\TL)\p_t[\psi]^\nnn&\text{ on }[0,T]\times\Sigma\\
\p_t[\psi]^\nnn = [v]^{\nnn,\pm}\cdot N^\nnn + v^{\nnn,\pm}\cdot [N]^{\nnl}&\text{ on }[0,T]\times\Sigma\\
v_3^{\nnn,\pm}=v_3^{\nnl,\pm}=\mb_3^{\nnn,\pm}=\mb_3^{\nnl,\pm}=0&\text{ on }[0,T]\times\Sigma^\pm\\
([v]^\nnn,[b]^\nnn,[q]^\nnn,[\psi]^\nnn)|_{t=0}=(\vec{0},\vec{0},0,0),
\end{cases}
\end{equation}where the source terms are defined by
\begin{align}
\fr_v^{\nnn,\pm}:=&~[\rho]^{\nnl,\pm}\p_t v^{\nnn,\pm} + [\rho\vb]^{\nnl,\pm}\cdot\cnab v^{\nnn,\pm} + [\rho V_{\NN}]^{\nnl,\pm}\p_3 v^{\nnn,\pm}\no\\
&~ - [\mb]^{\nnl,\pm}\cdot\cnab b^{\nnn,\pm} - [\bm{B}_\NN]^{\nnl,\pm}\p_3 b^{\nnn,\pm},\\
\fr_q^{\nnn,\pm}:=&~[\ffp]^{\nnl,\pm}\p_t q^{\nnn,\pm} + [\ffp\vb]^{\nnl,\pm}\cdot\cnab q^{\nnn,\pm} + [\ffp V_\NN]^{\nnl,\pm}\p_3 q^{\nnn,\pm}\no\\
&- ([\ffp]^{\nnl,\pm}\p_t b^{\nnn,\pm} + [\ffp\vb]^{\nnl,\pm}\cdot\cnab b^{\nnn,\pm} + [\ffp V_\NN]^{\nnl,\pm}\p_3 b^{\nnn,\pm})\cdot\mb^{\nnn,\pm}\no\\
& - (\ffpm)^{\nnl}D_t^{\vp^\nnl}b^{\nnn,\pm}\cdot[\mb]^{\nnl,\pm},\\
\fr_p^{\nnn,\pm}:=&~[\vb]^{\nnl,\pm}\cdot\cnab b^{\nnn,\pm} + [V_{\NN}]^{\nnl,\pm}\p_3 b^{\nnn,\pm} - [\mb]^{\nnl,\pm}\cdot\cnab v^{\nnn,\pm} - [\bm{B}_\NN]^{\nnl,\pm}\p_3 v^{\nnn,\pm}\no\\
&+  [\mb]^{\nnl,\pm}(\nab^{\vp^\nnl} \cdot v^{\nnn,\pm}),\\
\fr_S^{\nnn,\pm}:=&~[\vb]^{\nnl,\pm}\cdot\cnab S^{\nnn,\pm} + [V_{\NN}]^{\nnl,\pm}\p_3 S^{\nnn,\pm},
\end{align}with
\begin{align*}
V_\NN^\nnn:=\frac{1}{\p_3\vp^\nnn}(v^\nnn\cdot\NN^\nnl-\p_t\vp^\nnn),\quad \bm{B}_\NN^\nnn:=\frac{1}{\p_3\vp^\nnn}(\mb^\nnn\cdot\NN^\nnn),\quad \nab^{[\vp]^\nnl} f^\nnn:= -[\NN/\p_3\vp]^\nnl\p_3 f^\nnn
\end{align*}

For $1\leq n\in\N^*$, we define the energy for the linear system \eqref{CMHDllkkgap} as follows
\begin{align}
[\Er^\kk]^\nnn(t):=&~[\Er^\kk]_3^\nnn(t)+\cdots+[\Er^\kk]_6^\nnn(t),\no\\
[\Er^\kk]_{3+l}^\nnn(t):=&\sum_\pm\sum_{k=0}^3\sum_{\len{\alpha}=2l}\ino{\eps^{2l}\p_t^k\TT^\alpha([v]^{\nnn,\pm},[b]^{\nnn,\pm},[q]^{\nnn,\pm},[S]^{\nnn,\pm})}_{3-k-l}^2 \no\\
&+\sum_{k=0}^{3+l}\bno{\sqrt{\kk}\eps^{2l}\p_t^k[\psi]^\nnn}_{5+l-k}^2 +\int_0^t\bno{\sqrt{\kk}\eps^{2l}\p_t^{4+l}[\psi]^\nnn}_{1}^2\dtau,~~0\leq l\leq 3,
\end{align}where $\TT^{\alpha}:=(\omega(x_3)\p_3)^{\alpha_4}\p_t^{\alpha_0}\p_1^{\alpha_1}\p_2^{\alpha_2}$ with the multi-index $\alpha=(\alpha_0,\alpha_1,\alpha_2,0,\alpha_4),~\lee \alpha\ree=\alpha_0+\alpha_1+\alpha_2+2\times0+\alpha_4$. It should be noted that the initial value of $[\Er^\kk]^\nnn$. Thus, we shall prove the following proposition in order for the strong convergence.
\begin{prop}\label{prop CMHDllgap}
There exists a time $T_\kk'>0$ depending on $\kk$ and $\Kr_0$, such that
\begin{align}
\forall 2\leq n\in\N^*,\quad \sup_{0\leq t\leq T_{\kk}'}[\Er^\kk]^{\nnn}(t)\leq\frac{1}{4}\left(\sup_{0\leq t\leq T_{\kk}'}[\Er^\kk]^{\nnl}(t)+\sup_{0\leq t\leq T_{\kk}'}[\Er^\kk]^{\nnll}(t)\right).
\end{align}
\end{prop}

\subsubsection*{Step 1: Div-Curl analysis and reduction of pressure}

The reduction of pressure follows in the same way as in Section \ref{sect divE4}. Invoking the momentum equation, we have
\[
-(\p_3\vp^{\nnn})^{-1}\p_3[q]^{\nnn,\pm} = \rho^{\nnn,\pm} D_t^{\vp^\nnn}[v]^{\nnn,\pm} - (\mb^{\nnn,\pm}\cdot\nab^{\vp^\nnn}) [b]^{\nnn,\pm} +\fr_v^{\nnn,\pm}+(\p_3\vp^{\nnn})^{-1}\p_3 q ^{\nnn,\pm}.
\]Then using $\ppr_i=\TP_i-\TP_i\pr\ppr_3$, we can convert $\p q$ to a spatial derivative of $v$ and $b$ plus the given term $\p_3 q ^{\nnn,\pm}$. 

For the div-curl analysis, using \eqref{divcurlTT}, we have for $0\leq l\leq 2,~0\leq k\leq 2-l$
\begin{align}
&\ino{\eps^{2l}\p_t^k\TT^\alpha ([v]^{\nnn,\pm},[b]^{\nnn,\pm})}_{3-k-l,\pm}^2\no\\
\leq C(\Kr_0)& \bigg(\|\eps^{2l}\p_t^k\TT^\alpha ([v]^{\nnn,\pm},[b]^{\nnn,\pm})\|_{0,\pm}^2 +\ino{\eps^{2l}\nab^{\vp^\nnn}\times\p_t^k\TT^\alpha ([v]^{\nnn,\pm},[b]^{\nnn,\pm})}_{2-k-l,\pm}^2 \no\\
&+\ino{\eps^{2l}\nab^{\vp^\nnn}\cdot\p_t^k\TT^\alpha ([v]^{\nnn,\pm},[b]^{\nnn,\pm})}_{2-k-l,\pm}^2+\ino{\TP^{3-k-l}\p_t^k\TT^\alpha ([v]^{\nnn,\pm},[b]^{\nnn,\pm})}_{0,\pm}^2\bigg).
\end{align} 

For the curl part, we again analyze the evolution equations of vorticity and current
\begin{align}
&\rho^\nnn D_t^{\vp^\nnn}(\nab^{\vp^\nnn}\times[v]^\nnn) - (\mb^\nnn\cdot\nab^{\vp^\nnn})(\nab^{\vp^\nnn}\times [b]^\nnn) \no\\
=& -\nab^{\vp^\nnn}\times\fr_v^\nnn - \nab^{\vp^\nnn}\rho^\nnn\times D_t^{\vp^\nnn}[v]^\nnn  +(\nab^{\vp^\nnn}\mb_j^{\nnn})\times(\p_j^{\vp^\nnn}[b]^\nnn)+\rho^\nnn [D_t^{\vp^\nnn},\nab^{\vp^\nnn}\times][v]^\nnn,\\
&D_t^{\vp^\nnn}(\nab^{\vp^\nnn}\times [b]^\nnn) - (\mb^\nnn\cdot\nab^{\vp^\nnn})(\nab^{\vp^\nnn}\times [v]^\nnn) - \mb^\nnn\times (\nab^{\vp^\nnn}(\nab^{\vp^\nnn}\cdot [v]^\nnn) )\no\\
=&- \nab^{\vp^\nnn}\times \fr_b^\nnn + [D_t^{\vp^\nnn},\nab^{\vp^\nnn}\times][b]^\nnn + (\nab^{\vp^\nnn}\mb_j^{\nnn})\times(\p_j^{\vp^\nnn}[v]^\nnn)  \no\\
&- \nab^{\vp^\nnn}\times(\mb^\nnn \nab^{[\vp]^\nnl}\cdot v^{\nnn,\pm})-(\nab^{\vp^\nnn}\times \mb^\nnn)(\nab^{\vp^\nnn}\cdot [v]^\nnn).
\end{align}

Mimicing the proof in Section \ref{sect divcurlll} and using the vanishing initial value of system \eqref{CMHDllkkgap}, we can prove
\begin{equation}
\begin{aligned}
&\ino{\eps^{2l}\nab^{\vp^\nnn}\times\p_t^k\TT^\alpha [v]^{\nnn,\pm}}_{2-k-l,\pm}^2+\ino{\eps^{2l}\nab^{\vp^\nnn}\times\p_t^k\TT^\alpha [b]^{\nnn,\pm}}_{2-k-l,\pm}^2 +\ino{\eps^{2l}\sqrt{(\ffpm)^\nnn} \mb^{\nnn,\pm}\times(\p_t^k\TT^\alpha [b]^\pm)}_{2-l-k,\pm}^2\\
\lesssim &~C(\Kr_0)\int_0^t \sum_{j=0}^l [E]_{3+j}^{\nnn}(\tau) + [E]_{3+l+1}^{\nnn}(\tau)\dtau.
\end{aligned}
\end{equation}  Similarly, the divergence of $[v]^\nnn$ can be converted to tangential derivatives of $[q]^\nnn$ and $[b]^\nnn$ by invoking the continuity equation, and the evolution equation of $\nabpr\cdot [b]^\nnn$ is 
\begin{align}
D_t^{\vp^\nnn}(\nab^{\vp^\nnn}\cdot [b]^\nnn)=&~(\p_i^{\vp^\nnn}\mb^\nnn_j)(\p_j^{\vp^\nnn} [v]_i^\nnn) - (\nab^{\vp^\nnn}\cdot\mb^\nnn)(\nab^{\vp^\nnn}\cdot [v]^\nnn)+ [D_t^{\vp^\nnn},\nabpr\cdot][b]^\nnn \no\\
&-\nab^{\vp^\nnn}\cdot(\fr_b^\nnn+\mb^\nnn \nab^{[\vp]^\nnl}\cdot v^{\nnn,\pm}),
\end{align}so the divergence part is controlled by
\begin{equation}
\begin{aligned}
&\ino{\eps^{2l}\nab^{\vp^\nnn}\cdot\p_t^k\TT^\alpha [v]^{\nnn,\pm}}_{2-k-l,\pm}^2+\ino{\eps^{2l}\nab^{\vp^\nnn}\cdot\p_t^k\TT^\alpha [b]^{\nnn,\pm}}_{2-k-l,\pm}^2\\
\lesssim &~\ino{\eps^{2l}(\ffpm)^\nnn\p_t^k\TT^\alpha D_t^{\vp^\nnn}([q]^\pm,[b]^\pm)}_{2-l-k,\pm}^2+C(\Kr_0)\int_0^t \sum_{j=0}^l [E]_{3+j}^{\nnn}(\tau)\dtau,
\end{aligned}
\end{equation} in which the first term will be reduced to tangential estimates.

\subsubsection*{Step 2: Tangential estimates}
It remains to prove the tangential estimates for $\TT^\gamma$-differentiated system where $\TT^\gamma=\TP^{3-l-k}\p_t^k\TT^\alpha$ satisfies $\alpha_3=0,\len{\alpha}=2l,~0\leq k\leq 3-l,~0\leq l\leq 3$. We shall introduce the Alinhac good unknowns $([\VV],[\BB],[\QQ])$ as below instead of directly taking tangential derivatives in \eqref{CMHDllkkgap}.
\begin{align*}
[\FF]^\nnn:=&~\FFr^\nnr-\FFr^\nnn =\TT^\gamma [f]^\nnn -\TT^\gamma\vp^\nnn\p_3^{\vp^\nnn} [f]^\nnn-\TT^\gamma\vp^\nnn \p_3^{[\vp]^\nnl}q^\nnn - \TT^\gamma[\vp]^\nnl \p_3^{\vp^\nnl}q^\nnn
\end{align*}and it satisfies
\begin{align*}
\TT^\gamma(\p_i^{\vp^\nnn}[f]^\nnn+\p_i^{[\vp]^\nnl}f^\nnn) = \p_i^{\vp^\nnn}[\FF]^\nnn + [\cc]_i^\nnn(f),\\ \TT^\gamma(D_t^{\vp^\nnn}[f]^\nnn+D_t^{[\vp]^\nnl}f^\nnn) = D_t^{\vp^\nnn}[\FF]^\nnn + [\dd]^\nnn (f),\\
\TT^\gamma((\mb^\nnn\cdot\nab^{\vp^\nnn})[f]^\nnn+(\mb^\nnn\cdot\nab^{[\vp]^\nnl})f^\nnn) =(\mb^\nnn\cdot\nab^{\vp^\nnn})[\FF]^\nnn + [\ccb]^\nnn (f)
\end{align*}with 
\begin{align*}
\|[\cc]_i^\nnn(f)-(\cc_i^\nnn(f^\nnr) - \cc_i^\nnl(f^\nnn)) \|_0^2\leq C(\Kr_0)([\Er^\kk]^\nnn(t)+[\Er^\kk]^\nnl(t)+[\Er^\kk]^\nnll(t)) \\
\|[\dd]^\nnn(f)-(\dd^\nnn(f^\nnr) - \dd^\nnl(f^\nnn))\|_0^2 \leq C(\Kr_0)([\Er^\kk]^\nnn(t)+[\Er^\kk]^\nnl(t)+[\Er^\kk]^\nnll(t))\\
\|[\ccb]^\nnn(f)-(\ccb^\nnn(f^\nnr) - \ccb^\nnl(f^\nnn))\|_0^2 \leq C(\Kr_0)([\Er^\kk]^\nnn(t)+[\Er^\kk]^\nnl(t)+[\Er^\kk]^\nnll(t))
\end{align*}where $\cc_i^\nnn(f^\mmm),\dd^\nnn(f^\mmm),\ccb^\nnn(f^\mmm)$ are defined by setting $\pr=\vp^\nnn,~\pd=\vp^\nnl,~f^\nnr=f,~f^\nnn=\fr,~\mb^\nnn=\mbr,~\mb^\nnl=\dot{\mb}$ in \eqref{AGU comm Cr}-\eqref{AGU comm Br}. This can be seen by substracting the corresponding identities of $\FFr$ with superscript $\nnl$ from the ones with superscript $\nnn$. The evolution equations of the good unknowns are (with $\pm$ dropped)
\begin{align}
\rho^\nnn D_t^{\vp^\nnn} [\VV]^\nnn - (\mb^\nnn\cdot\nab^{\vp^\nnn}) [\BB]^\nnn + \nab^{\vp^\nnn}[\QQ]^\nnn =& -\cc^\nnn(q^\nnr) +\cc^\nnl(q^\nnn) +\ccb^\nnn(b^\nnr)-\ccb^\nnl(b^\nnn) + [\RR]_v\\
\ffp^\nnn D_t^{\vp^\nnn}[\QQ]^\nnn - \ffp^\nnn D_t^{\vp^\nnn}[\BB]^\nnn\cdot\mb^\nnn +  \nab^{\vp^\nnn}\cdot[\VV]^\nnn=&-\cc_i^\nnn(v_i^\nnr) +\cc_i^\nnl(v_i^\nnn) + [\RR]_q\\
 D_t^{\vp^\nnn} [\BB]^\nnn - (\mb^\nnn\cdot\nab^{\vp^\nnn}) [\VV]^\nnn + \mb^\nnn(\nab^{\vp^\nnn}\cdot[\VV]^\nnn) =&~\ccb^\nnn(v^\nnr)-\ccb^\nnl(v^\nnn) + [\RR]_v -\mb^\nnn\left(\cc_i^\nnn(v_i^\nnr) - \cc_i^\nnl(v_i^\nnn)\right) + [\RR]_b
\end{align}where $[\RR]$ terms are controllable in $L^2(\Om)$ by $$\|[\RR]\|_{0}^2\leq C(\Kr_0)([\Er^\kk]^\nnn(t)+[\Er^\kk]^\nnl(t)+[\Er^\kk]^\nnll(t)).$$
The boundary conditions of these good unknowns on the interface $\Sigma$ are
\begin{align}
[\QQ]^\nnn:=&~\sigma\TT^\gamma\left(\h(\psi^\nnn)-\h(\psi^\nnl)\right) - \kk (1-\TL)^2 \TT^\gamma[\psi]^\nnn - \kk(1-\TL)\p_t\TT^\gamma[\psi]^\nnn \no\\
& -\TT^\gamma\psi^\nnn\jump{ \p_3[q]^\nnn} - \TT^\gamma[\psi]^\nnl\jump{\p_3q^\nnn}\\
[\VV]^\nnn\cdot N^\nnn:=&~\TT^\gamma\p_t[\psi]^\nnn + [\vb]^\nnn\cdot\cnab\TT^\gamma\psi^\nnn +  (\vb^\nnn\cdot\cnab)\TT^\gamma[\psi]^\nnl + \TT^\gamma \vb^\nnn\cdot\cnab [\psi]^\nnl - [\WW]^\nnn\\
[\WW]^\nnn:=&~(\p_3[v]^\nnn\cdot N^\nnn)\TT^\gamma \psi^\nnn+ (\p_3v^\nnn\cdot N^\nnn)\TT^\gamma [\psi]^\nnl +\left[\TT^\gamma, N_i^\nnn, v_i^\nnr\right] - \left[\TT^\gamma, N_i^\nnl, v_i^\nnn\right] 
\end{align}

Given $0\leq l\leq 3$, following Section \ref{sect ETTll}, we can similarly prove that
\begin{align}
&\sum_\pm\ddt\frac12\iopm\eps^{2l}\rho^\nnn|[\VV]^{\nnn,\pm}|^2+|[\BB]^{\nnn,\pm}|^2+(\ffpm)^{\nnn,\pm}([\QQ]^{\nnn,\pm}-[\BB]^{\nnn,\pm}\cdot\mb^{\nnn,\pm})^2 \dvt^\nnn \no \\
=&~[\ST]^\nnn+[\ST']^\nnn+[\VS]^\nnn+[\RT]^\nnn+\sum_\pm[\RT]^{\nnn,\pm}+([ZB]^{\nnn,\pm}+[Z]^{\nnn,\pm}) \\
&+C(\Kr_0)([\Er^\kk]^\nnn(t)+[\Er^\kk]^\nnl(t)+[\Er^\kk]^\nnll(t)) \no
\end{align}where the term $[\Er^\kk]^\nnl+[\Er^\kk]^\nnll$ is produced from the estimates of $[\vp]^\nnl, [\vp]^\nnll$. The above terms on the right side are defined by
\begin{align}
\label{def [ST]} [\ST]^\nnn:=&~\eps^{4l}\is \TT^\gamma\jump{[q]^\nnn}\,\p_t\TT^\gamma[\psi]^\nnn\dx',\\
\label{def [ST']} [\ST']^\nnn:=&~\eps^{4l}\is \TT^\gamma\jump{[q]^\nnn}\,([\vb^+]^\nnn\cdot\cnab)\TT^\gamma\psi^\nnn\dx' +\eps^{4l}\is\TT^\gamma\jump{[q]^\nnn}\,(\vb^{\nnn,+}\cdot\cnab)\TT^\gamma[\psi]^\nnl \dx' ,\\
\label{def [VS]}[\VS]^\nnn:=&~\eps^{4l}\is\TT^\gamma [q]^{\nnn,-}\,(\jump{\vb}^\nnn\cdot\cnab)\TT^\gamma\psi^\nnn\dx'+\eps^{4l}\is\TT^\gamma [q]^{\nnn,-}\,\left(\jump{\vb^{\nnn}}\cdot\cnab\right)\TT^\gamma[\psi]^\nnl \dx' ,\\
\label{def [RT]}[\RT]^\nnn:=&-\eps^{4l}\is\left(\jump{\p_3 [q]^\nnn}\TT^\gamma\psi^\nnn+\TT^\gamma[\psi]^\nnl\jump{\p_3q^\nnn}\right)\,\p_t\TT^\gamma\psi\dx',\\
\label{def [RT']}[\RT]^{\nnn,\pm}:=&\mp\eps^{4l}\is \left(\p_3 [q]^{\nnn,\pm}\,\TT^\gamma\psi^\nnn+\TT^\gamma[\psi]^\nnl\p_3q^{\nnn,\pm}\right)\,\left((\vb^\pm\cdot\cnab)\TT^\gamma\psr+(\vb^{\nnn,\pm}\cdot\cnab)\TT^\gamma[\psi]^\nnl\right)\dx',\\
\label{def [ZB]}[ZB]^{\nnn,\pm}:=&\mp\eps^{4l}\is [\QQ]^{\nnn,\pm}[\WW]^{\nnn,\pm}\dx',~~[Z]^{\nnn,\pm}=-\iopm\eps^{4l}[\QQ]^{\nnn,\pm}(\cc_i^\nnn(v_i^{\nnr,\pm})-\cc_i^\nnl(v_i^{\nnn,\pm}))\dvt^\nnn.
\end{align}

\subsubsection*{Step 3: Boundary regularity of $[\psi]$}

The analysis of the boundary integrals is still similar to Section \ref{sect ETTbdryll}. Since $\omega(x_3)=0$ on $\Sigma$, we can rewrite $\p_t^k\TT^\alpha$ to be $\p_t^k\TP^{3+l-k}$ for $0\leq k\leq 3+l,~0\leq l\leq3$. Then the term $[\ST]^\nnn$ gives the regularity of $[\psi]^\nnn$ after inserting the jump condition for $[q]^\nnn$
\begin{align}
\int_0^t[\ST]^\nnn\dtau\lesssim -\bno{\sqrt{\kk}\eps^{2l}\p_t^k[\psi]^\nnn}_{5+l-k}^2\bigg|^t_0 -\int_0^t\bno{\sqrt{\kk}\eps^{2l}\p_t^{k+1}[\psi]^\nnn}_{4+l-k}^2 +\frac{\sigma}{\kk}C(\Kr_0)\int_0^t [\Er^\kk]^\nnn(\tau)+[\Er^\kk]^\nnl(\tau)\dtau.
\end{align} The term $[\ST']^\nnn$ can be controlled by inserting the jump condition for $[q]^\nnn$ and then integrating by parts $\cnab\cdot,1-\TL,\sqrt{1-\TL}$ in the three terms in $[q]^\nnn$ respectively. This is essentially the same as shown in Section \ref{sect ETTbdryll}, so we only list the result
\begin{align}
\int_0^t[\ST]^\nnn\dtau\lesssim\sigma C(\Kr_0,\kk^{-1})t+C(\Kr_0)\int_0^t[\Er^\kk]^\nnn(\tau)+[\Er^\kk]^\nnl(\tau)\dtau
\end{align}
The terms $[\RT]^\nnn,[\RT]^{\nnn,\pm}$ are also controlled directly with the help of $\kk$-weighted enhanced regularity. The term $[\VS]^\nnn$ is also controlled directly by integrating by parts for one tangential derivative in $\p_t^k\TP^{3+l-k}$ as in Section \ref{sect ETTbdryll}. Finally, for the term $([ZB]^{\nnn,\pm}+[Z]^{\nnn,\pm})$, we still have the previously-used cancellation structure
 \begin{align}
[ZB]^{\nnn,\pm}+[Z]^{\nnn,\pm}\eql&\mp \is\eps^{4l} [\QQ]^{\nnn,\pm}\left[\p_t^k\TP^{3+l-k},N_i^\nnn,v_i^{\nnn,\pm}\right]\dx' -\iopm\eps^{4l}[\QQ]^{\pm}\cc^\nnn_i(v_i^{\nnn,\pm})\dvt^\nnn \no\\
&\pm \is\eps^{4l} [\QQ]^{\nnn,\pm}\left[\p_t^k\TP^{3+l-k},N_i^\nnl,v_i^{\nnr,\pm}\right]\dx' + \iopm\eps^{4l}[\QQ]^{\nnn,\pm}\cc^\nnl_i(v_i^{\nnn,\pm})\dvt^\nnn.
\end{align}Mimicing the proof in step 4 in Section \ref{sect E8tt}, we have
\begin{align}
&\mp \is\eps^{4l} [\QQ]^{\nnn,\pm}\left[\p_t^k\TP^{3+l-k},N_i^\nnn,v_i^{\nnn,\pm}\right]\dx' -\iopm\eps^{4l}[\QQ]^{\pm}\cc^\nnn_i(v_i^{\nnn,\pm})\dvt^\nnn \no\\
\eql&~ \iopm \eps^{4l}\p_3^{\vp^\nnn}[\QQ]^{\nnn,\pm}\left[\p_t^k\TP^{3+l-k},\NN_i^{\nnn},v_i^{\nnr,\pm}\right]\dvr,
\end{align}whose time integral can be directly controlled by $$\delta[\Er^\kk]^\nnn(t) + C(\Kr_0,\kk^{-1})\int_0^t [\Er^\kk]^\nnn(\tau)+[\Er^\kk]^\nnl(\tau)\dtau$$ after integrating by parts for one tangential derivative in $\p_t^k\TP^{3+l-k}$. Similar estimate applies to the second line of $[ZB]^{\nnn,\pm}+[Z]^{\nnn,\pm}$: 
\begin{align*}
&\int_0^t\left(\pm \is\eps^{4l} [\QQ]^{\nnn,\pm}\left[\p_t^k\TP^{3+l-k},N_i^\nnl,v_i^{\nnr,\pm}\right]\dx' + \iopm\eps^{4l}[\QQ]^{\nnn,\pm}\cc^\nnl_i(v_i^{\nnn,\pm})\dvt^\nnn\right)\dtau \no\\
\lesssim&~\delta[\Er^\kk]^\nnn(t) + C(\Kr_0,\kk^{-1})\int_0^t [\Er^\kk]^\nnn(\tau)+[\Er^\kk]^\nnl(\tau)+[\Er^\kk]^\nnll(\tau)\dtau
\end{align*}

\subsubsection*{Step 4: Convergence}
Summarizing the above estimates and using $[\Er^\kk]^\nnn(0)=0$, we obtain the energy inequality 
\[
[\Er^\kk]^\nnn(t)\lesssim \delta[\Er^\kk]^\nnn(t) +  C(\Kr_0,\kk^{-1})\int_0^t [\Er^\kk]^\nnn(\tau)+[\Er^\kk]^\nnl(\tau)+[\Er^\kk]^\nnll(\tau)\dtau.
\]Choosing $0<\delta\ll1$ suitably small, the $\delta$-term can be absorbed by the left side. Thus, there exists a time $T_\kk'>0$ depending on $\kk,\Kr_0$ and independent of $n$, such that
\begin{equation}
\sup_{0\leq t\leq T_{\kk}'}[\Er^\kk]^{\nnn}(t)\leq\frac{1}{4}\left(\sup_{0\leq t\leq T_{\kk}'}[\Er^\kk]^{\nnl}(t)+\sup_{0\leq t\leq T_{\kk}'}[\Er^\kk]^{\nnll}(t)\right),
\end{equation}and thus we know by induction that
\begin{equation}
\sup_{0\leq t\leq T_{\kk}'}[\Er^\kk]^{\nnn}(t)\leq C(\Kr_0,\kk^{-1})/2^{n-1}\to 0 \text{ as }n\to+\infty.
\end{equation} 
Hence, for any fixed $\kk>0$, the sequence of approximate solutions $\{(v^{\nnn,\pm},b^{\nnn,\pm},\mb^{\nnn,\pm}, q^{\nnn,\pm},\psi^\nnn)\}_{n\in\N^*}$ has a strongly convergent subsequence. We write the limit function to be  $\{(v^{[\infty],\pm},b^{[\infty],\pm},\mb^{[\infty],\pm} q^{[\infty],\pm},\psi^{[\infty]})\}_{n\in\N^*}$.

\subsection{Well-posedness of the nonlinear approximate problem}\label{sect recoverkk}
We now record the existence of a unique solution to \eqref{CMHD0kk} in the following proposition.
\begin{prop}\label{thm CMHD0kk lwp}
Fix $\kk>0$. Assume the initial data $v_0^{\kk,\pm}, b_0^{\kk,\pm}, q_0^{\kk,\pm}, S_0^{\kk,\pm}\in H_*^8(\Om^\pm)$ and $\psi_0^{\kk}\in H^{10}(\Sigma)$ satisfy the compatibility conditions \eqref{comp cond kk} up to 7-th order, the constraints $\nab^{\varphi_0^\kk}\cdot b_0^{\kk,\pm}=0$ in $\Om^\pm$ and $b^{\kk,\pm}\cdot N|_{\{t=0\}\times(\Sigma\cup\Sigma^\pm)}=0$ and $|\psi_0^\kk|_{L^\infty(\Sigma)}\leq 1$. Then there exists a time $T_{\kk}'>0$ depending on $\kk$ and the initial data, such that system \eqref{CMHD0kk} admits a unique solution $v^{\kk,\pm}, b^{\kk,\pm}, q^{\kk,\pm}, S^{\kk,\pm}\in H_*^8(\Om^\pm)$ and $\psi^{\kk}\in H^{10}(\Sigma)$ satisfying the estimates
\[
\sup_{0\leq t\leq T_{\kk}'}E^\kk(t) \leq C(\kk^{-1})P(E^\kk(0)),
\]where $E^\kk(t)$ is defined by \eqref{energy lwp kk}.
\end{prop}
\begin{proof}
The limit functions obtained in Section \ref{sect picard}, denoted by $(v^{[\infty],\pm},b^{[\infty],\pm},\mb^{[\infty],\pm},q^{[\infty],\pm},S^{[\infty],\pm},\psi^{[\infty]})$ satisfy the following system.
\begin{equation}\label{CMHDllkklimit}
\begin{cases}
\rho^{[\infty],\pm} D_t^{\vp^{[\infty]},\pm} v^{{[\infty]},\pm} - (\mb^{[\infty],\pm}\cdot\nab^{\vp^{[\infty]}}) b^{{[\infty]},\pm}+\nab^{\vp^{[\infty]}} q^{{[\infty]},\pm}=0&~~\text{ in }[0,T]\times \Omega^\pm,\\
(\ffpm)^{[\infty]} D_t^{\vp^{{[\infty]}},\pm} q^{{[\infty]},\pm} - (\ffpm)^{[\infty]} D_t^{\vp^{{[\infty]}},\pm} b^{{[\infty]},\pm}\cdot \mb^{{[\infty]},\pm} +\nab^{\vp^{[\infty]}}\cdot v^{{[\infty]},\pm}=0 &~~\text{ in }[0,T]\times \Omega^\pm,\\
D_t^{\vp^{{[\infty]}},\pm} b^{{[\infty]},\pm}-(\mb^{[\infty],\pm}\cdot\nab^{\vp^{[\infty]}}) v^{{[\infty]},\pm}+\mb^{{[\infty]},\pm}\nab^{\vp^{[\infty]}}\cdot v^{{[\infty]},\pm}=0&~~\text{ in }[0,T]\times \Omega^\pm,\\
D_t^{\vp^{{[\infty]}},\pm} S^{{[\infty]},\pm}=0&~~\text{ in }[0,T]\times \Omega^\pm,\\
\jump{q^{[\infty]}}=\sigma\h(\psi^{{[\infty]}})- \kk (1-\TL)^2\psi^{{[\infty]}} - \kk (1-\TL)\p_t\psi^{{[\infty]}}&~~\text{ on }[0,T]\times\Sigma, \\
\p_t \psi^{{[\infty]}} = v^{{[\infty]},\pm}\cdot N^{{[\infty]}} &~~\text{ on }[0,T]\times\Sigma,\\
v_3^{{[\infty]},\pm}=0&~~\text{ on }[0,T]\times\Sigma^\pm,\\
(v^{{[\infty]},\pm},b^{{[\infty]},\pm},q^{{[\infty]},\pm},S^{{[\infty]},\pm},\psi^{{[\infty]}})|_{t=0}=(v_0^{\kk,\pm}, b_0^{\kk,\pm}, q_0^{\kk,\pm}, S_0^{\kk,\pm},\psi_0^{\kk}),
\end{cases}
\end{equation}where $\rho^{[\infty]}$ is defined via the equation of state $\rho=\rho(p,S)$ and $p^{[\infty]}:=q^{[\infty]}-\frac12|b^{[\infty]}|^2$. Also we have
\begin{align*}
D_t^{\vp^{[\infty]},\pm}=&~\p_t+\vb^{{[\infty]},\pm}\cdot\cnab+\frac{1}{\p_3\vp^{[\infty]}}(v^{{[\infty]},\pm}\cdot \NN^{{[\infty]}}-\p_t\vp^{[\infty]})\p_3,\\
 \mb^{[\infty],\pm}\cdot\nab^{\vp^{[\infty]}}=&~\mbc^{{[\infty]},\pm}\cdot\cnab+\frac{1}{\p_3\vp^{[\infty]}}(\mb^{{[\infty]},\pm}\cdot N^{[\infty]})\p_3.
\end{align*} 

For each fixed $\kk>0$, we want to prove that the limit system \eqref{CMHDllkklimit} exactly coincides with the nonlinear approximate problem \eqref{CMHD0kk}. If we compare the concrete form of each equation, we find that it remains to show   $b^{[\infty],\pm}=\mb^{[\infty],\pm}$ in $\Om^\pm$.

According to the definition of $\mb^\nnn$ in \eqref{modified b}, the limit function satisfies $\mb_i^{[\infty],\pm}=b_i^{[\infty],\pm}$ for $i=1,2$ and
\[
\mb_3^{[\infty],\pm}=b_3^{[\infty],\pm}+\mathfrak{R}_T^\pm\left(b_1^{[\infty],\pm}\TP_1\psi^{[\infty]}+b_2^{[\infty],\pm}\TP_2\psi^{[\infty]}-b_3^{[\infty],\pm}\right)\big|_{\Sigma}\Rightarrow \mb_3^{[\infty],\pm}\cdot N^{{[\infty]}}|_{\Sigma}=0.
\] Since Lemma \ref{trace} implies that $\mathfrak{R}_T^\pm(0)=0$, then the remaining step is to show $b^{[\infty],\pm}\cdot N^{[\infty]}|_{\Sigma} = 0$ holds with in the lifespan of the solution to \eqref{CMHDllkklimit} provided  $b^{[\infty],\pm}\cdot N^{[\infty]}|_{t=0} = 0$ on $\Sigma$. On $\Sigma$, we compute that
\begin{align*}
&D_t^{\vp^{{[\infty]}},\pm}(b^{{[\infty]},\pm}\cdot N^{[\infty]})=D_t^{\vp^{{[\infty]}},\pm}b^{{[\infty]},\pm}\cdot N^{[\infty]}+b^{{[\infty]},\pm}\cdot D_t^{\vp^{{[\infty]}},\pm}N^{[\infty]}\\
=&\underbrace{(\mbc^{{[\infty]},\pm}\cdot\cnab)}_{=\bc^{{[\infty]},\pm}\cdot\cnab}v^{{[\infty]},\pm}\cdot N^{[\infty]} + \underbrace{(\mb^{{[\infty]},\pm}\cdot N^{[\infty]})}_{=0\text{ on }\Sigma}\p_3v^{{[\infty]},\pm}\cdot N^{[\infty]} + \underbrace{(\mb^{{[\infty]},\pm}\cdot N^{[\infty]})}_{=0\text{ on }\Sigma}(\nab^{\vp^{[\infty]}}\cdot v^{{[\infty]},\pm}) \\
&- \bc_i^{{[\infty]},\pm}\TP_i\p_t\psi^{{[\infty]}} - \bc_i^{{[\infty]},\pm}\vb_j^{{[\infty]},\pm}\TP_j\TP_i\psi^{{[\infty]}} \\
=&~(\bc^{{[\infty]},\pm}\cdot\cnab)\underbrace{\left(v_3^{{[\infty]},\pm}-\vb_j^{{[\infty]},\pm}\TP_j\psi^{{[\infty]}}\right)}_{=\p_t\psi^{{[\infty]}}} +  \bc_i^{{[\infty]},\pm}\vb_j^{{[\infty]},\pm}\TP_j\TP_i\psi^{{[\infty]}}- \bc_i^{{[\infty]},\pm}\TP_i\p_t\psi^{{[\infty]}} - \bc_i^{{[\infty]},\pm}\vb_j^{{[\infty]},\pm}\TP_j\TP_i\psi^{{[\infty]}}=0.
\end{align*} Thus standard $L^2$ energy estimate shows that
\begin{align}
\ddt\is\bno{b^{{[\infty]},\pm}\cdot N^{[\infty]}}^2\dx' =& \is(\cnab\cdot\vb^{{[\infty]},\pm})\bno{b^{{[\infty]},\pm}\cdot N^{[\infty]}}^2\dx'\leq |\TP v^{{[\infty]},\pm}|_{L^{\infty}}\bno{b^{{[\infty]},\pm}\cdot N^{[\infty]}}_0^2.
\end{align} Since $b^{[\infty],\pm}\cdot N^{[\infty]}|_{t=0} = 0$ on $\Sigma$, we conclude that $b^{[\infty],\pm}\cdot N^{[\infty]}= 0$ always holds on $\Sigma$ by using Gr\"onwall's inequality. Plugging it back to the expression of $\mb_3^{[\infty],\pm}$, we find $\mb_3^{{[\infty]},\pm}=b_3^{{[\infty]},\pm}$ in $\Om^\pm$ as desired. Then we can replace $\mb$ by $b$ in the limit system \eqref{CMHDllkklimit} to get the following one.
\begin{equation}\label{CMHDllkkinf}
\begin{cases}
\rho^{[\infty],\pm} D_t^{\vp^{[\infty]},\pm} v^{{[\infty]},\pm} - (b^{[\infty],\pm}\cdot\nab^{\vp^{[\infty]}}) b^{{[\infty]},\pm}+\nab^{\vp^{[\infty]}} q^{{[\infty]},\pm}=0&~~\text{ in }[0,T]\times \Omega^\pm,\\
(\ffpm)^{[\infty]} D_t^{\vp^{{[\infty]}},\pm} q^{{[\infty]},\pm} - (\ffpm)^{[\infty]} D_t^{\vp^{{[\infty]}},\pm} b^{{[\infty]},\pm}\cdot b^{{[\infty]},\pm} +\nab^{\vp^{[\infty]}}\cdot v^{{[\infty]},\pm}=0 &~~\text{ in }[0,T]\times \Omega^\pm,\\
D_t^{\vp^{{[\infty]}},\pm} b^{{[\infty]},\pm}-(b^{[\infty],\pm}\cdot\nab^{\vp^{[\infty]}}) v^{{[\infty]},\pm}+b^{{[\infty]},\pm}\nab^{\vp^{[\infty]}}\cdot v^{{[\infty]},\pm}=0&~~\text{ in }[0,T]\times \Omega^\pm,\\
D_t^{\vp^{{[\infty]}},\pm} S^{{[\infty]},\pm}=0&~~\text{ in }[0,T]\times \Omega^\pm,\\
\jump{q^{[\infty]}}=\sigma\h(\psi^{{[\infty]}})- \kk (1-\TL)^2\psi^{{[\infty]}} - \kk (1-\TL)\p_t\psi^{{[\infty]}}&~~\text{ on }[0,T]\times\Sigma, \\
\p_t \psi^{{[\infty]}} = v^{{[\infty]},\pm}\cdot N^{{[\infty]}},\quad b^{{[\infty]},\pm}\cdot N^{{[\infty]}}=0&~~\text{ on }[0,T]\times\Sigma,\\
v_3^{{[\infty]},\pm}=b_3^{{[\infty]},\pm}=0&~~\text{ on }[0,T]\times\Sigma^\pm,\\
(v^{{[\infty]},\pm},b^{{[\infty]},\pm},q^{{[\infty]},\pm},S^{{[\infty]},\pm},\psi^{{[\infty]}})|_{t=0}=(v_0^{\kk,\pm}, b_0^{\kk,\pm}, q_0^{\kk,\pm}, S_0^{\kk,\pm},\psi_0^{\kk}),
\end{cases}
\end{equation}
Finally, the divergence constraint $\nab^{\vp^{[\infty]}}\cdot b^{[\infty],\pm}=0$ in $\Om^\pm$ automatically holds thanks to the second equation, the fourth equation in \eqref{CMHDllkkinf} and  $\nab^{\vp_0^{\kk}} \cdot b_0^{\kk,\pm}=0$ in $\Om^\pm$. Thus, the limit functions $\{(v^{[\infty],\pm},b^{[\infty],\pm}, q^{[\infty],\pm},\psi^{[\infty]})\}_{n\in\N^*}$ introduced in \eqref{CMHDllkkinf} exactly give the solution to the nonlinear $\kk$-problem \eqref{CMHD0kk} in the time interval $[0,T_\kk']$ for each fixed $\kk>0$. The uniqueness follows from a parallel argument in Section \ref{sect picard}. 
\end{proof}

\section{Well-posedness and incompressible limit}\label{sect LWP}
\subsection{Well-posedness of compressible current-vortex sheets with surface tension}
We are ready to prove the local well-posedness of the original system \eqref{CMHD0} for 3D compressible current-vortex sheets with fixed surface tension coefficient $\sigma>0$. Recall that we introduce the nonlinear approximate system \eqref{CMHD0kk} indexed by $\kk>0$. In Section \ref{sect LWPkk}, we use Galerkin approximation and Picard iteration to prove the well-posedness of \eqref{CMHD0kk} for each fixed $\kk>0$. The lifespan for \eqref{CMHD0kk} may rely on $\kk>0$. Then we prove the uniform-in-$\kk$ estimates for \eqref{CMHD0kk} \textit{without loss of regularity} so that we can extend the solution of \eqref{CMHD0kk} to a $\kk$-independent lifespan $[0,T]$.  In Appendix \ref{sect comp cond}, we construct the initial data of \eqref{CMHD0kk} that converges to the given initial data of \eqref{CMHD0} as $\kk\to 0$. Thus, by taking $\kk\to 0$, we obtain the local existence of the original system \eqref{CMHD0} and the energy estimates for $E(t)$ defined in \eqref{energy lwp} without loss of regularity. 

It remains to prove the uniqueness. Namely, we assume $(v^{[1],\pm},b^{[1],\pm},q^{[1],\pm},\psi^{[1]})$ and $(v^{[2],\pm},b^{[2],\pm},q^{[2],\pm},\psi^{[2]})$ are two solutions to \eqref{CMHD0} \textit{with the same initial data}. Define $[f]:=f^{[1]}-f^{[2]}$, and we need to prove $([v]^{\pm},[b]^{\pm},[q]^{\pm},[\psi])$ are identically zero. In fact, the argument for uniqueness is substantially similar to the analysis in Section \ref{sect picard}. The only difference is that the boundary regularity is now given by the surface tension instead of the $\kk$-regularization terms.  This has been studied in the previous paper \cite[Section 6]{LuoZhang2022CWWST} by Luo and the author and we refer to  \cite[Section 6]{LuoZhang2022CWWST} for details.

\subsection{Incompressible limit of compressible current-vortex sheets with surface tension}\label{sect limit sigma}
Next, we justify the incompressible limit of the solution obtained above, that is the limiting behavior of the local-in-time solution of \eqref{CMHD0} as $\eps\to 0$. Given $\sigma> 0$, we introduce the equations of $(\xi^{\sigma}, w^{\pm,\sigma},h^{\pm,\sigma})$ describing the motion of incompressible non-uniform current-vortex sheets together with a transport equation of entropy $\FS^{\sigma}$
\begin{equation} \label{IMHDVS}
\begin{cases}
\rr^{\pm,\sigma}(\p_t+w^{\pm,\sigma}\cdot\nab^{\Xi^\sigma})w^{\pm,\sigma}- (h^{\pm,\sigma}\cdot\nab^{\Xi^\sigma})h^{\pm,\sigma}+\nab^{\Xi^\sigma} \Pi^{\pm,\sigma}=0&~~~ \text{in}~[0,T]\times \Om^\pm,\\
\nab^{\Xi^\sigma}\cdot w^{\pm,\sigma}=0&~~~ \text{in}~[0,T]\times  \Om^\pm,\\
(\p_t+w^{\pm,\sigma}\cdot\nab^{\Xi^\sigma}) h^{\pm,\sigma}=(h^{\pm,\sigma}\cdot\nab^{\Xi^\sigma})w^{\pm,\sigma}&~~~ \text{in}~[0,T]\times  \Om^\pm,\\
\nab^{\Xi^\sigma}\cdot h^{\pm,\sigma}=0&~~~ \text{in}~[0,T]\times  \Om^\pm,\\
(\p_t+w^{\pm,\sigma}\cdot\nab^{\Xi^\sigma})\mathfrak{S}^{\pm,\sigma}=0&~~~ \text{in}~[0,T]\times  \Om^\pm,\\
\jump{\Pi^{\sigma}}=\sigma\cnab \cdot \left( \frac{\cnab \xi^{\sigma}}{\sqrt{1+|\cnab\xi^{\sigma}|^2}}\right) &~~~\text{on}~[0,T]\times\Sigma,\\
\p_t \xi^{\sigma} = w^{\pm,\sigma}\cdot N^{\sigma} &~~~\text{on}~[0,T]\times\Sigma,\\
h^{\pm,\sigma}\cdot N^{\sigma}=0&~~~\text{on}~[0,T]\times\Sigma,\\
(w^{\pm,\sigma},h^{\pm,\sigma},\mathfrak{S}^{\pm,\sigma},\xi^{\sigma})|_{t=0}=(w_0^{\pm,\sigma},h_0^{\pm,\sigma},\mathfrak{S}_0^{\pm,\sigma}, \xi_0^{\sigma}), 
\end{cases}
\end{equation}where $\Xi^{\sigma}(t,x) = x_3+\chi(x_3) \xi^\sigma(t,x')$ to be the extension of $\xi^\sigma$ in $\Omega$ and $ N^\sigma:=(-\TP_1\xi^\sigma, -\TP_2\xi^\sigma, 1)^\top$. The quantity $\Pi^\pm:=\bar{\Pi}^\pm+\frac12|h^\pm|^2$ represent the total pressure for the incompressible equations with $\bar{\Pi}^\pm$ the fluid pressure functions. The quantity $\rr^\pm$ satisfies the evolution equation $(\p_t+w^\pm\cdot\nabp)\rr^\pm=0$ with initial data $\rr_0^\pm:=\rho^\pm(0,\mathfrak{S}_0^\pm)$. 

Denote $(\psi^\es, v^{\pm,\es}, b^{\pm,\es}, \rho^{\pm,\es}, S^{\pm,\es})$ to be the solution of \eqref{CMHD0} (indexed by $\sigma$ and $\eps$) with initial data $(\psi_0^\es, v_0^{\pm,\es}, b_0^{\pm,\es}, \rho_0^{\pm,\es}, S_0^{\pm,\es})$. For fixed $\sigma>0$, we want to show the convergence from the solutions to \eqref{CMHD0} to the solution to \eqref{IMHDVS} as $\eps\to 0$ provided the convergence of initial data. We assume
\begin{enumerate}
\item (Constraints for compressible initial data) The sequence of initial data $(\psi_0^\es, v_0^{\pm,\es}, b_0^{\pm,\es}, \rho_0^{\pm,\es}, S_0^{\pm,\es})\in H^{9.5}(\Sigma)\times(H_*^8(\Om^\pm))^4$ of \eqref{CMHD0} satisfy the constraints $\nabp\cdot b_0^{\pm,\es}=0$ in $\Om^\pm$, $b^{\pm,\es}\cdot N^\sigma|_{t=0}=0$ on $\Sigma\cup \Sigma^\pm$, the compatibility conditions \eqref{comp cond} up to 7-th order, $|\psi_0^{\es}|\leq 1$ and $|\jump{\vb_0}|>0$. 
\item (Convergence of initial data) $(\psi_0^\es, v_0^{\pm,\es}, b_0^{\pm,\es}, \rho_0^{\pm,\es}, S_0^{\pm,\es})\to (\xi_0^\sigma, w_0^{\pm,\sigma}, h_0^{\pm,\sigma}, \rr_0^{\pm,\sigma}, \FS_0^{\pm,\sigma})$ in $H^{5.5}(\Sigma)\times(H^4(\Om^\pm))^4$.
\item (Constraints for incompressible initial data) The incompressible data $(\xi_0^\sigma, w_0^{\pm,\sigma}, h_0^{\pm,\sigma}, \rr_0^{\pm,\sigma}, \FS_0^{\pm,\sigma})\in H^5(\Sigma)\times(H^4(\Om^\pm))^4$ satisfies the constraints $\nab^{\xi^\sigma_0}\cdot h_0^\pm=0$ in $\Om^\pm$, $h^{\pm,\sigma}\cdot N^\sigma|_{t=0}=0$ on $\Sigma\cup \Sigma^\pm$, $|\xi_0^{\es}|\leq 2$ and $\jump{\bar{w}_0}>0$.
\end{enumerate}
Under these assumptions, we can prove that there exists a time $T_\sigma>0$ that depends on $\sigma$ and initial data and is independent of Mach number $\eps$, such that the corresponding solutions to \eqref{CMHD0} converge to the solution to \eqref{IMHDVS} as the Mach number $\eps\to 0$
\begin{align*}
&(\psi^\es, v^{\pm,\es}, b^{\pm,\es}, \rho^{\pm,\es}, S^{\pm,\es})\to (\xi^{\sigma}, w^{\pm,\sigma},h^{\pm,\sigma},\rr^{\pm,\sigma},\FS^{\pm,\sigma})\\
&\text{ strongly in }C([0,T_\sigma];H_{\text{loc}}^{5.5-\delta}(\Sigma)\times (H_{\text{loc}}^{4-\delta}(\Om^\pm))^4),\text{ and  weakly-* in }L^{\infty}([0,T_\sigma];H^{5.5}(\Sigma)\times (H^{4}(\Om^\pm))^4).
\end{align*} In fact, according to estimates obtained in Theorem \ref{thm STLWP}, we already have the uniform-in-$\eps$ boundedness for $\psi^\es,~v^{\pm,\es}$, $b^{\pm,\es},~S^{\pm,\es}$ as well as their first-order time derivatives. Thus, using Aubin-Lions compactness lemma, the above convergence is a straightforward result of uniform-in-$\eps$ estimates. Theorem \ref{thm CMHDlimit1} is proven.

\begin{appendix}
\section{Reynolds transport theorems}\label{sect transport}
We record the Reynolds transport theorems used in this paper. For the proof, we refer to Luo-Zhang \cite[Appendix A]{LuoZhang2022CWWST}
\begin{lem}\label{time deriv transport pre}
Let $f,g$ be smooth functions defined on $[0,T]\times \Omega$. Then:
\begin{align}
\ddt \int_\Omega fg \p_3\vp\dx= \int_\Omega (\pp_t f)g\p_3 \vp\dx +\int_\Omega f(\pk_t g)\p_3\vp\dx+\int_{x_3=0}fg\p_t \psi\dx',\label{time deriv transport pre tilde}\\
\ddt \int_\Omega fg \p_3\pr\dx= \int_\Omega (\p_t^{\pr} f)g\p_3 \pr\dx +\int_\Omega f(\p_t^{\pr} g)\p_3\pr\dx+\int_{x_3=0}fg\p_t \psr\dx'.\label{time deriv transport pre tildering}
\end{align}
\end{lem}

\begin{lem}[\textbf{Integration by parts for covariant derivatives}] \label{int by parts lem}
Let $f, g$ be defined as in Lemma \ref{time deriv transport pre}. Then:
\begin{align}
\int_\Omega (\pp_i f)g \p_3 \vp \dx= -\int_\Omega f(\pp_i g)\p_3 \vp\dx+ \int_{x_3=0} fg N_i\dx',\label{int by parts tilde}\\ 
\int_\Omega (\ppr_i f)g \p_3 \pr \dx= -\int_\Omega f(\ppr_i g)\p_3 \pr\dx+ \int_{x_3=0} fg \npr_i\dx'.\label{int by parts tildering}
\end{align}
\end{lem}

The following theorem holds.
\begin{thm}[\textbf{Reynolds transport theorem}]\label{transport thm nonlinear}
Let $f$ be a smooth function defined on $[0, T]\times\Omega$. Then:
\begin{align}
\ddt  \io \rho |f|^2 \p_3\vp\dx = \io \rho (\Dtp f)f\p_3 \vp\dx.  \label{transpt nonlinear}
\end{align}
\end{thm}

Theorem \ref{transport thm nonlinear} leads to the following two corollaries. The first one records the integration by parts formula for $\Dtp$. 
\begin{cor}[\textbf{Reynolds transport theorem - a variant}] \label{transport thm without rho}
It holds that
\begin{align}
\ddt  \io fg \p_3\vp\dx =  \io (\Dtp f)g \p_3\vp\dx+\io f(\Dtp g)\p_3\vp\dx+\io (\nabp\cdot v) fg\p_3\vp\dx. \label{transpt nonlinear without rho}
\end{align}
\end{cor}

 The second corollary concerns the transport theorem as well as the integration by parts formula for the linearized material derivative $\Dtpr$. 

\begin{cor}[\textbf{Reynolds transport theorem for linearized $\kk$-problem}] \label{transport thm linearized}
Let $\Dtpr:=\p_t+(\vbr\cdot\cnab)+\frac{1}{\p_3\pr}(\vr\cdot\Npd-\p_t\pr)\p_3$ be the linearized material derivative. Then:
\begin{align}
\frac{1}{2}\ddt \io \rhor |f|^2 \p_3\pr\dx =& \io \rhor (\Dtpkr f)f\p_3 \pr\dx
+\frac{1}{2}\io \left( \Dtpkr \rhor +\rhor\nabpkr\cdot \vr\right) |f|^2\p_3\pr\dx \label{transpt linearized} \\
&+\frac{1}{2}\io \rhor |f|^2 \left(\p_3(\vbr\cdot\cnab)(\pr-\pd)\right)\dx.  \nonumber
\end{align}
\begin{align}
\frac{1}{2}\ddt  \io |f|^2 \p_3\pr\dx = & \io (\Dtpr f)f \p_3\pr\dx+\frac{1}{2}\io \nabpr\cdot \vr |f|^2\p_3\pr\dx \label{transpt linearized without rho}\\
&+\frac{1}{2}\io |f|^2 \left(\p_3(\vbr\cdot\cnab)(\pr-\pd)\right)\dx.\nonumber
\end{align}
\end{cor}

\section{Preliminary lemmas about Sobolev inequalities}\label{sect lemma}

\begin{lem}[Hodge-type elliptic estimates]\label{hodgeTT}
For any sufficiently smooth vector field $X$ and $s\geq 1$, one has
\begin{align}
\label{divcurlTT}\|X\|_s^2\leq C(|\psi|_s,|\cnab\psi|_{W^{1,\infty}})\left(\|X\|_0^2+\|\nabp\cdot X\|_{s-1}^2+\|\nabp\times X\|_{s-1}^2+\|\TP^{\alpha}X\|_0^2\right),\\
\label{divcurlNN}\|X\|_s^2\leq C'(|\psi|_{s+\frac12},|\cnab\psi|_{W^{1,\infty}})\left(\|X\|_0^2+\|\nabp\cdot X\|_{s-1}^2+\|\nabp\times X\|_{s-1}^2+|X\cdot N|_{s-\frac12}^2\right),\\
\label{divcurltt}\|X\|_s^2\leq C''(|\psi|_{s+\frac12},|\cnab\psi|_{W^{1,\infty}})\left(\|X\|_0^2+\|\nabp\cdot X\|_{s-1}^2+\|\nabp\times X\|_{s-1}^2+|X\times N|_{s-\frac12}^2\right),
\end{align} 
for any multi-index $\alpha$ with $|\alpha|=s$. The constant $C(|\psi|_s,|\cnab\psi|_{W^{1,\infty}})>0$ depends linearly on $|\psi|_s^2$ and the constants $C'(|\psi|_{s+\frac12},|\cnab\psi|_{W^{1,\infty}})>0$ and $C'(|\psi|_{s+\frac12},|\cnab\psi|_{W^{1,\infty}})>0$ depend linearly on $|\psi|_{s+\frac12}^2$. 
\end{lem}

\begin{lem}[Normal trace lemma]\label{ntrace}
For any sufficiently smooth vector field $X$ and $s\geq 0$, one has
\begin{align}
\bno{X\cdot N}_{s-\frac12}^2\lesssim C'''(|\psi|_{s+\frac12},|\cnab\psi|_{W^{1,\infty}}) \left(\|\jp^s X\|_0^2+\|\nabp\cdot X\|_{s-1}^2\right)
\end{align} 
where the constant $C'''(|\psi|_{s+\frac12},|\cnab\psi|_{W^{1,\infty}})>0$ depends linearly on $|\psi|_{s+\frac12}^2$.
\end{lem}

We list two lemmas for the estimates of traces in the anisotropic Sobolev spaces. Define $L_T^2(H_*^m(\Om^\pm))=\bigcap\limits_{k=0}^mH^k((-\infty,T];H_*^{m-k}(\Om^\pm))$ with the norm $\|u\|_{m,*,T,\pm}:=\int_{-\infty}^T\|u(t)\|_{m,*,\pm}^2\dt$. Similarly, we define $L_T^2(H^m(\Sigma))=\bigcap\limits_{k=0}^mH^k((-\infty,T];H^{m-k}(\Sigma))$ with the norm $|u|_{m,T}:=\int_{-\infty}^T|u(t)|_{m}^2\dt$.

\begin{lem}[Trace lemma for anisotropic Sobolev spaces, {\cite[Lemma 3.4]{TW2020MHDLWP}}]\label{trace}
Let $m\geq 1,~m\in\N^*$, then we have the following trace lemma for the anisotropic Sobolev space.
\begin{enumerate}
\item If $f\in L_T^2(H_*^{m+1}(\Omega^\pm))$, then its trace $f|_{\Sigma}$ belongs to $L_T^2(H^{m}(\Omega^\pm))$ and satisfies
\[
|f|_{m,T}\lesssim\|f\|_{m+1,*,T,\pm}.
\]

\item There exists a linear continuous operator $\mathfrak{R}_{T}^\pm: L_T^2(H^m(\Sigma)) \to L_T^2(H_*^{m+1}(\Om^\pm))$ such that $(\mathfrak{R}_{T}^\pm g)|_{\Sigma}=g$ and \[
\|\mathfrak{R}_{T}^\pm g\|_{m+1,*,T,\pm}\lesssim|g|_{m,T}.
\]
\end{enumerate}
\end{lem}
\begin{proof}
The proof for the above lemma can be found in \cite[Theorem 1]{anisotropictrace} when we replace $(-\infty,T)$ by $(-\infty,\infty)$. In our case, we can prove the same result by doing Sobolev extension. Namely, given $f\in L_T^2(H_*^{m+1}(\Om^+))$, we can extend it to $F(t,x):R\times \Om^+\to\R$ such that
\[
\|f\|_{m+1,*,T,+}\lesssim\|F(t,x)\|_{H_*^{m+1}(\R\times\Om^+)}\lesssim\|f\|_{m+1,*,T,+}.
\]We can apply \cite[Theorem 1]{anisotropictrace} to $F$, and then do the truncation in $(-\infty,T]$
\[
|f|_{m,T}\lesssim|F|_{H^m(\R\times\Sigma)}\lesssim\|F(t,x)\|_{H_*^{m+1}(\R\times\Om^+)}\lesssim\|f\|_{m+1,*,T,+}.
\]
\end{proof}
There is one derivative loss in the above trace lemma, which is 1/2-order more than the trace lemma for standard Sobolev spaces. Indeed, for $\Om^\pm$ defined in this paper, we have the following estimate that will be applied to control the non-characteristic variables $q, v\cdot\NN$ and $b\cdot\NN$.
\begin{lem}[An estimate for traces of non-characteristic variables]\label{nctrace}
Let $\Om^\pm:=\T^{d-1}\times\{0\lessgtr  x_d\lessgtr \pm H\}$, $\Sigma=\T^{d-1}\times\{x_d=0\}$ and $\Sigma^\pm=\T^{d-1}\times\{\pm H\}$. Let $\TT^\alpha=(\omega(x_d)\p_d)^{\alpha_{d+1}}\p_t^{\alpha_0}\TP_1^{\alpha_1}\cdots\TP_{d-1}^{\alpha_{d-1}}\p_d^{\alpha_d}$ with $\len{\alpha}:=\alpha_0+\cdots+\alpha_{d-1}+2\alpha_d+\alpha_{d+1}=m-1,~m\in\N^*$. Let $q^\pm(t,x)\in H_*^m(\Om)$ satisfy $\|q^\pm(t)\|_{m,*,\pm}+\|\p_d q^\pm(t)\|_{m-1,*,\pm}<\infty$ for any $0\leq t\leq T$ and let $f^\pm\in H_*^{2}(\Om^\pm)\cap H^{\frac32}(\Om^\pm)$ be a function vanishing on $\Sigma^\pm$. Then we have
\begin{equation}\label{q IBP}
\is (\jp^{\frac12}\TT^\gamma q^\pm) \, (\jp f^\pm)\dx'\leq (\|\p_d q^\pm\|_{m-1,*,\pm}+\|q^\pm\|_{m,*,\pm})\|\jp^{\frac12} f^\pm\|_{1,\pm}
\end{equation}
In particular, for $s\geq 1$, we have the following inequality for any $g^\pm\in H_*^s(\Om^\pm)$ with $g^\pm|_{\Sigma^\pm}=0$.
\[
|g^\pm|_{s-1/2}^2\leq \|\jp^s g^\pm\|_{0,\pm}\|\jp^{s-1}\p_d g^\pm\|_{0,\pm}\leq \|g^\pm\|_{s,*,\pm}\|\p_dg^\pm\|_{s-1,*,\pm}.
\]
\end{lem}
\begin{proof}
This is a direct consequence of Gauss-Green formula. Note that the unit exterior normal vectors for $\Om^\pm$ are $(0,\cdots,0,\mp 1)^\top$ respectively, so we have
\begin{equation}
\begin{aligned}
&\is (\jp^{\frac12}\TT^\gamma q^\pm) \, (\jp f^\pm)\dx'= \mp\iopm (\p_d\TT^\gamma q^\pm)\,(\jp^{\frac32} f^\pm) +(\jp\TT^\gamma q^\pm)\,(\jp^{\frac12}\p_d f^\pm) \dx\\
 \leq&~ (\|\p_d q^\pm\|_{m-1,*,\pm}+\|q^\pm\|_{m,*,\pm})\|\jp^{\frac12} f^\pm\|_{1,\pm}
\end{aligned}
\end{equation}  In particular, let $q^\pm=g^\pm$ and $f^\pm=\jp^{s-\frac32} g^\pm$ in \eqref{q IBP} and we get
\begin{align*}
|g^\pm|_{s-1/2}^2=\is (\jp^{s-1/2}g^\pm)(\jp^{s-1/2}g^\pm)\dx'=\mp2\iopm (\p_d\jp^{s-1/2} g^\pm)(\jp^{s-1/2}g^\pm)\dx\overset{\jp^{1/2}}{=}~\mp2\iopm  (\p_d\jp^{s-1} g^\pm)(\jp^{s} g^\pm)\dx.
\end{align*}
\end{proof}

The following lemma concerns the Sobolev embeddings.
\begin{lem}[{\cite[Lemma 3.3]{TW2020MHDLWP}}]\label{embedding}
We have the following inequalities
\begin{align*}
H^m(\Omega^\pm)\hookrightarrow H_*^m(\Omega^\pm)\hookrightarrow& H^{\lfloor m/2\rfloor}(\Omega^\pm),~~\forall m\in\N^*;\\
\|u\|_{L^{\infty}(\Omega^\pm)}\lesssim\|u\|_{H_*^3(\Omega^\pm)},~~&\|u\|_{W^{1,\infty}(\Omega^\pm)}\lesssim\|u\|_{H_*^5(\Omega^\pm)},~~|u|_{W^{1,\infty}(\Omega^\pm)}\lesssim\|u\|_{H_*^{5}(\Omega^\pm)}.
\end{align*}
\end{lem}

We also need the following Kato-Ponce type multiplicative Sobolev inequality.
\begin{lem}[{\cite{KatoPonce1988}}]\label{KatoPonce}
Let $J=(1-\Delta)^{1/2}$, $s\geq 0$. Then the following estimates hold:
\begin{equation}\label{product}
\|J^s(fg)\|_{L^2}\lesssim \|f\|_{W^{s,p_1}}\|g\|_{L^{p_2}}+\|f\|_{L^{q_1}}\|g\|_{W^{s,q_2}},
\end{equation}where $1/2=1/p_1+1/p_2=1/q_1+1/q_2$ and $2\leq p_1,q_2<\infty$.
\begin{equation}\label{commutator}
\|[J^s,f]g\|_{L^p}\lesssim \|\p f\|_{L^{\infty}}\|J^{s-1}g\|_{L^p}+\|J^s f\|_{L^p}\|g\|_{L^{\infty}}
\end{equation}where  $s\geq 0$ and $1<p<\infty$.
\end{lem}

We also need the following transport-type estimate in order to close the uniform estimates for the nonlinear approximate system.
\begin{lem}\label{parabolic}
Let $f(t)\in W^{1,1}(0,T)$ and $g\in L^1(0,T)$ and $\kk>0$. Assume that $f(t)+\kk f'(t)\leq g(t)$ holds for a.e. $t\in(0,T).$ Then for any $t\in(0,T)$, we have $\sup_{\tau\in[0,t]}f(\tau)\leq f(0)+\mathop{\mathrm{ess~sup}}\limits_{\tau\in(0,t)}|g(\tau)|.$
\end{lem}

\section{Construction of initial data satisfying the compatibility conditions}\label{sect comp cond}
Given initial data $(v_0^\pm,b_0^\pm,q_0^\pm,S_0^\pm,\psi_0)$ of the original current-vortex sheets problem \eqref{CMHD0} satisfying the compatibility conditions \eqref{comp cond} up to 7-th order, we need to construct a sequence of initial data $(v_0^{\kk,\pm},b_0^{\kk,\pm},q_0^{\kk,\pm},S_0^{\kk,\pm},\psi_0^{\kk})$ to the nonlinear $\kk$-approximate system \eqref{CMHD0kk} satisfying the compatibility conditions \eqref{comp cond kk} up to 7-th order that converge to the given data as $\kk\to 0_+$.

\subsection{Reformulation of the compatibility conditions}
Let us first ignore the $\kk$-regularization terms and consider the compatibility conditions \eqref{comp cond} for the original system. Also, let us omit the fixed boundaries $\Sigma^\pm$, omit the density functions, consider the isentropic case and write $\eps^2=\ffpm$ for convenience. The heuristic idea is that the odd $(m=2r+1)$ order compatibility condition is rewritten to be $$-\jump{\cbi^{r+1}(\lapi)^r(\nabpp\cdot v_0)}=\cdots \quad\text{ on }\Sigma$$  and the even  $(m=2r)$ order compatibility condition is rewritten to be $$\jump{\cbi^{r}(\lapi)^r q_0}=\cdots \quad\text{ on }\Sigma$$ with $\cbi:=\eps^{-2}+|b_0|^2$. Such reformulation is convenient for us to add $\kk$-perturbation terms to construct the desired data for \eqref{CMHD0kk}. More specifically, let us start with the zero-th order compatibility conditions:
\begin{align}
\jump{q_0}=\sigma\h(\psi_0),\quad \psi_t|_{t=0}=v_0^\pm\cdot N_0=v_{03}^\pm-\vb_0^\pm\cdot\TP\psi_0.
\end{align} The first-order compatibility conditions are
\begin{align}
\p_t\jump{q}|_{t=0}=\sigma\p_t\h(\psi)|_{t=0},\quad \psi_{tt}|_{t=0}=\p_t(v^\pm\cdot N)|_{t=0},
\end{align} which are not easy to compute, especially the first one. The left side is equal to
\begin{align*}
\p_t q^+ -\p_t q^- =&~\Dtbu q^+- \Dtbl q^- - (\vb^+\cdot\cnab)\jump{q} - (\jump{\vb}\cdot\cnab) q^-.
\end{align*}Using the continuity equation, the evolution equation of $b$, we get 
\[
\Dtb q = -\eps^{-2} (\nabp\cdot v) +\Dtb b\cdot b = -\underbrace{(\eps^{-2}+|b|^2)}_{=:\cb}(\nabp\cdot v) + \bcp v\cdot b\text{ on }\Sigma,
\]and thus the time-differentiated jump condition becomes
\begin{align*}
\jump{\cb(\nabpp\cdot v_0)}=&\jump{(\bc_0\cdot\cnab)v_0\cdot b_0}- (\jump{\vb_0}\cdot\cnab) q_0^-+\Dtbu (\sigma\h(\psi))|_{t=0}  \text{ on }\Sigma.
\end{align*}Here and thereafter, we will repeatedly use $\Dtbpm \psi= v_3^\pm$ on $\Sigma$ and omit lots of redundant terms in order for simplicity of notations. For example, we will write $\h(\psi)\sim \TL\psi$, write $(1-\TL)$ to be $-\TL$, and omit the commutators between $\Dtbu$ and $\h,(1-\TL)$, the density function $\rho$. Indeed, later we will see that the concrete form of those omitted term is not important, and we just need to find out the major term as in \cite[Appendix A]{Zhang2021CMHD}. Under this setting, we have
\begin{align}
\jump{\cb(\nabpp\cdot v_0)}\sim &\jump{(\bc_0\cdot\cnab)v_0\cdot b_0}- (\jump{\vb_0}\cdot\cnab) q_0^- +\sigma \TL v_{03}^+ \text{ on }\Sigma.
\end{align}
For higher-order compatibility conditions, we invoke the wave equation for total pressure $q^\pm$ to get (cf. \cite[Appendix A.1]{Zhang2021CMHD})
\begin{align}
(\Dtb)^2 q = \cb \lapp q +\MH_0(v,b) +\NH_0(v,b)\quad\text{ on }\Sigma,
\end{align}where $$\MH_0(v,b)=-\bcp^2 q+\bcp^2b\cdot b+\RR_0(v,b),~~~\NH_0(v,b)=\pp_i v^j\pp_j v^i-\pp_i b^j\pp_jb^i$$ and $\RR_0(v,b)$ only contains the first-order derivatives of $b,v$ with the form
\[
\RR_0(v,b)=P_0(b)((\p^{i_1} v)(\p^{i_2} v) + (\p^{j_1} b)(\p^{j_2} b))
\] where $P_0(b)$ is a polynomial of $b$ only containing cubic and quadratic terms and $(i_1,i_2,j_1,j_2)=(0,0,1,1)$ or $(1,1,0,0)$. Taking substraction between the equation of $q^+$ and the equation of $q^-$, we get
\[
\jump{(\Dtb)^2 q}|_{t=0} = \jump{\cbi \lapi q_0}  + \jump{\MH_0(v_0,b_0)+\NH_0(v_0,b_0)}\quad\text{ on }\Sigma.
\]Then using $\Dtbu=\Dtbl+(\jump{\vb}\cdot\cnab)$, we get 
\begin{align*}
\jump{(\Dtb)^2 q}|_{t=0} = (\Dtbu)^2(\sigma \h(\psi))|_{t=0} +\TT^{2}_{\jump{v}} q^-|_{t=0},
\end{align*}where each $\TT_{\jump{v}} $ represents either of $\Dtbl$ and $(\jump{\vb}\cdot\cnab)$. So, the second-order compatibility condition is reformulated as
\begin{align}
\jump{\cbi \lapi q_0} =&~  (\Dtbu)^2(\sigma \h(\psi))|_{t=0} +\TT^{2}_{\jump{v}} q^-|_{t=0} - \jump{\MH_0(v_0,b_0)+\NH_0(v_0,b_0)}\no\\
\sim&~-\sigma \TL\p_3 q_0^+ + \sigma \TL \bcpu b_{03}^++\TT^{2}_{\jump{v}} q^-|_{t=0}- \jump{\MH_0(v_0,b_0)+\NH_0(v_0,b_0)} \quad\text{ on }\Sigma.
\end{align}
Taking one more material derivative in the wave equation and again use the continuity equation, we get
\begin{align}
(\Dtb)^3 q \sim -\cb^2\lapp(\nabp\cdot v) +\eps^{-2}\bcp^2(\nabp\cdot v) + \MH_1(v,b,q)+\NH_1(v,b,q)
\end{align}where the concrete form of $\MH_1,\NH_1$ will be specified later. Recursively, after long and tedious calculations (cf. \cite[(A.4)-(A.7)]{Zhang2021CMHD}), we find that the time-differentiated wave equation (restricted on $\{t=0\}\times\Sigma$) can be expressed as
\begin{align}
\label{wave odd} m=2r+1,&~~-\cbi^{r+1}(\lapi)^r(\nabpp\cdot v_0)=(\Dtb)^{2r+1}q+\sum_{j=0}^{r}(\lapi)^j(\MH_{2r-1-2j}(v_0,b_0,q_0)+\NH_{2r-1-2j}(v_0,b_0,q_0))\text{ on }\Sigma,\\
\label{wave even} m=2r,&~~\cbi^r(\lapi)^r q_0=(\Dtb)^{2r}q+\sum_{j=0}^{r-1}(\lapi)^j(\MH_{2r-2-2j}(v_0,b_0,q_0)+\NH_{2r-2-2j}(v_0,b_0,q_0))\text{ on }\Sigma,
\end{align}where $\MH_{-1}(v_0,b_0):=-(\bc_0\cdot\cnab)v_0\cdot b_0$ and $\NH_{-1}:=0$, and for $r\geq 1$ we define
\begin{align}
m=2r-1,~~\MH_{2r-1}(v_0,b_0,q_0)=&(\bar{b}_0\cdot\cnab)^2(\lapi)^{r-1}(\nabpp\cdot v_0)+\sum_{l=2}^{r+1}\underbrace{b_0^{i_1}\cdots b_0^{i_{2l}}(\nabla^{2r+1}v_0)}_{<2^{l}\text{ terms}}+\RR_{2r-1}(v_0,b_0,q_0),\label{MH odd}\\
m=2r,~~\MH_{2r}(v_0,b_0,q_0)=&~-(\bar{b}_0\cdot\cnab)^2(\lapi)^{r}q_0+\RR_{2r}(v_0,b_0,q_0),\nonumber\\
&+\sum_{l=2}^{r+1}\underbrace{(\bar{b}_0\cdot\cnab)^{r+2}(\nab^r b_0) b_0^{i_1}\cdots b_0^{i_{2l}}+(\bar{b}_0\cdot\cnab)^{2}(\nab^{2r} q_0) b_0^{j_1}\cdots b_0^{j_{2l}}}_{<2^{l}\text{ terms }}\label{MH even};
\end{align}and the term $\RR_m$, where every top-order term has $(m+1)$-th order derivative, has the following form
\[
\RR_m(v_0,b_0,q_0)=P_k(b_0)\left(C^{m}_{i_1\cdots i_p,j_1\cdots j_n, k_1\cdots k_l}(\nab^{i_1} v_0)\cdots(\nab^{i_p} v_0)(\nab^{j_1} b_0)\cdots(\nab^{j_n} b_0)(\nab^{k_1} q_0)\cdots(\nab^{k_l} q_0)\right),
\]where $\nab$ may represent either of $\nab^{\vp_0}$ or $\p$, and $P_k(\cdot)$ is a polynomial of its arguments and the lowest power is 4 and the indices above satisfy
\begin{align*}
1\leq i_1,\cdots, i_p,j_1,\cdots, j_n\leq k+1, 0\leq k_1,\cdots, k_l \leq m+1,\\
i_1+\cdots+i_p+j_1+\cdots+j_n+k_1+\cdots+k_l=m+1.
\end{align*}
The term $\NH_m(v_0,b_0,q_0)$ has the following form
\begin{align}
\NH_m(v_0,b_0,q_0)=&~P_{m,1}(b_0)(\nab^{1+2\lfloor\frac{m}{2}\rfloor}v_0)(\nab v_0)+P_{m,2}(b_0)(\nab^{2\lceil \frac{m}{2}\rceil} q_0)(\nab v_0)+P_{k,0}(b_0)(\nab^{m+1}b_0)(\nab v_0)\no\\ 
&+P'_m(b_0)D^{m}_{i_1\cdots i_p,j_1\cdots j_n, k_1\cdots k_l}\left((\nab^{i_1} v_0)\cdots(\nab^{i_p} v_0)(\nab^{j_1} b_0)\cdots(\nab^{j_n} b_0)(\nab^{k_1} q_0)\cdots(\nab^{k_l} q_0)\right), \label{NH}
\end{align}where $P_{m,1}(\cdot), P_{m,2}(\cdot),P_m'(\cdot)$ are polynomials of their arguments and $P_{m,0}(\cdot)$ is a polynomial of its arguments and the lowest power is 2. The indices above satisfy
\begin{align*}
1\leq i_1,\cdots, i_p,j_1,\cdots, j_n\leq k, 0\leq k_1,\cdots, k_l \leq m,\\
i_1+\cdots+i_p+j_1+\cdots+j_n+k_1+\cdots+k_l=m+1.
\end{align*}
Next we take the difference between the equations \eqref{wave odd}-\eqref{wave even} in $\Om^+$ and those in $\Om^-$ and restrict the equation on $\{t=0\}\times\Sigma$ to get the jump condition in the $m$-th order compatibility conditions
\begin{align}
\label{CH odd} m=2r+1,&~-\jump{\cbi^{r+1}(\lapi)^r(\nabpp\cdot v_0)}=\jump{(\Dtb)^{2r+1}q}+\sum_{j=0}^{r}\jump{(\lapi)^j(\MH_{2r-1-2j}(v_0,b_0,q_0)+\NH_{2r-1-2j}(v_0,b_0,q_0))}\text{ on }\Sigma,\\
\label{CH even} m=2r,&~\jump{\cbi^r(\lapi)^r q_0}=\jump{(\Dtb)^{2r}q}+\sum_{j=0}^{r-1}\jump{(\lapi)^j(\MH_{2r-2-2j}(v_0,b_0,q_0)+\NH_{2r-2-2j}(v_0,b_0,q_0))}\text{ on }\Sigma.
\end{align}Then using $\Dtbu=\Dtbl+(\jump{\vb}\cdot\cnab)$, we get
\[
\jump{(\Dtb)^m q}=(\Dtbu)^m\jump{q}+\TT_{\jump{\vb}}^{m}q^-|_{t=0},
\]where each $\TT_{\jump{\vb}}$ represents either $(\Dtbl)$ or $(\jump{\vb}\cdot\cnab)$. Using the jump condition for $\jump{q}$, we have
\begin{align}
m=2r:&~(\Dtbu)^{2r}\jump{q}\sim\sigma\TL(\Dtbu)^{2r-1}v_3^+\sim \sigma\cbi^{r-1}\TL(\lapi)^{r-1}\p_3 q_0^+ + \sigma\TL\sss_{2r-1}(v_0^+,b_0^+,q_0^+)\\
m=2r+1:&~(\Dtbu)^{2r+1}\jump{q}\sim\sigma\TL(\Dtbu)^{2r}v_3^+\sim -\sigma\cbi^{r}\TL(\lapi)^{r-1}\p_3(\nabpp\cdot v_0^+) + \sigma\TL\sss_{2r}(v_0^+,b_0^+,q_0^+)
\end{align}where the leading-order terms in $\sss_m$ are
\begin{align}
\sss_{2r-1}\eql (\cbi)^{r-2}(\bc_0\cdot\cnab)^2(\lapi)^{r-2}\p_3q_0^+,\quad \sss_{2r}\eql -(\cbi)^{r-1}(\bc_0\cdot\cnab)^2(\lapi)^{r-2}\p_3(\nabpp\cdot v_0^+).
\end{align} Thus, the compatibility conditions for the original current-vortex sheets system \eqref{CMHD0} are reformulated as
\begin{align}
\label{ccd odd} m=2r+1,&~~-\jump{\cbi^{r+1}(\lapi)^r(\nabpp\cdot v_0)}\sim\sum_{j=0}^{r}\jump{(\lapi)^j(\MH_{2r-1-2j}(v_0,b_0,q_0)+\NH_{2r-1-2j}(v_0,b_0,q_0))}\\
\no&+\TT_{\jump{\vb}}^{2r+1}q^-|_{t=0}-\sigma\cbi^{r}\TL(\lapi)^{r-1}\p_3(\nabpp\cdot v_0^+) + \sigma\TL\sss_{2r}(v_0^+,b_0^+,q_0^+)\quad\text{ on }\Sigma,\\
\label{ccd even} m=2r,&~~\jump{\cbi^r(\lapi)^r q_0}\sim\sum_{j=0}^{r-1}\jump{(\lapi)^j(\MH_{2r-2-2j}(v_0,b_0,q_0)+\NH_{2r-2-2j}(v_0,b_0,q_0))}\\
&\no+\TT_{\jump{\vb}}^{2r}q^-|_{t=0}+\sigma\cbi^{r-1}\TL(\lapi)^{r-1}\p_3 q_0^+ + \sigma\TL\sss_{2r-1}(v_0^+,b_0^+,q_0^+))\quad\text{ on }\Sigma.
\end{align}Note that the time-differentiated kinematic boundary condition is already implicitly used when deriving the above compatibility conditions. Similarly, the compatibility conditions for the $\kk$-approximate problem \eqref{CMHD0kk} are reformulated as
\begin{align}
\label{ccdkk odd} m=2r+1,&~~-\jump{\cbi^{r+1}(\lapi)^r(\nabpp\cdot v_0^\kk)}\sim\sum_{j=0}^{r}\jump{(\lapi)^j(\MH_{2r-1-2j}(v_0^\kk,b_0^\kk,q_0^\kk)+\NH_{2r-1-2j}(v_0^\kk,b_0^\kk,q_0^\kk))}\no\\
&+\TT_{\jump{\vb}}^{2r+1}q^-|_{t=0}-\sigma\cbi^{r}\TL(\lapi)^{r-1}\p_3(\nabpp\cdot v_0^{\kk,+})+\kk\cbi^{r}\TL^2(\lapi)^{r-1}\p_3(\nabpp\cdot v_0^{\kk,+})+\kk\cbi^{r}\TL(\lapi)^{r}\p_3 q_0^+\no\\
&+ (\sigma\TL-\kk\TL^2)\sss_{2r}(v_0^{\kk,+},b_0^{\kk,+},q_0^{\kk,+})+\kk\TL\sss_{2r+1}(v_0^{\kk,+},b_0^{\kk,+},q_0^{\kk,+})\quad\text{ on }\Sigma,\\
\label{ccdkk even} m=2r,&~~\jump{\cbi^r(\lapi)^r q_0^\kk}\sim\sum_{j=0}^{r-1}\jump{(\lapi)^j(\MH_{2r-2-2j}(v_0^\kk,b_0^\kk,q_0^\kk)+\NH_{2r-2-2j}(v_0^\kk,b_0^\kk,q_0^\kk))} \no\\
&+\TT_{\jump{\vb}}^{2r}q^-|_{t=0}+\sigma\cbi^{r-1}\TL(\lapi)^{r-1}\p_3 q_0^+ - \kk\cbi^{r-1}\TL^2(\lapi)^{r-1}\p_3 q_0^+ - \kk\cbi^{r}\TL(\lapi)^{r-1}\p_3(\nabpp\cdot v_0^+) \no\\
&+(\sigma\TL-\kk\TL^2)\sss_{2r-1}(v_0^{\kk,+},b_0^{\kk,+},q_0^{\kk,+})+\kk\TL\sss_{2r}(v_0^{\kk,+},b_0^{\kk,+},q_0^{\kk,+})\quad\text{ on }\Sigma.
\end{align}

\subsection{Construction of the converging initial data}
Given initial data $(v_0^\pm,b_0^\pm,q_0^\pm,S_0^\pm,\psi_0)$ of \eqref{CMHD0} satisfying the compatibility conditions \eqref{ccd odd}-\eqref{ccd even} up to 7-th order, we now construct the initial data $(v_0^{\kk,\pm},b_0^{\kk,\pm},q_0^{\kk,\pm},S_0^{\kk,\pm},\psi_0^{\kk})$ to \eqref{CMHD0kk} satisfying the compatibility conditions \eqref{ccdkk odd}-\eqref{ccdkk even} up to 7-th order that converge to the given data as $\kk\to 0_+$. To do this, we just need to \textbf{equally distribute the $\kk$-term to the solution in $\Om^+$ and the solution in $\Om^-$}.

\subsubsection{Recover the 0-th order and the 1-st order compatibility conditions}

First, we pick $b_0^{\kk,\pm}=b_0^\pm$, $\psi_0^{\kk}=\psi_0$. We define $\p_t\psi^\kk|_{t=0}:=v_0^\pm\cdot N_0$ and $\p_t b^\pm|_{t=0}=(b_0^\pm\cdot\nabpp)v_0^\pm-b_0^\pm(\nabpp\cdot v_0^\pm)$ in $\Om^\pm$. Then the constraints for the magnetic field are automatically satisfied. Now, we construct $q_0^{(0)}$ such that $(v_0^\pm,b_0^\pm,q_0^{(0),\pm},\psi_0)$ satisfies the 0-th order compatibility condition \eqref{ccdkk even}. The function $q_0^{(1),\pm}$ is set to be the solution to the poly-harmonic equation
\begin{equation}
\begin{cases}
\lap^2 q_0^{(0),\pm} = \lap^2 q_0^\pm&~~~\text{ in }\Om^\pm\\
q_0^{(0),\pm}=q_0^\pm\mp\frac12\kk\TL^2\psi_0\pm\frac12\kk\TL(v_0^\pm\cdot N_0)&~~~\text{ on }\Sigma\\
\p_3 q_0^{(0),\pm}=\p_3 q_0^{\pm}&~~~\text{ on }\Sigma\\
\p_3^j q_0^{(0),\pm}=\p_3^j q_0^{\pm},~~0\leq j\leq 1&~~~\text{ on }\Sigma^\pm.
\end{cases}
\end{equation} Then for $s\geq 4$, we have
\[
\|q_0^{(0),\pm}-q_0^{\pm}\|_{s,\pm}\lesssim \kk |\TL^2\psi_0| _{s-0.5}+\kk|\TL(v_0^\pm\cdot N_0)|_{s-0.5}\to 0\quad\text{ as }\kk\to 0.
\]
With this $q_0^{(0)}$, we define $\p_t^2\psi|_{t=0}=\p_t(v^\pm\cdot N)|_{t=0}$ via $(v_0^\pm,b_0^\pm,q_0^{(0),\pm},\psi_0)$ on $\Sigma$. (Note that $\p_t v\cdot N|_{t=0}$ already includes $\p_3 q_0$. Only when we have $\p_3 q_0^{(0),\pm}=\p_3 q_0^{\pm}$ on $\Sigma$ can we keep the jump condition $\jump{\p_t(v\cdot N)}=0$.)  and also define the corresponding $\p_t^2 b|_{t=0}$ in $\Om^\pm$ via the evolution equation of $b$. Thus, the $\p_t$-differentiated boundary constraint for $b\cdot N$ is also satisfied.

Now we introduce $v_0^{(0),\pm}$ such that $(v_0^{(0),\pm},b_0^\pm,q_0^{(0),\pm},\psi_0)$ satisfies the 1-st order compatibility condition \eqref{ccdkk odd}. We define $\vb_{0i}^{(0),\pm}=\vb_{0i}^\pm$ for $i=1,2$ and define $v_{03}^{(0),\pm}$ via the following poly-harmonic equation
\begin{equation}
\begin{cases}
\lap^3 v_{03}^{(0),\pm} = \lap^3 v_{03}^\pm&~~~\text{ in }\Om^\pm\\
\cbi (\nabpp\cdot v_{03}^{(0),\pm})=(\nabpp\cdot\cbi v_{03}^{\pm}) \mp \frac12(\jump{\vb_0}\cdot\cnab)(q_0^{(0),\pm}-q_0^\pm)\mp\frac{\kk}{2}\TL^2 v_{03}^+  \pm \frac{\kk}{2}\TL\p_3 q_0^{(0),+}
&~~~\text{ on }\Sigma\\
v_{03}^{(0),\pm}= v_{03}^{\pm},\quad \p_3^2 v_{03}^{(0),\pm}=\p_3^2 v_{03}^{\pm}&~~~\text{ on }\Sigma\\
\p_3^j v_{03}^{(0),\pm}=\p_3^j v_{03}^{\pm},~~0\leq j\leq 2&~~~\text{ on }\Sigma^\pm.
\end{cases}
\end{equation}
It is also straightforward to see the convergece for $s\geq 6$ 
\[
\|v_0^{(0),\pm}-v_0^{\pm}\|_{s,\pm}\lesssim |q_0^{(0),\pm}-q_0^\pm|_{s-0.5} + \kk(|v_{03}^+|_{s+2.5}+|\p_3 q_0^+|_{s+0.5}).
\]

\subsubsection{Higher-order compatibility conditions}
For $r\geq 1$, we can inductively define $q_0^{(r),\pm}$ such that $(v_0^{(r-1),\pm},b_0^\pm,q_0^{(r),\pm},\psi_0)$ satisfies the compatibility condition up to $2r$-th order
\begin{equation}
\begin{cases}
\lap^{2r+2} q_0^{(r),\pm} = \lap^{2r+2} q_0^{(r-1),\pm}&~~~\text{ in }\Om^\pm\\
\cbi^r (\lapi)^r q_0^{(r),\pm}=\cbi^r (\lapi)^r q_0^{(r-1),\pm}\\
\q+\sum\limits_{j=0}^{r-1}(\lapi)^j\left((\MH_{2r-2-2j}+\NH_{2r-2-2j})(v_0^{(r-1),\pm},b_0^\pm,q_0^{(r),\pm})- (\MH_{2r-2-2j}+\NH_{2r-2-2j})(v_0^{(r-2),\pm},b_0^\pm,q_0^{(r-1),\pm})\right)\\
\q\pm\frac12\bigg((\TT^{2r}_{\jump{\vb^{(r-1)}}} q^{(r),-}-\TT^{2r}_{\jump{\vb^{(r-2)}}}q^{(r-1),-})+\sigma\cbi^{r-1}\underbrace{\TL(\lapi)^{r-1}\p_3(q_0^{(r),+}-q_0^{(r-1),+})}_{=0}\\
\q\q+\sigma\TL\big(\sss_{2r-1}(v_0^{(r-1),+},b_0^+,q_0^{(r),+})-\sss_{2r-1}(v_0^{(r-2),+},b_0^+,q_0^{(r-1),+})\big)\bigg)\\
\q\mp\frac{\kk}{2}\left(\cbi^{r-1}\underbrace{\TL^2(\lapi)^{r-1}\p_3 (q_0^{(r),+}-q_0^{(r-1),+})}_{=0}-\TL\cbi^r(\lapi)^{r-1}\p_3\nabpp\cdot(v_0^{(r-1),+}-v_0^{(r-2),+})\right)\\
\q\mp\frac{\kk}{2}\left((\TL^2\sss_{2r-1}-\TL\sss_{2r})(v_0^{(r-1),+},b_0^+,q_0^{(r),+})-(\TL^2\sss_{2r-1}-\TL\sss_{2r})(v_0^{(r-2),+},b_0^+,q_0^{(r-1),+})\right)&~~~\text{ on }\Sigma\\
\p_3^j q_0^{(r),\pm}=\p_3^j q_0^{(r-1),\pm},~~0\leq j\leq 2r+1, j\neq 2r&~~~\text{ on }\Sigma\\
\p_3^j q_0^{(r),\pm}=\p_3^j q_0^{(r-1),\pm},~~0\leq j\leq 2r+1&~~~\text{ on }\Sigma^\pm,
\end{cases}
\end{equation} and define $\vb_0^{(r),\pm}=\vb_0^{(r-1),\pm}$ and $v_{03}^{(r),\pm}$ such that $(v_0^{(r),\pm},b_0^\pm,q_0^{(r),\pm},\psi_0)$ satisfies the compatibility condition up to $(2r+1)$-th order
\begin{equation}
\begin{cases}
\lap^{2r+3} v_{03}^{(r),\pm} = \lap^{2r+3} v_{03}^{(r-1),\pm}&~~~\text{ in }\Om^\pm\\
-\cbi^r(\lapi)^r (\nabpp\cdot v_0^{(r),\pm})=-\cbi^r(\lapi)^r (\nabpp\cdot v_0^{(r-1),\pm})\\
\q+\sum\limits_{j=0}^{r}(\lapi)^j\left((\MH_{2r-1-2j}+\NH_{2r-1-2j})(v_0^{(r),\pm},b_0^\pm,q_0^{(r),\pm})- (\MH_{2r-1-2j}+\NH_{2r-1-2j})(v_0^{(r-1),\pm},b_0^\pm,q_0^{(r-1),\pm})\right)\\
\q\pm\frac12\bigg((\TT^{2r+1}_{\jump{\vb^{(r)}}} q^{(r),-}-\TT^{2r}_{\jump{\vb^{(r-1)}}}q^{(r-1),-})-\sigma\cbi^{r-1}\underbrace{\TL(\lapi)^{r-1}\p_3\nabpp\cdot( v_{03}^{(r),+}- v_{03}^{(r-1),+})}_{=0}\\
\q\q+\sigma\TL\big(\sss_{2r}(v_0^{(r),+},b_0^+,q_0^{(r),+})-\sss_{2r}(v_0^{(r-1),+},b_0^+,q_0^{(r-1),+})\big)\bigg)\\
\q\pm\frac{\kk}{2}\left(\TL^2\cbi^{r-1}\underbrace{(\lapi)^{r-1}\p_3 \nabpp\cdot(v_0^{(r),+}-v_0^{(r-1),+})}_{=0}-\TL\cbi^r(\lapi)^{r-1}\p_3(q_0^{(r),+}-q_0^{(r-1),+})\right)\\
\q\mp\frac{\kk}{2}\left((\TL^2\sss_{2r}-\TL\sss_{2r+1})(v_0^{(r),+},b_0^+,q_0^{(r),+})-(\TL^2\sss_{2r}-\TL\sss_{2r+1})(v_0^{(r-1),+},b_0^+,q_0^{(r-1),+})\right)&~~~\text{ on }\Sigma\\
\p_3^j v_{03}^{(r),\pm}=\p_3^j v_{03}^{(r-1),\pm},~~0\leq j\leq 2r+2, j\neq 2r+1&~~~\text{ on }\Sigma\\
\p_3^j v_{03}^{(r),\pm}=\p_3^j v_{03}^{(r-1),\pm},~~0\leq j\leq 2r+2&~~~\text{ on }\Sigma^\pm.
\end{cases}
\end{equation} 
Since we require the compatibility conditions up to 7-th order, we can stop at $r=3$ and define $(v_0^{\kk,\pm},b_0^{\kk,\pm},q_0^{\kk,\pm},S_0^{\kk,\pm},\psi_0^{\kk})$ to be $(v_0^{(3),\pm},b_0^{\pm},q_0^{(3),\pm},S_0^{\pm},\psi_0)$. It is also straightforward to see the convergence after long and tedious calculations: For $s\geq 2\times(2r+3)=18$, we have the convergence as $\kk\to 0$
\begin{align*}
\ino{(v_0^{\kk,\pm},q_0^{\kk,\pm})-(v_0^{\pm},q_0^{\pm})}_{s,\pm}\lesssim&~ P(\|v_0^\pm,b_0^\pm,q_0^\pm,S_0^\pm\|_{s+1,\pm})\left(\kk|\psi_0|_{s+3.5}+\sum_{j=0}^{r}\kk|(\lapi)^{j}v_0^\pm|_{s+1.5-2j}+\kk|(\lapi)^{(j-1)_+}\p_3 q_0^\pm|_{s+0.5-2j}\right)\to 0,
\end{align*}provided that the given initial data is sufficiently regular. Specifically, picking $s=18$, the given data is required to satisfy $\|(v_0^{\pm},b_0^{\pm},q_0^{\pm},S_0^{\pm})\|_{20,\pm}+|\psi_0|_{21.5}<+\infty$. We may assume the given data belongs to $C^{\infty}$-class for convenience.
\end{appendix}

\paragraph*{Data avaliability.} Data sharing is not applicable as no datasets were generated or analyzed during the current study.

\subsection*{Ethics Declarations}
\paragraph*{Conflict of interest.} On behalf of all authors, the corresponding author states that there is no conflict of interest.


\begin{thebibliography}{99}
\addcontentsline{toc}{section}{References}
\setlength{\itemsep}{0.5ex}
\begin{spacing}{0.9}
\bibitem{Alinhac1989good}
Alinhac, S. 
\newblock Existence d'ondes de rar\'efaction pour des syst\`emes quasi-lin\'eaires hyperboliques multidimensionnels.(French. English summary) [Existence of rarefaction waves for multidimensional hyperbolic quasilinear systems].
\newblock {\em Commun. Partial Differ. Equ.}, 14(2), 173-230, 1989.

\bibitem{AM2007VS}
Ambrose, D., Masmoudi, N.
\newblock Well-posedness of 3D vortex sheets with surface tension.
\newblock  {\em Commun. Math. Sci.}, 5(2), 391-430, 2007.

\bibitem{Axford2D}
Axford, W. I.
\newblock Note on a problem of magnetohydrodynamic stability.
\newblock  {\em Canad. J. Phys.}, 40(5), 654-655, 1962.

\bibitem{CVSphy}
Baranov, V. B., Krasnobaev, K. V., Kulikovsky, A. G.
\newblock  A model of interaction of the solar wind with the interstellar medium.
\newblock {\em Sov. Phys. Dokl.}, Vol. 15, 791–793 (1971).

\bibitem{Blokhinshock}
Blokhin, A.
\newblock Estimation of the energy integral of a mixed problem for gas dynamics equations with boundary conditions on the shock wave. 
\newblock {\em Siberian Math. J.}, 22, 501–523 (1981).

\bibitem{ChenWangCMHDVS}
Chen, G.-Q., Wang, Y.-G.
\newblock Existence and Stability of Compressible Current-Vortex Sheets in Three-Dimensional Magnetohydrodynamics.
\newblock {\em Arch. Rational Mech. Anal.}, 187(3), 369-408, 2008.

\bibitem{ChenSX}
Chen, S.-X.
\newblock Initial boundary value problems for quasilinear symmetric hyperbolic systems
with characteristic boundary.
\newblock {\em Translated from Chin. Ann. Math.} 3(2), 222–232 (1982). {\em Front. Math. China}, 2(1), 87–102 (2007).

\bibitem{CCS2007}
Cheng, C.-H. A., Coutand, D., Shkoller, S.
\newblock On the Motion of Vortex Sheets with Surface Tension in Three-Dimensional Euler Equations with Vorticity.
\newblock  {\em Commun. Pure Appl. Math.}, 61(12), 1715-1752, 2007.

\bibitem{CL2000}
Christodoulou, D., Lindblad, H. 
\newblock On the motion of the free surface of a liquid.
\newblock  {\em Commun. Pure Appl. Math.}, 53(12), 1536--1602, 2000.

\bibitem{CMST2012MHDVS}
Coulombel, J.-F., Morando, A., Secchi, P., Trbeschi, P.
\newblock A priori Estimates for 3D Incompressible Current-Vortex Sheets.
\newblock  {\em Commun. Math. Phys.}, 311(1), 247-275, 2012.

\bibitem{Secchi2004CVS}
Coulombel, J.-F., Secchi, P.
\newblock The stability of compressible vortex sheets in two space dimensions.
\newblock {\em Indiana Univ. Math. J.}, 53(4), 941-1012, 2004.

\bibitem{Secchi2008CVS}
Coulombel, J.-F., Secchi, P.
\newblock Nonlinear compressible vortex sheets in two space dimensions.
\newblock {\em Ann. Sci. de l'\'Ecole Norm.(4)}, 41(1), 85-139, 2008.


\bibitem{DL2019limit}
Disconzi, M. M., Luo, C.
\newblock On the incompressible limit for the compressible free-boundary Euler equations with surface tension in the case of a liquid.
\newblock {\em Arch. Rational Mech. Anal.}, 237(2), 829-897, 2020.

\bibitem{Ebin1987}
Ebin, D. G.
\newblock The equations of motion of a perfect fluid with free boundary are not well posed.
\newblock  {\em Commun. Partial Differ. Equ.}, 12(10), 1175--1201, 1987.

\bibitem{Ebin1988}
Ebin, D. G.
\newblock Ill-posedness of the Rayleigh–Taylor and Helmholtz problems for incompressible fluids.
\newblock {\em Commum. Partial Differ. Equ.}, 13(10), 1265-1295, 1988.

\bibitem{Evans}
Evans, L. C.
\newblock Partial Differential Equations, 2nd edition.
\newblock Graduate Studies in Mathematics 19, AMS, 2010.

\bibitem{Miles1}
Fejer, J. A., Miles, J. W. 
\newblock  On the stability of a plane vortex sheet with respect to threedimensional disturbances.
\newblock {\em J. Fluid Mech.}, 15, 335–336 (1963).

\bibitem{MHDphy}
Goedbloed, H., Keppens, R., Poedts, S.
\newblock {\em Magnetohydrodynamics of Laboratory and Astrophysical plasmas.}
\newblock  Cambridge University Press, 2020.

\bibitem{Guaxi1}
Gu, X.
\newblock Well-posedness of axially symmetric incompressible ideal magnetohydrodynamic equations with vacuum under the non-collinearity condition.
\newblock  {\em Commun. Pure \& Appl. Anal.}, 18(2), 569-602, 2019.

\bibitem{Guaxi2}
Gu, X.
\newblock Well-posedness of axially symmetric incompressible ideal magnetohydrodynamic equations with vacuum under the Rayleigh-Taylor sign condition.
\newblock  arXiv: 1712.02152, preprint, 2017.

\bibitem{GuLuoZhang2021MHD0ST}
Gu, X., Luo, C., Zhang, J.
\newblock Zero Surface Tension Limit of the Free-Boundary Problem in Incompressible Magnetohydrodynamics.
\newblock {\em Nonlinearity}, 35(12), 6349-6398, 2022.

\bibitem{GuLuoZhang2021MHDST}
Gu, X., Luo, C., Zhang, J.
\newblock Local Well-posedness of the Free-Boundary Incompressible Magnetohydrodynamics with Surface Tension.
\newblock {\em J. Math. Pures Appl.}, Vol. 182, 31-115, 2024.

\bibitem{GuWang2016LWP}
Gu, X., Wang, Y.
\newblock On the construction of solutions to the free-surface incompressible ideal magnetohydrodynamic equations.
\newblock  {\em J. Math. Pures Appl.}, Vol. 128: 1-41, 2019.

\bibitem{GuWang2023LWP}
Gu, X., Wang, Y.
\newblock Well-posedness and Low Mach Number Limit of the Free Boundary Problem for the Euler-Fourier System.
\newblock arXiv:2308.15332, preprint.


\bibitem{HaoLuo2014priori}
Hao, C., Luo, T.
\newblock A priori estimates for free boundary problem of incompressible inviscid magnetohydrodynamic flows.
\newblock {\em Arch. Rational Mech. Anal.}, 212(3), 805--847, 2014.

\bibitem{HaoLuo2018ill}
Hao, C., Luo, T.
\newblock Ill-posedness of free boundary problem of the incompressible ideal MHD.
\newblock  {\em Commun. Math. Phys.}, 376(1), 259-286, 2020.

\bibitem{HaoYang2023MHDST}
Hao, C., Yang, S.
\newblock On the motion of the closed free surface in three-dimensional incompressible ideal MHD with surface tension.
\newblock arXiv:2312.09473, preprint.

\bibitem{KatoPonce1988}
Kato, T., Ponce, G.
\newblock Commutator estimates and the Euler and Navier-Stokes equations.
\newblock {\em Commun. Pure  Appl. Math.}, 41(7), 891-907, 1988.

\bibitem{lopatinskii}
Kreiss, H.-O.
\newblock Initial boundary value problems for hyperbolic systems.
\newblock {\em Commun. Pure  Appl. Math.}, 23(3), 277-298, 1970.

\bibitem{Lax1960LWP}
Lax, P. D., Phillips, R. S. 
\newblock Local boundary conditions for dissipative symmetric linear differential operators.
\newblock {\em Commun. Pure. Appl. Math.}, 13(3), 427--455, 1960.

\bibitem{sb}
Li, C., Li, H.
\newblock Well-posedness of the free-boundary problem in incompressible MHD with surface tension.
\newblock {\em Calc. Var. \& PDE}, 61: 191, 2022.

\bibitem{LL2018priori}
Lindblad, H., Luo, C.
\newblock A priori estimates for the compressible Euler equations for a liquid with free surface boundary and the incompressible limit.
\newblock {\em Commun. Pure Appl. Math.}, 71(7), 1273-1333, 2018.

\bibitem{Zhang2021CMHD}
Lindblad, H., Zhang, J.
\newblock Anisotropic Regularity of the Free-Boundary Problem in Compressible Ideal Magnetohydrodynamics.
\newblock {\em Arch. Rational Mech. Anal.}, 247(5), no. 89: 94 pp., 2023.

\bibitem{LiuXin2023MHDVS}
Liu, S., Xin, Z.
\newblock Local Well-posedness of the Incompressible Current-Vortex Sheet Problems.
\newblock arXiv:2309.03534, preprint, 2023.

\bibitem{LiuXin2023MHDPV}
Liu, S., Xin, Z.
\newblock On the Free Boundary Problems for the Ideal Incompressible MHD Equations.
\newblock arXiv:2311.06581, preprint, 2023.

\bibitem{Luo2018CWW}
Luo, C. 
\newblock On the Motion of a Compressible Gravity Water Wave with Vorticity.
\newblock {\em Ann. PDE}, 4(2), 2506-2576, 2018.

\bibitem{LuoZhang2019MHDST}
Luo, C., Zhang, J.
\newblock A priori estimates for the incompressible free-boundary magnetohydrodynamics equations with surface tension.
\newblock {\em SIAM J. Math. Anal.}, 53(2), 2595-2630, 2021.

\bibitem{LuoZhang2020CWW}
Luo, C., Zhang, J. 
\newblock Local well-posedness for the motion of a compressible gravity water wave with vorticity.
\newblock {\em J. Differ. Equ.}, Vol. 332, 333-403, 2022.

\bibitem{LuoZhang2022CWWST}
Luo, C., Zhang, J. 
\newblock Compressible Gravity-Capillary Water Waves: Local Well-Posedness, Incompressible and Zero-Surface-Tension Limits.
\newblock arXiv: 2211.03600, preprint.

\bibitem{Majdashock1}
Majda, A.
\newblock The stability of multi-dimensional shock fronts.
\newblock {\em Mem. Am. Math. Soc.}, 41(275) (1983).

\bibitem{Majdashock2}
Majda, A.
\newblock The existence of multi-dimensional shock fronts.
\newblock  {\em Mem. Am. Math. Soc.}, 43(281) (1983).

\bibitem{Majdalimit}
Majda, A.
\newblock {\em Compressible Fluids Flow and Systems of Conservation Laws in Several Space Variables.}
\newblock Applied Mathematical Sciences, Vol. 53, Springer-Verlag New York, 1984.


\bibitem{Miles2}
Miles, J. W. 
\newblock   On the disturbed motion of a plane vortex sheet.
\newblock {\em J. Fluid Mech.}, 4, 538-552. (1958)

\bibitem{MHDSTphy2}
Molokov, S., Reed, C. B.
\newblock {\em Review of Free-Surface MHD Experiments and Modeling.}
\newblock United States: N. p., 2000. 

\bibitem{MSTY2023CMHDVS}
Morando, A., Secchi, P., Trebeschi, P., Yuan, D.
\newblock Nonlinear Stability and Existence of Two-Dimensional Compressible Current-Vortex Sheets.
\newblock  {\em Arch. Rational Mech. Anal.}, 247(3), No. 50, 1-83, 2023. 

\bibitem{MTT2018MHDCD}
Morando, A., Trakhinin, Y., Trebeschi, P.
\newblock Local existence of MHD contact discontinuities.
\newblock  {\em Arch. Rational Mech. Anal.}, 228(2), 691-742, 2018. 

\bibitem{OS1998MHDill}
Ohno, M., Shirota, T.
\newblock On the initial-boundary-value problem for the linearized equations of magnetohydrodynamics.
\newblock  {\em Arch. Rational Mech. Anal.}, 144(3), 259-299, 1998.

\bibitem{anisotropictrace}
Ohno, M., Shizuta, Y., Yanagisawa, T.
\newblock The trace theorem on anisotropic Sobolev spaces.
\newblock  {\em Tohoku Math. J.}, 46(3), 393-401, 1994.

\bibitem{Rauch1985}
Rauch, J.
\newblock {\em Symmetric Positive Systems with Boundary Characteristic of Constant Multiplicity.}
\newblock Trans. Amer. Math. Soc., 291(1), 167-187, 1985.

\bibitem{Secchi1995}
Secchi, P. 
\newblock Well-posedness for Mixed Problems for the Equations of Ideal Magneto-hydrodynamics.
\newblock  {\em Archiv. der. Math.}, 64(3), 237–245, 1995.

\bibitem{Secchi1996}
Secchi, P.
\newblock Well-posedness of characteristic symmetric hyperbolic systems.
\newblock  {\em Arch. Rational Mech. Anal.}, 134(2), 155-197, 1996.

\bibitem{Secchi1996-2}
Secchi, P.
\newblock The initial-boundary value problem for linear symmetric hyperbolic systems with characteristic boundary of constant multiplicity.
\newblock {\em Differential Integral Equations}, 9(4), 671-700, 1996.

\bibitem{Secchi2013CMHDLWP}
Secchi, P., Trakhinin, Y. 
\newblock Well-posedness of the plasma--vacuum interface problem.
\newblock  {\em Nonlinearity}, 27(1), 105-169, 2013.

\bibitem{SZ3}
Shatah, J., Zeng, C.
\newblock Local well-posedness for fluid interface problems.
\newblock  {\em Arch. Rational Mech. Anal.}, 199(2), 653-705, 2011.

\bibitem{Stevens2016CVS}
Stevens, B.
\newblock Short-time structural stability of compressible vortex sheets with surface tension.
\newblock {\em Arch. Rational Mech. Anal.}, 222(2), 603-730, 2016.

\bibitem{SWZ2015MHDLWP}
Sun, Y., Wang, W., Zhang, Z. 
\newblock Nonlinear Stability of the Current-Vortex Sheet to the Incompressible MHD Equations.
\newblock  {\em Commun. Pure Appl. Math.}, 71(2), 356-403, 2018.

\bibitem{SWZ2017MHDLWP}
Sun, Y., Wang, W., Zhang, Z.
\newblock Well-posedness of the Plasma-Vacuum Interface Problem for Ideal Incompressible MHD.
\newblock  {\em Arch. Rational Mech. Anal.}, 234(1), 81-113, 2019.

\bibitem{SyrovatskiiMHD}
Syrovatski\u{\i}, S. I.
\newblock The stability of tangential discontinuities in a magnetohydrodynamic medium. (In Russian)
\newblock {\em Z. Eksperim. Teoret. Fiz.}, 24, 622-629. (1953)

\bibitem{SyrovatskiiEuler}
Syrovatski\u{\i}, S. I.
\newblock Instability of tangential discontinuities in a compressible medium. (In Russian)
\newblock {\em Z. Eksperim. Teoret. Fiz.}, 27, 121-123. (1954)

\bibitem{Trakhinin2005CMHDVS}
Trakhinin, Y.
\newblock Existence of compressible current-vortex sheets: variable coefficients linear analysis.
\newblock {\em Arch. Rational Mech. Anal.}, 177(3), 331-366, 2005.

\bibitem{Trakhinin2008CMHDVS}
Trakhinin, Y.
\newblock The existence of current-vortex sheets in ideal compressible magnetohydrodynamics.
\newblock {\em Arch. Rational Mech. Anal.}, 191(2), 245-310, 2009.


\bibitem{TW2020MHDLWP}
Trakhinin, Y., Wang, T. 
\newblock Well-posedness of Free Boundary Problem in Non-relativistic and Relativistic Ideal Compressible Magnetohydrodynamics.
\newblock {\em Arch. Rational Mech. Anal.}, 239(2), 1131-1176, 2021.

\bibitem{TW2021MHDSTLWP}
Trakhinin, Y., Wang, T. 
\newblock Well-Posedness for the Free-Boundary Ideal Compressible Magnetohydrodynamic Equations with Surface Tension.
\newblock {\em Math. Ann.}, 383(1)-(2), 761-808.

\bibitem{TW2021MHDCDST}
Trakhinin, Y., Wang, T. 
\newblock Nonlinear Stability of MHD Contact Discontinuities with Surface Tension.
\newblock {\em Arch. Rational Mech. Anal.}, 243(2), 1091-1149, 2022.

\bibitem{TW2022MHDPVST}
Trakhinin, Y., Wang, T. 
\newblock Well-Posedness for Moving Interfaces with Surface Tension in Ideal Compressible MHD.
\newblock {\em SIAM J. Math. Anal.}, 54(6), 5888-5921, 2022.

\bibitem{WZ2023CMHDlimit}
Wang, J., Zhang, J.
\newblock Incompressible Limit of Compressible Ideal MHD Flows inside a Perfectly Conducting Wall.
\newblock {\em J. Differ. Equ.}, Vol.425, 846-894, 2025.

\bibitem{WangYu2013CMHDVS}
Wang, Y.-G., Yu, F.
\newblock Stabilization effect of magnetic fields on two-dimensional compressible current-vortex sheets.
\newblock {\em Arch. Rational Mech. Anal.}, 208(2), 341-389, 2013.

\bibitem{WangXinMHDCD}
Wang, Y., Xin, Z. 
\newblock Existence of Multi-dimensional Contact Discontinuities for the Ideal Compressible Magnetohydrodynamics.
\newblock  {\em Commun. Pure Appl. Math.}, 77(1), 583-629, 2024.

\bibitem{Wu2004VS}
Wu, S.
\newblock Mathematical analysis of vortex sheets.
\newblock {\em Commun. Pure Appl. Math.}, 59(8), 1065-1206, 2005.

\bibitem{1991MHDfirst}
Yanagisawa, T., Matsumura, A. 
\newblock The fixed boundary value problems for the equations of ideal magnetohydrodynamics with a perfectly conducting wall condition.
\newblock {\em Commun. Math. Phys.}, 136(1), 119-140, 1991.

\bibitem{Zhang2020CRMHD}
Zhang, J. 
\newblock Local Well-posedness of the Free-Boundary Problem in Compressible Resistive Magnetohydrodynamics.
\newblock  {\em Calc. Var. Partial Differ. Equ.}, 62(4), no.124: 60 pp., 2023.

\bibitem{Zhang2021elasto}
Zhang, J. 
\newblock Local Well-posedness and Incompressible Limit of the Free-Boundary Problem in Compressible Elastodynamics.
\newblock {\em Arch. Rational Mech. Anal.}, 244(3), 599-697, 2022.

\bibitem{Zhang2023CMHDVS2}
Zhang, J. 
\newblock On the Incompressible Limit of Current-Vortex Sheets with or without Surface Tension.
\newblock arXiv:2405.00421, preprint, 2024.
\end{spacing}
\end{thebibliography}
\end{document}